\documentclass[11pt]{article}
\usepackage[tbtags]{amsmath}
\usepackage{bbm}
\usepackage{amssymb}
\usepackage{amsthm}
\usepackage{stmaryrd}
\usepackage[misc]{ifsym}
\usepackage{graphicx}
\usepackage{tikz}
\usepackage{color}
\usepackage{float}
\usetikzlibrary{shapes,arrows}
\usepackage{cases}
\numberwithin{equation}{section}
\newtheorem{proposition}{Proposition}[section]
\newtheorem{definition}{Definition}[section]
\newtheorem{theorem}{Theorem}[section]
\newtheorem{lemma}{Lemma}[section]
\newtheorem{remark}{Remark}[section]

\normalsize
\usepackage{hyperref,amsmath,amssymb,amscd}

\title{\bf The Global Maximum Principle for Progressive Optimal Control of Partially Observed Forward- Backward Stochastic Systems with Random Jumps
\thanks{This work is supported by National Natural Science Foundations of China (Grant Nos. 11971266, 11831010, 11571205), and Shandong Provincial Natural Science Foundations (Grant Nos. ZR2020ZD24, ZR2019ZD42).}}
\author{\normalsize Yueyang Zheng\thanks{\it School of Mathematics, Shandong University, Jinan 250100, P.R.China, E-mail: zhengyueyang0106@163.com} , Jingtao Shi\thanks{Corresponding author. \it School of Mathematics, Shandong University, Jinan 250100, P.R.China, E-mail: shijingtao@sdu.edu.cn}}

\begin{document}
\maketitle

\noindent{\bf Abstract:}
In this paper, we study a partially observed progressive optimal control problem of forward-backward stochastic differential equations with random jumps, where the control domain is not necessarily convex, and the control variable enter into all the coefficients. In our model, the observation equation is not only driven by a Brownian motion but also a Poisson random measure, which also have correlated noises with the state equation. For preparation, we first derive the existence and uniqueness of the solutions to the fully coupled forward-backward stochastic system with random jumps and its estimation in $L^\beta(\beta\geq2)$-space under some assumptions, and the non-linear filtering equation of partially observed stochastic system with random jumps. Then we derive the partially observed global maximum principle with random jumps. To show its applications, a partially observed linear quadratic progressive optimal control problem with random jumps is investigated, by the maximum principle and stochastic filtering. State estimate feedback representation of the optimal control is given in a more explicit form by introducing some ordinary differential equations.

\vspace{1mm}


\vspace{1mm}


\section{Introduction}

In modern stochastic control theory, it is acknowledged that the main way to obtain the {\it global} maximum principle is the spike variational technique when the control domain is not necessarily convex. But different from the classical stochastic control problem, when the diffusion term contains the control variable, first-order expansion is not enough to formulate the variational equation for the state process $x^u$ which generally is the adapted solution to the following controlled {\it stochastic differential equation} (SDE for short):
\begin{equation*}\label{sde}
\left\{
\begin{aligned}
dx_t^u&=b(t,x_t^u,u_t)dt+\sigma(t,x_t^u,u_t)dW_t,\quad t\in[0,T],\\
x_0^u&=x,
\end{aligned}
\right.
\end{equation*}
where $W$ is a Brownian motion. The main reason is the It\^o's integral $\int_t^{t+\epsilon}\sigma(s,x_s^u,u_s)dW_s$ is only of order $O(\epsilon^{\frac{1}{2}})$. In 1990, Peng \cite{Peng90} first introduced the second-order term in the Taylor's expansion of the variation and completely obtained the global maximum principle for the stochastic control problem.

Consider the following controlled {\it backward stochastic differential equation} (BSDE for short), which is coupled with the above controlled SDE:
\begin{equation}\label{bsde}
\left\{
\begin{aligned}
-dy_t^u&=g(t,x_t^u,y_t^u,z_t^u,u_t)dt-z_t^udW_t,\quad t\in[0,T],\\
y_T^u&=\phi(x_T^u).
\end{aligned}
\right.
\end{equation}
Pardoux and Peng \cite{PP90} first obtained its existence and uniqueness result under the assumption that the generator $g$ is linear growth and Lipschitz continuous in $(y,z)$. Independently, Duffie and Epstein \cite{DE92} introduced the notion of recursive utilities in continuous time which is a type of BSDE where $g$ is independent of $z$. In El Karoui et al. \cite{KPQ97}, a general stochastic control problem with recursive utilities by the solutions to BSDEs, where the generator $g$ contains $z$, is studied. When the control domain is convex, the global maximum principle reduces to the so-called {\it local} maximum principle which is derived by the classical convex variation instead of the spike variation. Peng \cite{Peng93} first established a local maximum principle for the stochastic recursive optimal control problem, where the state equation of the problem is a controlled {\it forward-backward stochastic differential equation} (FBSDE for short). Wu \cite{Wu98} proved a local maximum principle for the optimal control problem of fully coupled FBSDEs. Wu \cite{Wu13} and Yong \cite{Yong10} obtained global maximum principles for optimal control problems of decoupled and fully coupled FBSDEs, respectively.

Next, consider the following {\it backward stochastic differential equation with random jumps} (BSDEP for short):
\begin{equation*}
\left\{
\begin{aligned}
-dy_t^u&=g(t,x_t^u,y_t^u,z_t^u,\tilde{z}^u_{(t,\cdot)},u_t)dt-z_t^udW_t-\int_\mathcal{E}\tilde{z}_{(t,e)}\tilde{N}(de,dt),\quad t\in[0,T],\\
y_T^u&=\phi(x_T^u),
\end{aligned}
\right.
\end{equation*}
where $\tilde{N}(\cdot,\cdot)$ is a Poisson random martingale measure. Under standard assumptions, Tang and Li \cite{TL94} first proved their existence and uniqueness results for general BSDEPs, which are derived as adjoint equations when they studied the global maximum principle for the controlled {\it stochastic differential equation with random jumps} (SDEP for short). Wu \cite{Wu99} studied the existence and uniqueness result of the solutions to fully coupled {\it forward-backward SDEP} (FBSDEP for short). Situ \cite{Situ91} first obtained the global maximum principle for the controlled SDEP where the control variable is not included in the jump coefficient. Tang and Li \cite{TL94} completely proved the global maximum principle where the control variable enters into both diffusion and jump coefficients. Shi and Wu \cite{SW10} extended the result of Peng \cite{Peng93} to discontinuous case, and proved a local maximum principle for a kind of controlled FBSDEPs. Recently, Song et al. \cite{STW20} found that, in the work of \cite{TL94}, the third estimate in (2.10) is not correct, which was also put forward by Situ \cite{Situ05} in the early, and the subsequent variational calculus then underestimated the influence of jumps. In \cite{STW20}, they overcome the deficiencies with a new spike variation technique and obtained a new global maximum principle for controlled SDEPs.

An inspiring thought appearing in the nonlinear expectation theory motivates us to study the optimal control problem of FBSDEP. We consider an SDEP as follows:
\begin{equation}\label{sdep}
\left\{
\begin{aligned}
 dx_t^u&=b(t,x_t^u,u_t)dt+\sigma(t,x_t^u,u_t)dW_t+\int_{\mathcal{E}}f(t,x_{t-}^u,u_t,e)\tilde{N}(de,dt),\quad t\in[0,T],\\
  x_0^u&=x_0,
\end{aligned}
\right.
\end{equation}
and the cost functional can be formulated as:
\begin{equation}\label{g_p-expectation}
J_{g^p}(u)=\mathcal{E}_{g^p}\bigg[\int_0^Tl(t,x_t^u,u_t)dt+\Phi(x_T^u)\bigg],
\end{equation}
where $\mathcal{E}_{g^p}[\cdot]$ is a nonlinear expectation called {\it $g^p$-expectation}, which is related to the BSDEP:
\begin{equation}
\left\{
\begin{aligned}
-d\eta_t&=g(t,\zeta_t,\lambda_{(t,\cdot)})dt-\zeta_tdW_t-\int_{\mathcal{E}}\lambda_{(t,e)}\tilde{N}(de,dt),\quad t\in[0,T],\\
\eta_T&=\xi.
\end{aligned}
\right.
\end{equation}
If we set $g(t,0,0)\equiv0$, then we can define a $g^p$-expectation in the following.
\begin{definition}\label{definition11}
For any $\xi\in L^2_{\mathcal{F}_T}$, we set $\mathcal{E}_{g^p}[\xi]=\eta_0$.
\end{definition}
It is known that in 1997, Peng \cite{Peng1997} introduced the notion of $g$-expectation, for a random variable $\xi$, it is defined as the initial value of a classical BSDE driven by a generator $g$ and with a terminal value $\xi$. More precisely, when we consider the following BSDE without jumps:
\begin{equation}
\left\{
\begin{aligned}
-dy_t&=g(t,y_t,z_t)dt-z_tdW_t,\quad t\in[0,T],\\
y_T&=\xi,
\end{aligned}
\right.
\end{equation}
then we denote $g$-expectation by the operator $\mathcal{E}_g[\cdot]$ and define $\mathcal{E}_g[\xi]=y_0$. More properties about $g$-expectation can be seen in \cite{Peng04}. In 2006, Royer \cite{Royer06} generalized the notion of $g$-expectation and introduced $f$-expectation related to some BSDEP driven by a generator $f$ and with terminal value $\xi$, under which the filtration is generated by both Brownian motion and Poisson random measure. Under some conditions, $f$-expectation, like $g$-expectation, preserve most properties of the classical expectation except the linearity. Here, in (\ref{g_p-expectation}) of this paper, we put a superscript $p$ on $g$, i.e. $g^p$, to represent some $f$-expectation with jumps which is distinct form the $g$-expectation without jumps. Moreover, $\mathcal{E}_{g^p}[\cdot]$ and $\mathcal{E}_{g}[\cdot]$ can be reduced to classical expectation $E[\cdot]$ provided that $g=0$.

Then the aim is to find an optimal control $\bar{u}\in\mathcal{U}_{ad}$ such that
\begin{equation}
J_{g^p}(\bar{u})=\inf_{u\in\,\mathcal{U}_{ad}}J_{g^p}(u),
\end{equation}
where $\mathcal{U}_{ad}$ is a certain admissible control set. Next, we consider that $\xi=\int_0^Tl(t,x_t^u,u_t)dt+\Phi(x_T^u)$, then we get
\begin{equation}
\eta_T=\int_0^Tl(t,x_t^u,u_t)dt+\Phi(x_T^u),
\end{equation}
and define
\begin{equation}
y^u_t=\eta_t-\int_0^tl(s,x_s^u,u_s)ds,\quad z^u_t=\zeta^u_t,\quad \tilde{z}^u_{(t,\cdot)}=\lambda^u_{(t,\cdot)},
\end{equation}
where $(y^u,z^u,\tilde{z}^u)$ satisfies the following BSDEP coupled with \eqref{sdep}:
\begin{equation}
\left\{
\begin{aligned}
-dy^u_t&=\big[g(t,z_t^u,\tilde{z}_{(t,\cdot)}^u)+l(t,x_t^u,u_t)\big]dt-z^u_tdW_t-\int_{\mathcal{E}}\tilde{z}^u_{(t,e)}\tilde{N}(de,dt),\quad t\in[0,T],\\
  y^u_T&=\Phi(x_T^u).
\end{aligned}
\right.
\end{equation}
According to the discussion above, we can transform the original optimal control problem with $g^p$-expectation into a problem of FBSDEP with classical expectation. In this case, the state equation becomes
\begin{equation}\label{sfbsdep}
\left\{
\begin{aligned}
 dx_t^u&=b(t,x_t^u,u_t)dt+\sigma(t,x_t^u,u_t)dW_t+\int_{\mathcal{E}}f(t,x_{t-}^u,u_t,e)\tilde{N}(de,dt),\\
-dy^u_t&=\big[g(t,z_t^u,\tilde{z}_{(t,\cdot)}^u)+l(t,x_t^u,u_t)\big]dt-z^u_tdW_t-\int_{\mathcal{E}}\tilde{z}^u_{(t,e)}\tilde{N}(de,dt),\quad t\in[0,T],\\
  x_0^u&=x_0,\ y^u_T=\Phi(x_T^u),
\end{aligned}
\right.
\end{equation}
with the cost functional
\begin{equation}\label{cfg}
J_{g^p}(u)=y_0^u,
\end{equation}
and the aim is transformed to minimize the cost functional \eqref{cfg} subject to \eqref{sfbsdep}, which can be regarded as a special recursive optimal control problem.

In the literature, when the control domain in nonconvex, it is difficult to derive the first-order and second-order Taylor's expansions for BSDE \eqref{bsde}, which is proposed as an open problem in Peng \cite{Peng98}. Recently, Hu \cite{Hu17} solved this open problem and obtained a novel global maximum principle. In \cite{Hu17}, the relationship among the terms of the first-order Taylor expansions, i.e.,
\begin{equation*}
\begin{aligned}
y^1_t&=p_tx^1_t,\ z^1_t&=p_t\delta\sigma(t)\mathbb{I}_{E_{\epsilon}}(t)+[\sigma_x(t)p_t+q_t]x^1_t,
\end{aligned}
\end{equation*}
is built, where $(p,q)$ is the solution to the first-order adjoint equation which possesses a linear generator. More importantly, a novel feature is that the variation of $z_t$ contains an additional term $p_t\delta\sigma(t)\mathbb{I}_{E_{\epsilon}}(t)$. Another novel idea is to apply Taylor's expansion to $\bar{z}_t+p_t\delta\sigma(t)\mathbb{I}_{E_{\epsilon}}(t)$ and then to derive the global maximum principle. Motivated by these ideas, Hu et al. \cite{HJX18} further extended the global maximum principle to the fully coupled controlled FBSDE. In \cite{HJX18}, they decoupled the fully coupled linear first-order and second-order forward-backward variation equations by establishing the following relationships:
\begin{equation*}
\begin{aligned}
y^1_t&=p_tx^1_t,\ z^1_t&=\Delta(t)\mathbb{I}_{E_{\epsilon}}(t)+k_1(t)x^1_t,
\end{aligned}
\end{equation*}
where $k_1(t)=(1-p_t\sigma_z(t))^{-1}[\sigma_x(t)p_t+\sigma_y(t)p_t^2+q_t]$ and $\Delta(t)$ satisfies an algebra equation
\begin{equation*}
\begin{aligned}
\Delta(t)=p_t\big[\sigma(t,\bar{x}_t,\bar{y}_t,\bar{z}_t+\Delta(t),u_t)-\sigma(t,\bar{x}_t,\bar{y}_t,\bar{z}_t,\bar{u}_t)\big],
\end{aligned}
\end{equation*}
with $(p,q)$ is the solution to the adjoint equation with a quadratic generator. They also applied Taylor's expansion to $\bar{z}_t+\Delta(t)\mathbb{I}_{E_{\epsilon}}(t)$ and derived the global maximum principle by assuming $q$ is bounded and unbounded, respectively.

However, in the above, the controlled systems are all considered under complete information. In practice, the state equation usually could not be observed directly. There are two kinds of incomplete information forms. The first one is that one could only observe a related stochastic process $Y$ which usually satisfies an SDE:
\begin{equation*}
\left\{
\begin{aligned}
dY_t&=h(t,x_t,u_t)dt+\sigma(t)dW_t,\quad t\in[0,T],\\
 Y_0&=0.
\end{aligned}
\right.
\end{equation*}
The second one is that the known information is generated by a standard Brownian noise, which, in fact, can be regarded as a special case of the first situation.

For stochastic control problems with partial information, there has been considerable literature about forward-backward stochastic controlled systems. Wang and Wu \cite{WW09} obtained a global maximum principle for partially observed stochastic recursive optimal control problem, under the assumption that control domain is not necessarily convex and the forward diffusion coefficient does not contain control variable. Wu \cite{Wu10} studied a partially observed stochastic control problem of FBSDEs and a local maximum principle was proved. Shi and Wu \cite{SW101} obtained a partially observed maximum principle of controlled fully coupled FBSDEs on the assumption that the forward diffusion coefficient does not contain the control variable and the control domain is not necessarily convex. Wang et al. \cite{WWX13} studied a partial observed stochastic control problem of FBSDEs with correlated noises between the system and the observation. Three versions of maximum principle were established by using a direct method, an approximation method and a Malliavin derivative method, respectively. By introducing a state decomposition and backward separation approach, Wang et al. \cite{WWX15} studied an {\it linear-quadratic} (LQ for short) stochastic control problem of FBSDEs, where the drift coefficient of the observation equation is linear with respect to state $x$, and the observation noise is correlated with the state noise. Some recent progress for partially observed problems, has been made in mean-field type controls (see \cite{ML16}, \cite{WXX17}, \cite{HM20}), differential games (see \cite{WY12}, \cite{HWW16}, \cite{WZ18}, \cite{XZZ19}, \cite{ZS20}), random jumps (see \cite{OS09}, \cite{TH02}, \cite{Xiao11}, \cite{Xiao13}, \cite{XW10}, \cite{ZXL18}), time delay (see \cite{LWW20}) and applications in finance (see \cite{XZZ20}, \cite{XZ07}). The partially observed stochastic control problem is usually associated with the state filtering and estimation technique, and we only list some main references such as \cite{LS77}, \cite{Ben92}, \cite{Xiong08}, \cite{WW08}, \cite{XW10}, \cite{WWX18}, \cite{Situ05}.

The motivation of our work in this paper comes from two aspects. The first one is Hu \cite{Hu17}, in which a global maximum principle for optimal control problem with recursive utilities has been obtained completely. The second one is Song et al. \cite{STW20}, in which a new global maximum principle for stochastic control problem of SDEPs, where the control variable enters into both the coefficients of diffusion and jump, has been derived by using a novel spike variation technique, which modifies a critical incorrectness in Tang and Li \cite{TL94}, which has been pointed out in the book of Situ \cite{Situ05} for over sixteen years. Combining the two aspects above with the partially observed feature, in this paper, we aim to study a global maximum principle for partially observed controlled forward-backward stochastic system with random jumps, where all the coefficients contain the control variable. We note that, very recently, Hao and Meng \cite{HM20} obtained a global maximum principle for optimal control of the general mean-field forward-backward stochastic system with jumps. However, the jump term in the SDEP of their paper does not contain the control variable and the cost functional is a special recursive case.

To our best knowledge, there are rare literature on the partially observed forward-backward stochastic optimal control problem with random jumps. In this paper, we initiate to prove the partially observed global maximum principle with random jumps, where the state noise is correlated with the observation noise. Our work distinguishes itself from the existing literatures in the following aspects:

(1) We extend the work in \cite{Hu17} into the partially observed case with random jumps. There exist correlated noises between the state equation and the observation equation, which make the original decoupled forward-backward stochastic system become a kind of fully-coupled one. We first prove the existence and uniqueness of the solution to fully coupled FBSDEPs in $L^2$-space on small interval. However, due to some technical difficulties, we only obtain the existence and uniqueness of the solution to the fully coupled FBSDEPs and give its estimate under {\bf (A4), (A6)} in $L^\beta(\beta>2)$-space, which, motivated by the recent work of Meng and Yang \cite{MY21arxiv} in which they proved that an adapted $L^2$-solution is an $L^\beta$-solution for some $\beta>2$, extends to a general case for any given terminal time $T$. For further resolving the difficulties, we relax the Assumption {\bf (A6)}, and obtain the $L^\beta(\beta>2)$-solution to fully coupled FBSDEPs and give its estimate under {\bf (A4), (A5)} by defined a new space $\tilde{F}^{2,\beta}[0,T]$ for any given $T$.

(2) Then we derive the partially observed global maximum principle. Compared to \cite{TL94} and \cite{TH02}, our result is both a correction and an extension. To our best knowledge, in the optimal control problems, we initiate to consider that the observation equation is not only driven by the Brownian motion, but also by the Poisson random measure. Moreover, the correlated noises between state and the observation equation not only come from the Brownian motion, but from the Poisson random measure. It is an extension of the case in \cite{XW10}, in which an LQ problem with special Poisson process and convex control domain is considered.

(3) Similar to \cite{Hu17}, we also guess that $z_t^{1,1},z_t^{2,1},\tilde{z}_{(t,e)}^{1,1},\tilde{z}_{(t,e)}^{2,1}$ have the following form, respectively:
\begin{equation}\label{Delta}
\begin{aligned}
z_t^{1,1}&=\Delta^1(t)\mathbbm{1}_{[\bar{t},\bar{t}+\epsilon]}+z_t^{1,1'},\ z_t^{2,1}=\Delta^2(t)\mathbbm{1}_{[\bar{t},\bar{t}+\epsilon]}+z_t^{2,1'},\\
\tilde{z}_{(t,e)}^{1,1}&=\tilde{\Delta}^1(t)\mathbbm{1}_{\mathcal{O}}+\tilde{z}_{(t,e)}^{1,1'},\ \tilde{z}_{(t,e)}^{2,1}=\tilde{\Delta}^2(t)\mathbbm{1}_{\mathcal{O}}+\tilde{z}_{(t,e)}^{2,1'},
\end{aligned}
\end{equation}
where $\Delta^1(\cdot),\Delta^2(\cdot),\tilde{\Delta}^1(\cdot),\tilde{\Delta}^2(\cdot)$ are four $\mathcal{F}_t$-adapted processes, and $z^{1,1'},z^{2,1'},\tilde{z}^{1,1'},\tilde{z}^{2,1'}$ have good estimates similarly as $x^1$. However, different from the results in \cite{Hu17} or \cite{HJX18}, we find $\tilde{\Delta}^1(\cdot)=\tilde{\Delta}^2(\cdot)=0$. The main reason is that we subtract the jump part in the random interval $\mathcal{O}$. There is no need for such a term $\tilde{\Delta}^1(t)\mathbbm{1}_{\mathcal{O}},\tilde{\Delta}^2(t)\mathbbm{1}_{\mathcal{O}}$ in the third and fourth equality in \eqref{Delta}, which is in accordance with the variational equation of $x$ in Section 3.1. Therefore, we mainly do Taylor expansions at $\bar{x}_t,\bar{y}_t,\bar{z}^1_t+p_t\delta\sigma_1(t)\mathbb{I}_{E_{\epsilon}}(t),\bar{z}^2_t+p_t\delta\sigma_2(t)\mathbb{I}_{E_{\epsilon}}(t)$ and $\bar{\tilde{z}}_t^1,\bar{\tilde{z}}_t^2$ for $\tilde{b}_1,\sigma_1,\sigma_2, f_1, f_2, l$, respectively.

(4) Compared to the work in \cite{Hu17}, we extend the cost functional into a general case, which is similar to the LQ part in \cite{HJX18}. However, two new adjoint equations are introduced. The variational inequality is firstly derived. Compared to the work in \cite{Hu17}, besides the vital first-order adjoint equation with Poisson measure, we can obtain four new adjoint equations, however, the second-order adjoint equation is a new BSDEP. Then we derive a new partially observed global maximum principle. More importantly, the maximum principle can reduce to the work \cite{LT95} when the forward-backward system with random jumps degenerate into the forward one without random jumps. Different from the framework in Li and Tang \cite{LT95} (or Tang \cite{Tang98}) where the observed process $Y$ is viewed as a Brownian motion at the beginning,  we not only consider a more natural framework(Hu et al. \cite{HNZ14}, Zhang et al. \cite{ZXL18}), but also prove the equivalence between the two frameworks. More details can be seen in Remark \ref{rem22}, Remark \ref{rem43}, Remark \ref{rem44}.

(5) We consider a special partially observed LQ forward-backward stochastic optimal control problem with random jumps. By the method of \cite{Situ05}, we first deduce a non-linear filtering equation for partially observed stochastic system with random jumps, which extends the case with Poisson process into the case with Poisson martingale measure. And we obtain the state estimate feedback of the optimal control in a closed-loop form. Moreover, an explicit representation for the optimal control is given by some {\it ordinary differential equations} (ODEs for short).

The rest of this paper is organized as follows. In Section 2, we formulate a general partially observed progressive optimal control problem of forward-backward stochastic system with random jumps, and prove the existence and uniqueness of the solution to the fully coupled FBSDEPs and its $L^\beta$-estimate under some assumptions and different spaces, for any $\beta\geq2$. In Section 3, we derive the partially observed global maximum principle of FBSDEPs with general cost functional. In Section 4, as an application, our theoretical results are applied to study a partially observed LQ optimal control problem of FBSDEPs. In a special case, we first give the maximum condition for the optimal control and verify its sufficiency. We then represent the optimal control as a state estimate feedback form. Moreover, we obtain a more explicit optimal control with some ODEs. Section 5 gives some concluding remarks. In the Appendix, we give some results about non-linear filtering equation of signal-observation system with Poisson martingale measure that will be used in the LQ problem in Section 4.

\section{Problem Formulation and Preliminaries}

Let $T>0$ be fixed. Consider a complete filtered probability space $(\Omega,\mathcal{F},(\mathcal{F}_t)_{0\leq t \leq T},\bar{P})$ and two one-dimensional independent standard Brownian motions $W^1$ and $\tilde{W}^2$ defined in $\mathbb{R}^2$ with $W^1_0=\tilde{W}^2_0=0$. Let $(\mathcal{E},\mathcal{B}(\mathcal{E}))$ be a Polish space with the $\sigma$-finite measure $\nu_1$ on $\mathcal{E}_1$, $\nu_2$ on $\mathcal{E}_2$ and $\mathcal{E}_1\subset\mathcal{E},\mathcal{E}_2\subset\mathcal{E}$. Suppose that $N_1(de,dt)$ is a Poisson random measure on $(\mathbb{R}^{+}\times\mathcal{E}_1,\mathcal{B}(\mathbb{R}^{+})\times\mathcal{B}(\mathcal{E}_1))$ under $\bar{P}$ and for any $E_1\in\mathcal{B}(\mathcal{E}_1)$, $\nu_1(E_1)<\infty$, then the compensated Poisson random measure is given by $\tilde{N}_1(de,dt)=N_1(de,dt)-\nu_1(de)dt$. Let $N_2(de,dt)$ be an integer-valued random measure and its predictable compensator is given by $\lambda(t,x_{t-},e)\nu_2(de)dt$, where the function $\lambda(t,x,e)\in[l,1),0< l<1$, and for any $E_2\in\mathcal{B}(\mathcal{E}_2)$, $\nu_2(E_2)<\infty$, then its compensated random measure is given by $\tilde{N}^{\prime}_2(de,dt)=N_2(de,dt)-\lambda(t,x_{t-},e)\nu_2(de)dt$. Moreover, $W^1,\tilde{W}^2,N_1,N_2$ are mutually independent under $\bar{P}$, and let $\mathcal{F}_t^{W^1},\mathcal{F}_t^{\tilde{W}^2},\mathcal{F}_t^{N_1},\mathcal{F}_t^{N_2}$ be the $\bar{P}$-completed natural filtrations generated by $W^1,\tilde{W}^2,N_1,N_2$, respectively. Set $\mathcal{F}_t:=\mathcal{F}_t^{W^1}\vee\mathcal{F}_t^{\tilde{W}^2}\vee\mathcal{F}_t^{N_1}\vee\mathcal{F}_t^{N_2}\vee\mathcal{N}$ and $\mathcal{F}=\mathcal{F}_T$, where $\mathcal{N}$ denotes the totality of $\bar{P}$-null sets. $\bar{E}$ denotes the expectation under the probability measure $\bar{P}$.

Different from \cite{TL94} and \cite{TH02}, the integrand of the stochastic integral in our paper is $\mathcal{E}$-progressive measurable rather than $\mathcal{E}$-predictable, which is referred to \cite{STW20}. We first introduce some preliminaries from \cite{STW20}.
\begin{definition}\label{definition21}
Suppose that $\mathcal{H}$ is an Euclidean space, and $\mathcal{B}(\mathcal{H})$ is the Borel $\sigma$-field on $\mathcal{H}$. Given $T>0$, a process $x:[0,T]\times\Omega\rightarrow\mathcal{H}$ is called \emph{progressive measurable (predictable)} if $x$ is $\mathcal{G}/\mathcal{B}(\mathcal{H})(\mathcal{P}/\mathcal{B}(\mathcal{H}))$ measurable, where $\mathcal{G}(\mathcal{P})$ is the corresponding progressive measurable (predictable) $\sigma$-field on $[0,T]\times\Omega$, and a process $x:[0,T]\times\Omega\times\mathcal{E}\rightarrow\mathcal{H}$ is called \emph{$\mathcal{E}$-progressive measurable ($\mathcal{E}$-predictable)} if $x$ is $\mathcal{G}\times\mathcal{B}(\mathcal{E})/\mathcal{B}(\mathcal{H})(\mathcal{P}\times\mathcal{B}(\mathcal{E})/\mathcal{B}(\mathcal{H}))$ measurable.
\end{definition}

Now, given a process $x$ which has RCLL paths, $x_{0-}:=0$ and $\Delta x_t:=x_t-x_{t-},t\geq0$, for $i=1,2$, let $m_i$ denote the measure on $\mathcal{F}\otimes\mathcal{B}([0,T])\otimes\mathcal{B}(\mathcal{E}_i)$ generated by $N_i$ that $m_i(A)=\bar{E}\int_0^T\int_{\mathcal{E}_i}\mathbb{I}_{A}N_i(de,dt)$. For any $\mathcal{F}\otimes\mathcal{B}([0,T])\otimes\mathcal{B}(\mathcal{E}_i)/\mathcal{B}(\mathbb{R})$ measurable integrable process $x$, we set $\mathbb{E}_i[x]:=\int xdm_i$ and denote by $\mathbb{E}_i[x|\mathcal{P}\otimes\mathcal{B}(\mathcal{E}_i)]$ the Radon-Nikodym derivatives w.r.t. $\mathcal{P}\otimes\mathcal{B}(\mathcal{E}_i)$. In fact, $\mathbb{E}_i$ is not an expectation (for $m_i$ is not a probability measure), but it owns similar properties to expectation. A more general definition of stochastic integral of random measure has been introduced by \cite{STW20} where the theory of dual predictable projection is utilized. Therefore, we omit the details here and give the following lemma directly.

\begin{lemma}\label{lemma21}
If $g$ is a positive $\mathcal{E}_i$-progressive measurable process that
\begin{equation*}\label{mathbbE}
\bar{E}\int_0^T\int_{\mathcal{E}_i}gN_i(de,dt)<\infty,\ i=1,2,
\end{equation*}
then we have the following results:
\begin{equation*}\begin{aligned}
&(i)\qquad \bigg(\int_0^{\cdot}\int_{\mathcal{E}_i}gN_i(de,dt)\bigg)_t^p=\int_0^t\int_{\mathcal{E}_i}\mathbb{E}_i\big[g|\mathcal{P}\otimes\mathcal{B}(\mathcal{E}_i)\big]\nu_i(de)dt,\\
&(ii)\qquad \int_0^T\int_{\mathcal{E}_i}g\tilde{N}_i(de,dt)=\int_0^T\int_{\mathcal{E}_i}gN_i(de,dt)-\bigg(\int_0^{\cdot}\int_{\mathcal{E}_i}gN_i(de,dt)\bigg)_T^p,\\
&(iii)\qquad \int_0^T\int_{\mathcal{E}_i}g\tilde{N}_i(de,dt)=\int_0^T\int_{\mathcal{E}_i}gN_i(de,dt)-\int_0^T\int_{\mathcal{E}_i}\mathbb{E}_i\big[g|\mathcal{P}\otimes\mathcal{B}(\mathcal{E}_i)\big]\nu_i(de)dt,\\
&(iv)\qquad \bar{E}\int_0^T\int_{\mathcal{E}_i}gN_i(de,dt)=\bar{E}\int_0^T\int_{\mathcal{E}_i}\mathbb{E}_i\big[g|\mathcal{P}\otimes\mathcal{B}(\mathcal{E}_i)\big]\nu_i(de)dt,\\
&(v)\qquad \Delta(g_{\cdot}\tilde{N}_i)_t=\int_{\mathcal{E}_i}gN_i(de,\{t\}),\\
&(vi)\qquad \big[g_{\cdot}\tilde{N}_{i,t},g_{\cdot}\tilde{N}_{i,t}\big]=\int_0^t\int_{\mathcal{E}_i}g^2N_i(de,dt),
\end{aligned}\end{equation*}
where $x^p$ is the dual predictable projection of $x$.
\end{lemma}

Consider the following state equation which is a controlled FBSDEP:
\begin{equation}\label{stateeq1}
\left\{
\begin{aligned}
 dx_t^u&=b_1(t,x_t^u,u_t)dt+\sigma_1(t,x_t^u,u_t)dW^1_t+\sigma_2(t,x_t^u,u_t)d\tilde{W}^2_t\\
       &\quad +\int_{\mathcal{E}_1}f_1(t,x_{t-}^u,u_t,e)\tilde{N}_1(de,dt)+\int_{\mathcal{E}_2}f_2(t,x_{t-}^u,u_t,e)\tilde{N}^{\prime}_2(de,dt),\\
-dy_t^u&=g\big(t,x_t^u,y_t^u,z^{1,u}_t,z^{2,u}_t,\int_{\mathcal{E}_1}\tilde{z}^{1,u}_{(t,e)}\nu_1(de),\int_{\mathcal{E}_2}\tilde{z}^{2,u}_{(t,e)}\nu_2(de),u_t\big)dt-z^{1,u}_tdW^1_t\\
       &\quad -z^{2,u}_tdW^2_t-\int_{\mathcal{E}_1}\tilde{z}^{1,u}_{(t,e)}\tilde{N}_1(de,dt)-\int_{\mathcal{E}_2}\tilde{z}^{2,u}_{(t,e)}\tilde{N}_2(de,dt),\quad t\in[0,T],\\
  x_0^u&=x_0,\ \ y_T^u=\phi(x_T^u),
\end{aligned}
\right.
\end{equation}
where, as we concluded in the following Remark \ref{rem22}, the first framework where probability measure $P$ is given originally is equivalent to the second framework (as presented in our paper) where probability measure $\bar{P}$ is given originally. Therefore, by the following Girsanov's theorem method, $W^1,W^2$ are mutually independent Brownian motions and $\tilde{N}_1, \tilde{N}_2$ are mutually independent Poisson martingale measures under $P$ and $W^1,\tilde{W}^2$ are mutually independent Brownian motions and $\tilde{N}_1, \tilde{N}^\prime_2$ are mutually independent Poisson martingale measures under $\bar{P}$. Then we give the fact that the stochastic process $W^2$ and Poisson random measure $\tilde{N}_2$ under $\bar{P}$ is standard Brownian motion and Poisson martingale measure under $P$, respectively. Equivalently, we also have the understanding that stochastic process $\tilde{W}^2$ and Poisson random measure $\tilde{N}^\prime_2$ under $P$ is standard Brownian motion and Poisson martingale measure under ${\bar{P}}$ as presented exactly in our paper, which is accordance with the non-jumps framework in \cite{WWX13}. Moreover, we have explicit representations of $W^2,\tilde{N}_2$ by $\tilde{W}^2,\tilde{N}^\prime_2$ as follows:
\begin{equation}\label{relation2}
\begin{aligned}
dW^2_t&=d\tilde{W}^2_t+\sigma_3^{-1}(t)b_2(t,\Theta^u(t),u_t)dt,\\
\tilde{N}_2(de,dt)&=\tilde{N}^\prime_2(de,dt)+(\lambda(t,x_{t-}^u,e)-1)\nu_2(de)dt.
\end{aligned}
\end{equation}
where $\sigma_3(t)$ and $b_2(t,\Theta^u(t),u_t)$ are defined in the following.

Suppose that the state process $(x^u,y^u,z^{1,u},z^{2,u},\tilde{z}^{1,u},\tilde{z}^{2,u})$ cannot be observed directly. Indeed, we can observe a related process $Y$ which is governed by the following SDEP:
\begin{equation}\label{observation}
\left\{
\begin{aligned}
dY_t^u&=b_2\big(t,x_t^u,y_t^u,z^{1,u}_t,z^{2,u}_t,\int_{\mathcal{E}_1}\tilde{z}^{1,u}_{(t,e)}\nu_1(de),\int_{\mathcal{E}_2}\tilde{z}^{2,u}_{(t,e)}\nu_2(de),u_t\big)dt\\
      &\quad +\sigma_3(t)d\tilde{W}^2_t+\int_{\mathcal{E}_2}f_3(t,e)\tilde{N}^{\prime}_2(de,dt),\quad t\in[0,T],\\
 Y_0^u&=0,
\end{aligned}
\right.
\end{equation}
and consider the cost functional as follows:
\begin{equation}\label{cf1}
J(u)=\bar{E}\bigg\{\int_0^Tl\big(t,x_t^u,y_t^u,z_t^{1,u},z_t^{2,u},\int_{\mathcal{E}_1}\tilde{z}_{(t,e)}^{1,u}\nu_1(de),
\int_{\mathcal{E}_2}\tilde{z}_{(t,e)}^{2,u}\nu_2(de),u_t\big)dt+\Phi(x_T^u)+\Gamma(y_0^u)\bigg\},
\end{equation}
where we suppose that $(b_1,\sigma_1,\sigma_2)(t,x,u):[0,T]\times \mathbb{R}\times \mathbb{R}\rightarrow \mathbb{R}$, $(f_1,f_2)(t,x,u,e):[0,T]\times \mathbb{R}\times \mathbb{R}\times\mathcal{E}_1(\mathcal{E}_2)\rightarrow \mathbb{R}$, and $(g,b_2,l)(t,x,y,z^1,z^2,\tilde{z}^1,\tilde{z}^2,u):[0,T]\times \mathbb{R}\times \mathbb{R}\times \mathbb{R}\times \mathbb{R}\times \mathbb{R}\times \mathbb{R}\times \mathbb{R}\rightarrow \mathbb{R}$, $\phi(x): \mathbb{R}\rightarrow \mathbb{R}$, $\sigma_3(t):[0,T]\rightarrow \mathbb{R}$, $f_3(t,e):[0,T]\times\mathcal{E}_2\rightarrow \mathbb{R}$, $\Phi(x):\mathbb{R}\rightarrow \mathbb{R}$, $\Gamma(y):\mathbb{R}\rightarrow \mathbb{R}$ are suitable maps.
The following assumptions are assumed.

{\bf (A1)}\quad
(1) $b_1,\sigma_1,\sigma_2$ are $\mathcal{G}\otimes\mathcal{B}(\mathbb{R})\otimes\mathcal{B}(\mathbb{R})/\mathcal{B}(\mathbb{R})$ measurable, $f_1,f_2$ are $\mathcal{G}\otimes\mathcal{B}(\mathbb{R})\otimes\mathcal{B}(\mathbb{R})\otimes\mathcal{B}(\mathcal{E}_1)(\mathcal{B}(\mathcal{E}_2))/\mathcal{B}(\mathbb{R})$ measurable, $\phi$ is $\mathcal{B}(\mathbb{R})/\mathcal{B}(\mathbb{R})$ measurable.

(2) $b_1,\sigma_1,\sigma_2,f_1$ and $f_2$ are twice continuously differentiable in $x$ with bounded first and second order derivatives, $\sigma_2$ is also bounded and there is a constant $C$ such that
\begin{equation*}
\big|(b_1,\sigma_1,\sigma_2,f_1,f_2)(t,x,u)\big|\leq C(1+|x|+|u|).
\end{equation*}

(3) $\phi$ are twice continuous differentiable in $x$ with bounded second order derivatives and there is a constant $C$ such that
\begin{equation*}
|\phi(x)|\leq C(1+|x|).
\end{equation*}

(4) For $\beta\geq2$, the following hold:
\begin{equation*}
\begin{aligned}
&\bigg(\int_0^T|b_1(t,0,0)|dt\bigg)^\beta<\infty,\quad \bigg(\int_0^T|\sigma_1(t,0,0)|^2dt\bigg)^\frac{\beta}{2}<\infty,\quad \bigg(\int_0^T|\sigma_2(t,0,0)|^2dt\bigg)^\frac{\beta}{2}<\infty,\\
&\bigg(\int_0^T\int_{\mathcal{E}_1}|f_1(t,0,0,e)|^2\nu_1(de)dt\bigg)^\frac{\beta}{2}<\infty,\quad \bigg(\int_0^T\int_{\mathcal{E}_2}|f_2(t,0,0,e)|^2\nu_2(de)dt\bigg)^\frac{\beta}{2}<\infty,\\ &\bigg(\int_0^T|b_2(t,0,0,0,0,0,0,0)|dt\bigg)^\beta<\infty,\quad \bigg(\int_0^T|g(t,0,0,0,0,0,0,0)|dt\bigg)^\beta<\infty.
\end{aligned}
\end{equation*}

(5) $b_2,g$ are $\mathcal{G}\otimes\mathcal{B}(\mathbb{R})\otimes\mathcal{B}(\mathbb{R})\otimes\mathcal{B}(\mathbb{R})\otimes\mathcal{B}(\mathbb{R})\otimes\mathcal{B}
(\mathbb{R})\otimes\mathcal{B}(\mathbb{R})\otimes\mathcal{B}(\mathbb{R})/
\mathcal{B}(\mathbb{R})$ measurable, and are twice continuously differentiable with respect to $(x,y,z^1,z^2,\tilde{z}^1,\tilde{z}^2)$; $b_2,g,Db_2,Dg,D^2b_2,D^2g$ are continuous in $(x,y,z^1,z^2,\tilde{z}^1,\tilde{z}^2,u)$; $b_2,Db_2,Dg,D^2b_2,D^2g$ are bounded, and $|\sigma_3^{-1}(t)|$ is bounded by some constant $C$, and
\begin{equation*}
\begin{aligned}
&\big|b_2(t,x,,y,z^1,z^2,\tilde{z}^1,\tilde{z}^2,u)\big|+\big|g(t,x,,y,z^1,z^2,\tilde{z}^1,\tilde{z}^2,u)\big|\\
&\ \leq C(1+|x|+|y|+|z^1|+|z^2|+||\tilde{z}^1||+||\tilde{z}^2||+|u|).
\end{aligned}
\end{equation*}

(6) For any $(t,x,e)\in[0,T]\times \mathbb{R}\times\mathcal{E}_2$, there exists a constant $C$ such that $|\lambda_x(t,x,e)|+|\lambda_{xx}(t,x,e)|\leq C$.

{\bf (A2)}\quad $l,\Phi,\Gamma$ are twice continuous differentiable in $(x,y,z^1,z^2,\tilde{z}^1,\tilde{z}^2)$ with bounded second order derivatives and there is a constant $C$ such that
\begin{equation*}
\begin{aligned}
&\big|l_{(x,y,z^1,z^2,\tilde{z}^1,\tilde{z}^2)}(t,x,y,z^1,z^2,\tilde{z}^1,\tilde{z}^2,u)\big|\leq C\big(1+|x|+|y|+|z^1|+|z^2|+||\tilde{z}^1||+||\tilde{z}^2||+|u|\big),\\
&\big|l(t,x,y,z^1,z^2,\tilde{z}^1,\tilde{z}^2,u)\big|\leq C\big(1+|x|^2+|y|^2+|z^1|^2+|z^2|^2+||\tilde{z}^1||^2+||\tilde{z}^2||^2+|u|^2\big),\\
&|\Phi(x)|\leq C(1+|x|^2),\quad |\Phi_x(x)|\leq C(1+|x|),\quad \Gamma(y)\leq C(1+|y|^2),\quad \Gamma_y(y)\leq C(1+|y|).
\end{aligned}
\end{equation*}
\begin{remark}\label{rem21}
We consider the deterministic coefficient in our paper, that is, fixing $(x,y,z^1,z^2,\\\tilde{z}^1,\tilde{z}^2,u)$, $b_1,b_2,\sigma_1,\sigma_2,g,f_1,f_2,l$ are deterministic function from $[0,T]\times\mathcal{E}_1(\mathcal{E}_2)$ to $\mathbb{R}$. In fact, we can regard $[0,T]$ as a kind of cylinder set $[0,T]\times\Omega$ in a binary form, and here, $\mathcal{G}(\mathcal{P})$ is the corresponding progressive measurable (predictable) $\sigma$-field on $[0,T]$ which can be viewed as a special case on $[0,T]\times\Omega$. The only difference is that the coefficients $b_1,b_2,\sigma_1,\sigma_2,g,f_1,f_2,l$ do not contain $\omega$ exactly.
\end{remark}

In our paper, for notational simplicity, we define $\Theta^{u}(t):=(x_t^u,y_t^u,z^{1,u}_t,z^{2,u}_t,\tilde{z}^{1,u}_{(t,e)},\tilde{z}^{2,u}_{(t,e)})$ and
\begin{equation*}
\begin{aligned}
&\tilde{g}(t,\Theta^{u}(t),u_t)
:=\tilde{g}\big(t,x_t^u,y_t^u,z^{1,u}_t,z^{2,u}_t,\int_{\mathcal{E}_1}\tilde{z}^{1,u}_{(t,e)}\nu_1(de),\int_{\mathcal{E}_2}\tilde{z}^{2,u}_{(t,e)}\nu_2(de),u_t\big),
\end{aligned}
\end{equation*}
for $\tilde{g}=g,b_2,l$ and $(t,x,y,z^1,z^2,\tilde{z}^1,\tilde{z}^2)\in[0,T]\times \mathbb{R}\times \mathbb{R}\times \mathbb{R}\times \mathbb{R}\times L^2(\mathcal{E}_1,\mathcal{B}(\mathcal{E}_1),\nu_1;\mathbb{R})\times L^2(\mathcal{E}_2,\mathcal{B}(\mathcal{E}_2),\nu_2;\mathbb{R})$. Here, the space $L^2(\mathcal{E}_i,\mathcal{B}(\mathcal{E}_i),\nu_i;\mathbb{R})$,\ $i=1,2$ are defined in the {\bf (A4)}, (3) in the following.

For $t\in[0,T]$, define ${\mathcal{F}}_t^Y:=\sigma\{Y_s;0\leq s\leq t\}$. For $U\subseteq\mathbb{R}$, the admissible control set is defined as follows:
\begin{equation}\label{ad_control_set}
\begin{aligned}
\mathcal{U}_{ad}[0,T]:=&\Big\{u\big|u_t\mbox{ is }\mathcal{F}_t^{Y}\mbox{-progressive }U\mbox{-valued process, such that }\sup_{0\leq t\leq T}\bar{E}\big[|u_t|^{\beta}\big]<\infty,\\
&\mbox{ for any $\beta>1$ }\mbox{and }\bar{E}\int_0^T|u_t|^2N_{i}(\mathcal{E}_{i},dt)<\infty,\mbox{\ for\ } i=1,2\Big\}.
\end{aligned}
\end{equation}

Our aim in this paper, is to select an optimal control $\bar{u}\in{\mathcal{U}_{ad}[0,T]}$ such that
\begin{equation*}
J(\bar{u})=\inf_{u\in\,\mathcal{U}_{ad}[0,T]}J(u).
\end{equation*}

\begin{remark}\label{rem22}
In our setting, for overcoming the circular dependence between the controlled filtration $\mathcal{F}_t^Y$ and the control $u$, we have adopted the Girsanov measure transformation method. We point out that our framework is different from that most used in (\cite{LT95}, \cite{WW09}, \cite{Wu10}, \cite{SW101}, \cite{WWX13}), but similar as (\cite{HNZ14}, \cite{ZXL18}). The obvious difference between these two kinds of frameworks depends on the original probability measure, that is, $P$ or $\bar{P}$, which is given first. In the first framework, for a given original probability measure $P$, the observation process $Y$ is given as a standard Brownian motion. However, in the second framework which is used in our paper, under a given original probability measure $\bar{P}$, the observation process $Y$ is a stochastic process which becomes a standard Brownian motion under the new probability measure $P$ by Girsanov's theorem.

In fact, by analyzing, we could obtain an important conclusion that these two kinds of frameworks are equivalent in some sense. First, no matter which kind of framework is considered in the existing literature, as a goal, the observation process $Y$ should be a control-independent standard Brownian motion under the probability measure $P$. Indeed, this goal is the key to deal with partially observed problems. More precisely, in the first framework, $W$ and observation process $Y$ are two uncontrolled Brownian motions, and the cost functional is usually defined under the new controlled probability measure $\bar{P}$. However, for obtaining necessary and sufficient maximum principles, a controlled $\bar{P}$ would bring the obstacle when using variation technique and It\^{o}'s formula. Therefore, the given uncontrolled $P$ and uncontrolled Brownian motion $Y$ by the measure transformation can be regarded as a ``bridge" to fill in this gap. In the second framework, $P$ is a new probability measure under which the original controlled stochastic process $Y$ becomes a uncontrolled Brownian motion, which is more natural than the given Brownian motion $Y$ at the beginning in the first framework. Moreover, the nature mainly shows that, it is usual that the Brownian motion can not be observed directly, but the role of Girsanov transformation shows that there exists another probability measure $P$ such that the observation process $Y$ can be turned into certain Brownian motion. So we think that considering $Y$ as a Brownian motion at first, although it is unnatural, is a technical demanding. Moreover, it means that the original $P$ (the new $\bar{P}$) in the first framework and the new $P$ (the original $\bar{P}$) in the second framework play the same roles and actually have the same effects. So these two kinds of frameworks are equivalent in some sense. We will show its equivalence in general mathematical sense in Remark \ref{rem44}.
\end{remark}

As analyzed above, in order to solve the problem, we first set
\begin{equation}\label{RN}
\begin{aligned}
\Gamma^u_t:=&\exp\bigg\{-\int_0^T\sigma_3^{-1}(t)b_2\big(t,\Theta^{u}(t),u_t\big)d\tilde{W}^2_t-\frac{1}{2}\int_0^T|\sigma_3^{-1}(t)b_2\big(t,\Theta^{u}(t),u_t\big)|^2dt\\
            &\qquad -\int_0^T\int_{\mathcal{E}_2}\log{\lambda(t,x_{t-}^u,e)}N_2(de,dt)-\int_0^T\int_{\mathcal{E}_2}(1-\lambda(t,x_{t-}^u,e))\nu_2(de)dt\bigg\}.
\end{aligned}
\end{equation}
The following assumption is necessary to guarantee the Girsanov measure transformation.
\begin{equation*}
{\bf (A3)}\qquad\qquad \bar{E}\bigg[\exp\bigg\{\int_0^T\int_{\mathcal{E}_2}\frac{(1-\lambda(t,x_{t-}^u,e))^2}{\lambda(t,x_{t-}^u,e)}\nu_2(de)dt\bigg\}\bigg]<\infty.
\end{equation*}
Under {\bf (A3)}, define
\begin{equation*}
\begin{aligned}
M(t):=&-\int_0^t\sigma_3^{-1}(s)b_2\big(s,\Theta^{u}(s),u_s\big)d\tilde{W}^2_s+\int_0^t\int_{\mathcal{E}_2}\frac{1-\lambda(s,x_{s-}^u,e)}{\lambda(s,x_{s-}^u,e)}\tilde{N}^{\prime}_2(de,ds),
\end{aligned}
\end{equation*}
which is a locally square integrable martingale. Moreover, $M_t-M_{t-}>-1,\ \bar{P}\mbox{-}a.s.$, and
\begin{equation*}
\begin{aligned}
&\bar{E}\bigg[\exp\bigg\{\frac{1}{2}\langle M^c,M^c\rangle_T+\langle M^d,M^d\rangle_T\bigg\}\bigg]\\
&=\bar{E}\bigg[\exp\bigg\{\frac{1}{2}\int_0^T\big|\sigma_3^{-1}(t)b_2\big(t,\Theta^{u}(t),u_t\big)\big|^2dt+\int_0^T\int_{\mathcal{E}_2}\frac{(1-\lambda(t,x_{t}^u,e))^2}{\lambda(t,x_{t}^u,e)}\nu_2(de)dt\bigg\}\bigg]<\infty,
\end{aligned}
\end{equation*}
where $M^c$ and $M^d$ are continuous and purely discontinuous martingale parts of $M$, respectively. Therefore, it follows that $\Gamma_t^u$, the Dol\'{e}ans-Dade exponential of $M$, is a martingale (Protter and Shimbo \cite{PS08}). Then we can define a probability measure $P$ via
\begin{equation*}
\frac{d\bar{P}}{dP}:=\tilde{\Gamma}_T^u\equiv(\Gamma_T^u)^{-1},
\end{equation*}
where $\tilde{\Gamma}^u$ satisfies the following equation:
\begin{equation}\label{RN2}
\left\{
\begin{aligned}
d\tilde{\Gamma}_t^u&=\tilde{\Gamma}_t^u\sigma_3^{-1}(t)b_2\big(t,\Theta^{u}(t),u_t\big)dW^2_t+\int_{\mathcal{E}_2}\tilde{\Gamma}_{t-}^u(\lambda(t,x_{t-}^u,e)-1)\tilde{N}_2(de,dt),\quad t\in[0,T],\\
 \tilde{\Gamma}_0^u&=1.
\end{aligned}
\right.
\end{equation}
Moreover, under the new probability measure $P$ and denoting the expectation with respect to $P$ by $E$, $W^1,W^2$ are mutually independent Brownian motions and $\tilde{N}_1, \tilde{N}_2$ are mutually independent Poisson martingale measures, where
\begin{equation}\label{relation1}
\begin{aligned}
&dW^2_t=d\tilde{W}^2_t+\sigma_3^{-1}(t)b_2(t,\Theta^{u}(t),u_t)dt,\quad \tilde{N}_2(de,dt)=N_2(de,dt)-\nu_2(de)dt.
\end{aligned}
\end{equation}
Observing $\tilde{N}_2$ and $\tilde{N}_2^{\prime}$, then their relationship can be built as \eqref{relation2}.

We introduce the following spaces:
\begin{equation*}
S^\beta[0,T]=\Big\{y\big|y\mbox{ has } RCLL\mbox{ paths, adapted and }E\Big[\sup_{0\leq t\leq T}|y_t|^\beta\Big]<\infty\Big\},
\end{equation*}
with norm $||y||_\beta^\beta:=E\big[\sup_{0\leq t\leq T}|y_t|^\beta\big]$;
\begin{equation*}
M^{2,\beta}_i[0,T]=\Big\{z^i\big|z^i\mbox{ is predictable and }E\bigg[\Big(\int_0^T|z^i_t|^2dt\Big)^{\frac{\beta}{2}}\bigg]<\infty\Big\},\ i=1,2,
\end{equation*}
with norm $||z^i||_{2,\beta}^\beta:=E\big[(\int_0^T|z^i_t|^2dt)^{\frac{\beta}{2}}\big]$;
\begin{equation*}
F^{2,\beta}_i[0,T]=\Big\{\tilde{z}^i\big|\tilde{z}^i\mbox{ is $\mathcal{E}_i$-predictable and }E\bigg[\Big(\int_0^T\int_{\mathcal{E}_i}|\tilde{z}^i_{(t,e)}|^2\nu_i(de)dt\Big)^{\frac{\beta}{2}}\bigg]<\infty\Big\},\ i=1,2,
\end{equation*}
with norm $||\tilde{z}^i||_{2,\beta}^\beta:=E\big[\big(\int_0^T\int_{\mathcal{E}_i}|\tilde{z}^i_{(t,e)}|^2\nu_i(de)dt\big)^{\frac{\beta}{2}}\big]$. Specially, for $\beta=2$, we set $S^2[0,T]$, $M^2_i[0,T]\equiv M^{2,2}_i[0,T]$ and $F^2_i[0,T]\equiv F^{2,2}_i[0,T]$ defined similarly, for $i=1,2$, as the above.

For simplicity, we set $\mathcal{M}^\beta[0,T]:=S^{\beta}[0,T]\times S^{\beta}[0,T]\times M_1^{2,\beta}[0,T]\times M_2^{2,\beta}[0,T]\times F_1^{2,\beta}[0,T]\times F_2^{2,\beta}[0,T]$ and $\mathcal{N}^\beta[0,T]:=S^{\beta}[0,T]\times M_1^{2,\beta}[0,T]\times M_2^{2,\beta}[0,T]\times F_1^{2,\beta}[0,T]\times F_2^{2,\beta}[0,T]$, for $\beta\geq2$.

Then substituting \eqref{relation2} into the state equation \eqref{stateeq1}, we get
\begin{equation}\label{stateeq2}
\left\{
\begin{aligned}
 dx_t^u&=\tilde{b}_1\big(t,\Theta^{u}(t),u_t\big)dt+\sum_{i=1}^2\sigma_i(t,x_t^u,u_t)dW^i_t+\sum_{i=1}^2\int_{\mathcal{E}_i}f_i(t,x_{t-}^u,u_t,e)\tilde{N}_i(de,dt),\\
-dy_t^u&=g\big(t,\Theta^{u}(t),u_t\big)dt-\sum_{i=1}^2z^{i,u}_tdW^i_t-\sum_{i=1}^2\int_{\mathcal{E}_i}\tilde{z}^{i,u}_{(t,e)}\tilde{N}_i(de,dt),\quad t\in[0,T],\\
  x_0^u&=x_0,\ \ y_T^u=\phi(x_T^u),
\end{aligned}
\right.
\end{equation}
where
\begin{equation*}
\begin{aligned}
\tilde{b}_1\big(t,\Theta^{u}(t),u\big)&:=b_1(t,x^u,u)-\int_{\mathcal{E}_2}f_2(t,x^u,u,e)(\lambda(t,x^u,e)-1)\nu_2(de)\\
                                      &\qquad-\sigma_3^{-1}(t)\sigma_2(t,x^u,u)b_2\big(t,\Theta^{u}(t),u\big).
\end{aligned}
\end{equation*}

\begin{remark}\label{rem23}
We note that, although we aim to study the decoupled forward-backward stochastic systems originally, the system becomes a fully coupled one due to the partially observed structure. Therefore, we need existence and uniqueness of the solutions and the corresponding $L^{\beta}(\beta\geq2)$-estimate for a series of variations needed in the deduction of the global maximum principle.
\end{remark}

For the $L^p$-estimate and existence and uniqueness of solutions to SDE(P)s, BSDE(P)s and FBSDE(P)s, let us mention a few results. Tang and Li \cite{TL94} studied the existence and uniqueness of solution to BSDEP in $L^2$ space (see also Barles et. al. \cite{BBP97}). Situ \cite{S97} derived an existence and uniqueness result for the $L^2$-adapted solution to BSDEP, where he considered that its terminal time is a bounded random stopping time and the coefficients are non-Lipschitzian. Wu \cite{Wu99} obtained the existence and uniqueness results of solution to fully coupled FBSDEP in $L^2$-space under the monotonicity condition. Then Wu \cite{Wu03} gave the existence and uniqueness result and a comparison theorem for such equation in stopping time (unbounded) duration. Yin and Mao \cite{YM08} dealt with a class of BSDEP with random terminal times and proved the existence and uniqueness result of adapted solution under the assumption of non-Lipschitzian coefficient. Quenez and Sulem \cite{QS13}  studied some properties of linear BSDEP and also gave the existence and uniqueness of solution to such a BSDEP with general driver in $L^p$-space $(p\geq2)$ and its $L^2$-estimate. Li and Wei \cite{LW14} studied some $L^p$ $(p\geq2)$ estimates for fully coupled FBSDEP which were proved under the monotonicity assumption for arbitrary time intervals and in a small time intervals by assuming that some Lipschitz constants are sufficiently small. Kruse and Popier \cite{KP16} gave the existence and uniqueness of solution to a BSDE with Brownian and Poisson noises in a general filtration in $L^p$-spaces $(p>1)$, especially for the case of $p\in(1,2)$ which has to be handled carefully, and some mistake was pointed out when $p<2$, which was corrected by themselves in \cite{KP17}. Yao \cite{Yao17} studied $L^p$-solution $(p\in(1,2))$ of a multi-dimensional BSDEP whose generator may not be Lipschitz continuous in $(y,z)$ and obtained its existence and uniqueness of solution with $p$-integrable terminal data. Geiss and Steinicke \cite{GS18}  proved existence and uniqueness of BSDEP allowing the coefficients in the linear growth- and monotonicity-condition for the generator to be random and time-dependent. In the $L^2$-case with linear growth, this also generalized the results in \cite{KP16}. Confortola \cite{Confortola19} obtained the existence and uniqueness in $L^p$ $(p>1)$ of solution to a BSDE driven by a marked point process on a bounded interval. Moreover, it is also necessary to mention the following recent works about $L^p$-theory of FBSDE. By firstly establishing a measurable global implicit function theorem, then, under the Lipschitz condition and monotonicity assumption, Xie and Yu \cite{XY20} obtained an $L^p(p>1)$-solution and its related $p$-th power estimates for coupled FBSDE with random coefficients on small durations. It is worth noting that Yong \cite{Yong20} investigated the problem of when an adapted $L^2$-solution of an FBSDE is an adapted $L^p(p>2)$-solution. For giving an affirmative answer, several important cases were explored, which is thought-provoking but is far from having a satisfactory answer. Therefore, some open questions worthy of studying are posed to construct a better theory of FBSDEs. For giving a positive answer to the problem proposed by Yong \cite{Yong20}, Meng and Yang \cite{MY21arxiv} proved that a unique $L^2$-solution of fully coupled FBSDEs is an $L^p$-solution under standard conditions on the coefficients for any given terminal time $T$, based on a key observation of the relation between $L^2$ and $L^p$ estimations of FBSDEs, by which we have developed a $L^p(p>2)$-theory, including $L^p$-solution and its $p$-th power estimation, of fully coupled FBSDEPs for any given terminal time $T$ as an extension to the case with Poisson jumps.

\subsection{The existence and uniqueness of the solution and $L^\beta$-estimate}

Consider the following fully coupled FBSDEP:
\begin{equation}\label{fbsdep1}
\left\{
\begin{aligned}
 dx_t&=b(t,\Theta(t))dt+\sum_{i=1}^2\sigma_i(t,\Theta(t))dW^i_t+\sum_{i=1}^2\int_{\mathcal{E}_i}f_i(t,\Theta(t-),e)\tilde{N}_i(de,dt),\\
-dy_t&=g(t,\Theta(t))dt-\sum_{i=1}^2z^i_tdW^i_t-\sum_{i=1}^2\int_{\mathcal{E}_i}\tilde{z}^i_{(t,e)}\tilde{N}_i(de,dt),\quad t\in[0,T],\\
  x_0&=x_0,\ \ y_T=\phi(x_T),
\end{aligned}
\right.
\end{equation}
where $\Theta(t):=(x_t,y_t,z^{1}_t,z^{2}_t,\tilde{z}^{1}_{(t,e)},\tilde{z}^{2}_{(t,e)})$, and the coefficients $b,\sigma_1,\sigma_2,f_1,f_2,g,\phi$ could be random.

{\bf (A4)}\quad (1) Let $(\Omega,\mathcal{F},\{\mathcal{F}_t\}_{t\geq0},P)$ be a complete probability space, on which standrad Brownian motions $\{W^1_t,W^2_t\}_{t\geq0}\in\mathbb{R}^2$ and Poisson random measures $N_i$ with the compensator $EN_i(de,dt)=\nu_i(de)dt$, for $i=1,2,$ are mutually independent.  Here $\nu_i$ is assumed to be a $\sigma$-finite L\'{e}vy measure on $(\mathcal{E}_i,\mathcal{B}(\mathcal{E}_i))$ with the property that $\int_{\mathcal{E}_i}(1\land|e|^2)\nu_i(de)<\infty$, for $i=1,2$.

(2) Define the space $L^2(\mathcal{E}_i,\mathcal{B}(\mathcal{E}_i),\nu_i;\mathbb{R}):=\big\{\tilde{z}^i_{(t,e)}\in \mathbb{R}\ \mbox{such that}\ \big[\int_{\mathcal{E}_i}|\tilde{z}^i_{(t,e)}|^2\nu_i(de)\big]^{\frac{1}{2}}<\infty\big\}$, for $i=1,2$.

(3) $b,\sigma_1,\sigma_2,g$ are $\mathcal{G}\otimes\mathcal{B}(\mathbb{R})\otimes\mathcal{B}(\mathbb{R})\otimes\mathcal{B}(\mathbb{R})\otimes\mathcal{B}(\mathbb{R})\otimes\mathcal{B}(L^2(\mathcal{E}_1,\mathcal{B}(\mathcal{E}_1),\nu_1;\mathbb{R}))\\
\otimes\mathcal{B}
(L^2(\mathcal{E}_2,\mathcal{B}(\mathcal{E}_2),\nu_2;\mathbb{R}))/\mathcal{B}(\mathbb{R})$ measurable, $f_1(f_2)$ is $\mathcal{G}\otimes\mathcal{B}(\mathbb{R})\otimes\mathcal{B}(\mathbb{R})\otimes\mathcal{B}(\mathbb{R})\otimes\mathcal{B}(\mathbb{R})\otimes\mathcal{B}(\mathbb{R})\\
\otimes\mathcal{B}(\mathbb{R})\otimes\mathcal{B}(\mathcal{E}_1)(\otimes\mathcal{B}(\mathcal{E}_2))/\mathcal{B}(\mathbb{R})$ measurable, and $\phi$ is $\mathcal{F}_T\otimes\mathcal{B}(\mathbb{R})/\mathcal{B}(\mathbb{R})$ measurable.

(4) $b,\sigma_1,\sigma_2,g$ are uniformly Lipschitz with respect to $(x,y,z^1,z^2,\tilde{z}^1,\tilde{z}^2)$, and $\phi(x)$ is uniformly Lipschitz with respect to $x\in \mathbb{R}$.

(5) For $\beta\geq2$, we have
\begin{equation*}
\begin{aligned}
&E\bigg(\int_0^T|b(t,\omega,0,0,0,0,0,0)|dt\bigg)^\beta+\sum_{i=1}^2\bigg(E\bigg(\int_0^T|\sigma_i(t,\omega,0,0,0,0,0,0)|^2dt\bigg)^{\frac{\beta}{2}}\\
&+E\bigg(\int_0^T\int_{\mathcal{E}_i}|f_i(t,\omega,0,0,0,0,0,0,e)|^2N_i(de,dt)\bigg)^{\frac{\beta}{2}}\bigg)+E\bigg(\int_0^T|g(t,\omega,0,0,0,0,0,0)|dt\bigg)^\beta<\infty,
\end{aligned}
\end{equation*}
and $E[|\phi(0)|^\beta]<\infty$.

(6) For any $t\in[0,T]$, $\Theta\in \mathbb{R}\times \mathbb{R}\times \mathbb{R}\times \mathbb{R}\times L^2(\mathcal{E}_1,\mathcal{B}(\mathcal{E}_1),\nu_1;\mathbb{R})\times L^2(\mathcal{E}_2,\mathcal{B}(\mathcal{E}_2),\nu_2;\mathbb{R})$, $P\mbox{-}a.s.$,
\begin{equation*}
|b(t,\Theta)|+\sum_{i=1}^2|\sigma_i(t,\Theta)|+|g(t,\Theta)|+|\phi(x)|\leq L(1+|x|+|y|+|z^1|+|z^2|+||\tilde{z}^1||+||\tilde{z}^2||),
\end{equation*}
and there exists a measurable function $\rho:\mathcal{E}_i(\subset\mathcal{E})\rightarrow\mathbb{R}^{+}$ with $\int_{\mathcal{E}_i}\rho^\beta(e)\nu_i(de)<\infty (\beta\geq2)$ such that, for any $t\in[0,T]$, $\Theta\in \mathbb{R}^6$ and $e\in\mathcal{E}_i$, $|f_i(t,\Theta,e)|\leq\rho(e)(1+|x|+|y|+|z^1|+|z^2|+|\tilde{z}^1|+|\tilde{z}^2|)$, i=1,2.

(7) For $i=1,2$, the Lipschitz constant $L_{\sigma_i}\geq0$ of diffusion $\sigma_i$ with respect to $(z^1,z^2,\tilde{z}^1,\tilde{z}^2)$ is sufficient small. That is, there exists some $L_{\sigma_i}\geq0$ small enough such that, for all $t\in[0,T]$, $\Theta_j=(x_j,y_j,z^1_j,z^2_j,\tilde{z}^1_j,\tilde{z}^2_j)\in \mathbb{R}^4\times L^2(\mathcal{E}_1,\mathcal{B}(\mathcal{E}_1),\nu_1;\mathbb{R})\times L^2(\mathcal{E}_2,\mathcal{B}(\mathcal{E}_2),\nu_2;\mathbb{R})$, $j=1,2$, $P\mbox{-}a.s.$,
\begin{equation*}
\big|\sigma_i(t,\Theta_1)-\sigma_i(t,\Theta_2)\big|\leq K(|x_1-x_2|+|y_1-y_2|)+L_{\sigma_i}(|z^1_1-z^1_2|+|z^2_1-z^2_2|+||\tilde{z}^1_1-\tilde{z}^1_2||+||\tilde{z}^2_1-\tilde{z}^2_2||).
\end{equation*}
Moreover, the Lipschitz coefficient $L_{f_i}$ of $f_i$ with respect to $(z^1,z^2,\tilde{z}^1,\tilde{z}^2)$ is sufficient small, i.e., there exists a function $L_{f_i}:\mathcal{E}_i\rightarrow\mathbb{R}^{+}$ with $\tilde{C}_{f_i}:=\max\big\{\sup_{e\in\mathcal{E}_i}L^{\beta}_{f_i}(e),\int_{\mathcal{E}_i}L^{\beta}_{f_i}(e)\nu_i(de)\big\}<\infty(\beta\geq2)$ sufficient small, and for all $t\in[0,T]$, $(\Theta_1,\Theta_2)\in \mathbb{R}^{12}$ and $e\in\mathcal{E}_i$, $P\mbox{-}a.s.$,
\begin{equation*}
\big|f_i(t,\Theta_1,e)-f_i(t,\Theta_2,e)\big|\leq\rho(e)(|x_1-x_2|+|y_1-y_2|)+L_{f_i}(e)(|z^1_1-z^1_2|+|z^2_1-z^2_2|+|\tilde{z}^1_1-\tilde{z}^1_2|+|\tilde{z}^2_1-\tilde{z}^2_2|).
\end{equation*}

Next, we first prove the existence and uniqueness of the solution to \eqref{fbsdep1}.

\begin{theorem}\label{the21}
Suppose that {\bf (A4)} holds, then there exists a constant $\tilde{T}>0$ only depending on $K,L_{\sigma_1},L_{\sigma_2},L_{f_1},L_{f_2},\rho$ such that, for every $0<T<\tilde{T}$, the fully coupled FBSDEP \eqref{fbsdep1} has a unique solution in $\mathcal{M}^2[0,T]$.
\end{theorem}

\begin{proof}
For any $(y^\prime,z^{1,\prime},z^{2,\prime},\tilde{z}^{1,\prime},\tilde{z}^{2,\prime})\in \mathcal{N}^2[0,T]$, it is classical that there exists a unique solution $(y,z^1,z^2,\tilde{z}^1,\tilde{z}^2)\in \mathcal{N}^2[0,T]$ to the following decoupled FBSDEP:
\begin{equation}\label{fbsdep2}
\left\{
\begin{aligned}
 dx_t&=b(t,x_t,y^\prime_t,z^{1,\prime}_t,z^{2,\prime}_t,\tilde{z}^{1,\prime}_{(t,e)},\tilde{z}^{2,\prime}_{(t,e)})dt+\sum_{i=1}^2\sigma_i(t,x_t,y^\prime_t,z^{1,\prime}_t,z^{2,\prime}_t,\tilde{z}^{1,\prime}_{(t,e)},\tilde{z}^{2,\prime}_{(t,e)})dW^i_t\\
     &\quad +\sum_{i=1}^2\int_{\mathcal{E}_i}f_i(t,x_t,y^\prime_t,z^{1,\prime}_t,z^{2,\prime}_t,\tilde{z}^{1,\prime}_{(t,e)},\tilde{z}^{2,\prime}_{(t,e)},e)\tilde{N}_i(de,dt),\\
-dy_t&=g(t,\Theta(t))dt-\sum_{i=1}^2z^i_tdW^i_t-\sum_{i=1}^2\int_{\mathcal{E}_i}\tilde{z}^i_{(t,e)}\tilde{N}_i(de,dt),\quad t\in[0,T],\\
  x_0&=x_0,\ \ y_T=\phi(x_T).
\end{aligned}
\right.
\end{equation}
Define a map $\mathbb{I}_1$ from $\mathcal{N}^2[0,T]$ to itself. Let $(y^\prime_i,z^{1,\prime}_i,z^{2,\prime}_i,\tilde{z}^{1,\prime}_i,\tilde{z}^{2,\prime}_i)\in \mathcal{N}^2[0,T]$ and $(y_i,z^1_i,z^2_i,\tilde{z}^1_i,\tilde{z}^2_i)\\=\mathbb{I}_1(y^\prime_i,z^{1,\prime}_i,z^{2,\prime}_i,\tilde{z}^{1,\prime}_i,\tilde{z}^{2,\prime}_i),i=1,2$ and define
\begin{equation}
\begin{aligned}
&\hat{y}^\prime=y^\prime_1-y^\prime_2,\ \hat{z}^{1,\prime}=z^{1,\prime}_1-z^{1,\prime}_2,\ \hat{z}^{2,\prime}=z^{2,\prime}_1-z^{2,\prime}_2,\ \hat{\tilde{z}}^{1,\prime}=\tilde{z}^{1,\prime}_1-\tilde{z}^{1,\prime}_2,\ \hat{\tilde{z}}^{2,\prime}=\tilde{z}^{2,\prime}_1-\tilde{z}^{2,\prime}_2\\
&\hat{y}=y_1-y_2,\ \hat{z}^1=z^1_1-z^1_2,\ \hat{z}^2=z^2_1-z^2_2,\ \hat{\tilde{z}}^1=\tilde{z}^1_1-\tilde{z}^1_2,\ \hat{\tilde{z}}^2=\tilde{z}^2_1-\tilde{z}^2_2,\ \hat{x}=x_1-x_2.
\end{aligned}
\end{equation}
By {\bf (A4)}, we get (We omit some time variables if there is no ambiguity.)
\begin{equation}
\begin{aligned}
&|\hat{x}_t|^2\leq C\bigg[\Big|\int_0^tL_b\big(|\hat{x}|+|\hat{y}^\prime|+|\hat{z}^{1,\prime}|+|\hat{z}^{2,\prime}|+||\hat{\tilde{z}}^{1,\prime}||+||\hat{\tilde{z}}^{2,\prime}||\big)dt\Big|^2\\
&\quad +\sum_{i=1}^2\Big|\int_0^t\big[K(|\hat{x}|+|\hat{y}^\prime|)+L_{\sigma_i}(|\hat{z}^{1,\prime}|+|\hat{z}^{2,\prime}|+||\hat{\tilde{z}}^{1,\prime}||+||\hat{\tilde{z}}^{2,\prime}||)\big]dW^i_t\Big|^2\\
&\quad +\sum_{i=1}^2\Big|\int_0^t\int_{\mathcal{E}_i}\big[\rho(e)(|\hat{x}|+|\hat{y}^\prime|)+L_{f_i}(e)(|\hat{z}^{1,\prime}|+|\hat{z}^{2,\prime}|+|\hat{\tilde{z}}^{1,\prime}|
 +|\hat{\tilde{z}}^{2,\prime}|)\big]\tilde{N}_i(de,dt)\Big|^2\bigg].
\end{aligned}
\end{equation}
Applying B-D-G's inequality and H\"{o}lder's inequality, we obtain
\begin{equation}
\begin{aligned}
&E\Big[\sup_{0\leq t\leq T}|\hat{x}_t|^2\Big]\leq CT E\bigg[\int_0^T\big(|\hat{x}|^2+|\hat{y}^\prime|^2+|\hat{z}^{1,\prime}|^2+|\hat{z}^{2,\prime}|^2+||\hat{\tilde{z}}^{1,\prime}||^2+||\hat{\tilde{z}}^{2,\prime}||^2\big)dt\bigg]\\
&\quad +C E\bigg[\int_0^T\big[K^2(|\hat{x}|^2+|\hat{y}^\prime|^2)+(L_{\sigma_1}^2+L_{\sigma_2}^2)(|\hat{z}^{1,\prime}|^2+|\hat{z}^{2,\prime}|^2+||\hat{\tilde{z}}^{1,\prime}||^2+||\hat{\tilde{z}}^{2,\prime}||^2)\big]dt\bigg]\\
&\quad +C\sum_{i=1}^2E\bigg[\int_0^T\int_{\mathcal{E}_i}\big[\rho^2(e)(|\hat{x}|^2+|\hat{y}^\prime|^2)+L_{f_i}^2(e)(|\hat{z}^{1,\prime}|^2
 +|\hat{z}^{2,\prime}|^2+|\hat{\tilde{z}}^{1,\prime}|^2+|\hat{\tilde{z}}^{2,\prime}|^2)\big]N_i(de,dt)\bigg]\\
&\leq CT E\bigg[T\sup_{0\leq t\leq T}(|\hat{x}|^2+|\hat{y}^\prime|^2)+\sum_{i=1}^2\bigg(\int_0^T|\hat{z}^{i,\prime}|^2dt+\int_0^T\int_{\mathcal{E}_i}|\hat{\tilde{z}}^{i,\prime}|^2\nu_i(de)dt\bigg)\bigg]\\
&\quad +C E\bigg[TK^2\sup_{0\leq t\leq T}(|\hat{x}|^2+|\hat{y}^\prime|^2)+(L_{\sigma_1}^2+L_{\sigma_2}^2)\sum_{i=1}^2\bigg(\int_0^T|\hat{z}^{i,\prime}|^2dt +\int_0^T\int_{\mathcal{E}_i}|\hat{\tilde{z}}^{i,\prime}|^2\nu_i(de)dt\bigg)\bigg]\\
&\quad +C E\bigg[T\sum_{i=1}^2\int_{\mathcal{E}_i}\rho^2(e)\nu_i(de)\sup_{0\leq t\leq T}(|\hat{x}|^2+|\hat{y}^\prime|^2)+\bigg(\sum_{i=1}^2\int_{\mathcal{E}_i}L_{f_i}^2(e)\nu_i(de)\bigg)\sum_{i=1}^2\int_0^T|\hat{z}^{i,\prime}|^2dt\\
&\qquad\qquad +\bigg(\sum_{i=1}^2\sup_{e\in\mathcal{E}_i}L_{f_i}^2(e)\bigg)\sum_{i=1}^2\int_0^T\int_{\mathcal{E}_i}|\hat{\tilde{z}}^{i,\prime}|^2\nu_i(de)dt\bigg]\\
&\leq \bigg[CT^2+CTK^2+CT\sum_{i=1}^2\int_{\mathcal{E}_i}\rho^2(e)\nu_i(de)\bigg]E\Big[\sup_{0\leq t\leq T}(|\hat{x}|^2+|\hat{y}^\prime|^2)\Big]\\
&\quad +\bigg[CT+C(L_{\sigma_1}^2+L_{\sigma_2}^2)+C\sum_{i=1}^2\int_{\mathcal{E}_i}L_{f_i}^2(e)\nu_i(de)\bigg]\sum_{i=1}^2E\int_0^T|\hat{z}^{i,\prime}|^2dt\\
&\quad +\bigg[CT+C(L_{\sigma_1}^2+L_{\sigma_2}^2)+C\sum_{i=1}^2\sup_{e\in\mathcal{E}_i}L_{f_i}^2(e)\bigg]\sum_{i=1}^2E\int_0^T\int_{\mathcal{E}_i}|\hat{\tilde{z}}^{i,\prime}|^2\nu_i(de)dt.
\end{aligned}
\end{equation}
Then we have
\begin{equation}
\begin{aligned}
&\Big\{1-\Big[CT^2+CTK^2+CT\sum_{i=1}^2\int_{\mathcal{E}_i}\rho^2(e)\nu_i(de)\Big]\Big\}E\Big[\sup_{0\leq t\leq T}|\hat{x}_t|^2\Big]\\
&\leq \Big[CT^2+CTK^2+CT\sum_{i=1}^2\int_{\mathcal{E}_i}\rho^2(e)\nu_i(de)\Big]E\Big[\sup_{0\leq t\leq T}|\hat{y}^\prime|^2\Big]+\Big[CT+C(L_{\sigma_1}^2+L_{\sigma_2}^2)\\
&\quad +C\sum_{i=1}^2\int_{\mathcal{E}_i}L_{f_i}^2(e)\nu_i(de)\Big]\sum_{i=1}^2E\int_0^T|\hat{z}^{i,\prime}|^2dt\\
&\quad+\Big[CT+C(L_{\sigma_1}^2+L_{\sigma_2}^2)+C\sum_{i=1}^2\sup_{e\in\mathcal{E}_i}L_{f_i}^2(e)\Big]\sum_{i=1}^2E\int_0^T\int_{\mathcal{E}_i}|\hat{\tilde{z}}^{i,\prime}|^2\nu_i(de)dt,
\end{aligned}
\end{equation}
and there exists $T_0$ such that $CT_0^2+CT_0K^2+CT_0\sum_{i=1}^2\int_{\mathcal{E}_i}\rho^2(e)\nu_i(de)<1$, we have, for $0<T<T_0$
\begin{equation}\label{hatx}
\begin{aligned}
&E\Big[\sup_{0\leq t\leq T}|\hat{x}_t|^2\Big]
\leq \frac{CT^2+CTK^2+CT\sum_{i=1}^2\int_{\mathcal{E}_i}\rho^2(e)\nu_i(de)}{1-\big[CT^2+CTK^2+CT\sum_{i=1}^2\int_{\mathcal{E}_i}\rho^2(e)\nu_i(de)\big]}E\Big[\sup_{0\leq t\leq T}|\hat{y}^\prime|^2\Big]\\
&\quad +\frac{\Big[CT+C(L_{\sigma_1}^2+L_{\sigma_2}^2)+C(\tilde{C}_{f_1}+\tilde{C}_{f_2})\Big]}{1-\big[CT^2+CTK^2+CT\sum_{i=1}^2\int_{\mathcal{E}_i}\rho^2(e)\nu_i(de)\big]}\sum_{i=1}^2E\int_0^T|\hat{z}^{i,\prime}|^2dt\\
&\quad +\frac{\Big[CT+C(L_{\sigma_1}^2+L_{\sigma_2}^2)+C(\tilde{C}_{f_1}+\tilde{C}_{f_2})\Big]}{1-\big[CT^2+CTK^2+CT\sum_{i=1}^2\int_{\mathcal{E}_i}\rho^2(e)\nu_i(de)\big]}\sum_{i=1}^2E\int_0^T\int_{\mathcal{E}_i}|\hat{\tilde{z}}^{i,\prime}|^2\nu_i(de)dt.
\end{aligned}
\end{equation}
Noting that
\begin{equation}
\begin{aligned}
-d\hat{y}_t&=\big[g(t,\Theta_1)-g(t,\Theta_2)\big]dt-\sum_{i=1}^2(z^i_1-z^i_2)dW^i_t-\sum_{i=1}^2\int_{\mathcal{E}_i}(\tilde{z}^i_1-\tilde{z}^i_2)\tilde{N}_i(de,dt)\\
           &=\big[g(t,\Theta_1)-g(t,x_2,y_1-(y_1-y_2),z^1_1-(z^1_1-z^1_2),z^2_1-(z^2_1-z^2_2),\tilde{z}^1_1-(\tilde{z}^1_1-\tilde{z}^1_2),\\
           &\quad\tilde{z}^2_1-(\tilde{z}^2_1-\tilde{z}^2_2))\big]dt -\sum_{i=1}^2(z^i_1-z^i_2)dW^i_t-\sum_{i=1}^2\int_{\mathcal{E}_i}(\tilde{z}^i_1-\tilde{z}^i_2)\tilde{N}_i(de,dt)\\
           &=\big[g(t,\Theta_1)-g(t,x_2,y_1-\hat{y},z^1_1-\hat{z}^1,z^2_1-\hat{z}^2,\tilde{z}^1_1-\hat{\tilde{z}}^1,\tilde{z}^2_1-\hat{\tilde{z}}^2)\big]dt\\
           &\quad-\sum_{i=1}^2\hat{z}^idW^i_t-\sum_{i=1}^2\int_{\mathcal{E}_i}\hat{\tilde{z}}^i\tilde{N}_i(de,dt),
\end{aligned}
\end{equation}
by standard estimates for solutions to BSDEPs (see \cite{BBP97}, \cite{KP16}, \cite{QS13} and its derivative under our assumptions, which can be regarded as a special case of Theorem 2.3 in the following), we have
\begin{equation}
\begin{aligned}
&E\bigg[\sup_{0\leq t\leq T}|\hat{y}_t|^2+\sum_{i=1}^2\bigg(\int_0^T|\hat{z}^i_t|^2dt+\int_0^T\int_{\mathcal{E}_i}|\hat{\tilde{z}}^i_{(t,e)}|^2\nu_i(de)dt\bigg)\bigg]\\
&\leq C E\bigg[\big|\phi(x_{1,T})-\phi(x_{2,T})\big|^2+\int_0^T\big|g(t,\Theta_1)-g(t,x_2,y_1,z^1_1,z^2_1,\tilde{z}^1_1,\tilde{z}^2_1)\big|^2dt\bigg]\\
&\leq C E\bigg[|C\hat{x}_T|^2+\int_0^TC^2|\hat{x}|^2dt\bigg]\leq C(1+T)E\bigg[\sup_{0\leq t\leq T}|\hat{x}_t|^2\bigg].
\end{aligned}
\end{equation}
Combined with \eqref{hatx}, we have
\begin{equation}
\begin{aligned}
&E\bigg[\sup_{0\leq t\leq T}|\hat{y}_t|^2+\sum_{i=1}^2\bigg(\int_0^T|\hat{z}^i_t|^2dt+\int_0^T\int_{\mathcal{E}_i}|\hat{\tilde{z}}^i_{(t,e)}|^2\nu_i(de)dt\bigg)\bigg]\\
&\leq C(1+T)\frac{C(T+T^2)+TK^2+L_{\sigma_1}^2+L_{\sigma_2}^2+\tilde{C}_{f_1}+\tilde{C}_{f_2}+T\sum_{i=1}^2\int_{\mathcal{E}_i}\rho^2(e)\nu_i(de)}{1-\big[CT^2+CTK^2+CT\sum_{i=1}^2\int_{\mathcal{E}_i}\rho^2(e)\nu_i(de)\big]}\\
&\quad \times E\bigg[\sup_{0\leq t\leq T}|\hat{y}^\prime|^2+\sum_{i=1}^2\bigg(\int_0^T|\hat{z}^{i,\prime}|^2dt+\int_0^T\int_{\mathcal{E}_i}|\hat{\tilde{z}}^{i,\prime}_{(t,e)}|^2\nu_i(de)dt\bigg)\bigg].
\end{aligned}
\end{equation}
We take Lipschitz coefficients $L_{\sigma_1},L_{\sigma_2},\tilde{C}_{f_1}$ and $\tilde{C}_{f_2}$ small sufficient, $C$ is a constant not related to $T$ but changed every step, therefore we can choose $0<\tilde{T}\leq T_0$ such that
\begin{equation}
 C(1+\tilde{T})\frac{C(\tilde{T}+\tilde{T}^2)+\tilde{T}K^2+L_{\sigma_1}^2+L_{\sigma_2}^2+\tilde{C}_{f_1}+\tilde{C}_{f_2}+\tilde{T}\sum_{i=1}^2\int_{\mathcal{E}_i}\rho^2(e)\nu_i(de)}{1-\big[C\tilde{T}^2+C\tilde{T}K^2+C\tilde{T}\sum_{i=1}^2\int_{\mathcal{E}_i}\rho^2(e)\nu_i(de)\big]}<1.
\end{equation}
Then, for any $0\leq T\leq\tilde{T}$, $\mathbb{I}_1$ is a contraction mapping on $[0,T]$ which has a unique fixed point $\mathbb{I}_1(y,z^1,z^2,\tilde{z}^1,\tilde{z}^2)=(y,z^1,z^2,\tilde{z}^1,\tilde{z}^2)$. Furthermore, FBSDEP \eqref{fbsdep1} has a unique solution $(x,y,z^1,z^2,\tilde{z}^1,\tilde{z}^2)$ on small interval $[0,T]$.  The proof is complete.
\end{proof}

\begin{remark}\label{rem24}
Theorem 2.1 only gives the existence and uniqueness of solutions to fully coupled FBSDEP when $\beta=2$. In fact, we encounter some difficulties in obtaining the counterpart in $\mathcal{M}^\beta[0,T]$ for $\beta>2$. However, a special fully coupled FBSDEP, where the jump coefficients do not contain $(z^1,z^2,\tilde{z}^1,\tilde{z}^2)$, can be proved to admit a unique solution $(x,y,z^1,z^2,\tilde{z}^1,\tilde{z}^2)$ in $\mathcal{M}^\beta[0,T]$ for $\beta\geq2$, and the proof is similar, we only list the main theorem in the following and omit the details.
\end{remark}

\begin{theorem}\label{the22}
Suppose that {\bf (A4)} holds, for any $\beta\geq2$, then there exists a constant $\tilde{T}>0$ depending on $K,L_{\sigma_1},L_{\sigma_2},\rho$ such that, for every $0\leq T\leq \tilde{T}$, the fully coupled FBSDEP \eqref{fbsdep1}, in which the jump coefficients do not contain $(z^1,z^2,\tilde{z}^1,\tilde{z}^2)$, has a unique solution in $\mathcal{M}^\beta[0,T]$.
\end{theorem}

\begin{remark}\label{rem25}
As we mentioned before, we fail to obtain the unique solvability of fully-coupled FBSDEP in $\mathcal{M}^\beta[0,T]$ for $\beta>2$, especially for jump coefficient $f_i$ containing $(z^1,z^2,\tilde{z}^1,\tilde{z}^2)$. More specificly, the critical obstacle appears in the norm defined in $F_i^{2,\beta}[0,T]$, which is not supportive enough for us to get the $L^\beta$-estimate for BSDEP for $\beta>2$. Therefore, for obtaining a better result for a more general fully-coupled FBSDEP ($f_i$ can contain $\tilde{z}^1,\tilde{z}^2$ but is independent of $z^1,z^2$), we should change the space $F_i^{2,\beta}[0,T]$ with an another norm.
\end{remark}

Next, we introduce an another kind of necessary solution space as follows to replace the original space $F_i^{2,\beta}[0,T]$:
\begin{equation*}
\tilde{F}_i^{2,\beta}[0,T]:=\bigg\{\tilde{z}^i\big|\tilde{z}^i\mbox{ is $\mathcal{E}_i$-predictable and }E\bigg[\bigg(\int_0^T\int_{\mathcal{E}_i}|\tilde{z}^i_{(t,e)}|^2N_i(de,dt)\bigg)^{\frac{\beta}{2}}\bigg]<\infty\bigg\},\ i=1,2,
\end{equation*}
with norm $||\tilde{z}^i||_{2,\beta}^\beta:=E\big[\big(\int_0^T\int_{\mathcal{E}_i}|\tilde{z}^i_{(t,e)}|^2N_i(de,dt)\big)^{\frac{\beta}{2}}\big]$.

We set $\tilde{\mathcal{M}}^\beta[0,T]:=S^{\beta}[0,T]\times S^{\beta}[0,T]\times M_1^{2,\beta}[0,T]\times M_2^{2,\beta}[0,T]\times \tilde{F}_1^{2,\beta}[0,T]\times \tilde{F}_2^{2,\beta}[0,T]$ and $\tilde{\mathcal{N}}^\beta[0,T]:=S^{\beta}[0,T]\times M_1^{2,\beta}[0,T]\times M_2^{2,\beta}[0,T]\times \tilde{F}_1^{2,\beta}[0,T]\times \tilde{F}_2^{2,\beta}[0,T]$. Specially, for $\beta=2$, it is easy to find that $\tilde{F}_i^{2,\beta}[0,T]=F_i^{2,\beta}[0,T], \tilde{\mathcal{M}}^\beta[0,T]=\mathcal{M}^\beta[0,T]$ and $\tilde{\mathcal{N}}^\beta[0,T]=\mathcal{N}^\beta[0,T]$.

{\bf (A5)}\quad For any $t\in[0,T]$, $(x,y,z^1,z^2,\tilde{z}^1,\tilde{z}^2)\in R^6$, the jump coefficient $f_i$ is independent of $(z^1,z^2)$ for $i=1,2$.

\begin{theorem}\label{the23}
Suppose that {\bf (A4), (A5)} hold, then there exists a constant $\tilde{T}>0$ depending on $K,L_{\sigma_1},L_{\sigma_2}, L_{f_1},L_{f_2}, \rho$ such that, for every $0\leq T\leq\tilde{T}$ and $\beta\geq2$, the fully coupled FBSDEP \eqref{fbsdep1} has a unique solution in $\tilde{\mathcal{M}}^{\beta}[0,T]$.
\end{theorem}

Before we prove the Theorem \ref{the23}, we should give the following necessary lemmas about the existence and uniqueness of solution and a priori estimate of BSDEP.

\begin{lemma}\label{lemma22}
Suppose that {\bf (A4)} hold, then for $\beta\geq2$, the following BSDEP
\begin{equation}\label{bsdep01}
\left\{
\begin{aligned}
-dy_t&=g(t,y_t,z^1_t,z^2_t,\tilde{z}^1_{(t,e)},\tilde{z}^2_{(t,e)})dt-\sum_{i=1}^2z^i_tdW^i_t-\sum_{i=1}^2\int_{\mathcal{E}_i}\tilde{z}^i_{(t,e)}\tilde{N}_i(de,dt),\quad t\in[0,T],\\
y_T&=\xi.
\end{aligned}
\right.
\end{equation}
has a unique solution in $\tilde{\mathcal{N}}^{\beta}[0,T]$.
\end{lemma}
\begin{proof}
We first prove the case that $g$ dose not contain $(y,z^1,z^2,\tilde{z}^1,\tilde{z}^2)$, then $y$ is given by the right-continuous version of $y_t=E\big[\xi+\int_t^Tg(s)ds|\mathcal{F}_t\big]$. Thus, for $\beta\geq2$, the martingale representation theorem for locally square integrable martingale in \cite{TL94} implies that there exist unique predictable, $\mathcal{E}_1$-predictable and $\mathcal{E}_2$-predictable $(z^1,z^2,\tilde{z}^1,\tilde{z}^2)\in M_1^{2,\beta}[0,T]\times M_2^{2,\beta}[0,T]\times \tilde{F}_1^{2,\beta}[0,T]\times \tilde{F}_2^{2,\beta}[0,T]$ satisfying:
\begin{equation}
E\bigg[\xi+\int_0^Tg(s)ds\bigg|\mathcal{F}_t\bigg]=y_0+\sum_{i=1}^2\bigg(\int_0^tz^i_sdW_s+\int_0^t\int_{\mathcal{E}_i}\tilde{z}^i_{(s,e)}\tilde{N}_i(de,ds)\bigg),\ P\mbox{-}a.s.
\end{equation}
Next, we have $|y_t|=E\big[|\xi|+\int_0^T|g(s)|ds|\mathcal{F}_t\big]$, and using martingale inequality (Ikeda and Watanabe \cite{IW81}), we get
\begin{equation}
E\bigg[\sup_{0\leq t\leq T}|y_t|^\beta\bigg]\leq C_\beta E\bigg[\bigg(|\xi|+\int_0^T|g(s)|ds\bigg)^\beta\bigg],
\end{equation}
where $C_\beta$ is a positive constant which depends only on the number $\beta$ and not related to $T$. Due to {\bf (A4)}, we continue to deduce that
\begin{equation}
E\bigg[\sup_{0\leq t\leq T}|y_t|^\beta\bigg]\leq C_\beta E\bigg[|\xi|^\beta+\bigg(\int_0^T|g(s)|ds\bigg)^\beta\bigg]=C_\beta \bigg[E|\xi|^\beta+E\bigg(\int_0^T|g(s)|ds\bigg)^\beta\bigg]<\infty,
\end{equation}
and by B-D-G's inequality, we have
\begin{equation*}
\begin{aligned}
&E\bigg[\bigg(\sum_{i=1}^2\bigg(\int_0^T|z^i_s|^2ds+\int_0^T\int_{\mathcal{E}_i}|\tilde{z}^i_{(s,e)}|^2N_i(de,ds)\bigg)\bigg)^{\frac{\beta}{2}}\bigg]\\
&\leq c_\beta E\bigg[\bigg(\sum_{i=1}^2\bigg|\int_0^Tz^i_sdW^i_s+\int_0^T\int_{\mathcal{E}_i}\tilde{z}^i_{(s,e)}\tilde{N}_i(de,ds)\bigg|\bigg)^\beta\bigg]\\
&\leq c_\beta E\bigg[\bigg|\xi+\int_0^Tg(s)ds-y_0\bigg|^\beta\bigg]\leq c_{\beta}E\bigg[|\xi|^{\beta}+\bigg(\int_0^T|g(s)|ds\bigg)^\beta\bigg],
\end{aligned}
\end{equation*}
where $c_\beta$ changed every step. Then we define a mapping $\mathcal{X}:\tilde{\mathcal{N}}^\beta[0,T]\rightarrow \tilde{\mathcal{N}}^\beta[0,T]$ and maps $(y^\prime,z^{1,\prime},z^{2,\prime},\tilde{z}^{1,\prime},\tilde{z}^{2,\prime})$ into the solution $(y,z^1,z^2,\tilde{z}^1,\tilde{z}^2)$ of the BSDEP associated with the generator $g(t,y^\prime,z^{1,\prime},z^{2,\prime},\tilde{z}^{1,\prime},\tilde{z}^{2,\prime})$. Given $(y^{1,\prime},z^{1,1,\prime},z^{2,1,\prime},\tilde{z}^{1,1,\prime},\tilde{z}^{2,1,\prime})$, $(y^{2,\prime},z^{1,2,\prime},z^{2,2,\prime},\tilde{z}^{1,2,\prime},\tilde{z}^{2,2,\prime})\in \tilde{\mathcal{N}}^\beta[0,T]$, define $(y^1,z^{1,1},z^{2,1},\tilde{z}^{1,1},\tilde{z}^{2,1})=\mathcal{X}(y^{1,\prime},z^{1,1,\prime},z^{2,1,\prime},\tilde{z}^{1,1,\prime},\tilde{z}^{2,1,\prime})$ and $(y^2,z^{1,2},z^{2,2},\tilde{z}^{1,2},\tilde{z}^{2,2})\\=\mathcal{X}(y^{2,\prime},z^{1,2,\prime},z^{2,2,\prime},\tilde{z}^{1,2,\prime},\tilde{z}^{2,2,\prime})$, and set
\begin{equation}
\begin{aligned}
&\tilde{y}^\prime=y^{1,\prime}-y^{2,\prime},\ \tilde{z}^{1,\prime}=z^{1,1,\prime}-z^{1,2,\prime},\ \tilde{z}^{2,\prime}=z^{2,1,\prime}-z^{2,2,\prime},\ \tilde{\tilde{z}}^{1,\prime}=\tilde{z}^{1,1,\prime}-\tilde{z}^{1,2,\prime},\ \tilde{\tilde{z}}^{2,\prime}=\tilde{z}^{2,1,\prime}-\tilde{z}^{2,2,\prime},\\
&\tilde{y}=y^1-y^2,\ \tilde{z}^1=z^{1,1}-z^{1,2},\ \tilde{z}^2=z^{2,1}-z^{2,2},\ \tilde{\tilde{z}}^1=\tilde{z}^{1,1}-\tilde{z}^{1,2},\ \tilde{\tilde{z}}^2=\tilde{z}^{2,1}-\tilde{z}^{2,2},
\end{aligned}
\end{equation}
where $(\tilde{y},\tilde{z}^1,\tilde{z}^2,\tilde{\tilde{z}}^1,\tilde{\tilde{z}}^2)$ satisfies
\begin{equation}
\begin{aligned}
\tilde{y}_t&=\int_t^T\big[g(t,y^{1,\prime},z^{1,1,\prime},z^{2,1,\prime},\tilde{z}^{1,1,\prime},\tilde{z}^{2,1,\prime})-g(t,y^{2,\prime},z^{1,2,\prime},z^{2,2,\prime},\tilde{z}^{1,2,\prime},\tilde{z}^{2,2,\prime})\big]ds\\
&\quad-\sum_{i=1}^2\int_t^T\tilde{z}^i_sdW^i_s-\sum_{i=1}^2\int_t^T\int_{\mathcal{E}_i}\tilde{\tilde{z}}^i_{s,e}\tilde{N}_i(de,ds),
\end{aligned}
\end{equation}
and the terminal value is equal to zero. For fixed $(y^{1,\prime},z^{1,1,\prime},z^{2,1,\prime},\tilde{z}^{1,1,\prime},\tilde{z}^{2,1,\prime})$ and \\$(y^{2,\prime},z^{1,2,\prime},z^{2,2,\prime},\tilde{z}^{1,2,\prime},\tilde{z}^{2,2,\prime})$, by the above estimate, we get
\begin{equation}
\begin{aligned}
&E\bigg[\sup_{0\leq t\leq T}|\tilde{y}_t|^\beta+\sum_{i=1}^2\bigg(\bigg(\int_0^T|\tilde{z}^i_s|^2ds\bigg)^{\frac{\beta}{2}}+\bigg(\int_0^T\int_{\mathcal{E}_i}|\tilde{\tilde{z}}^i|^2N_i(de,ds)\bigg)^{\frac{\beta}{2}}\bigg)\bigg]\\
&\leq C_\beta E\bigg[\bigg(\int_0^T\big|g(t,y^{1,\prime},z^{1,1,\prime},z^{2,1,\prime},\tilde{z}^{1,1,\prime},\tilde{z}^{2,1,\prime})-g(t,y^{2,\prime},z^{1,2,\prime},z^{2,2,\prime},\tilde{z}^{1,2,\prime},\tilde{z}^{2,2,\prime})\big|ds\bigg)^\beta\bigg]\\
&\leq C_\beta E\bigg[\bigg(\int_0^TC_g(|\tilde{y}^\prime|+|\tilde{z}^{1,\prime}|+|\tilde{z}^{2,\prime}|+||\tilde{\tilde{z}}^{1,\prime}||+||\tilde{\tilde{z}}^{2,\prime}||)ds\bigg)^\beta\bigg]\\
&\leq T^{\frac{\beta}{2}}C_\beta E\bigg[\bigg(\int_0^TC_g^2(|\tilde{y}^\prime|^2+|\tilde{z}^{1,\prime}|^2+|\tilde{z}^{2,\prime}|^2+||\tilde{\tilde{z}}^{1,\prime}||^2+||\tilde{\tilde{z}}^{2,\prime}||^2)ds\bigg)^{\frac{\beta}{2}}\bigg]\\
&\leq T^{\frac{\beta}{2}}C_g^\beta C_\beta E\bigg[T^{\frac{\beta}{2}}\sup_{0\leq t\leq T}|\tilde{y}^\prime|^\beta+\sum_{i=1}^2\bigg(\bigg(\int_0^T|\tilde{z}^{i,\prime}|^2dt\bigg)^{\frac{\beta}{2}}
+\bigg(\int_0^T\int_{\mathcal{E}_i}|\tilde{\tilde{z}}^{i,\prime}|^2N_i(de,dt)\bigg)^{\frac{\beta}{2}}\bigg)\bigg].
\end{aligned}
\end{equation}
Choosing $T$ such that $C_\beta C_g^\beta T^{\frac{\beta}{2}}<1$ and $C_\beta C_g^\beta T^\beta<1$, then the  mapping, from $\tilde{\mathcal{N}}^\beta[0,T]$ to itself, is a contraction, and admits a fixed point which corresponds to the solution to BSDEP in $\tilde{\mathcal{N}}^\beta[0,T]$. For arbitrary $T$, we can obtain the unique solution $(y,z^1,z^2,\tilde{z}^1,\tilde{z}^2)$ by subdividing interval $[0,T]$ into finite number of pieces. The proof is complete.
\end{proof}

Moreover, we need the following a priori estimate.
\begin{lemma}\label{lemma23}
Let $T>0$ and $\xi^1,\xi^2\in L^2_{\mathcal{F}_T}(\Omega)$, and $g^1,g^2$ satisfy {\bf (A4)}, let $(y^1,z^{1,1},z^{2,1},\tilde{z}^{1,1},\tilde{z}^{2,1})$ and $(y^2,z^{1,2},z^{2,2},\tilde{z}^{1,2},\tilde{z}^{2,2})$ be the solutions to the BSDEP \eqref{bsdep01} associated with the terminal conditions $\xi^1,\xi^2$ and the generators $g^1,g^2$, respectively. Then for $\beta\geq2$, there exist some constant $C_{\beta,g,T,\delta}$ depending on $\beta,T,\delta$ and Lipschitz constants $C_g$ such that:
\begin{equation}\label{apriori01}
\begin{aligned}
&E\bigg[\sup_{0\leq t\leq T}|\tilde{y}_t|^\beta\bigg]+E\bigg[\sum_{i=1}^2\bigg(\bigg(\int_0^T|\tilde{z}^i|^2dt\bigg)^{\frac{\beta}{2}}
+\bigg(\int_0^T\int_{\mathcal{E}_i}|\tilde{\tilde{z}}^i|^2N_i(de,dt)\bigg)^{\frac{\beta}{2}}\bigg)\bigg]\\
&\leq C_{\beta,g,T,\delta}E\bigg[|\xi^1-\xi^2|^\beta+\bigg(\int_0^T\big|g^1(s,y^2,z^{1,2},z^{2,2},\tilde{z}^{1,2},\tilde{z}^{2,2})\\
&\qquad\qquad\qquad -g^2(t,y^2,z^{1,2},z^{2,2},\tilde{z}^{1,2},\tilde{z}^{2,2})\big|ds\bigg)^\beta\bigg].
\end{aligned}
\end{equation}
\end{lemma}
\begin{proof}
It is easy to derive that $(\tilde{y},\tilde{z}^1,\tilde{z}^2,\tilde{\tilde{z}}^1,\tilde{\tilde{z}}^2)$ satisfies:
\begin{equation}\label{bsdetildey}
\left\{
\begin{aligned}
-d\tilde{y}_t&=\big[g^1(t,y^1,z^{1,1},z^{2,1},\tilde{z}^{1,1},\tilde{z}^{2,1})-g^2(t,y^2,z^{1,2},z^{2,2},\tilde{z}^{1,2},\tilde{z}^{2,2})\big]dt\\
&\quad-\sum_{i=1}^2\tilde{z}^idW^i_t-\sum_{i=1}^2\int_{\mathcal{E}_i}\tilde{\tilde{z}}^i\tilde{N}_i(de,dt),\quad t\in[0,T],\\
\tilde{y}_T&=\xi^1-\xi^2.
\end{aligned}
\right.
\end{equation}
Applying It\^{o}'s formula to $|\tilde{y}|^2$, we have
\begin{equation}
\begin{aligned}
&|\tilde{y}_0|^2+\sum_{i=1}^2\bigg(\int_0^T|\tilde{z}^i|^2dt+\int_0^T\int_{\mathcal{E}_i}|\tilde{\tilde{z}}^i|^2N_i(de,dt)\bigg)\\
&=|\tilde{y}_T|^2+\int_0^T2\tilde{y}\big[g^1(t,y^1,z^{1,1},z^{2,1},\tilde{z}^{1,1},\tilde{z}^{2,1})-g^2(t,y^2,z^{1,2},z^{2,2},\tilde{z}^{1,2},\tilde{z}^{2,2})\big]dt\\
&\qquad -\sum_{i=1}^2\int_0^T2\tilde{y}\tilde{z}^idW^i_t-\sum_{i=1}^2\int_0^T\int_{\mathcal{E}_i}2\tilde{y}\tilde{\tilde{z}}^i\tilde{N}_i(de,dt),
\end{aligned}
\end{equation}
and
\begin{equation*}
\begin{aligned}
&\sum_{i=1}^2\bigg(\int_0^T|\tilde{z}^i|^2dt+\int_0^T\int_{\mathcal{E}_i}|\tilde{\tilde{z}}^i|^2N_i(de,dt)\bigg)\\
&\leq|\xi^1-\xi^2|^2+\int_0^T2\tilde{y}\big[g^1(t,y^1,z^{1,1},z^{2,1},\tilde{z}^{1,1},\tilde{z}^{2,1})-g^1(t,y^2,z^{1,2},z^{2,2},\tilde{z}^{1,2},\tilde{z}^{2,2})\\
&\quad+g^1(t,y^2,z^{1,2},z^{2,2},\tilde{z}^{1,2},\tilde{z}^{2,2})-g^2(t,y^2,z^{1,2},z^{2,2},\tilde{z}^{1,2},\tilde{z}^{2,2})\big]dt\\
&\quad-\sum_{i=1}^2\int_0^T2\tilde{y}\tilde{z}^idW^i_t-\sum_{i=1}^2\int_0^T\int_{\mathcal{E}_i}2\tilde{y}\tilde{\tilde{z}}^i\tilde{N}_i(de,dt)\\
&\leq |\xi^1-\xi^2|^2+\int_0^T2\tilde{y}\big[C_g(|\tilde{y}|+|\tilde{z}^1|+|\tilde{z}^2|+||\tilde{\tilde{z}}^1||+||\tilde{\tilde{z}}^2||)\big]dt\\
&\quad+\int_0^T2\tilde{y}\big[g^1(t,y^2,z^{1,2},z^{2,2},\tilde{z}^{1,2},\tilde{z}^{2,2})-g^2(t,y^2,z^{1,2},z^{2,2},\tilde{z}^{1,2},\tilde{z}^{2,2})\big]dt\\
&\quad+\sum_{i=1}^2\bigg(\bigg|\int_0^T2\tilde{y}\tilde{z}^idW^i_t\bigg|+\bigg|\int_0^T\int_{\mathcal{E}_i}2\tilde{y}\tilde{\tilde{z}}^i\tilde{N}_i(de,dt)\bigg|\bigg)\\
\end{aligned}
\end{equation*}
\begin{equation}
\begin{aligned}
&\leq |\xi^1-\xi^2|^2+\int_0^T\big[2C_g|\tilde{y}|^2+4C_g^2|\tilde{y}|^2+\frac{1}{2}|\tilde{z}^1|^2+\frac{1}{2}|\tilde{z}^2|^2+\frac{1}{2}||\tilde{\tilde{z}}^1||^2+\frac{1}{2}||\tilde{\tilde{z}}^2||^2\big]dt\\
&\quad+2\sup_{0\leq t\leq T}|\tilde{y}|\int_0^T\big[g^1(t,y^2,z^{1,2},z^{2,2},\tilde{z}^{1,2},\tilde{z}^{2,2})-g^2(t,y^2,z^{1,2},z^{2,2},\tilde{z}^{1,2},\tilde{z}^{2,2})\big]dt\\
&\quad+\sum_{i=1}^2\bigg(\bigg|\int_0^T2\tilde{y}\tilde{z}_idW^i_t\bigg|+\bigg|\int_0^T\int_{\mathcal{E}_i}2\tilde{y}\tilde{\tilde{z}}^i\tilde{N}_i(de,dt)\bigg|\bigg)\\
&\leq |\xi^1-\xi^2|^2+(2C_g+4C_g^2)T\sup_{0\leq t\leq T}|\tilde{y}|^2+\frac{1}{2}\bigg[\sum_{i=1}^2\bigg(\int_0^T|\tilde{z}^i|^2dt+\int_0^T\int_{\mathcal{E}_i}|\tilde{\tilde{z}}^i|^2N_i(de,dt)\bigg)\bigg]\\
&\quad+2\delta\sup_{0\leq t\leq T}|\tilde{y}|^2+\frac{1}{2\delta}\bigg(\int_0^T[g^1(t,y^2,z^{1,2},z^{2,2},\tilde{z}^{1,2},\tilde{z}^{2,2})-g^2(t,y^2,z^{1,2},z^{2,2},\tilde{z}^{1,2},\tilde{z}^{2,2})]dt\bigg)^2\\
&\quad+\sum_{i=1}^2\bigg(\bigg|\int_0^T2\tilde{y}\tilde{z}^idW^i_t\bigg|+\bigg|\int_0^T\int_{\mathcal{E}_i}2\tilde{y}\tilde{\tilde{z}}^i\tilde{N}_i(de,dt)\bigg|\bigg).
\end{aligned}
\end{equation}
Then we get
\begin{equation}\label{ineq2}
\begin{aligned}
&E\bigg[\sum_{i=1}^2\bigg(\bigg(\int_0^T|\tilde{z}^i|^2dt\bigg)^{\frac{\beta}{2}}+\bigg(\int_0^T\int_{\mathcal{E}_i}|\tilde{\tilde{z}}^i|^2N_i(de,dt)\bigg)^{\frac{\beta}{2}}\bigg)\bigg]\\
&\leq C_\beta\bigg\{E|\xi^1-\xi^2|^\beta+(2C_g+4C_g^2)^{\frac{\beta}{2}}T^{\frac{\beta}{2}}E\bigg[\sup_{0\leq t\leq T}|\tilde{y}|^\beta\bigg]+\delta^{\frac{\beta}{2}}E\bigg[\sup_{0\leq t\leq T}|\tilde{y}|^\beta\bigg]\\
&\quad+\frac{1}{\delta^{\frac{\beta}{2}}}E\bigg(\int_0^T[g^1(t,y^2,z^{1,2},z^{2,2},\tilde{z}^{1,2},\tilde{z}^{2,2})-g^2(t,y^2,z^{1,2},z^{2,2},\tilde{z}^{1,2},\tilde{z}^{2,2})]dt\bigg)^\beta\\
&\quad+E\sum_{i=1}^2\bigg|\int_0^T\tilde{y}\tilde{z}^idW^i_t\bigg|^{\frac{\beta}{2}}
+E\sum_{i=1}^2\bigg|\int_0^T\int_{\mathcal{E}_i}\tilde{y}\tilde{\tilde{z}}^i\tilde{N}_i(de,dt)\bigg|^{\frac{\beta}{2}}\bigg\}.
\end{aligned}
\end{equation}
Note that
\begin{equation}
\begin{aligned}
E\bigg|\int_0^T\tilde{y}\tilde{z}^idW^i_t\bigg|^{\frac{\beta}{2}}&\leq CE\bigg(\int_0^T|\tilde{y}|^2|\tilde{z}^i|^2dt\bigg)^{\frac{\beta}{4}}
\leq C E\bigg[\sup_{0\leq t\leq T}|\tilde{y}_t|^{\frac{\beta}{2}}\bigg(\int_0^T|\tilde{z}^i|^2dt\bigg)^{\frac{\beta}{4}}\bigg]\\
&\leq \frac{C^2}{2}E\bigg[\sup_{0\leq t\leq T}|\tilde{y}_t|^\beta\bigg]+\frac{1}{2}E\bigg[\bigg(\int_0^T|\tilde{z}^i|^2dt\bigg)^{\frac{\beta}{2}}\bigg],
\end{aligned}
\end{equation}
and
\begin{equation}
\begin{aligned}
&E\bigg|\int_0^T\int_{\mathcal{E}_i}\tilde{y}\tilde{\tilde{z}}^i\tilde{N}_i(de,dt)\bigg|^{\frac{\beta}{2}}
\leq C E\bigg(\int_0^T\int_{\mathcal{E}_i}|\tilde{y}|^2|\tilde{\tilde{z}}^i|^2N_i(de,dt)\bigg)^{\frac{\beta}{4}}\\
&\leq C E\bigg[\sup_{0\leq t\leq T}|\tilde{y}_t|^{\frac{\beta}{2}}\bigg(\int_0^T\int_{\mathcal{E}_i}|\tilde{\tilde{z}}^i|^2N_i(de,dt)\bigg)^{\frac{\beta}{4}}\bigg]\\
&\leq \frac{C^2}{2}E\bigg[\sup_{0\leq t\leq T}|\tilde{y}_t|^\beta\bigg]+\frac{1}{2}E\bigg[\bigg(\int_0^T\int_{\mathcal{E}_i}|\tilde{\tilde{z}}^i|^2N_i(de,dt)\bigg)^{\frac{\beta}{2}}\bigg].\\
\end{aligned}
\end{equation}
Then the inequality \eqref{ineq2} becomes
\begin{equation}\label{ineq3}
\begin{aligned}
&E\bigg[\sum_{i=1}^2\bigg(\bigg(\int_0^T|\tilde{z}^i|^2dt\bigg)^{\frac{\beta}{2}}+\bigg(\int_0^T\int_{\mathcal{E}_i}|\tilde{\tilde{z}}^i|^2N_i(de,dt)\bigg)^{\frac{\beta}{2}}\bigg)\bigg]\\
&\leq C_\beta\bigg\{E|\xi^1-\xi^2|^\beta+\big[(2C_g+4C_g^2)^{\frac{\beta}{2}}T^{\frac{\beta}{2}}+\delta^{\frac{\beta}{2}}+C^2\big]E\bigg[\sup_{0\leq t\leq T}|\tilde{y}_t|^\beta\bigg]\\
&\quad+\frac{1}{\delta^{\frac{\beta}{2}}}E\bigg(\int_0^T[g^1(t,y^2,z^{1,2},z^{2,2},\tilde{z}^{1,2},\tilde{z}^{2,2})-g^2(t,y^2,z^{1,2},z^{2,2},\tilde{z}^{1,2},\tilde{z}^{2,2})]dt\bigg)^\beta\bigg\}.
\end{aligned}
\end{equation}
Meanwhile, we can obtain from the integral form of \eqref{bsdetildey} that
\begin{equation}
\begin{aligned}
|\tilde{y}|^\beta&\leq C_\beta\bigg\{|\xi^1-\xi^2|^\beta+\bigg|\int_t^T[g^1(s,y^1,z^{1,1},z^{2,1},\tilde{z}^{1,1},\tilde{z}^{2,1})-g^2(t,y^2,z^{1,2},z^{2,2},\tilde{z}^{1,2},\tilde{z}^{2,2})]ds\bigg|^\beta\\
&\qquad +\sum_{i=1}^2\bigg(\bigg|\int_t^T\tilde{z}^idW^i_s\bigg|^\beta+\bigg|\int_t^T\int_{\mathcal{E}_i}\tilde{\tilde{z}}^i\tilde{N}_i(de,ds)\bigg|^\beta\bigg)\bigg\}.\\
\end{aligned}
\end{equation}
By B-D-G's inequality, we have
\begin{equation}
\begin{aligned}
&E\bigg[\sup_{0\leq t\leq T}|\tilde{y}_t|^\beta\bigg]\leq C_\beta\bigg\{E|\xi^1-\xi^2|^\beta+E\bigg|\int_0^T[g^1(s,y^1,z^{1,1},z^{2,1},\tilde{z}^{1,1},\tilde{z}^{2,1})\\
&\qquad -g^2(t,y^2,z^{1,2},z^{2,2},\tilde{z}^{1,2},\tilde{z}^{2,2})]ds\bigg|^\beta+E\sum_{i=1}^2\bigg(\int_0^T|\tilde{z}^i|^2ds\bigg)^{\frac{\beta}{2}}\\
&\qquad +E\sum_{i=1}^2\bigg(\int_0^T\int_{\mathcal{E}_i}|\tilde{\tilde{z}}^i|^2N_i(de,ds)\bigg)^{\frac{\beta}{2}}\bigg\}\\
&\leq C_\beta\bigg\{E|\xi^1-\xi^2|^\beta+E\bigg(\int_0^TC_g(|\tilde{y}|+|\tilde{z}^1|+|\tilde{z}^2|+||\tilde{\tilde{z}}^1||+||\tilde{\tilde{z}}^2||)ds\bigg)^\beta\\
&\qquad +E\bigg(\int_0^T|g^1(s,y^2,z^{1,2},z^{2,2},\tilde{z}^{1,2},\tilde{z}^{2,2})-g^2(t,y^2,z^{1,2},z^{2,2},\tilde{z}^{1,2},\tilde{z}^{2,2})|ds\bigg)^\beta\\
&\qquad +E\sum_{i=1}^2\bigg(\int_0^T|\tilde{z}^i|^2ds\bigg)^{\frac{\beta}{2}}+E\sum_{i=1}^2\bigg(\int_0^T\int_{\mathcal{E}_i}|\tilde{\tilde{z}}^i|^2N_i(de,ds)\bigg)^{\frac{\beta}{2}}\bigg\}.
\end{aligned}
\end{equation}
Substituting \eqref{ineq3}, we can rearrange the above inequality as follows
\begin{equation}
\begin{aligned}
E\bigg[\sup_{0\leq t\leq T}|\tilde{y}_t|^\beta\bigg]&\leq C_{\beta,g,T} E|\xi^1-\xi^2|^\beta+C_{\beta,g, T,\delta}E\bigg(\int_0^T|g^1(s,y^2,z^{1,2},z^{2,2},\tilde{z}^{1,2},\tilde{z}^{2,2})\\
&\quad-g^2(t,y^2,z^{1,2},z^{2,2},\tilde{z}^{1,2},\tilde{z}^{2,2})|ds\bigg)^\beta+C_{\beta,g,T,\delta}E\bigg[\sup_{0\leq t\leq T}|\tilde{y}_t|^\beta\bigg].
\end{aligned}
\end{equation}
Let $1-C_{\beta,g,T,\delta}>0$ and we get
\begin{equation}
\begin{aligned}
E\bigg[\sup_{0\leq t\leq T}|\tilde{y}_t|^\beta\bigg]&\leq C_{\beta,g,T,\delta}\bigg[E|\xi^1-\xi^2|^\beta+E\bigg(\int_0^T|g^1(s,y^2,z^{1,2},z^{2,2},\tilde{z}^{1,2},\tilde{z}^{2,2})\\
&\qquad\qquad\quad -g^2(t,y^2,z^{1,2},z^{2,2},\tilde{z}^{1,2},\tilde{z}^{2,2})|ds\bigg)^\beta\bigg].
\end{aligned}
\end{equation}
Finally, the a priori estimate is given in the following
\begin{equation}
\begin{aligned}
&E\bigg[\sup_{0\leq t\leq T}|\tilde{y}_t|^\beta\bigg]+E\bigg[\sum_{i=1}^2\bigg(\bigg(\int_0^T|\tilde{z}^i|^2dt\bigg)^{\frac{\beta}{2}}
+\bigg(\int_0^T\int_{\mathcal{E}_i}|\tilde{\tilde{z}}^i|^2N_i(de,dt)\bigg)^{\frac{\beta}{2}}\bigg)\bigg]\\
&\leq C_{\beta,g,T,\delta}E\bigg[|\xi^1-\xi^2|^\beta+\bigg(\int_0^T|g^1(s,y^2,z^{1,2},z^{2,2},\tilde{z}^{1,2},\tilde{z}^{2,2})\\
&\qquad\qquad\qquad -g^2(t,y^2,z^{1,2},z^{2,2},\tilde{z}^{1,2},\tilde{z}^{2,2})|ds\bigg)^\beta\bigg].\\
\end{aligned}
\end{equation}
The proof is complete.
\end{proof}

We now back to the proof of Theorem \ref{the23}.
\begin{proof}
For any $(y^\prime,z^{1,\prime},z^{2,\prime},\tilde{z}^{1,\prime},\tilde{z}^{2,\prime})\in \tilde{\mathcal{N}}^\beta[0,T]$, it is easy to get a unique solution $(y,z^1,z^2,\tilde{z}^1,\tilde{z}^2)\\\in \tilde{\mathcal{N}}^\beta[0,T]$ to the decoupled FBSDEP as follows:
\begin{equation}\label{fbsdep01}
\left\{
\begin{aligned}
 dx_t&=b(t,x_t,y^\prime_t,z^{1,\prime}_t,z^{2,\prime}_t,\tilde{z}^{1,\prime}_{(t,e)},\tilde{z}^{2,\prime}_{(t,e)})dt+\sum_{i=1}^2\sigma_i(t,x_t,y^\prime_t,z^{1,\prime}_t,z^{2,\prime}_t,\tilde{z}^{1,\prime}_{(t,e)},\tilde{z}^{2,\prime}_{(t,e)})dW^i_t\\
     &\quad +\sum_{i=1}^2\int_{\mathcal{E}_i}f_i(t,x_t,y^\prime_t,\tilde{z}^{1,\prime}_{(t,e)},\tilde{z}^{2,\prime}_{(t,e)},e)\tilde{N}_i(de,dt),\\
-dy_t&=g(t,\Theta(t))dt-\sum_{i=1}^2z^i_tdW^i_t-\sum_{i=1}^2\int_{\mathcal{E}_i}\tilde{z}^i_{(t,e)}\tilde{N}_i(de,dt),\quad t\in[0,T],\\
  x_0&=x_0,\ \ y_T=\phi(x_T),
\end{aligned}
\right.
\end{equation}
by solving forward SDEP first in $S^\beta[0,T]$ and then solving BSDEP in $\tilde{\mathcal{N}}^\beta[0,T]$, sequentially.
Define a map $\mathcal{X}_1$ from $\tilde{\mathcal{N}}^\beta[0,T]$ to itself. Let $(y^\prime_i,z^{1,\prime}_i,z^{2,\prime}_i,\tilde{z}^{1,\prime}_i,\tilde{z}^{2,\prime}_i)\in \tilde{\mathcal{N}}^\beta[0,T]$ and $(y_i,z^1_i,z^2_i,\tilde{z}^1_i,\tilde{z}^2_i)=\mathcal{X}_1(y^\prime_i,z^{1,\prime}_i,z^{2,\prime}_i,\tilde{z}^{1,\prime}_i,\tilde{z}^{2,\prime}_i),i=1,2$ and define
\begin{equation}
\begin{aligned}
&\hat{y}^\prime=y^\prime_1-y^\prime_2,\ \hat{z}^{1,\prime}=z^{1,\prime}_1-z^{1,\prime}_2,\ \hat{z}^{2,\prime}=z^{2,\prime}_1-z^{2,\prime}_2,\ \hat{\tilde{z}}^{1,\prime}=\tilde{z}^{1,\prime}_1-\tilde{z}^{1,\prime}_2,\ \hat{\tilde{z}}^{2,\prime}=\tilde{z}^{2,\prime}_1-\tilde{z}^{2,\prime}_2,\\
&\hat{y}=y_1-y_2,\ \hat{z}^1=z^1_1-z^1_2,\ \hat{z}^2=z^2_1-z^2_2,\ \hat{\tilde{z}}^1=\tilde{z}^1_1-\tilde{z}^1_2,\ \hat{\tilde{z}}^2=\tilde{z}^2_1-\tilde{z}^2_2,\ \hat{x}=x_1-x_2.
\end{aligned}
\end{equation}
By {\bf (A4)}, we get
\begin{equation}
\begin{aligned}
|\hat{x}_t|^\beta&\leq C_\beta\bigg\{\bigg|\int_0^tC\big(|\hat{x}|+|\hat{y}^\prime|+|\hat{z}^{1,\prime}|+|\hat{z}^{2,\prime}|+||\hat{\tilde{z}}^{1,\prime}||+||\hat{\tilde{z}}^{2,\prime}||\big)ds\bigg|^\beta\\
&\qquad\quad+\sum_{i=1}^2\bigg|\int_0^t\big[K(|\hat{x}|+|\hat{y}^\prime|)+L_{\sigma_i}(|\hat{z}^{1,\prime}|+|\hat{z}^{2,\prime}|+||\hat{\tilde{z}}^{1,\prime}||+||\hat{\tilde{z}}^{2,\prime}||)\big]dW^i_s\bigg|^\beta\\
&\qquad\quad +\sum_{i=1}^2\bigg|\int_0^t\int_{\mathcal{E}_i}\big[\rho(e)(|\hat{x}|+|\hat{y}^\prime|)+L_{f_i}(e)(|\hat{\tilde{z}}^{1,\prime}|+|\hat{\tilde{z}}^{2,\prime}|)\big]\tilde{N}_i(de,ds)\bigg|^\beta\bigg\},
\end{aligned}
\end{equation}
and
\begin{equation}
\begin{aligned}
&E\bigg[\sup_{0\leq t\leq T}|\hat{x}_t|^\beta\bigg]\leq C_\beta\bigg\{(CT^{\frac{\beta}{2}}+L_{\sigma_1}^\beta+L_{\sigma_2}^\beta) \sum_{i=1}^2E\bigg(\int_0^T|\hat{z}^{i,\prime}|^2dt\bigg)^{\frac{\beta}{2}}\\
&\qquad+\big(CT^{\frac{\beta}{2}}+L_{\sigma_1}+L_{\sigma_2}+\tilde{C}_{f_1}+\tilde{C}_{f_2}\big)\sum_{i=1}^2E\bigg(\int_0^T\int_{\mathcal{E}_i}|\hat{\tilde{z}}^{i,\prime}|^2N_i(de,dt)\bigg)^{\frac{\beta}{2}}\\
&\qquad+\bigg(C T^\beta+K^\beta T^{\frac{\beta}{2}}+T\sum_{i=1}^2\int_{\mathcal{E}_i}\rho^\beta(e)\nu_i(de)\bigg)E\bigg[\sup_{0\leq t\leq T}|\hat{x}_t|^\beta+\sup_{0\leq t\leq T}|\hat{y}^\prime_t|^\beta\bigg]\bigg\}.
\end{aligned}
\end{equation}
And there exists a $\hat{T}$ such that $1-C_\beta\bigg(C \hat{T}^\beta+K^\beta\hat{T}^{\frac{\beta}{2}}+\hat{T}\sum_{i=1}^2\int_{\mathcal{E}_i}\rho^\beta(e)\nu_i(de)\bigg)>0$, then we have, for $T\leq\hat{T}$,
\begin{equation}\label{hatxbeta}
\begin{aligned}
&E\bigg[\sup_{0\leq t\leq T}|\hat{x}_t|^\beta\bigg]\leq \frac{C_\beta (C T^{\frac{\beta}{2}}+L_{\sigma_1}^\beta+L_{\sigma_2}^\beta)}{1-C_\beta\bigg(C T^\beta+K^\beta T^{\frac{\beta}{2}}+T\sum_{i=1}^2\int_{\mathcal{E}_i}\rho^\beta(e)\nu_i(de)\bigg)}\sum_{i=1}^2E\bigg(\int_0^T|\hat{z}^{i,\prime}|^2dt\bigg)^{\frac{\beta}{2}}\\
&\quad+\frac{C_\beta(C T^{\frac{\beta}{2}}+L_{\sigma_1}^\beta+L_{\sigma_2}^\beta+\tilde{C}_{f_1}+\tilde{C}_{f_2})}{1-C_\beta\bigg(C T^\beta+K^\beta T^{\frac{\beta}{2}}
+T\sum_{i=1}^2\int_{\mathcal{E}_i}\rho^\beta(e)\nu_i(de)\bigg)}\sum_{i=1}^2E\bigg(\int_0^T\int_{\mathcal{E}_i}|\hat{\tilde{z}}^{i,\prime}|^2N_i(de,dt)\bigg)^{\frac{\beta}{2}}\\
&\quad+\frac{C_\beta\bigg(C T^\beta+K^\beta T^{\frac{\beta}{2}}+T\sum_{i=1}^2\int_{\mathcal{E}_i}\rho^\beta(e)\nu_i(de)\bigg)}{1-C_\beta\bigg(C T^\beta+K^\beta T^{\frac{\beta}{2}}+T\sum_{i=1}^2\int_{\mathcal{E}_i}\rho^\beta(e)\nu_i(de)\bigg)}E\bigg[\sup_{0\leq t\leq T}|\hat{y}^\prime_t|^\beta\bigg].
\end{aligned}
\end{equation}
Next, we can get the BSDEP of $(\hat{y},\hat{z}^1,\hat{z}^2,\hat{\tilde{z}}^1,\hat{\tilde{z}}^2)$ in the following:
\begin{equation}
\left\{
\begin{aligned}
-d\hat{y}_t&=\big[g(t,\Theta_1)-g(t,\Theta_2)\big]dt-\sum_{i=1}^2\hat{z}^i_tdW^i_t-\sum_{i=1}^2\int_{\mathcal{E}_i}\hat{\tilde{z}}^i_{(t,e)}\tilde{N}_i(de,dt),\quad t\in[0,T],\\
\hat{y}_T&=\phi(x_{1T})-\phi(x_{2T}).
\end{aligned}
\right.
\end{equation}
By the a priori estimate \eqref{apriori01} in Lemma \ref{lemma23}, we have
\begin{equation}\label{hatyzzbeta}
\begin{aligned}
&E\bigg[\sup_{0\leq t\leq T}|\hat{y}_t|^\beta\bigg]+E\bigg[\sum_{i=1}^2\bigg(\bigg(\int_0^T|\hat{z}^i|^2dt\bigg)^{\frac{\beta}{2}}+\bigg(\int_0^T\int_{\mathcal{E}_i}|\hat{\tilde{z}}^i|^2N_i(de,dt)\bigg)^{\frac{\beta}{2}}\bigg)\bigg]\\
&\leq C\bigg[E|\phi(x_{1T})-\phi(x_{2T})|^\beta+E\bigg(\int_0^T|g(t,x_1,y_1,z^1_1,z^2_1,\tilde{z}^1_1,\tilde{z}^2_1)-g(t,x_2,y_1,z^1_1,z^2_1,\tilde{z}^1_1,\tilde{z}^2_1)|dt\bigg)^\beta\bigg]\\
&\leq C\bigg[E\big[C\hat{x}_T\big]^\beta+E\bigg(\int_0^T\hat{x}_tdt\bigg)^\beta\bigg]
\leq C\big(1+T^\beta\big)E\bigg[\sup_{0\leq t\leq T}|\hat{x}|^\beta\bigg].
\end{aligned}
\end{equation}
Then substituting \eqref{hatxbeta} into \eqref{hatyzzbeta}, we can obtain
\begin{equation}
\begin{aligned}
&E\bigg[\sup_{0\leq t\leq T}|\hat{y}_t|^\beta\bigg]+E\bigg[\sum_{i=1}^2\bigg(\bigg(\int_0^T|\hat{z}^i|^2dt\bigg)^{\frac{\beta}{2}}
+\bigg(\int_0^T\int_{\mathcal{E}_i}|\hat{\tilde{z}}^i|^2N_i(de,dt)\bigg)^{\frac{\beta}{2}}\bigg)\bigg]\\
&\leq \frac{C_\beta\big(1+T^\beta\big)\Big[C (T^{\frac{\beta}{2}}+T^\beta)+L_{\sigma_1}^\beta+L_{\sigma_2}^\beta+\tilde{C}_{f_1}+\tilde{C}_{f_2}+K^\beta T^{\frac{\beta}{2}}+T\sum_{i=1}^2\int_{\mathcal{E}_i}\rho^\beta(e)\nu_i(de)\Big]}
{1-C_\beta\Big[C T^\beta+K^\beta T^{\frac{\beta}{2}}+T\sum_{i=1}^2\int_{\mathcal{E}_i}\rho^\beta(e)\nu_i(de)\Big]}\\
&\quad \times E\bigg[\sum_{i=1}^2\bigg(\bigg(\int_0^T|\hat{z}^{i,\prime}|^2dt\bigg)^{\frac{\beta}{2}}
+\bigg(\int_0^T\int_{\mathcal{E}_i}|\hat{\tilde{z}}^{i,\prime}|^2N_i(de,dt)\bigg)^{\frac{\beta}{2}}\bigg)+\sup_{0\leq t\leq T}|\hat{y}^\prime_t|^\beta\bigg].
\end{aligned}
\end{equation}
We take Lipschitz coefficients $L_{\sigma_1},L_{\sigma_2},\tilde{C}_{f_1}$ and $\tilde{C}_{f_2}$ small enough, $C$ is a constant not related to $T$ but changed every step. Thus, there exists a $\tilde{T}\leq\hat{T}$ small enough such that
\begin{equation}
\begin{aligned}
\frac{C_\beta\big(1+ \tilde{T}^\beta\big)\Big[C (\tilde{T}^{\frac{\beta}{2}}+\tilde{T}^\beta)+L_{\sigma_1}^\beta+L_{\sigma_2}^\beta+\tilde{C}_{f_1}+\tilde{C}_{f_2}+K^\beta \tilde{T}^{\frac{\beta}{2}}+\tilde{T}\sum_{i=1}^2\int_{\mathcal{E}_i}\rho^\beta(e)\nu_i(de)\Big]}
{1-C_\beta\Big[C \tilde{T}^\beta+K^\beta\tilde{T}^{\frac{\beta}{2}}+\tilde{T}\sum_{i=1}^2\int_{\mathcal{E}_i}\rho^\beta(e)\nu_i(de)\Big]}<1.
\end{aligned}
\end{equation}
Then for any $0\leq T\leq \tilde{T}$, $\mathcal{X}_1$ is a contraction on $[0,T]$, which has a unique fixed point $\mathcal{X}_1(y,z^1,z^2,\tilde{z}^1,\tilde{z}^2)=(y,z^1,z^2,\tilde{z}^1,\tilde{z}^2)$. Moreover, FBSDEP (\ref{fbsdep1}) admits a unique solution $(x,y,z^1,z^2,\tilde{z}^1,\tilde{z}^2)$ on small interval $[0,T]$. The proof is complete.
\end{proof}
 The following assumption is needed.

{\bf (A6)}\quad For any $t\in[0,T]$, $(x,y,z^1,z^2,\tilde{z}^1,\tilde{z}^2)\in R^6$, the jump coefficient $f_i$ is independent of $(z^1,z^2,\tilde{z}^1,\tilde{z}^2)$ for $i=1,2$.

If we consider \eqref{fbsdep1} in a parameterized form by initial condition $(t,\xi)\in[0,T]\times L^2(\Omega,\mathcal{F}_t,P;\\\mathbb{R})$:
\begin{equation}\label{fbsdept1}
\left\{
\begin{aligned}
 dx^{t,\xi}_s&=b(s,\Theta_s^{t,\xi})ds+\sum_{i=1}^2\sigma_i(s,\Theta_s^{t,\xi})dW^i_s+\sum_{i=1}^2\int_{\mathcal{E}_i}f_i(s,\Theta_{s-}^{t,\xi},e)\tilde{N}_i(de,ds),\\
-dy^{t,\xi}_s&=g(s,\Theta_s^{t,\xi})ds-\sum_{i=1}^2z^{i,t,\xi}_sdW^i_s-\sum_{i=1}^2\int_{\mathcal{E}_i}\tilde{z}^{i,t,\xi}_{(s,e)}\tilde{N}_i(de,ds),\quad s\in[t,t+\delta],\\
  x^{t,\xi}_t&=\xi,\ \ y^{t,\xi}_{t+\delta}=\phi(x^{t,\xi}_{t+\delta}),
\end{aligned}
\right.
\end{equation}
where $\Theta_s^{t,\xi}=(x^{t,\xi}_s,y^{t,\xi}_s,z^{1,t,\xi}_s,z^{2,t,\xi}_s,\tilde{z}^{1,t,\xi}_{(s,e)},\tilde{z}^{2,t,\xi}_{(s,e)})$. Then, similar to {\bf Theorem 3.4} in Li and Wei \cite{LW14}, we can also obtain the following estimate with respect to the initial value in a conditional expectation form.

\begin{lemma}\label{Lbetaestimationt}
Let assumption {\bf (A4),(A6)} hold. Then, for every $\beta\geq2$, there exists a sufficiently small constant $\tilde{\delta}$ depending on $(K,L_{\sigma_1},L_{\sigma_2},\rho,L_{f_1},L_{f_2})$ and constant $C_{\beta,K,\rho,\sigma_1,\sigma_2,f_1,f_2,L}$ depending on $(\beta,K,\rho,L_{\sigma_1},L_{\sigma_2},L_{f_1},L_{f_2},L)$ such that for every $0\leq\delta\leq\tilde{\delta}$ and $\xi\in L^\beta(\Omega,\mathcal{F}_t,P;\mathbb{R})$,
\begin{equation}\label{Lbeta estimate1}
\begin{aligned}
&E\bigg[\sup_{t\leq s\leq t+\delta}|x^{t,\xi}_s|^{\beta}+\sup_{t\leq s\leq t+\delta}|y^{t,\xi}_s|^{\beta}+\sum_{i=1}^2\bigg(\int_t^{t+\delta}|z^{i,t,\xi}_s|^2ds\bigg)^{\frac{\beta}{2}}\\
&\quad+\sum_{i=1}^2\bigg(\int_t^{t+\delta}\int_{\mathcal{E}_i}|\tilde{z}^{i,t,\xi}_{(s,e)}|^2\nu_i(de)ds\bigg)^{\frac{\beta}{2}}\bigg|\mathcal{F}_t\bigg]
\leq C_{\beta,K,\rho,\sigma_1,\sigma_2,f_1,f_2,L}(1+|\xi|^{\beta}),\ P\mbox{-}a.s..
\end{aligned}
\end{equation}
\end{lemma}

Motivated by the {\bf Theorem 2.1} in the recent result in \cite{MY21arxiv}, by the critical and thought-provoking technique in their paper, we give an extension result to the case with Poisson jumps. Then we introduce the first main theorem in our paper.
\begin{theorem}\label{MYextension}
Let {\bf (A4),(A6)} hold, and assume that, for every $\xi,\tilde{\xi}\in L^\beta(\Omega,\mathcal{F}_t,P;\mathbb{R})$, the $L^2$-estimations of FBSDEP \eqref{fbsdept1} are given as follows:

$\mathrm{(i)}$
\begin{equation}\label{L2estimation1}
\begin{aligned}
&E\bigg[\sup_{t\leq s\leq T}|x^{t,\xi}_s|^2+\sup_{t\leq s\leq T}|y^{t,\xi}_s|^2+\sum_{i=1}^2\int_t^T|z^{i,t,\xi}_s|^2ds+\sum_{i=1}^2\int_t^T\int_{\mathcal{E}_i}|\tilde{z}^{i,t,\xi}_{(s,e)}|^2\nu_i(de)ds\bigg|\mathcal{F}_t\bigg]\\
&\qquad\leq K_1(1+|\xi|^2),
\end{aligned}
\end{equation}
$\mathrm{(ii)}$
\begin{equation}\label{L2estimation2}
\begin{aligned}
&E\bigg[\sup_{t\leq s\leq T}|x^{t,\xi}_s-x^{t,\tilde{\xi}}_s|^2+\sup_{t\leq s\leq T}|y^{t,\xi}_s-y^{t,\tilde{\xi}}_s|^2+\sum_{i=1}^2\int_t^T|z^{i,t,\xi}_s-z^{i,t,\tilde{\xi}}_s|^2ds\\
&\qquad+\sum_{i=1}^2\int_t^T\int_{\mathcal{E}_i}|\tilde{z}^{i,t,\xi}_{(s,e)}-\tilde{z}^{i,t,\tilde{\xi}}_{(s,e)}|^2\nu_i(de)ds\bigg|\mathcal{F}_t\bigg]\leq K_1|\xi-\tilde{\xi}|^2,
\end{aligned}
\end{equation}
where constant $K_1$ is positive and independent of $t\in[0,T]$.

Then FBSDEP \eqref{fbsdept1} admits a unique $L^\beta(\beta>2)$-solution with $t=0$ and any given terminal time $T$, i.e.,
\begin{equation}\label{Lbetaresult1}
\begin{aligned}
&E\bigg[\sup_{0\leq s\leq T}|x^{0,\xi}_s|^\beta+\sup_{0\leq s\leq T}|y^{0,\xi}_s|^\beta+\sum_{i=1}^2\bigg(\int_0^T|z^{i,0,\xi}_s|^2ds\bigg)^{\frac{\beta}{2}}+\sum_{i=1}^2\bigg(\int_0^T\int_{\mathcal{E}_i}|\tilde{z}^{i,0,\xi}_{(s,e)}|^2\nu_i(de)ds\bigg)^{\frac{\beta}{2}}\bigg]\\
&\qquad\leq K_2(1+|\xi|^\beta),
\end{aligned}
\end{equation}
\begin{equation}\label{Lbetaresult2}
\begin{aligned}
&E\bigg[\sup_{0\leq s\leq T}|x^{0,\xi}_s-x^{0,\tilde{\xi}}_s|^\beta+\sup_{0\leq s\leq T}|y^{0,\xi}_s-y^{0,\tilde{\xi}}_s|^\beta+\sum_{i=1}^2\bigg(\int_0^T|z^{i,0,\xi}_s-z^{i,0,\tilde{\xi}}_s|^2ds\bigg)^{\frac{\beta}{2}}\\
&\qquad+\sum_{i=1}^2\bigg(\int_0^T\int_{\mathcal{E}_i}|\tilde{z}^{i,0,\xi}_{(s,e)}-\tilde{z}^{i,0,\tilde{\xi}}_{(s,e)}|^2\nu_i(de)ds\bigg)^{\frac{\beta}{2}}\bigg]\leq K_2|\xi-\tilde{\xi}|^\beta,
\end{aligned}
\end{equation}
\end{theorem}
\begin{proof}
First, the given $L^2$-estimations $\mathrm{(i),(ii)}$ come to show the existence and uniqueness of $L^2$-solution to FBSDE \eqref{fbsdept1} on $[t,T]$ for any given initial time and its corresponding initial value $(t,\xi)\in L^2(\Omega,\mathcal{F}_t,P;\mathbb{R})$.

Next, we investigate the FBSDEP \eqref{fbsdept1} at the initial time $t=0$,
\begin{equation}
\left\{
\begin{aligned}
 dx^{0,\xi}_s&=b(s,\Theta_s^{0,\xi})ds+\sum_{i=1}^2\sigma_i(s,\Theta_s^{0,\xi})dW^i_s+\sum_{i=1}^2\int_{\mathcal{E}_i}f_i(s,\Theta_{s-}^{0,\xi},e)\tilde{N}_i(de,ds),\\
-dy^{0,\xi}_s&=g(s,\Theta_s^{0,\xi})ds-\sum_{i=1}^2z^{i,0,\xi}_sdW^i_s-\sum_{i=1}^2\int_{\mathcal{E}_i}\tilde{z}^{i,0,\xi}_{(s,e)}\tilde{N}_i(de,ds),\quad s\in[0,T],\\
  x^{0,\xi}_0&=\xi,\ \ y^{0,\xi}_T=\phi(x^{0,\xi}_T).
\end{aligned}
\right.
\end{equation}
Observing the estimation \eqref{L2estimation2}, a key derivation shows
\begin{equation}
\begin{aligned}
&E\bigg[\sup_{t\leq s\leq T}|x^{0,\xi}_s-x^{0,\tilde{\xi}}_s|^2+\sup_{t\leq s\leq T}|y^{0,\xi}_s-y^{0,\tilde{\xi}}_s|^2+\sum_{i=1}^2\int_t^T|z^{i,0,\xi}_s-z^{i,0,\tilde{\xi}}_s|^2ds\\
&\qquad+\sum_{i=1}^2\int_t^T\int_{\mathcal{E}_i}|\tilde{z}^{i,0,\xi}_{(s,e)}-\tilde{z}^{i,0,\tilde{\xi}}_{(s,e)}|^2\nu_i(de)ds\bigg|\mathcal{F}_t\bigg]\leq K_1|x^{0,\xi}_t-x^{0,\tilde{\xi}}_t|^2,
\end{aligned}
\end{equation}
then we get
\begin{equation}
\begin{aligned}
&|y^{0,\xi}_t-y^{0,\tilde{\xi}}_t|^2\leq E\bigg[\sup_{t\leq s\leq T}|y^{0,\xi}_s-y^{0,\tilde{\xi}}_s|^2\bigg|\mathcal{F}_t\bigg]\leq K_1|x^{0,\xi}_t-x^{0,\tilde{\xi}}_t|^2,
\end{aligned}
\end{equation}
which implies the uniformly Lipschitz condition of $y^{0,\xi}_t$ with respect to $x^{0,\xi}_t$ for any $t\in[0,T]$, i.e.,
\begin{equation}\label{uniform Lip}
\begin{aligned}
&|y^{0,\xi}_t-y^{0,\tilde{\xi}}_t|\leq \sqrt{K_1}|x^{0,\xi}_t-x^{0,\tilde{\xi}}_t|,\ t\in[0,T].
\end{aligned}
\end{equation}
Then, for any given $T$, we divide the whole interval $[0,T]$ into several small intervals with same interval length $\delta$, i.e., $[0,\delta],[\delta,2\delta],\dots,[(k-1)\delta,k\delta]$, and without loss of generality, we assume $T=k\delta$ and $k$ is positive integer. For the first small interval $[0,\delta]$, under assumption {\bf (A4),(A6)} and uniformly Lipschitz condition \eqref{uniform Lip} with terminal time $\delta$, then, by Lemma \ref{Lbetaestimationt}, there exists a constant $\delta>0$ such that
\begin{equation}
\begin{aligned}
&E\bigg[\sup_{0\leq s\leq \delta}|x^{0,\xi}_s|^{\beta}+\sup_{0\leq s\leq \delta}|y^{0,\xi}_s|^{\beta}+\sum_{i=1}^2\bigg(\bigg(\int_0^{\delta}|z^{i,0,\xi}_s|^2ds\bigg)^{\frac{\beta}{2}}+\bigg(\int_0^{\delta}\int_{\mathcal{E}_i}|\tilde{z}^{i,0,\xi}_{(s,e)}|^2\nu_i(de)ds\bigg)^{\frac{\beta}{2}}\bigg)\bigg|\mathcal{F}_0\bigg]\\
&\leq C_{K,\rho,\sigma_1,\sigma_2,f_1,f_2,L,K_1}(1+|x^{0,\xi}_0|^{\beta}),
\end{aligned}
\end{equation}
where the constant $C_{K,\rho,\sigma_1,\sigma_2,f_1,f_2,L,K_1}$ depends on $K_1,L$ and Lipschitz constants $(K,\rho,L_{\sigma_1},L_{\sigma_2},\\L_{f_1},L_{f_2})$. Note that the coefficients of FBSDEP \eqref{fbsdept1} satisfy the same assumption in the interval $[0,T]$. Therefore, by the induction for the following interval $[(j-1)\delta,j\delta],1\leq j\leq k$ and combining the uniformly Lipschitz condition \eqref{uniform Lip} for terminal time $2\delta,3\delta,\cdots,k\delta$, respectively, we have, for $1\leq j\leq k$,
\begin{equation}
\begin{aligned}
&E\bigg[\sup_{s\in[(j-1)\delta,j\delta]}|x^{0,\xi}_s|^{\beta}+\sup_{s\in[(j-1)\delta,j\delta]}|y^{0,\xi}_s|^{\beta}+\sum_{i=1}^2\bigg(\int_{(j-1)\delta}^{j\delta}|z^{i,0,\xi}_s|^2ds\bigg)^{\frac{\beta}{2}}\\
&\qquad+\sum_{i=1}^2\bigg(\int_{(j-1)\delta}^{j\delta}\int_{\mathcal{E}_i}|\tilde{z}^{i,0,\xi}_{(s,e)}|^2\nu_i(de)ds\bigg)^{\frac{\beta}{2}}\bigg|\mathcal{F}_{(j-1)\delta}\bigg]\leq C_{K,\rho,\sigma_1,\sigma_2,f_1,f_2,L,K_1}(1+|x^{0,\xi}_{(j-1)\delta}|^{\beta}).
\end{aligned}
\end{equation}
Now, by \eqref{uniform Lip}, we have obtained these $L^\beta$-estimates by Lemma \ref{Lbetaestimationt} for small intervals $[(j-1)\delta,j\delta],1\leq j\leq k$, then we show that how to connect them together.

When considering the case of $j=1,2$, respectively, we have
\begin{equation}\label{0delta1delta}
\begin{aligned}
&E\bigg[\sup_{0\leq s\leq \delta}|x^{0,\xi}_s|^{\beta}+\sup_{0\leq s\leq \delta}|y^{0,\xi}_s|^{\beta}+\sum_{i=1}^2\bigg(\bigg(\int_0^{\delta}|z^{i,0,\xi}_s|^2ds\bigg)^{\frac{\beta}{2}}\\
&\quad +\bigg(\int_0^{\delta}\int_{\mathcal{E}_i}|\tilde{z}^{i,0,\xi}_{(s,e)}|^2\nu_i(de)ds\bigg)^{\frac{\beta}{2}}\bigg)\bigg|\mathcal{F}_0\bigg]\leq C_{K,\rho,\sigma_1,\sigma_2,f_1,f_2,L,K_1}(1+|\xi|^{\beta}),
\end{aligned}
\end{equation}
and
\begin{equation}\label{1delta2delta}
\begin{aligned}
&E\bigg[\sup_{\delta\leq s\leq 2\delta}|x^{0,\xi}_s|^{\beta}+\sup_{\delta\leq s\leq 2\delta}|y^{0,\xi}_s|^{\beta}+\sum_{i=1}^2\bigg(\bigg(\int_{\delta}^{2\delta}|z^{i,0,\xi}_s|^2ds\bigg)^{\frac{\beta}{2}}\\
&\quad +\bigg(\int_{\delta}^{2\delta}\int_{\mathcal{E}_i}|\tilde{z}^{i,0,\xi}_{(s,e)}|^2\nu_i(de)ds\bigg)^{\frac{\beta}{2}}\bigg)\bigg|\mathcal{F}_{\delta}\bigg]
\leq C_{K,\rho,\sigma_1,\sigma_2,f_1,f_2,L,K_1}(1+|x^{0,\xi}_{\delta}|^{\beta}),
\end{aligned}
\end{equation}
we consider
\begin{equation}
\begin{aligned}
&E\big[C_{K,\rho,\sigma_1,\sigma_2,f_1,f_2,L,K_1}(1+|x^{0,\xi}_{\delta}|^\beta)\big|\mathcal{F}_0\big]\\
&\leq C_{K,\rho,\sigma_1,\sigma_2,f_1,f_2,L,K_1}\big[1+C_{K,\rho,\sigma_1,\sigma_2,f_1,f_2,L,K_1}(1+|\xi|^\beta)\big]\\
&\leq (C_{K,\rho,\sigma_1,\sigma_2,f_1,f_2,L,K_1}+C_{K,\rho,\sigma_1,\sigma_2,f_1,f_2,L,K_1}^2)(1+|\xi|^\beta),
\end{aligned}
\end{equation}
and
\begin{equation}
\begin{aligned}
&C_{K,\rho,\sigma_1,\sigma_2,f_1,f_2,L,K_1}(1+|\xi|^\beta)+E\big[C_{K,\rho,\sigma_1,\sigma_2,f_1,f_2,L,K_1}(1+|x^{0,\xi}_{\delta}\big|^\beta)|\mathcal{F}_0\big]\\
& \leq C^{(2)}_{K,\rho,\sigma_1,\sigma_2,f_1,f_2,L,K_1}(1+|\xi|^\beta)
\end{aligned}
\end{equation}
where $C^{(2)}_{K,\rho,\sigma_1,\sigma_2,f_1,f_2,L,K_1}=2C_{K,\rho,\sigma_1,\sigma_2,f_1,f_2,L,K_1}+C_{K,\rho,\sigma_1,\sigma_2,f_1,f_2,L,K_1}^2$.

Taking conditional expectation $E[\cdot|\mathcal{F}_0]$ on both sides of \eqref{1delta2delta}, then adding it to \eqref{0delta1delta}, we get
\begin{equation}
\begin{aligned}
&E\bigg[\sup_{0\leq s\leq 2\delta}|x^{0,\xi}_s|^{\beta}+\sup_{0\leq s\leq 2\delta}|y^{0,\xi}_s|^{\beta}+\sum_{i=1}^2\bigg(\bigg(\int_0^{\delta}|z^{i,0,\xi}_s|^2ds\bigg)^{\frac{\beta}{2}}+\bigg(\int_{\delta}^{2\delta}|z^{i,0,\xi}_s|^2ds\bigg)^{\frac{\beta}{2}}\bigg)\\
&\qquad+\sum_{i=1}^2\bigg(\bigg(\int_0^{\delta}\int_{\mathcal{E}_i}|\tilde{z}^{i,0,\xi}_{(s,e)}|^2\nu_i(de)ds\bigg)^{\frac{\beta}{2}}+\bigg(\int_{\delta}^{2\delta}\int_{\mathcal{E}_i}|\tilde{z}^{i,0,\xi}_{(s,e)}|^2\nu_i(de)ds\bigg)^{\frac{\beta}{2}}\bigg)\bigg|\mathcal{F}_0\bigg]\\
&\leq C^{(2)}_{K,\rho,\sigma_1,\sigma_2,f_1,f_2,L,K_1}(1+|\xi|^{\beta}),
\end{aligned}
\end{equation}
by utilizing the inequality
\begin{equation}
(a+b)^k\leq 2^k(a^k+b^k),\ a,b\geq0,\ k\geq1.
\end{equation}
Then we have, for $i=1,2$,
\begin{equation}
\begin{aligned}
&2^{-\frac{\beta}{2}}\bigg(\int_0^{2\delta}|z^{i,0,\xi}_s|^2ds\bigg)^{\frac{\beta}{2}}\leq\bigg(\int_0^{\delta}|z^{i,0,\xi}_s|^2ds\bigg)^{\frac{\beta}{2}}+\bigg(\int_{\delta}^{2\delta}|z^{i,0,\xi}_s|^2ds\bigg)^{\frac{\beta}{2}},\\
&2^{-\frac{\beta}{2}}\bigg(\int_0^{2\delta}\int_{\mathcal{E}_i}|\tilde{z}^{i,0,\xi}_{(s,e)}|^2\nu_i(de)ds\bigg)^{\frac{\beta}{2}}\leq\bigg(\int_0^{\delta}\int_{\mathcal{E}_i}|\tilde{z}^{i,0,\xi}_{(s,e)}|^2\nu_i(de)ds\bigg)^{\frac{\beta}{2}}+\bigg(\int_{\delta}^{2\delta}\int_{\mathcal{E}_i}|\tilde{z}^{i,0,\xi}_{(s,e)}|^2\nu_i(de)ds\bigg)^{\frac{\beta}{2}}.
\end{aligned}
\end{equation}
Let $\tilde{C}^{(2)}_{K,\rho,\sigma_1,\sigma_2,f_1,f_2,L,K_1}=2^{\frac{\beta}{2}}C^{(2)}_{K,\rho,\sigma_1,\sigma_2,f_1,f_2,L,K_1}$, it follows that
\begin{equation}
\begin{aligned}
&E\bigg[\sup_{0\leq s\leq 2\delta}|x^{0,\xi}_s|^{\beta}+\sup_{0\leq s\leq 2\delta}|y^{0,\xi}_s|^{\beta}+\sum_{i=1}^2\bigg(\bigg(\int_0^{2\delta}|z^{i,0,\xi}_s|^2ds\bigg)^{\frac{\beta}{2}}\\
&\qquad+\bigg(\int_0^{2\delta}\int_{\mathcal{E}_i}|\tilde{z}^{i,0,\xi}_{(s,e)}|^2\nu_i(de)ds\bigg)^{\frac{\beta}{2}}\bigg)\bigg|\mathcal{F}_0\bigg]\leq \tilde{C}^{(2)}_{K,\rho,\sigma_1,\sigma_2,f_1,f_2,L,K_1}(1+|\xi|^{\beta}).
\end{aligned}
\end{equation}
Next, we continue to consider the case of $j=3$,
\begin{equation}
\begin{aligned}
&E\bigg[\sup_{2\delta\leq s\leq 3\delta}|x^{0,\xi}_s|^{\beta}+\sup_{2\delta\leq s\leq 3\delta}|y^{0,\xi}_s|^{\beta}+\sum_{i=1}^2\bigg(\bigg(\int_{2\delta}^{3\delta}|z^{i,0,\xi}_s|^2ds\bigg)^{\frac{\beta}{2}}\\
&\qquad+\bigg(\int_{2\delta}^{3\delta}\int_{\mathcal{E}_i}|\tilde{z}^{i,0,\xi}_{(s,e)}|^2\nu_i(de)ds\bigg)^{\frac{\beta}{2}}\bigg)\bigg|\mathcal{F}_{2\delta}\bigg]\leq C_{K,\rho,\sigma_1,\sigma_2,f_1,f_2,L,K_1}(1+|x^{0,\xi}_{2\delta}|^{\beta}).
\end{aligned}
\end{equation}
Similarly,  we have
\begin{equation}
\begin{aligned}
&E\bigg[\sup_{0\leq s\leq 3\delta}|x^{0,\xi}_s|^{\beta}+\sup_{0\leq s\leq 3\delta}|y^{0,\xi}_s|^{\beta}+\sum_{i=1}^2\bigg(\bigg(\int_0^{3\delta}|z^{i,0,\xi}_s|^2ds\bigg)^{\frac{\beta}{2}}\\
&\qquad+\bigg(\int_0^{3\delta}\int_{\mathcal{E}_i}|\tilde{z}^{i,0,\xi}_{(s,e)}|^2\nu_i(de)ds\bigg)^{\frac{\beta}{2}}\bigg)\bigg|\mathcal{F}_0\bigg]\leq \tilde{C}^{(3)}_{K,\rho,\sigma_1,\sigma_2,f_1,f_2,L,K_1}(1+|\xi|^{\beta}).
\end{aligned}
\end{equation}
By induction for $j\geq4$, we obtain
\begin{equation}
\begin{aligned}
&E\bigg[\sup_{0\leq s\leq T}|x^{0,\xi}_s|^{\beta}+\sup_{0\leq s\leq T}|y^{0,\xi}_s|^{\beta}+\sum_{i=1}^2\bigg(\bigg(\int_0^T|z^{i,0,\xi}_s|^2ds\bigg)^{\frac{\beta}{2}}\\
&\qquad +\bigg(\int_0^T\int_{\mathcal{E}_i}|\tilde{z}^{i,0,\xi}_{(s,e)}|^2\nu_i(de)ds\bigg)^{\frac{\beta}{2}}\bigg)\bigg|\mathcal{F}_0\bigg]\leq \tilde{C}^{(k)}_{K,\rho,\sigma_1,\sigma_2,f_1,f_2,L,K_1}(1+|\xi|^{\beta}).
\end{aligned}
\end{equation}
Finally, by taking $K_2=\tilde{C}^{(k)}_{K,\rho,\sigma_1,\sigma_2,f_1,f_2,L,K_1}$, then we prove the $L^\beta$-estimation \eqref{Lbetaresult1} and can also obtain \eqref{Lbetaresult2} similarly, which guarantee the existence and uniqueness of $L^\beta$-solution to FBSDEP \eqref{fbsdept1} with $t=0$, respectively.
\end{proof}

Now, we can give the $L^\beta(\beta>2)$-solution and its estimate for FBSDEP \eqref{fbsdept1} in $\tilde{\mathcal{M}}^\beta[0,T]$. The following lemma can be obtained similar to Lemma \ref{Lbetaestimationt} by referring to the result in \cite{LW14}.

\begin{lemma}\label{Lbetaestimationt2}
Let assumption {\bf (A4),(A5)} hold. Then, for every $\beta\geq2$, there exists a sufficiently small constant $\tilde{\delta}$ depending on $(K,L_{\sigma_1},L_{\sigma_2},\rho,L_{f_1},L_{f_2})$ and constant $C_{\beta,K,\rho,\sigma_1,\sigma_2,f_1,f_2,L}$ depending on $(\beta,K,\rho,L_{\sigma_1},L_{\sigma_2},L_{f_1},L_{f_2},L)$ such that for every $0\leq\delta\leq\tilde{\delta}$ and $\xi\in L^\beta(\Omega,\mathcal{F}_t,P;\mathbb{R})$,
\begin{equation}\label{Lbeta estimate2}
\begin{aligned}
&E\bigg[\sup_{t\leq s\leq t+\delta}|x^{t,\xi}_s|^{\beta}+\sup_{t\leq s\leq t+\delta}|y^{t,\xi}_s|^{\beta}+\sum_{i=1}^2\bigg(\bigg(\int_t^{t+\delta}|z^{i,t,\xi}_s|^2ds\bigg)^{\frac{\beta}{2}}\\
&\qquad+\bigg(\int_t^{t+\delta}\int_{\mathcal{E}_i}|\tilde{z}^{i,t,\xi}_{(s,e)}|^2N_i(de,ds)\bigg)^{\frac{\beta}{2}}\bigg)\bigg|\mathcal{F}_t\bigg]
\leq C_{\beta,K,\rho,\sigma_1,\sigma_2,f_1,f_2,L}(1+|\xi|^{\beta}),\ P\mbox{-}a.s..
\end{aligned}
\end{equation}
\end{lemma}
Similarly, we can also obtain the following important theorem based on {\bf Theorem 2.1} in \cite{MY21arxiv}, which shows that, for any given terminal time $T$, under some $L^2$-estimate, the $L^2$-solution are the $L^\beta$-solution in $\tilde{\mathcal{M}}^\beta[0,T](\beta>2)$. So we omit the detailed proof.

\begin{theorem}
Let {\bf (A4),(A5)} hold, and assume that, for every $\xi,\tilde{\xi}\in L^\beta(\Omega,\mathcal{F}_t,P;\mathbb{R})$, the $L^2$-estimations of FBSDEP \eqref{fbsdept1} are given as follows:

$\mathrm{(i)}$
\begin{equation}
\begin{aligned}
&E\bigg[\sup_{t\leq s\leq T}|x^{t,\xi}_s|^2+\sup_{t\leq s\leq T}|y^{t,\xi}_s|^2+\sum_{i=1}^2\bigg(\int_t^T|z^{i,t,\xi}_s|^2ds+\int_t^T\int_{\mathcal{E}_i}|\tilde{z}^{i,t,\xi}_{(s,e)}|^2N_i(de,ds)\bigg)\bigg|\mathcal{F}_t\bigg]\\
&\qquad\leq K_1(1+|\xi|^2),
\end{aligned}
\end{equation}
$\mathrm{(ii)}$
\begin{equation}
\begin{aligned}
&E\bigg[\sup_{t\leq s\leq T}|x^{t,\xi}_s-x^{t,\tilde{\xi}}_s|^2+\sup_{t\leq s\leq T}|y^{t,\xi}_s-y^{t,\tilde{\xi}}_s|^2+\sum_{i=1}^2\bigg(\int_t^T|z^{i,t,\xi}_s-z^{i,t,\tilde{\xi}}_s|^2ds\\
&\qquad+\int_t^T\int_{\mathcal{E}_i}|\tilde{z}^{i,t,\xi}_{(s,e)}-\tilde{z}^{i,t,\tilde{\xi}}_{(s,e)}|^2N_i(de,ds)\bigg)\bigg|\mathcal{F}_t\bigg]\leq K_1|\xi-\tilde{\xi}|^2,
\end{aligned}
\end{equation}
where constant $K_1$ is positive and independent of $t\in[0,T]$.

Then FBSDEP \eqref{fbsdept1} admits a unique $L^\beta(\beta>2)$-solution with $t=0$ and any given terminal time $T$, i.e.,
\begin{equation}
\begin{aligned}
&E\bigg[\sup_{0\leq s\leq T}|x^{0,\xi}_s|^\beta+\sup_{0\leq s\leq T}|y^{0,\xi}_s|^\beta+\sum_{i=1}^2\bigg(\bigg(\int_0^T|z^{i,0,\xi}_s|^2ds\bigg)^{\frac{\beta}{2}}\\
&\qquad +\bigg(\int_0^T\int_{\mathcal{E}_i}|\tilde{z}^{i,0,\xi}_{(s,e)}|^2N_i(de,ds)\bigg)^{\frac{\beta}{2}}\bigg)\bigg]\leq K_2(1+|\xi|^\beta),
\end{aligned}
\end{equation}
\begin{equation}
\begin{aligned}
&E\bigg[\sup_{0\leq s\leq T}|x^{0,\xi}_s-x^{0,\tilde{\xi}}_s|^\beta+\sup_{0\leq s\leq T}|y^{0,\xi}_s-y^{0,\tilde{\xi}}_s|^\beta+\sum_{i=1}^2\bigg(\bigg(\int_0^T|z^{i,0,\xi}_s-z^{i,0,\tilde{\xi}}_s|^2ds\bigg)^{\frac{\beta}{2}}\\
&\qquad+\bigg(\int_0^T\int_{\mathcal{E}_i}|\tilde{z}^{i,0,\xi}_{(s,e)}-\tilde{z}^{i,0,\tilde{\xi}}_{(s,e)}|^2N_i(de,ds)\bigg)^{\frac{\beta}{2}}\bigg)\bigg]\leq K_2|\xi-\tilde{\xi}|^\beta.
\end{aligned}
\end{equation}
\end{theorem}

Then we give the $L^{\beta}(\beta\geq2)$-estimate for fully-coupled FBSDEP.

\begin{theorem}\label{the24}
Suppose that {\bf (A4), (A6)} hold. Then, for any $\beta\geq2$, suppose that $(x,y,z^1,z^2,\tilde{z}^1,\\\tilde{z}^2)$ is the solution to \eqref{fbsdep1}, there exist a sufficiently small constant $\tilde{\tilde{T}}>0$ depending on $K,L_{\sigma_1},L_{\sigma_2},\rho$ and some constant $C_{\beta,K,\rho,\sigma_1,\sigma_2}$ depending on $\beta$ and Lipschitz constants $K,\rho,L_{\sigma_1}$, $L_{\sigma_2}$, such that, for every $0\leq T\leq\tilde{\tilde{T}}$,
\begin{equation}\label{Lbeta estimate1}
\begin{aligned}
&E\bigg[\sup_{0\leq t\leq T}|x_t|^{\beta}+\sup_{0\leq t\leq T}|y_t|^{\beta}+\sum_{i=1}^2\bigg(\bigg(\int_0^T|z^i_t|^2dt\bigg)^{\frac{\beta}{2}}
 +\bigg(\int_0^T\int_{\mathcal{E}_i}|\tilde{z}^i|^2\nu_i(de)dt\bigg)^{\frac{\beta}{2}}\bigg)\bigg]\\
&\leq C_{\beta,K,\rho,\sigma_1,\sigma_2}E\bigg[|x_0|^{\beta}+|\phi(0)|^{\beta}+\bigg(\int_0^T|g(t,0,0,0,0,0,0)|dt\bigg)^\beta+\bigg(\int_0^T|b(t,0,0,0,0,0,0)|dt\bigg)^\beta\\
&\quad +\sum_{i=1}^2\bigg(\bigg(\int_0^T|\sigma_i(t,0,0,0,0,0,0)|^2dt\bigg)^{\frac{\beta}{2}}+\bigg(\int_0^T\int_{\mathcal{E}_i}|f_i(t,0,0,e)|^2N_i(de,dt)\bigg)^{\frac{\beta}{2}}\bigg)\bigg].
\end{aligned}
\end{equation}
Moreover, Let $(x^1,y^1,z^{1,1},z^{2,1},\tilde{z}^{1,1},\tilde{z}^{2,1}),(x^2,y^2,z^{1,2},z^{2,2},\tilde{z}^{1,2},\tilde{z}^{2,2})$ be two solutions with different initial values $x^1_0,x^2_0$, respectively. Then we also give an estimate with respect to the initial value in the following
\begin{equation}\label{difference estimate}
\begin{aligned}
&E\bigg[\sup_{0\leq t\leq T}|x^1_t-x^2_t|^{\beta}+\sup_{0\leq t\leq T}|y^1_t-y^2_t|^{\beta}+\sum_{i=1}^2\bigg(\bigg(\int_0^T|z^{i,1}_t-z^{i,2}_t|^2dt\bigg)^{\frac{\beta}{2}}\\
&\quad+\bigg(\int_0^T\int_{\mathcal{E}_i}|\tilde{z}^{i,1}-\tilde{z}^{i,2}|^2\nu_i(de)dt\bigg)^{\frac{\beta}{2}}\bigg)\bigg]\leq C_{\beta,K,\rho,\sigma_1,\sigma_2}|x^1_0-x^2_0|^{\beta}.
\end{aligned}
\end{equation}
\end{theorem}

\begin{proof}
The second estimate \eqref{difference estimate} are the direct result of \eqref{Lbeta estimate1}, so we mainly prove the first estimate. We should note that, as we analyzed in Remark \ref{rem23}, under {\bf (A6)}, we have obtained the existence and uniqueness of the solution to the fully coupled FBSDEP \eqref{fbsdep1}, which can guarantee the $\beta$-integrability in the following. From the backward equation in \eqref{fbsdep1}, $y$ is given by the right-continuous version of $y_t=E\big[\phi(x_T)+\int_t^Tg(t,x,y,z^1,z^2,\tilde{z}^1,\tilde{z}^2)ds|\mathcal{F}_t\big]$. Since $\beta\geq2$, by the martingale representation theorem in \cite{TL94} for local square integrable martingale, $(z^1,z^2,\tilde{z}^1,\tilde{z}^2)$ corresponds to the unique pair of $\mathcal{E}_1,\mathcal{E}_2$-predictable processes satisfying
\begin{equation*}
\begin{aligned}
&E\Big[\phi(x_T)+\int_0^Tg(t,x,y,z^1,z^2,\tilde{z}^1,\tilde{z}^2)ds\big|\mathcal{F}_t\Big]\\
&\qquad=y_0+\sum_{i=1}^2\bigg(\int_0^tz^i_sdW^i_s+\int_0^t\int_{\mathcal{E}_i}\tilde{z}^i_{(s,e)}\tilde{N}_i(de,ds)\bigg),\ P\mbox{-}a.s.,
\end{aligned}
\end{equation*}
we have $|y_t|\leq E\big[|\phi(x_T)|+\int_0^T|g(t,x,y,z^1,z^2,\tilde{z}^1,\tilde{z}^2)|ds|\mathcal{F}_t\big]$. Using martingale inequality, we get
\begin{equation}\label{ybeta}
\begin{aligned}
&E\Big[\sup_{0\leq t\leq T}|y_t|^\beta\Big]\leq E\bigg[\sup_{0\leq t\leq T}\bigg|E\Big[|\phi(x_T)|+\int_0^T|g(t,\Theta)|ds\Big|\mathcal{F}_t\Big]\bigg|^\beta\bigg]\\
&\leq C_\beta E\bigg[\Big|E\Big[|\phi(x_T)|+\int_0^T|g(t,\Theta)|ds\Big|\mathcal{F}_T\Big]\Big|^\beta\bigg]\\
&\leq C_\beta E\Big[\Big||\phi(x_T)|+\int_0^T|g(t,\Theta)|ds\Big|^\beta\Big]\\
&\leq C_\beta E\bigg[|\phi(x_T)-\phi(0)+\phi(0)|^\beta+\Big|\int_0^T|g(t,\Theta)-g(t,0,0,0,0,0,0)+g(t,0,0,0,0,0,0)|ds\Big|^\beta\bigg]\\
&\leq C_\beta E\bigg[\big|C|x_T|+\phi(0)\big|^\beta+\Big|\int_0^T|C(|x|+|y|+|z^1+|z^2|+||\tilde{z}^1||+||\tilde{z}^2||)\\
&\qquad\qquad +g(t,0,0,0,0,0,0)|ds\Big|^\beta\bigg]\\
&\leq C_\beta E\bigg[C^\beta|x_T|^\beta+|\phi(0)|^\beta+\Big|\int_0^TC(|x|+|y|+|z^1|+|z^2|+||\tilde{z}^1||+||\tilde{z}^2||)ds\Big|^\beta\\
&\qquad+\Big|\int_0^T|g(t,0,0,0,0,0,0)|ds\Big|^\beta\bigg]\\
&\leq C_\beta E\bigg[C^\beta\sup_{0\leq t\leq T}|x_t|^\beta+|\phi(0)|^\beta+\Big|\int_0^TC|x|ds\bigg|^\beta+\bigg|\int_0^TC|y|ds\Big|^\beta\\
&\qquad\quad +\sum_{i=1}^2\bigg(\Big|\int_0^TC|z^i|ds\Big|^\beta+\Big|\int_0^TC||\tilde{z}^i||ds\Big|^\beta\bigg)+\Big|\int_0^T|g(t,0,0,0,0,0,0)|ds\Big|^\beta\bigg]\\
&\leq C_\beta E\bigg[(C^\beta+T^\beta C^\beta)\sup_{0\leq t\leq T}|x_t|^\beta+|\phi(0)|^{\beta}
+T^\beta C^{\beta}\sup_{0\leq t\leq T}|y_t|^\beta+C^\beta T^{\frac{\beta}{2}}\sum_{i=1}^2\Big|\int_0^T|z^i|^2ds\Big|^{\frac{\beta}{2}}\\
&\qquad\quad +C^\beta T^{\frac{\beta}{2}}\sum_{i=1}^2\Big|\int_0^T\int_{\mathcal{E}_i}|\tilde{z}^i|^2\nu_i(de)ds\Big|^{\frac{\beta}{2}}+\Big|\int_0^T|g(t,0,0,0,0,0,0)|ds\Big|^\beta\bigg].
\end{aligned}
\end{equation}
We need the following lemma (\cite{LW14}, Lemma 3.1).
\begin{lemma}\label{lemma24}
For any $\beta\geq2$, we have
\begin{equation}
E\bigg[\bigg(\int_0^T\int_{\mathcal{E}_i}|\tilde{z}^i|^2\nu_i(de)ds\bigg)^{\frac{\beta}{2}}\bigg]\leq C_\beta E\bigg[\bigg(\int_0^T\int_{\mathcal{E}_i}|\tilde{z}^i|^2N_i(de,ds)\bigg)^{\frac{\beta}{2}}\bigg],\ i=1,2.
\end{equation}
\end{lemma}
Then, by Lemma \ref{lemma22} and B-D-G's inequality, we can get the estimate
\begin{equation}\label{ztildezestimate}
\begin{aligned}
&E\bigg[\sum_{i=1}^2\bigg(\bigg(\int_0^T|z^i|^2ds\bigg)^{\frac{\beta}{2}}+\bigg(\int_0^T\int_{\mathcal{E}_i}|\tilde{z}^i|^2\nu_i(de)ds\bigg)^{\frac{\beta}{2}}\bigg)\bigg]\\
&\leq E\bigg[\sum_{i=1}^2\bigg(\int_0^T|z^i|^2ds\bigg)^{\frac{\beta}{2}}\bigg]+C_\beta E\bigg[\sum_{i=1}^2\bigg(\int_0^T\int_{\mathcal{E}_i}|\tilde{z}^i|^2N_i(de,ds)\bigg)^{\frac{\beta}{2}}\bigg]\\
&\leq C_\beta E\bigg[\sup_{0\leq t\leq T}\bigg|\sum_{i=1}^2\int_0^tz^i_sdW^i_s+\sum_{i=1}^2\int_0^t\int_{\mathcal{E}_i}\tilde{z}^i_{(s,e)}\tilde{N}_i(de,ds)\bigg|^{\beta}\bigg]\\
&\leq C_\beta E\bigg[\sup_{0\leq t\leq T}\bigg(|y_t|^\beta+|y_0|^\beta+\bigg|\int_0^tg(s,\Theta(s))ds\bigg|^\beta\bigg)\bigg]\\
&\leq C_\beta E\bigg[\sup_{0\leq t\leq T}|y_t|^\beta+\bigg(\int_0^T|g(s,\Theta(s))|ds\bigg)^\beta\bigg]+E\bigg[|\phi(x_T)|^\beta+\bigg|\int_0^Tg(s,\Theta(s))ds\bigg|^\beta\bigg]\\
&\leq C_\beta E\bigg[\sup_{0\leq t\leq T}|y_t|^\beta+\bigg(\int_0^T|g(s,\Theta)-g(s,0,0,0,0,0,0)+g(s,0,0,0,0,0,0)|ds\bigg)^{\beta}\bigg]\\
&\quad +E\bigg[C^\beta|x_T|^{\beta}+|\phi(0)|^\beta+\bigg|\int_0^Tg(s,\Theta)-g(s,0,0,0,0,0,0)+g(s,0,0,0,0,0,0)ds\bigg|^\beta\bigg]\\
&\leq C_\beta E\bigg[(1+C^{\beta}T^\beta)\sup_{0\leq t\leq T}|y_t|^\beta+(C^\beta+C^\beta T^\beta)\sup_{0\leq t\leq T}|x_t|^\beta+|\phi(0)|^\beta\\
&\quad +C^\beta T^{\frac{\beta}{2}}\sum_{i=1}^2\bigg(\bigg(\int_0^T|z^i|^2ds\bigg)^{\frac{\beta}{2}}+\bigg(\int_0^T\int_{\mathcal{E}_i}|\tilde{z}^i|^2\nu_i(de)ds\bigg)^{\frac{\beta}{2}}\bigg)\\
&\quad +\bigg(\int_0^T|g(s,0,0,0,0,0,0)|ds\bigg)^\beta\bigg].
\end{aligned}
\end{equation}
From \eqref{ybeta}, we have
\begin{equation*}
\begin{aligned}
&(1-T^{\beta}C_\beta)E\bigg[\sup_{0\leq t\leq T}|y_t|^\beta\bigg]\leq C_\beta E\bigg[(C^\beta+T^\beta C^\beta)\sup_{0\leq t\leq T}|x_t|^\beta+|\phi(0)|^\beta\\
&\qquad+C^\beta T^{\frac{\beta}{2}}\sum_{i=1}^2\bigg(\bigg|\int_0^T|z^i|^2ds\bigg|^{\frac{\beta}{2}}+\bigg|\int_0^T\int_{\mathcal{E}_i}|\tilde{z}^i|^2\nu_i(de)ds\bigg|^{\frac{\beta}{2}}\bigg)+\bigg|\int_0^T|g(t,0,0,0,0,0,0)|ds\bigg|^\beta\bigg].
\end{aligned}
\end{equation*}
Choosing $T_1$ small enough such that $T_1^{\beta}C_{\beta}<1$, and for every $0\leq T\leq T_1$, we get
\begin{equation}\label{yestimate2}
\begin{aligned}
&E\bigg[\sup_{0\leq t\leq T}|y_t|^\beta\bigg]\leq C_{\beta,T}E\bigg[(C^{\beta}+T^{\beta}C^{\beta})\sup_{0\leq t\leq T}|x_t|^\beta
+|\phi(0)|^{\beta}+C^{\beta}T^{\frac{\beta}{2}}\sum_{i=1}^2\bigg(\bigg|\int_0^T|z^i|^2ds\bigg|^{\frac{\beta}{2}}\\
&\qquad+\bigg|\int_0^T\int_{\mathcal{E}_i}|\tilde{z}^i|^2\nu_i(de)ds\bigg|^{\frac{\beta}{2}}\bigg)+\bigg|\int_0^T|g(t,0,0,0,0,0,0)|ds\bigg|^\beta\bigg].
\end{aligned}
\end{equation}
From \eqref{ztildezestimate}, we also have
\begin{equation*}
\begin{aligned}
&(1-C_\beta T^{\frac{\beta}{2}})E\bigg[\sum_{i=1}^2\bigg(\bigg(\int_0^T|z^i|^2ds\bigg)^{\frac{\beta}{2}}+\bigg(\int_0^T\int_{\mathcal{E}_i}|\tilde{z}^i|^2\nu_i(de)ds\bigg)^{\frac{\beta}{2}}\bigg)\bigg]\\
&\leq C_\beta E\bigg[(1+C^\beta T^\beta)\sup_{0\leq t\leq T}|y_t|^\beta+(C^\beta+C^\beta T^\beta)\sup_{0\leq t\leq T}|x_t|^\beta\\
&\qquad\quad +|\phi(0)|^\beta+\bigg(\int_0^T|g(s,0,0,0,0,0,0)|ds\bigg)^\beta\bigg].
\end{aligned}
\end{equation*}
Then combining with the estimate \eqref{yestimate2}, we get
\begin{equation*}
\begin{aligned}
&(1-C_{\beta,T}T^{\frac{\beta}{2}})E\bigg[\sum_{i=1}^2\bigg(\bigg(\int_0^T|z^i|^2ds\bigg)^{\frac{\beta}{2}}+\bigg(\int_0^T\int_{\mathcal{E}_i}|\tilde{z}^i|^2\nu_i(de)ds\bigg)^{\frac{\beta}{2}}\bigg)\bigg]\\
&\leq C_{\beta,T}E\bigg[(C^\beta+T^\beta C^\beta)\sup_{0\leq t\leq T}|x_t|^\beta+|\phi(0)|^\beta+\bigg|\int_0^T|g(t,0,0,0,0,0,0)|ds\bigg|^\beta\bigg].
\end{aligned}
\end{equation*}
Choosing $0\leq T_2\leq T_1$ small enough such that $C_{\beta,T_2}T_2^{\frac{\beta}{2}}<1$, and for every $0\leq T\leq T_2$, we get
\begin{equation}\label{ztildezestimate2}
\begin{aligned}
&E\bigg[\sum_{i=1}^2\bigg(\bigg(\int_0^T|z^i|^2ds\bigg)^{\frac{\beta}{2}}+\bigg(\int_0^T\int_{\mathcal{E}_i}|\tilde{z}^i|^2\nu_i(de)ds\bigg)^{\frac{\beta}{2}}\bigg)\bigg]\\
&\leq C_{\beta,T}E\bigg[(C^\beta+T^\beta C^\beta)\sup_{0\leq t\leq T}|x_t|^\beta+|\phi(0)|^\beta+\bigg|\int_0^T|g(t,0,0,0,0,0,0)|ds\bigg|^\beta\bigg],
\end{aligned}
\end{equation}
where the constant $C_{\beta,T}$ changed every step. Next, we begin to deal with the forward equation in \eqref{fbsdep1}. First, since
\begin{equation*}
\begin{aligned}
|x_t|^\beta&=C_\beta\bigg[|x_0|^\beta+\bigg|\int_0^tb(s,\Theta)ds\bigg|^\beta+\sum_{i=1}^2\bigg|\int_0^t\sigma_i(s,\Theta)dW^i_s\bigg|^\beta\\
           &\qquad +\sum_{i=1}^2\bigg|\int_0^t\int_{\mathcal{E}_i}f_i(s,x,y,e)\tilde{N}_i(de,ds)\bigg|^{\beta}\bigg],
\end{aligned}
\end{equation*}
by B-D-G's inequality and H\"{o}lder's inequality, we have
\begin{equation*}
\begin{aligned}
&E\bigg[\sup_{0\leq t\leq T}|x_t|^\beta\bigg]\leq C_\beta E\big[|x_0|^\beta\big]+C_\beta E\bigg[\int_0^T|b(s,\Theta)|ds\bigg]^\beta\\
&\quad +C_\beta E\bigg[\sum_{i=1}^2\bigg(\int_0^T|\sigma_i(s,\Theta)|^2ds\bigg)^{\frac{\beta}{2}}\bigg]
+C_\beta E\bigg[\sum_{i=1}^2\bigg(\int_0^T\int_{\mathcal{E}_i}|f_i(s,x,y,e)|^2N_i(de,ds)\bigg)^{\frac{\beta}{2}}\bigg]\\
&\leq C_\beta E\big[|x_0|^\beta\big]+C_\beta E\bigg[\int_0^T|b(s,\Theta)-b(s,0,0,0,0,0,0)+b(s,0,0,0,0,0,0)|ds\bigg]^\beta\\
&\quad +C_\beta E\bigg[\sum_{i=1}^2\bigg(\int_0^T|\sigma_i(s,\Theta)-\sigma_i(s,0,0,0,0,0,0)+\sigma_i(s,0,0,0,0,0,0)|^2ds\bigg)^{\frac{\beta}{2}}\bigg]\\
&\quad +C_\beta E\bigg[\sum_{i=1}^2\bigg(\int_0^T\int_{\mathcal{E}_i}|f_i(s,x,y,e)-f_i(s,0,0,e)+f_i(s,0,0,e)|^2N_i(de,ds)\bigg)^{\frac{\beta}{2}}\bigg]\\
\end{aligned}
\end{equation*}
\begin{equation}\label{xestimate2}
\begin{aligned}
&\leq C_\beta E[|x_0|^{\beta}]+C_\beta E\bigg[\bigg(\int_0^TC(|x|+|y|+|z^1|+|z^2|+||\tilde{z}^1||+||\tilde{z}^2||)ds\bigg)^{\beta}\bigg]\\
&\quad +C_\beta E\bigg[\bigg(\int_0^T|b(s,0,0,0,0,0,0)|ds\bigg)^{\beta}\bigg]+C_\beta E\bigg[\sum_{i=1}^2\bigg(\int_0^TK^2(|x|^2+|y|^2)\\
&\quad +L^2_{\sigma_i}(|z^1|^2+|z^2|^2+||\tilde{z}^1||^2+||\tilde{z}^2||^2)+|\sigma_i(s,0,0,0,0,0,0)|^2ds\bigg)^{\frac{\beta}{2}}\bigg]\\
&\quad +C_\beta E\bigg[\sum_{i=1}^2\bigg(\int_0^T\int_{\mathcal{E}_i}\rho^2(e)(|x|^2+|y|^2)+|f_i(s,0,0,e)|^2N_i(de,ds)\bigg)^{\frac{\beta}{2}}\bigg]\\
&\leq C_\beta E\big[|x_0|^\beta\big]+C_\beta E\bigg[C^\beta\bigg(T^\beta\sup_{0\leq t\leq T}|x|^\beta
+T^\beta\sup_{0\leq t\leq T}|y|^\beta+T^{\frac{\beta}{2}}\sum_{i=1}^2\bigg(\bigg(\int_0^T|z^i|^2ds\bigg)^{\frac{\beta}{2}}\\
&\qquad +\bigg(\int_0^T\int_{\mathcal{E}_i}|\tilde{z}^i|^2\nu_i(de)ds\bigg)^{\frac{\beta}{2}}\bigg)\bigg)\bigg]+C_\beta E\bigg[\bigg(\int_0^T|b(s,0,0,0,0,0,0)|ds\bigg)^\beta\bigg]\\
&\quad +C_\beta E\bigg[\sum_{i=1}^2\bigg(\int_0^TK^2(|x|^2+|y|^2)+L^2_{\sigma_i}(|z^1|^2+|z^2|^2+||\tilde{z}^1||^2+||\tilde{z}^2||^2)ds\bigg)^{\frac{\beta}{2}}\\
&\qquad+\sum_{i=1}^2\bigg(\int_0^T|\sigma_i(s,0,0,0,0,0,0)|^2ds\bigg)^{\frac{\beta}{2}}\bigg]\\
&\quad +C_\beta E\bigg[\sum_{i=1}^2\bigg(\int_0^T\int_{\mathcal{E}_i}\rho^2(e)(|x|^2+|y|^2)N_i(de,ds)\bigg)^{\frac{\beta}{2}}\\
&\qquad+\sum_{i=1}^2\bigg(\int_0^T\int_{\mathcal{E}_i}|f_i(s,0,0,e)|^2N_i(de,ds)\bigg)^{\frac{\beta}{2}}\bigg]\\
&\leq C_\beta E\big[|x_0|^\beta\big]+C_\beta E\bigg[C^\beta\bigg(T^\beta\sup_{0\leq t\leq T}|x|^\beta+T^\beta\sup_{0\leq t\leq T}|y|^\beta+T^{\frac{\beta}{2}}\sum_{i=1}^2\bigg(\bigg(\int_0^T|z^i|^2ds\bigg)^{\frac{\beta}{2}}\\
&\qquad +\bigg(\int_0^T\int_{\mathcal{E}_i}|\tilde{z}^i|^2\nu_i(de)ds\bigg)^{\frac{\beta}{2}}\bigg)\bigg)\bigg]+C_\beta E\bigg[\bigg(\int_0^T|b(s,0,0,0,0,0,0)|ds\bigg)^\beta\bigg]\\
&\quad +C_\beta E\bigg[T^{\frac{\beta}{2}}K^\beta\sup_{0\leq t\leq T}(|x|^{\beta}+|y|^\beta)+(L^\beta_{\sigma_1}+L^\beta_{\sigma_2})\sum_{i=1}^2\bigg(\bigg(\int_0^T|z^i|^2ds\bigg)^{\frac{\beta}{2}}\\
&\qquad+\bigg(\int_0^T\int_{\mathcal{E}_i}|\tilde{z}^i|^2\nu_i(de)ds\bigg)^{\frac{\beta}{2}}\bigg)+\sum_{i=1}^2\bigg(\bigg(\int_0^T|\sigma_i(s,0,0,0,0,0,0)|^2ds\bigg)^{\frac{\beta}{2}}\\
&\qquad+\bigg(\int_0^T\int_{\mathcal{E}_i}|f_i(s,0,0,e)|^2N_i(de,ds)\bigg)^{\frac{\beta}{2}}\bigg)\bigg]\\
&\quad +C_{\beta,T}E\bigg[\sum_{i=1}^2\int_0^T\int_{\mathcal{E}_i}\rho^{\beta}(e)(|x|^\beta+|y|^{\beta})\nu_i(de)ds\bigg].
\end{aligned}
\end{equation}
Here we have used the fact that, for $i=1,2$,
\begin{equation}
E\bigg[\bigg(\int_0^T\int_{\mathcal{E}_i}\rho^2(e)(|x|^2+|y|^2)N_i(de,ds)\bigg)^{\frac{\beta}{2}}\bigg]\leq C_{\beta,T}E\bigg[\int_0^T\int_{\mathcal{E}_i}\rho^{\beta}(e)(|x|^{\beta}+|y|^{\beta})\nu_i(de)ds\bigg].
\end{equation}
Indeed, we have
\begin{equation*}
E\bigg[\bigg(\int_0^T\int_{\mathcal{E}_i}\rho^2(e)|x|^2N_i(de,ds)\bigg)^{\frac{\beta}{2}}\bigg]\leq C_{\beta,T}E\bigg[\int_0^T\int_{\mathcal{E}_i}\rho^{\beta}(e)|x|^{\beta}\nu_i(de)ds\bigg],\ i=1,2.
\end{equation*}
Set $H_{i,t}=\int_0^t\int_{\mathcal{E}_i}\rho^2(e)|x|^2N_i(de,ds)$, and
\begin{equation*}
\begin{aligned}
H_{i,T}^{\frac{\beta}{2}}&=\sum_{s\leq T}(H^{\frac{\beta}{2}}_{i,s}-H^{\frac{\beta}{2}}_{i,s-})=\sum_{s\leq T}\int_{\mathcal{E}_i}\bigg(|H_{i,s-}+\rho^2(e)|x|^2|^{\frac{\beta}{2}}-H^{\frac{\beta}{2}}_{i,s-}\bigg)N_i(de,\{s\})\\
&=\int_0^T\int_{\mathcal{E}_i}\bigg(|H_{i,s-}+\rho^2(e)|x|^2|^{\frac{\beta}{2}}-H^{\frac{\beta}{2}}_{i,s-}\bigg)N_i(de,ds)\\
&\leq C_\beta\int_0^T\int_{\mathcal{E}_i}\bigg(H^{\frac{\beta}{2}}_{i,s-}+\rho^\beta(e)|x|^\beta\bigg)N_i(de,ds),
\end{aligned}
\end{equation*}
then
\begin{equation*}
\begin{aligned}
EH_{i,T}^{\frac{\beta}{2}}&\leq C_\beta E\bigg[\int_0^T\int_{\mathcal{E}_i}\bigg(H^{\frac{\beta}{2}}_{i,s-}+\rho^\beta(e)|x|^\beta\bigg)\nu_i(de)ds\bigg]\\
&\leq C_\beta E\int_0^TH^{\frac{\beta}{2}}_{i,s}ds+C_\beta E\bigg[\int_0^T\int_{\mathcal{E}_i}\rho^\beta(e)|x|^\beta\nu_i(de)ds\bigg].
\end{aligned}
\end{equation*}
From Gronwall's inequality, we get
\begin{equation}
E\bigg[\bigg(\int_0^T\int_{\mathcal{E}_i}\rho^2(e)|x|^2N_i(de,ds)\bigg)^{\frac{\beta}{2}}\bigg]\leq C_{\beta,T}E\bigg[\int_0^T\int_{\mathcal{E}_i}\rho^\beta(e)|x|^\beta\nu_i(de)ds\bigg].
\end{equation}
Then we can manage \eqref{xestimate2} as follows
\begin{equation*}
\begin{aligned}
&E\bigg[\sup_{0\leq t\leq T}|x_t|^\beta\bigg]\leq C_\beta E\big[|x_0|^\beta\big]
+\Big[C_\beta(T^\beta C^\beta+T^{\frac{\beta}{2}}K^{\beta})+C_{\beta,T}\sum_{i=1}^2\int_{\mathcal{E}_i}\rho^\beta(e)\nu_i(de)\Big]\\
&\quad\times\bigg(E\bigg[\sup_{0\leq t\leq T}|x|^\beta\bigg]+E\bigg[\sup_{0\leq t\leq T}|y|^\beta\bigg]\bigg)\\
&\quad +C_\beta(T^{\frac{\beta}{2}}C^\beta+L^\beta_{\sigma_1}+L^\beta_{\sigma_2})E\bigg[\sum_{i=1}^2\bigg(\bigg(\int_0^T|z^i|^2ds\bigg)^{\frac{\beta}{2}}
+\bigg(\int_0^T\int_{\mathcal{E}_i}|\tilde{z}^i|^2\nu_i(de)ds\bigg)^{\frac{\beta}{2}}\bigg)\bigg]\\
&\quad +C_\beta E\bigg[\bigg(\int_0^T|b(s,0,0,0,0,0,0)|ds\bigg)^\beta+\sum_{i=1}^2\bigg(\int_0^T|\sigma_i(s,0,0,0,0,0,0)|^2ds\bigg)^{\frac{\beta}{2}}\\
&\qquad +\sum_{i=1}^2\bigg(\int_0^T\int_{\mathcal{E}_i}|f_i(s,0,0,e)|^2N_i(de,ds)\bigg)^{\frac{\beta}{2}}\bigg].
\end{aligned}
\end{equation*}
From \eqref{yestimate2} and \eqref{ztildezestimate2}, we then have
\begin{equation*}
\begin{aligned}
&E\bigg[\sup_{0\leq t\leq T}|x_t|^\beta\bigg]\\
&\leq C_\beta E\big[|x_0|^\beta\big]+\bigg\{\Big[C_\beta(T^\beta C^\beta+T^{\frac{\beta}{2}}K^\beta)+C_{\beta,T}\sum_{i=1}^2\int_{\mathcal{E}_i}\rho^\beta(e)\nu_i(de)\Big]
\Big[1+C_{\beta,T}(C^\beta+T^\beta C^\beta)\Big]\\
&\qquad +C_\beta(T^{\frac{\beta}{2}}C^\beta+L^\beta_{\sigma_1}+L^\beta_{\sigma_2})C_{\beta,T}(C^\beta+T^\beta C^\beta)\bigg\}E\bigg[\sup_{0\leq t\leq T}|x|^\beta\bigg]
 +\Big[C_\beta(T^\beta C^\beta+T^{\frac{\beta}{2}}K^\beta)\\
&\qquad +C_{\beta,T}\sum_{i=1}^2\int_{\mathcal{E}_i}\rho^\beta(e)\nu_i(de)+C_\beta(T^{\frac{\beta}{2}}C^\beta+L^\beta_{\sigma_1}+L^\beta_{\sigma_2})\Big]C_{\beta,T}E|\phi(0)|^\beta
+\Big[C_\beta(T^\beta C^\beta+T^{\frac{\beta}{2}}K^\beta)\\
&\qquad +C_{\beta,T}\sum_{i=1}^2\int_{\mathcal{E}_i}\rho^{\beta}(e)\nu_i(de)+C_{\beta}(T^{\frac{\beta}{2}}C^{\beta}+L^{\beta}_{\sigma_1}+L^{\beta}_{\sigma_2})\Big]C_{\beta,T}
E\bigg[\bigg(\int_0^T|g(t,0,0,0,0,0,0)|ds\bigg)^{\beta}\bigg]\\
&\quad +C_\beta E\bigg[\bigg(\int_0^T|b(s,0,0,0,0,0,0)|ds\bigg)^\beta+\sum_{i=1}^2\bigg(\int_0^T|\sigma_i(s,0,0,0,0,0,0)|^2ds\bigg)^{\frac{\beta}{2}}\\
&\qquad +\sum_{i=1}^2\bigg(\int_0^T\int_{\mathcal{E}_i}|f_i(s,0,0,e)|^2N_i(de,ds)\bigg)^{\frac{\beta}{2}}\bigg]\\
&\leq C_\beta E\big[|x_0|^\beta\big]+C_{\beta,K,\rho,\sigma_1,\sigma_2,T}E\Big[\sup_{0\leq t\leq T}|x|^\beta\bigg]+\Big[C_\beta(T^\beta C^\beta+T^{\frac{\beta}{2}}K^\beta)\\
&\qquad +C_{\beta,T}\sum_{i=1}^2\int_{\mathcal{E}_i}\rho^\beta(e)\nu_i(de)+C_\beta(T^{\frac{\beta}{2}}C^\beta+L^\beta_{\sigma_1}+L^\beta_{\sigma_2})\Big]C_{\beta,T}E|\phi(0)|^\beta
+\Big[C_{\beta}(T^{\beta}C^{\beta}+T^{\frac{\beta}{2}}K^{\beta})\\
&\qquad +C_{\beta,T}\sum_{i=1}^2\int_{\mathcal{E}_i}\rho^\beta(e)\nu_i(de)+C_\beta(T^{\frac{\beta}{2}}C^\beta+L^\beta_{\sigma_1}
+L^\beta_{\sigma_2})\Big]C_{\beta,T}E\bigg[\bigg(\int_0^T|g(t,0,0,0,0,0,0)|ds\bigg)^\beta\bigg]\\
&\quad +C_\beta E\bigg[\bigg(\int_0^T|b(s,0,0,0,0,0,0)|ds\bigg)^\beta+\sum_{i=1}^2\bigg(\int_0^T|\sigma_i(s,0,0,0,0,0,0)|^2ds\bigg)^{\frac{\beta}{2}}\\
&\qquad +\sum_{i=1}^2\bigg(\int_0^T\int_{\mathcal{E}_i}|f_i(s,0,0,e)|^2N_i(de,ds)\bigg)^{\frac{\beta}{2}}\bigg].
\end{aligned}
\end{equation*}
Therefore, we have
\begin{equation*}
\begin{aligned}
&(1-C_{\beta,K,\rho,\sigma_1,\sigma_2,T})E\bigg[\sup_{0\leq t\leq T}|x_t|^\beta\bigg]\leq C_\beta E\big[|x_0|^\beta\big]\\
&\quad +\Big[C_\beta(T^\beta C^\beta+T^{\frac{\beta}{2}}K^\beta)
+C_{\beta,T}\sum_{i=1}^2\int_{\mathcal{E}_i}\rho^\beta(e)\nu_i(de)+C_\beta(T^{\frac{\beta}{2}}C^\beta+L^\beta_{\sigma_1}+L^\beta_{\sigma_2})\Big]C_{\beta,T}E|\phi(0)|^\beta\\
&\quad +\Big[C_\beta(T^\beta C^\beta+T^{\frac{\beta}{2}}K^\beta)+C_{\beta,T}\sum_{i=1}^2\int_{\mathcal{E}_i}\rho^\beta(e)\nu_i(de)
+C_\beta(T^{\frac{\beta}{2}}C^\beta+L^\beta_{\sigma_1}+L^\beta_{\sigma_2})\Big]C_{\beta,T}\\
\end{aligned}
\end{equation*}
\begin{equation*}
\begin{aligned}
&\qquad \times E\bigg[\bigg(\int_0^T|g(t,0,0,0,0,0,0)|ds\bigg)^\beta\bigg]+C_\beta E\bigg[\bigg(\int_0^T|b(s,0,0,0,0,0,0)|ds\bigg)^\beta\\
&\qquad +\sum_{i=1}^2\bigg(\bigg(\int_0^T|\sigma_i(s,0,0,0,0,0,0)|^2ds\bigg)^{\frac{\beta}{2}}+\bigg(\int_0^T\int_{\mathcal{E}_i}|f_i(s,0,0,e)|^2N_i(de,ds)\bigg)^{\frac{\beta}{2}}\bigg)\bigg].
\end{aligned}
\end{equation*}
Choosing $0\leq T_3\leq T_2$ and $L_{\sigma}$ small enough such that $C_{\beta,K,\rho,\sigma_1,\sigma_2,T_3}<1$, and for every $0\leq T\leq T_3$, we get
\begin{equation}\label{xestimate22}
\begin{aligned}
E\bigg[\sup_{0\leq t\leq T}|x_t|^\beta\bigg]&\leq C_{\beta,K,\rho,\sigma_1,\sigma_2,T}E\bigg[|x_0|^{\beta}+|\phi(0)|^{\beta}+\bigg(\int_0^T|g(t,0,0,0,0,0,0)|ds\bigg)^{\beta}\\
& +\bigg(\int_0^T|b(s,0,0,0,0,0,0)|ds\bigg)^{\beta}+\sum_{i=1}^2\bigg(\int_0^T|\sigma_i(s,0,0,0,0,0,0)|^2ds\bigg)^{\frac{\beta}{2}}\\
&+\sum_{i=1}^2\bigg(\int_0^T\int_{\mathcal{E}_i}|f_i(s,0,0,e)|^2N_i(de,ds)\bigg)^{\frac{\beta}{2}}\bigg].
\end{aligned}
\end{equation}
Finally, from (\ref{yestimate2}), (\ref{ztildezestimate2}) and (\ref{xestimate22}), we can choose $\bar{T}=\min\{T_1,T_2,T_3\}$, and for every $0\leq T\leq\bar{T}$ such that
\begin{equation*}
\begin{aligned}
&E\bigg[\sup_{0\leq t\leq T}|x_t|^\beta+\sup_{0\leq t\leq T}|y_t|^\beta+\sum_{i=1}^2\bigg(\bigg(\int_0^T|z^i|^2ds\bigg)^{\frac{\beta}{2}}
+\bigg(\int_0^T\int_{\mathcal{E}_i}|\tilde{z}^i|^2\nu_i(de)ds\bigg)^{\frac{\beta}{2}}\bigg)\bigg]\\
&\leq C_{\beta,K,\rho,\sigma_1,\sigma_2}E\bigg[|x_0|^\beta+|\phi(0)|^\beta+\bigg(\int_0^T|g(t,0,0,0,0,0,0)|ds\bigg)^\beta\\
&\quad +\bigg(\int_0^T|b(s,0,0,0,0,0,0)|ds\bigg)^\beta+\sum_{i=1}^2\bigg(\int_0^T|\sigma_i(s,0,0,0,0,0,0)|^2ds\bigg)^{\frac{\beta}{2}}\\
\end{aligned}
\end{equation*}
\begin{equation}
\begin{aligned}
&\quad +\sum_{i=1}^2\bigg(\int_0^T\int_{\mathcal{E}_i}|f_i(s,0,0,e)|^2N_i(de,ds)\bigg)^{\frac{\beta}{2}}\bigg],
\end{aligned}
\end{equation}
where the positive constant $C_{\beta,K,\rho,\sigma_1,\sigma_2}$ depends on $\beta$ and Lipschitz constants $K,\rho,L_{\sigma_1},L_{\sigma_2}$. The proof is complete.
\end{proof}

\begin{remark}\label{rem26}
Similar to Theorem \ref{the24}, under {\bf (A4), (A5)}, we can also obtain the $L^\beta(\beta\geq2)$-estimate for fully-coupled FBSDEP \eqref{fbsdep1} in $\tilde{\mathcal{M}}^\beta[0,T]$, we immediately give the similar result which is more general than Theorem \ref{the24} as follows.
\end{remark}

\begin{theorem}\label{the25}
Suppose that {\bf (A4), (A5)} hold. Then, for any $\beta\geq2$, suppose that $(x,y,z^1,z^2,\tilde{z}^2,\tilde{z}^2)$ is the solution to \eqref{fbsdep1}, there exists a sufficiently small constant $\tilde{\tilde{T}}>0$ depending on $K,L_{\sigma_1},L_{\sigma_2},L_{f_1},\\L_{f_2},\rho$ and some constant $C_{\beta,K,\rho,\sigma_1,\sigma_2,f_1,f_2}$ depending on $\beta$ and Lipschitz constants $K,\rho,L_{\sigma_1},L_{\sigma_2},L_{f_1},\\L_{f_2}$ such that, for every $0\leq T\leq\tilde{\tilde{T}}$
\begin{equation}
\begin{aligned}
&E\bigg[\sup_{0\leq t\leq T}|x_t|^{\beta}+\sup_{0\leq t\leq T}|y_t|^{\beta}+\sum_{i=1}^2\bigg(\bigg(\int_0^T|z^i|^2dt\bigg)^{\frac{\beta}{2}}+\bigg(\int_0^T\int_{\mathcal{E}_i}|\tilde{z}^i|^2N_i(de,dt)\bigg)^{\frac{\beta}{2}}\bigg)\bigg]\\
&\leq C_{\beta,K,\rho,\sigma_1,\sigma_2,f_1,f_2}E\bigg[|x_0|^{\beta}+|\phi(0)|^{\beta}+\bigg(\int_0^T|g(t,0,0,0,0,0,0)|dt\bigg)^{\beta}\\
&\quad +\bigg(\int_0^T|b(s,0,0,0,0,0,0)|dt\bigg)^{\beta}+\sum_{i=1}^2\bigg(\int_0^T|\sigma_i(s,0,0,0,0,0,0)|^2dt\bigg)^{\frac{\beta}{2}}\\
&\quad +\sum_{i=1}^2\bigg(\int_0^T\int_{\mathcal{E}_i}|f_i(s,0,0,0,0,e)|^2N_i(de,dt)\bigg)^{\frac{\beta}{2}}\bigg].
\end{aligned}
\end{equation}
\end{theorem}
In fact, due to Lemma \ref{lemma24}, Theorem \ref{the25} deduces to Theorem \ref{the24} when changing {\bf (A5)} into {\bf (A6)}.

\begin{remark}\label{rem27}
Theorem \ref{the23} and Theorem \ref{the25} are better and more general results compared to Theorem \ref{the22} and Theorem \ref{the24}, when extending the stronger Assumption {\bf (A6)} into a weaker one {\bf (A5)}. Moreover, they are nice results for studying the stochastic control problem of fully coupled FBSDEPs in the future. However, in our setting of this paper, the partially observed stochastic control problem involves a special fully coupled FBSDEP, and it is enough to solve them by Theorem \ref{the22} and Theorem \ref{the24}.

Therefore, we still use the space $\mathcal{M}^\beta[0,T]$ in the following discussion.
\end{remark}

\subsection{Decoupling random field method for solvability of FBSDEPs}

In the following, under {\bf (A1)}, for any $u\in\mathcal{U}_{ad}[0,T]$, we show that the state equation \eqref{stateeq2} admits a unique solution $(x,y,z^1,z^2,\tilde{z}^1,\tilde{z}^2)$ for any given time $T$ under $P$.

Observing that \eqref{stateeq2} is a kind of special fully coupled FBSDEP in which $\sigma_i,f_i$ do not contain $(z^1,z^2,\tilde{z}^1,\tilde{z}^2)$, then we can give the following theorem which is similar to the main result of \cite{SP20}

\begin{theorem}
Let ${\bf (A4)}$ hold or ${\bf (A1)}$ hold without control $u$, then there exists an $L^2$-solution $(x,y,z^{1},z^{2},\tilde{z}^{1},\tilde{z}^{2})$ to FBSDEP:
\begin{equation}\label{fbstate1}
\left\{
\begin{aligned}
 dx_t&=\tilde{b}_1\big(t,\Theta(t)\big)dt+\sum_{i=1}^2\sigma_i(t,x_t)dW^i_t+\sum_{i=1}^2\int_{\mathcal{E}_i}f_i(t,x_{t-},e)\tilde{N}_i(de,dt),\\
-dy_t&=g\big(t,\Theta(t)\big)dt-\sum_{i=1}^2z^{i}_tdW^i_t-\sum_{i=1}^2\int_{\mathcal{E}_i}\tilde{z}^{i}_{(t,e)}\tilde{N}_i(de,dt),\quad t\in[0,T],\\
  x_0&=x_0,\ \ y_T=\phi(x_T),
\end{aligned}
\right.
\end{equation}
 such that $x$ satisfies the following equation
\begin{equation}\label{xt}
\begin{aligned}
x_t&=x_0+\int_0^t\tilde{b}_1\Big(s,x_s,\theta(s,x_s),\partial_x\theta(s,x_s)\sigma_1(s,x_s),\partial_x\theta(s,x_s)\sigma_2(s,x_s),\\
   &\qquad\qquad\qquad \theta(s,x_s+f_1(s,x_s,e))-\theta(s,x_s),\theta(s,x_s+f_2(s,x_s,e))-\theta(s,x_s)\Big)ds\\
   &\quad +\sum_{i=1}^2\int_0^t\sigma_i(s,x_s)dW^i_s+\sum_{i=1}^2\int_0^t\int_{\mathcal{E}_i}f_i(s,x_{s-},e)\tilde{N}_i(de,ds),
\end{aligned}
\end{equation}
where, under some suitable conditions (see \cite{SP20}), $\theta(t,x)$ is the unique $C^{1,2}_b([0,T],\mathbb{R})$-solution with bounded $\partial_x\theta$ and $\partial_{xx}\theta$ to the following partial integro-differential equation (PIDE for short):
\begin{equation}\label{PIDE}
\left\{
\begin{aligned}
&\partial_x\theta\Big\{\tilde{b}_1\big(t,x,\theta,\partial_x\theta\sigma_1(t,x),\partial_x\theta\sigma_2(t,x),\theta(t,x+f_1(t,x,e))-\theta,\theta(t,x+f_2(t,x,e))-\theta\big)\\
&-\sum_{i=1}^2\int_{\mathcal{E}_i}f_i(t,x,e)\nu_i(de)\Big\}+\frac{1}{2}\sum_{i=1}^2\partial^2_{xx}\theta(\sigma_i(t,x))^2\\
&+g\big(t,x,\theta,\partial_x\theta\sigma_1(t,x),\partial_x\theta\sigma_2(t,x),\theta(t,x+f_1(t,x,e))-\theta,\theta(t,x+f_2(t,x,e))-\theta\big)\\
&+\sum_{i=1}^2\int_{\mathcal{E}_i}(\theta(t,x+f_i(t,x,e))-\theta)\nu_i(de)+\partial_t\theta=0,\qquad \theta(T,x)=\phi(x),
\end{aligned}
\right.
\end{equation}
where $\theta,\partial_t\theta,\partial_x\theta,\partial^2_{xx}\theta$ are everywhere evaluated at $(t,x)$(here, we omit the arguments $(t,x)$ for simplicity). Moreover, $y,z^1,z^2,\tilde{z}^1,\tilde{z}^2$ are explicitly expressed via the solution $\theta$ to PIDE \eqref{PIDE} by the formulas
\begin{equation}\label{relationPIDE}
\begin{aligned}
y_t=\theta(t,x_t),\ z^{i}_t=\partial_x\theta(t,x_t)\sigma_i(t,x_t),\ \tilde{z}^{i}_t=\theta(t,x_{t-}+f_i(t,x_{t-},e))-\theta(t,x_{t-}),\ i=1,2,
\end{aligned}
\end{equation}
and the $L^2$-solution $(x,y,z^{1},z^{2},\tilde{z}^{1},\tilde{z}^{2})$ is unique in $\mathcal{M}^2[0,T]$.
\end{theorem}
\begin{proof}
First, defining $\tilde{\tilde{b}}_1(t,x)=\tilde{b}_1\big(t,x,\theta(t,x),\partial_x\theta(t,x)\sigma_1(t,x),\partial_x\theta(t,x)\sigma_2(t,x),\theta(t,x+f_1(t,x,e))\\-\theta(t,x),\theta(t,x+f_2(t,x,e))-\theta(t,x)\big)$, the assumption {\bf (A4)} (or {\bf (A1)}) and boundedness of $\partial_x\theta,\partial_{xx}\theta$ can guarantee that $\tilde{\tilde{b}}_1, \sigma_1,\sigma_2,f_1,f_2$ satisfy the Lipschitz condition w.r.t $x$ and
\begin{equation}
E\int_0^T\Big|\tilde{\tilde{b}}_1(t,0)\Big|^2dt+\sum_{i=1}^2E\int_0^T|\sigma_i(t,0)|^2dt+\sum_{i=1}^2E\int_{\mathcal{E}_i}\int_0^T|f_i(t,0,e)|^2\nu_i(de)dt<\infty.
\end{equation}
Then \eqref{xt} admits a unique solution $x$. Next, by applying time-dependent It\^{o}'s formula ({\bf Lemma 6}, \cite{SP20}) to $\theta(t,x)$ and combining \eqref{PIDE} satisfied by $\theta(t,x)$, we can prove the existence of solution $y,z^1,z^2,\tilde{z}^1,\tilde{z}^2$ defined by \eqref{relationPIDE} of BSDEP in \eqref{fbstate1}. Finally, the proof of uniqueness is similar to that of \cite{SP20}, so we omit the details.
\end{proof}

The idea above mainly motivates us to consider the decoupling function $\theta$, so we study the following FBSDEP with random coefficients and initial value $x_0=\xi$:
\begin{equation}\label{stateeq20}
\left\{
\begin{aligned}
x^{0,\xi}_t&=\xi+\int_0^t\tilde{b}_1\big(r,\omega,\Theta^{0,\xi}(r)\big)dr+\sum_{i=1}^2\int_0^t\sigma_i(r,\omega,x^{0,\xi}_r)dW^i_r\\
           &\quad +\sum_{i=1}^2\int_0^t\int_{\mathcal{E}_i}f_i(r,\omega,x^{0,\xi}_{r-},e)\tilde{N}_i(de,dr),\\
y^{0,\xi}_t&=\phi(\omega,x_T^{0,\xi})+\int_t^Tg\big(r,\omega,\Theta^{0,\xi}(r)\big)dr-\sum_{i=1}^2\int_t^Tz^{i,0,\xi}_rdW^i_r\\
           &\quad -\sum_{i=1}^2\int_t^T\int_{\mathcal{E}_i}\tilde{z}^{i,0,\xi}_{(r,e)}\tilde{N}_i(de,dr),\quad t\in[0,T],
\end{aligned}
\right.
\end{equation}
where coefficients $\tilde{b}_1,\sigma_1,\sigma_2,f_1,f_2,g$ and $\phi$ are random, and $\sigma_1,\sigma_2,f_1,f_2$ are independent of \\$(y,z^1,z^2,\tilde{z}^1,\tilde{z}^2)$.

Let $0\leq t_1<t_2\leq T$, $\eta$ be an $\mathcal{F}_{t_1}$-measurable square integrable random variable, and $\varphi(\omega,x)$ be a random field such that $\varphi$ is $\mathcal{F}_{t_2}$-measurable for any fixed $x$ and uniformly Lipschitz continuous in $x$ with a Lipschitz constant $\bar{K}$. Consider the following FBSDEP over $[t_1,t_2]$:
\begin{equation}\label{stateeq21}
\left\{
\begin{aligned}
x^{0,\xi}_t&=\eta+\int_{t_1}^t\tilde{b}_1\big(r,\omega,\Theta^{0,\xi}(r)\big)dr+\sum_{i=1}^2\int_{t_1}^t\sigma_i(r,\omega,x^{0,\xi}_r)dW^i_r\\
           &\quad+\sum_{i=1}^2\int_{t_1}^t\int_{\mathcal{E}_i}f_i(r,\omega,x^{0,\xi}_{r-},e)\tilde{N}_i(de,dr),\\
y^{0,\xi}_t&=\varphi(\omega,x_{t_2}^{0,\xi})+\int_t^{t_2}g\big(r,\omega,\Theta^{0,\xi}(r)\big)dr-\sum_{i=1}^2\int_t^{t_2}z^{i,0,\xi}_rdW^i_r\\
           &\quad -\sum_{i=1}^2\int_t^{t_2}\int_{\mathcal{E}_i}\tilde{z}^{i,0,\xi}_{(r,e)}\tilde{N}_i(de,dr),\quad t\in[t_1,t_2].
\end{aligned}
\right.
\end{equation}
Note that, by applying Theorem \ref{the21}, there exists a constant $\tilde{T}(\bar{K})$ depending on the Lipschitz constants in {\bf (A4)} and Lipschitz constant $\bar{K}$ of $\varphi$ such that whenever $t_2-t_1\leq\tilde{T}(\bar{K})$, FBSDEP \eqref{stateeq21} admits a unique $L^2$-solution.

Now, we give the following result, whose non-jump case is introduced by Cvitani\'c and Zhang \cite{CZ12}(see also {\bf Lemma 4.1} \cite{MY21arxiv}).

\begin{theorem}\label{RF}
Let {\bf (A4)} hold, and there exists a random field $\theta(t,\omega,x)$ such that

(i)\ $\theta(T,\omega,x)=\phi(\omega,x)$;

(ii)\ For each $(t,x)$, $\theta$ is $\mathcal{F}_t$-measurable;

(iii)\ $|\theta(t,x_1)-\theta(t,x_2)|\leq \bar{K}|x_1-x_2|$,\ $x_1,x_2\in\mathbb{R}$;

(iv)\ For any $0\leq t_1<t_2\leq T$ such that $|t_2-t_1|\leq \tilde{T}(\bar{K})$, where $\tilde{T}(\bar{K})$, depending on the Lipschitz constants in {\bf (A4)} and Lipschitz constant $\bar{K}$, is a sufficient small constant, the unique solution to FBSDEP \eqref{stateeq21} over $[t_1,t_2]$ with terminal condition $\varphi(\omega,x):=\theta(t_2,\omega,x)$ satisfies $y^{0,\xi}_{t_1}=\theta(t_1,\omega,x^{0,\xi}_{t_1})$.

Then, for any given $T>0$, FBSDEP \eqref{stateeq20} admits a unique $L^2$-solution on $[0,T]$ and $y^{0,\xi}_t=\theta(t,\omega,x_t^{0,\xi})$.
\end{theorem}
\begin{proof}
{\bf Existence.} Let $0=t_0<\cdots<t_n=T$ be a partition of $[0,T]$ such that $t_k-t_{k-1}\leq \tilde{T}(\bar{K})$ for $k=1,\dots,n$. Denote $x_0^{0,\xi,0}:=\xi$. For $k=1,\dots,n$, let $(x^k,y^k,z^{1,k},z^{2,k},\tilde{z}^{1,k},\tilde{z}^{2,k})$ be the unique solution to the following FBSDEP over $[t_{k-1},t_k]$:
\begin{equation}\label{stateeq22}
\left\{
\begin{aligned}
x^{0,\xi,k}_s&=x^{0,\xi,k-1}_{t_{k-1}}+\int_{t_{k-1}}^s\tilde{b}_1\big(r,\omega,\Theta^{0,\xi,k}(r)\big)dr+\sum_{i=1}^2\int_{t_{k-1}}^s\sigma_i(r,\omega,x^{0,\xi,k}_r)dW^i_r\\
             &\quad +\sum_{i=1}^2\int_{t_{k-1}}^s\int_{\mathcal{E}_i}f_i(r,\omega,x^{0,\xi,k}_{r-},e)\tilde{N}_i(de,dr),\\
y^{0,\xi,k}_s&=\theta(t_k,\omega,x_{t_k}^{0,\xi,k})+\int_s^{t_k}g\big(r,\omega,\Theta^{0,\xi,k}(r)\big)dr-\sum_{i=1}^2\int_s^{t_k}z^{i,0,\xi,k}_rdW^i_r\\
             &\quad -\sum_{i=1}^2\int_s^{t_k}\int_{\mathcal{E}_i}\tilde{z}^{i,0,\xi,k}_{(r,e)}\tilde{N}_i(de,dr).
\end{aligned}
\right.
\end{equation}
Define
\begin{equation}
\begin{aligned}
x_t^{0,\xi}:=&\sum_{k=1}^nx_t^{0,\xi,k}1_{[t_{k-1},t_k)}(t)+x_T^{0,\xi,n}1_{\{T\}}(t),\\
y_t^{0,\xi}:=&\sum_{k=1}^ny_t^{0,\xi,k}1_{[t_{k-1},t_k)}(t)+y_T^{0,\xi,n}1_{\{T\}}(t),\\
z_t^{i,0,\xi}:=&\sum_{k=1}^nz_t^{i,0,\xi,k}1_{[t_{k-1},t_k)}(t)+z_T^{i,0,\xi,n}1_{\{T\}}(t),\ i=1,2,\\
\tilde{z}_{(t,e)}^{i,0,\xi}:=&\sum_{k=1}^n\tilde{z}_{(t,e)}^{i,0,\xi,k}1_{[t_{k-1},t_k)}(t)+\tilde{z}_{(T,e)}^{i,0,\xi,n}1_{\{T\}}(t),\ i=1,2.\\
\end{aligned}
\end{equation}
Note that we should consider a necessary situation that the jumps occur at the time $t_k,k=1,\dots,n$. Without loss of generality, we consider that the jump appears at $t_k$ and for forward equation, $x_t^{0,\xi,k}$ is an adapted RCLL process on $[t_{k-1},t_k], k=1,\dots,n$. Due to the appearance of  jump at $t_k$, it can be understood by the following equation from \eqref{stateeq22}:
\begin{equation}\label{1}
\begin{aligned}
(i)\ x^{0,\xi,k+1}_{t_k}&=x^{0,\xi,k}_{t_k}+\int_{t_k}^{t_k}\tilde{b}_1\big(r,\omega,\Theta^{0,\xi,k+1}(r)\big)dr+\sum_{i=1}^2\int_{t_k}^{t_k}\sigma_i(r,\omega,x^{0,\xi,k+1}_r)dW^i_r\\
                        &\quad+\sum_{i=1}^2\int_{t_k}^{t_k}\int_{\mathcal{E}_i}f_i(r,\omega,x^{0,\xi,k+1}_{r-},e)\tilde{N}_i(de,dr)\\
                        &=x^{0,\xi,k}_{t_k},
\end{aligned}
\end{equation}
\begin{equation}\label{2}
\begin{aligned}
(ii)\ x^{0,\xi,k}_{t_k}&=x^{0,\xi,k-1}_{t_{k-1}}+\int_{t_{k-1}}^{t_k}\tilde{b}_1\big(r,\omega,\Theta^{0,\xi,k}(r)\big)dr+\sum_{i=1}^2\int_{t_{k-1}}^{t_k}\sigma_i(r,\omega,x^{0,\xi,k}_r)dW^i_r\\
                       &\quad +\sum_{i=1}^2\int_{t_{k-1}}^{t_k}\int_{\mathcal{E}_i}f_i(r,\omega,x^{0,\xi,k}_{r-},e)\tilde{N}_i(de,dr).
\end{aligned}
\end{equation}
In (i), $x_{t_k}^{0,\xi,k+1}$ can be regarded as the initial value of $x_t^{0,\xi,k+1}$ on $[t_k,t_{k+1}]$. In (ii), $x_{t_k}^{0,\xi,k}$ can be regraded as the terminal value of $x_t^{0,\xi,k}$ on $[t_{k-1},t_k]$.

Because there exists a jump at $t_k$, the integrals, especially for $\int_{t_k}^{t_k}\int_{\mathcal{E}_i}f_i(r,\omega,x^{0,\xi,k+1}_{r-},e)\tilde{N}_i(de,\\dr)$ in (i), are equal to 0. However, the integrals, like $\int_{t_{k-1}}^{t_k}\int_{\mathcal{E}_i}f_i(r,\omega,x^{0,\xi,k}_{r-},e)\tilde{N}_i(de,dr)$ in (ii), are not equal to 0.

In particular, when $f_i\equiv1$, $\int_{t_k}^{t_k}\int_{\mathcal{E}_i}1\tilde{N}_i(de,dr)=\tilde{N}_i(\mathcal{E}_i,t_k)=0$,\ i.e., the quantity of jump at $t_k$ which is the initial time of $[t_k,t_{k+1}]$, is 0 for (i). And $\int_{t_{k-1}}^{t_k}\int_{\mathcal{E}_i}1\tilde{N}_i(de,dr)=\tilde{N}_i(\mathcal{E}_i,[t_{k-1},t_k])=1$,\ i.e., the quantity of jump on time interval $[t_{k-1},t_k]$ is 1 for (ii), and the jump just occur at $t_k$ which is the terminal time of $[t_{k-1},t_k]$. We have similar analysis for $y_t^{0,\xi,k}$ of the backward equation.

Based on the above analysis, we have $x_{t_{k-1}}^{0,\xi,k}=x_{t_{k-1}}^{0,\xi,k-1}$ and $y_{t_{k-1}}^{0,\xi,k}=\theta(t_{k-1},\omega,x^{0,\xi,k}_{t_{k-1}})=\theta(t_{k-1},\omega,x^{0,\xi,k-1}_{t_{k-1}})=y_{t_{k-1}}^{0,\xi,k-1}$.

Therefore, for any two adjacent intervals, $[t_{k-1},t_k]$ and $[t_k,t_{k+1}], k=1,\dots,n-1$, the initial value of $(x^{0,\xi,k+1},y^{0,\xi,k+1})$ at initial $t_k$ of current interval $[t_k,t_{k+1}]$ is equal to the terminal value of $(x^{0,\xi,k},y^{0,\xi,k})$ at terminal $t_k$ of previous interval $[t_{k-1},t_k]$. Moreover, we can connect these small pieces naturally and can check straightforwardly that $(x^{0,\xi},y^{0,\xi},z^{1,0,\xi},z^{2,0,\xi},\tilde{z}^{1,0,\xi},\tilde{z}^{2,0,\xi})$ solves \eqref{stateeq20}.

{\bf Uniqueness.} The proof is similar to that in \cite{CZ12}.

Finally, for any $t\in[t_{k-1},t_k]$, consider the FBSDEP on $[t,t_k]$, we see that $y^{0,\xi}_t=\theta(t,\omega,x_t^{0,\xi})$. The proof is complete.
\end{proof}

Now, we give a sufficient condition for the existence of the random field $\theta$ in Theorem \ref{RF}.
\begin{theorem}
Let {\bf (A4)} hold. $\tilde{b}_1,\sigma_1,\sigma_2,f_1,f_2$ are continuously differentiable in $(x,y,z^1,z^2,\\\tilde{z}^1,\tilde{z}^2)$ and assume that
\begin{equation}\label{Condition1}
\tilde{b}_{1y}+\sum_{i=1}^2\tilde{b}_{1z^i}\sigma_{ix}+\sum_{i=1}^2\tilde{b}_{1\tilde{z}^i}\int_{\mathcal{E}_i}f_{ix}\nu_i(de)=0.
\end{equation}
Then there exists a random field $\theta$ satisfying the conditions in Theorem \ref{RF}, and consequently FBSDEP \eqref{stateeq20} admits a unique solution.
\end{theorem}
\begin{proof}
{\bf Step 1.} Let $K$ denote the Lipschitz constant of $\tilde{b}_1,\sigma_1,\sigma_2,f_1,f_2,g$ with respect to $(x,y,z^1,z^2,\tilde{z}^1,\tilde{z}^2)$ and $K_0$ the Lipschitz constant of $\phi$ with respect to $x$. Denote
\begin{equation}
\bar{K}:=e^{\hat{K}T}(1+K_0)-1,
\end{equation}
where $\hat{K}:=(2+\sum_{i=1}^2\nu_i(\mathcal{E}_i))K^2+2K$ and let $\tilde{T}(\bar{K})$ be the constant introduced in Theorem \ref{RF}. Fix a partition $0=t_0<\cdots<t_n=T$ such that $\Delta t_k=t_k-t_{k-1}\leq\tilde{T}(\bar{K}),k=1,\dots,n$. For each $(t,x)\in[t_{n-1},t_n]\times\mathbb{R}$, consider the following FBSDEP on $[t,t_n]$:
\begin{equation}\label{stateeq23}
\left\{
\begin{aligned}
x^{t,x}_s&=x+\int_{t}^s\tilde{b}_1\big(r,\omega,\Theta^{t,x}(r)\big)dr+\sum_{i=1}^2\int_{t}^s\sigma_i(r,\omega,x^{t,x}_r)dW^i_r\\
         &\quad +\sum_{i=1}^2\int_{t}^s\int_{\mathcal{E}_i}f_i(r,\omega,x^{t,x}_{r-},e)\tilde{N}_i(de,dr),\\
y^{t,x}_s&=\theta(t_n,\omega,x_{t_n}^{t,x})+\int_s^{t_n}g\big(r,\omega,\Theta^{t,x}(r)\big)dr-\sum_{i=1}^2\int_s^{t_n}z^{i,t,x}_rdW^i_r\\
         &\quad -\sum_{i=1}^2\int_s^{t_n}\int_{\mathcal{E}_i}\tilde{z}^{i,t,x}_{(r,e)}\tilde{N}_i(de,dr).
\end{aligned}
\right.
\end{equation}
By Theorem \ref{the21}, FBSDEP \eqref{stateeq23} admits a unique solution. Define $\theta(t,x):=y^{t,x}_t, t\in[t_{n-1},t_n]$.

{\bf Step 2.} Given $x_1,x_2$ and $t\in[t_{n-1},t_n]$, denote $\Delta x:=x_1-x_2$, $\Delta\Theta:=\Theta^{t,x_1}-\Theta^{t,x_2}$. Then $\Delta\Theta$ satisfies the following FBSDEP:
\begin{equation}\label{stateeq24}
\left\{
\begin{aligned}
\Delta x_s&=\Delta x+\int_t^s[\tilde{b}^\prime_{1x}(r)\Delta x_r+\tilde{b}^\prime_{1y}(r)\Delta y_r+\tilde{b}^\prime_{1z^1}(r)\Delta z^1_r+\tilde{b}^\prime_{1z^2}(r)\Delta z^2_r\\
          &\quad +\tilde{b}^\prime_{1\tilde{z}^1}(r)\int_{\mathcal{E}_1}\Delta \tilde{z}^1_{(r,e)}\nu_1(de)+\tilde{b}^\prime_{1\tilde{z}^2}(r)\int_{\mathcal{E}_2}\Delta \tilde{z}^2_{(r,e)}\nu_2(de)]dr\\
          &\quad +\sum_{i=1}^2\int_t^s\sigma^\prime_{ix}(r)\Delta x_rdW^i_r+\sum_{i=1}^2\int_t^s\int_{\mathcal{E}_i}f^\prime_{ix}(r,e)\Delta x_{r-}\tilde{N}_i(de,dr)\\
\Delta y_s&=\phi^\prime_x(t_n)\Delta x_{t_n}+\int_s^{t_n}[g^\prime_{x}(r)\Delta x_r+g^\prime_{y}(r)\Delta y_r+g^\prime_{z^1}(r)\Delta z^1_r+g^\prime_{z^2}(r)\Delta z^2_r\\
          &\quad +g^\prime_{\tilde{z}^1}(r)\int_{\mathcal{E}_1}\Delta \tilde{z}^1_{(r,e)}\nu_1(de)+g^\prime_{\tilde{z}^2}(r)\int_{\mathcal{E}_2}\Delta \tilde{z}^2_{(r,e)}\nu_2(de)]dr\\
          &\quad -\sum_{i=1}^2\int_s^{t_n}\Delta z^i_rdW^i_r-\sum_{i=1}^2\int_s^{t_n}\int_{\mathcal{E}_i}\Delta \tilde{z}^i_{(r,e)}\tilde{N}_i(de,dr),
\end{aligned}
\right.
\end{equation}
where $\tilde{b}^\prime_{1x}(r):=\int_0^1\tilde{b}_{1x}(r,\omega,\Theta^{t,x_2}(r)+\lambda(\Theta^{t,x_1}(r)-\Theta^{t,x_2}(r)))d\lambda$ and $\tilde{b}^\prime_{1y}(r),\tilde{b}^\prime_{1z^i}(r),\tilde{b}^\prime_{1\tilde{z}^i}(r),\sigma^\prime_{ix}(r),\\f^\prime_{ix}(r,e),g^\prime_{x}(r),g^\prime_{y}(r),g^\prime_{z^i}(r),g^\prime_{\tilde{z}^i}(r),\phi^\prime_x(t_n), i=1,2,$ have similar definition. Moreover, $|\alpha_\beta(r,e)|\leq K$, for $\alpha=\tilde{b}_1,\sigma_1,\sigma_2,f_1,f_2$ and $g$, $\beta=x,y,z^1,z^2,\tilde{z}^1,\tilde{z}^2$, and $|\phi_x^\prime(t_n)|\leq K_0$. Then, FBSDEP \eqref{stateeq24} also satisfies {\bf (A4)}. Here, we note that FBSDEP \eqref{stateeq24} is on time interval $[t,t_n]$, and all the arguments in Theorem \ref{the21} and Theorem \ref{the24} hold by replacing the expectation with conditional expectation. By Theorem \ref{the24}, we have
\begin{equation}
\begin{aligned}
E\bigg[\sup_{t\leq s\leq t_n}(|\Delta x_s|^2+|\Delta y_s|^2)+\sum_{i=1}^2\bigg(\int_t^{t_n}|\Delta z^i_s|^2ds+\int_t^{t_n}\int_{\mathcal{E}_i}|\Delta\tilde{z}^i_{(s,e)}|^2\nu_i(de)ds\bigg)\bigg|\mathcal{F}_t\bigg]\leq C|\Delta x|^2,
\end{aligned}
\end{equation}
where the precise form of constant $C$ will be considered later. In particular, this implies
\begin{equation}
|\theta(t,x_1)-\theta(t,x_2)|^2=|\Delta y_t|^2\leq E[|\Delta y_t|^2|\mathcal{F}_t]\leq E[\sup_{t\leq s\leq t_n}|\Delta y_s|^2|\mathcal{F}_t]\leq C|\Delta x|^2
\end{equation}
i.e.,
\begin{equation}\label{UFC}
|\theta(t,x_1)-\theta(t,x_2)|\leq \sqrt{C}|x_1-x_2|.
\end{equation}
That is, for $t\in[t_{n-1},t_n]$, $\theta(t,x)$ is uniformly Lipschitz continuous in $x$.

{\bf Step 3.} Let $\eta=\sum_{j=1}^mx_j1_{E_j}$, where $x_1,\cdots,x_m\in\mathbb{R}$ and $E_1,\cdots,E_m\in\mathcal{F}_t$ form a partition of $\Omega$. One can check straightforwardly that $\Theta^{t,\eta}:=\sum_{j=1}^m\Theta^{t,x_j}1_{E_j}$ satisfies FBSDEP \eqref{stateeq23} with initial value $x^{t,\eta}_t:=\eta$. In particular, this implies that
\begin{equation}
y^{t,\eta}_t=\sum_{j=1}^my^{t,x_j}_t1_{E_j}=\sum_{j=1}^m\theta(t,\omega,x_j)1_{E_j}=\theta(t,\omega,\sum_{j=1}^mx_j1_{E_j})=\theta(t,\omega,\eta).
\end{equation}
Moreover, for general $\mathcal{F}_t$-measurable square integrable $\eta$, there exist $\{\eta_m,m\geq1\}$ taking the above form such that $\lim_{m\rightarrow\infty}E[|\eta_m-\eta|^2]=0$. Denote again by $\Theta^{t,\eta}$ the unique solution to FBSDEP \eqref{stateeq23} with initial value $x^{t,\eta}_t=\eta$. By the previous analysis, we obtain that
\begin{equation}
E[|y^{t,\eta}-y^{t,\eta_m}|^2]\leq CE[|\eta-\eta_m|^2]\rightarrow0.
\end{equation}
Since $\theta(t,x)$ is Lipschitz continuous in $x$, we get
\begin{equation}
y_t^{t,\eta}=\lim_{m\rightarrow\infty}y_t^{t,\eta_m}=\lim_{m\rightarrow\infty}\theta(t,\omega,\eta_m)=\theta(t,\omega,\lim_{m\rightarrow\infty}\eta_m)=\theta(t,\omega,\eta).
\end{equation}

{\bf Step 4.} We now get a more precise estimate for constant $C$ in \eqref{UFC}. Since $\tilde{b}_1,\sigma_1,\sigma_2,f_1,f_2,g$ and $\phi$ are continuously differentiable, by standard arguments, one can see that $\theta(t,x)$ is differentiable in $x$ with $\theta_x(t,x)=\nabla y^{t,x}_t$ where

\begin{equation}\label{stateeq25}
\left\{
\begin{aligned}
\nabla x_s&=1+\int_t^s[\tilde{b}_{1x}(r)\nabla x_r+\tilde{b}_{1y}(r)\nabla y_r+\tilde{b}_{1z^1}(r)\nabla z^1_r+\tilde{b}_{1z^2}(r)\nabla z^2_r\\
          &\quad +\tilde{b}_{1\tilde{z}^1}(r)\int_{\mathcal{E}_1}\nabla \tilde{z}^1_{(r,e)}\nu_1(de)+\tilde{b}_{1\tilde{z}^2}(r)\int_{\mathcal{E}_2}\nabla \tilde{z}^2_{(r,e)}\nu_2(de)]dr\\
          &\quad +\sum_{i=1}^2\int_t^s\sigma_{ix}(r)\nabla x_rdW^i_r+\sum_{i=1}^2\int_t^s\int_{\mathcal{E}_i}f_{ix}(r,e)\nabla x_{r-}\tilde{N}_i(de,dr)\\
\nabla y_s&=\phi_x(\omega,x_{t_n})\nabla x_{t_n}+\int_s^{t_n}[g_{x}(r)\nabla x_r+g_{y}(r)\nabla y_r+g_{z^1}(r)\nabla z^1_r+g_{z^2}(r)\nabla z^2_r\\
          &\quad +g_{\tilde{z}^1}(r)\int_{\mathcal{E}_1}\nabla \tilde{z}^1_{(r,e)}\nu_1(de)+g_{\tilde{z}^2}(r)\int_{\mathcal{E}_2}\nabla \tilde{z}^2_{(r,e)}\nu_2(de)]dr\\
          &\quad -\sum_{i=1}^2\int_s^{t_n}\nabla z^i_rdW^i_r-\sum_{i=1}^2\int_s^{t_n}\int_{\mathcal{E}_i}\nabla \tilde{z}^i_{(r,e)}\tilde{N}_i(de,dr),
\end{aligned}
\right.
\end{equation}
where $\tilde{b}_{1x}(r):=\tilde{b}_{1x}(r,\omega,\Theta(r))$ and other terms are defined similarly.

Denote
\begin{equation}
\begin{aligned}
\hat{y}_s:&=\nabla y_s(\nabla x_s)^{-1},\ \bar{z}^i_s:=\nabla z^i_s(\nabla x_s)^{-1},\ \bar{\tilde{z}}^i_{(s,e)}:=\nabla\tilde{z}^i_{(s,e)}(\nabla x_s)^{-1},\\
\hat{z}^i_s:&=\bar{z}^i_s-\sigma_{ix}\hat{y}_s,\ \hat{\tilde{z}}^i_{(s,e)}:=\bar{\tilde{z}}^i_{(s,e)}(1+f_{ix})^{-1}-\hat{y}_sf_{ix}(1+f_{ix})^{-1},\ i=1,2.
\end{aligned}
\end{equation}
Applying It\^o's formula, we obtain
\begin{equation}
\begin{aligned}
d(\nabla x_s)^{-1}&=-(\nabla x_s)^{-1}\bigg[\tilde{b}_{1x}+\tilde{b}_{1y}\hat{y}_s+\sum_{i=1}^2\bigg(\tilde{b}_{1z^i}\bar{z}^i_s+\tilde{b}_{1\tilde{z}^i}\int_{\mathcal{E}_i}\bar{\tilde{z}}^i_{(s,e)}\nu_i(de)-\sigma^2_{ix}\\
                  &\quad -\int_{\mathcal{E}_i}f_{ix}^2(1+f_{ix})^{-1}\nu_i(de)\bigg)\bigg]ds-(\nabla x_s)^{-1}\sum_{i=1}^2\sigma_{ix}dW^i_s\\
                  &\quad -(\nabla x_s)^{-1}\sum_{i=1}^2\int_{\mathcal{E}_i}f_{ix}(1+f_{ix})^{-1}\tilde{N}_i(de,ds),\\
        d\hat{y}_s&=-\hat{y}_s\bigg[\tilde{b}_{1x}+\tilde{b}_{1y}\hat{y}_s+\sum_{i=1}^2\bigg(\tilde{b}_{1z^i}\bar{z}^i_s+\tilde{b}_{1\tilde{z}^i}\int_{\mathcal{E}_i}\bar{\tilde{z}}^i_{(s,e)}\nu_i(de)-\sigma_{ix}^2\\
                  &\quad -\int_{\mathcal{E}_i}f^2_{ix}(1+f_{ix})^{-1}\nu_i(de)\bigg)\bigg]ds-\bigg[g_x+g_y\hat{y}_s+\sum_{i=1}^2g_{z^i}\bar{z}^i_s\\
                  &\quad +\sum_{i=1}^2g_{\tilde{z}^i}\int_{\mathcal{E}_i}\bar{\tilde{z}}^i_{(s,e)}\nu_i(de)\bigg]ds-\sum_{i=1}^2\int_{\mathcal{E}_i}f_{ix}(1+f_{ix})^{-1}\bar{\tilde{z}}^i_{(s,e)}\nu_i(de)ds\\
                  &\quad -\sum_{i=1}^2\sigma_{ix}\bar{z}^i_sds+\sum_{i=1}^2\hat{z}^i_sdW^i_s+\sum_{i=1}^2\int_{\mathcal{E}_i}\hat{\tilde{z}}^i_{(s,e)}\tilde{N}_i(de,ds).
\end{aligned}
\end{equation}
Note that $\bar{z}^i_s=\hat{z}^i_s+\sigma_{ix}\hat{y}_s$, $\bar{\tilde{z}}^i_{(s,e)}=\hat{\tilde{z}}^i_{(s,e)}(1+f_{ix})+\hat{y}_sf_{ix}$, $i=1,2$, we obtain
\begin{equation}
\begin{aligned}
d\hat{y}_s&=-\bigg[g_x+\bigg(\tilde{b}_{1x}+g_y+\sum_{i=1}^2g_{z^i}\sigma_{ix}+\sum_{i=1}^2g_{\tilde{z}^i}\int_{\mathcal{E}_i}f_{ix}\nu_i(de)\bigg)\hat{y}_s\\
          &\quad +\bigg(\tilde{b}_{1y}+\sum_{i=1}^2\tilde{b}_{1z^i}\sigma_{ix}+\sum_{i=1}^2\tilde{b}_{1\tilde{z}^i}\int_{\mathcal{E}_i}f_{ix}\nu_i(de)\bigg)\hat{y}^2_s\bigg]ds\\
          &\quad -\sum_{i=1}^2\big[\tilde{b}_{1z^i}\hat{y}_s+g_{z^i}+\sigma_{ix}\big]\hat{z}^i_sds+\sum_{i=1}^2\hat{z}^i_sdW^i_s+\sum_{i=1}^2\int_{\mathcal{E}_i}\hat{\tilde{z}}^i_{(s,e)}\tilde{N}_i(de,ds)\\
          &\quad -\sum_{i=1}^2\bigg[(\tilde{b}_{1\tilde{z}^i}\hat{y}_s+g_{\tilde{z}^i})\int_{\mathcal{E}_i}(1+f_{ix})\hat{\tilde{z}}^i_{(s,e)}\nu_i(de)+\int_{\mathcal{E}_i}f_{ix}\hat{\tilde{z}}^i_{(s,e)}\nu_i(de)\bigg]ds.
\end{aligned}
\end{equation}
We note that $\hat{y}_s=\theta_x(s,x_s)$ and is bounded for $s\in[t,t_n]\subset[t_{n-1},t_n]$. Denote
\begin{equation}
\begin{aligned}
d\hat{P}:&=\exp\bigg\{\sum_{i=1}^2\int_t^{t_n}(\tilde{b}_{1z^i}\hat{y}_s+g_{z^i}+\sigma_{ix})dW^i_s-\frac{1}{2}\sum_{i=1}^2\int_t^{t_n}|\tilde{b}_{1z^i}\hat{y}_s+g_{z^i}+\sigma_{ix}|^2ds\\
         &\quad +\sum_{i=1}^2\int_t^{t_n}\int_{\mathcal{E}_i}\log[(\tilde{b}_{1\tilde{z}^i}\hat{y}_s+g_{\tilde{z}^i}+1)(1+f_{ix})]N_i(de,ds)\\
         &\quad +\sum_{i=1}^2\int_t^{t_n}\int_{\mathcal{E}_i}[1-(\tilde{b}_{1\tilde{z}^i}\hat{y}_s+g_{\tilde{z}^i}+1)(1+f_{ix})]\nu_i(de)ds\bigg\}dP,
\end{aligned}
\end{equation}
and introduce
\begin{equation}
\begin{aligned}
\Gamma_s&=\exp\bigg\{\int_t^s\bigg[\tilde{b}_{1x}+g_y+\sum_{i=1}^2g_{z^i}\sigma_{ix}+\sum_{i=1}^2g_{\tilde{z}^i}\int_{\mathcal{E}_i}f_{ix}\nu_i(de)\\
        &\quad +\bigg(\tilde{b}_{1y}+\sum_{i=1}^2\tilde{b}_{1z^i}\sigma_{ix}+\sum_{i=1}^2\tilde{b}_{1\tilde{z}^i}\int_{\mathcal{E}_i}f_{ix}\nu_i(de)\bigg)\hat{y}_r\bigg]dr\bigg\}.
\end{aligned}
\end{equation}
Then we have
\begin{equation}
\hat{y}_t=E^{\hat{P}}\bigg[\Gamma_{t_n}\hat{y}_{t_n}+\int_t^{t_n}g_x\Gamma_sds\bigg|\mathcal{F}_t\bigg].
\end{equation}
Note that $|\hat{y}_{t_n}|=|\phi_x(x_{t_n})|\leq K_0$ and combine with condition \eqref{Condition1}
\begin{equation}
\begin{aligned}
&\tilde{b}_{1x}+g_y+\sum_{i=1}^2g_{z^i}\sigma_{ix}+\sum_{i=1}^2g_{\tilde{z}^i}\int_{\mathcal{E}_i}f_{ix}\nu_i(de)+\bigg(\tilde{b}_{1y}+\sum_{i=1}^2\tilde{b}_{1z^i}\sigma_{ix}+\sum_{i=1}^2\tilde{b}_{1\tilde{z}^i}\int_{\mathcal{E}_i}f_{ix}\nu_i(de)\bigg)\hat{y}_r\\
&=\tilde{b}_{1x}+g_y+\sum_{i=1}^2g_{z^i}\sigma_{ix}+\sum_{i=1}^2g_{\tilde{z}^i}\int_{\mathcal{E}_i}f_{ix}\nu_i(de)\\
&\leq \bigg(2+\sum_{i=1}^2\nu_i(\mathcal{E}_i)\bigg)K^2+2K=\hat{K},
\end{aligned}
\end{equation}
thus we have $\Gamma_s\leq e^{\hat{K}(s-t)}$. Then,
\begin{equation}
\hat{y}_t\leq e^{\hat{K}(t_n-t)}K_0+\int_t^{t_n}Ke^{\hat{K}(s-t)}ds\leq e^{\hat{K}(t_n-t)}K_0+e^{\hat{K}(t_n-t)}-1=e^{\hat{K}(t_n-t)}(K_0+1)-1,
\end{equation}
which implies that
\begin{equation}
|\theta_x(t,x)|\leq \tilde{K}:=e^{\hat{K}(t_n-t_{n-1})}(K_0+1)-1,\ t\in[t_{n-1},t_n].
\end{equation}

{\bf Step 5.} We note that $\tilde{K}\leq \bar{K}$. Assume that $\theta$ is defined on $[t_k,t_n]$ with $|\theta_x(t_k,x)|\leq K_k\leq\bar{K}$. Define $\theta(t,x):=y^{t,x}_t$ for $t\in[t_{k-1},t_k]$, $k=n-1,n-2,\cdots,1$, where
\begin{equation}
\left\{
\begin{aligned}
x^{t,x}_s&=x+\int_{t}^s\tilde{b}_1\big(r,\omega,\Theta^{t,x}(r)\big)dr+\sum_{i=1}^2\int_{t}^s\sigma_i(r,\omega,x^{t,x}_r)dW^i_r\\
         &\quad +\sum_{i=1}^2\int_{t}^s\int_{\mathcal{E}_i}f_i(r,\omega,x^{t,x}_{r-},e)\tilde{N}_i(de,dr),\\
y^{t,x}_s&=\theta(t_k,\omega,x_{t_k}^{t,x})+\int_s^{t_k}g\big(r,\omega,\Theta^{t,x}(r)\big)dr-\sum_{i=1}^2\int_s^{t_k}z^{i,t,x}_rdW^i_r\\
         &\quad -\sum_{i=1}^2\int_s^{t_k}\int_{\mathcal{E}_i}\tilde{z}^{i,t,x}_{(r,e)}\tilde{N}_i(de,dr),
\end{aligned}
\right.
\end{equation}
by the analysis in {\bf Step 4}, we can obtain
\begin{equation}
|\theta_x(t,x)|\leq K_{k-1}:=e^{\hat{K}(t_k-t_{k-1})}(K_k+1)-1, t\in[t_{k-1},t_k].
\end{equation}
By induction, note that $K_n=K_0$, we get
\begin{equation}
K_k=e^{\hat{K}(t_n-t_k)}(K_0+1)-1\leq \bar{K},\ k=1,\dots,n.
\end{equation}

Therefore, the backward induction can continue until $k=1$, and thus we may define $\theta$ on $[0,T]$ and $|\theta_x(t,x)|\leq\bar{K}$ for all $t\in[0,T]$. Finally, it is clear that $\theta$ satisfies the other requirements of Theorem \ref{RF} and the existence of $\theta$ is proved.
\end{proof}

Then next lemma shows the $L^2$-estimation of FBSDEP \eqref{stateeq20}.

\begin{lemma}\label{L^2estimation}
Let {\bf (A4)} hold, and $\theta(t,x)$ satisfies the conditions in Theorem \ref{RF}. Then we obtain the $L^2$-estimations of FBSDEP \eqref{stateeq20}
\begin{equation}
\begin{aligned}
&E\bigg[\sup_{0\leq s\leq T}|x^{0,\xi}_s-x^{0,\tilde{\xi}}_s|^2+\sup_{0\leq s\leq T}|y^{0,\xi}_s-y^{0,\tilde{\xi}}_s|^2+\sum_{i=1}^2\int_0^T|z^{i,0,\xi}_s-z^{i,0,\tilde{\xi}}_s|^2ds\\
&\qquad+\sum_{i=1}^2\int_0^T\int_{\mathcal{E}_i}|\tilde{z}^{i,0,\xi}_{(s,e)}-\tilde{z}^{i,0,\tilde{\xi}}_{(s,e)}|^2\nu_i(de)ds\bigg|\mathcal{F}_0\bigg]\leq C|\xi-\tilde{\xi}|^2,
\end{aligned}
\end{equation}
\begin{equation}
\begin{aligned}
&E\bigg[\sup_{t\leq s\leq T}|x^{0,\xi}_s-x^{0,\tilde{\xi}}_s|^2+\sup_{t\leq s\leq T}|y^{0,\xi}_s-y^{0,\tilde{\xi}}_s|^2+\sum_{i=1}^2\int_t^T|z^{i,0,\xi}_s-z^{i,0,\tilde{\xi}}_s|^2ds\\
&\qquad+\sum_{i=1}^2\int_t^T\int_{\mathcal{E}_i}|\tilde{z}^{i,0,\xi}_{(s,e)}-\tilde{z}^{i,0,\tilde{\xi}}_{(s,e)}|^2\nu_i(de)ds\bigg|\mathcal{F}_t\bigg]\leq C|x^{0,\xi}_t-x^{0,\tilde{\xi}}_t|^2,
\end{aligned}
\end{equation}
where constant $C$ depends on $\beta,K,\rho$ and $L$ in {\bf (A4)}.
\begin{proof}
Based on the idea in the proof of Theorem \ref{MYextension}, it is easy to prove the two estimations.
\end{proof}
\end{lemma}
Therefore, for any $u\in\mathcal{U}_{ad}[0,T]$ and any given terminal time $T$, under {\bf (A1)} and the conditions in Theorem \ref{RF}, and combining Theorem \ref{RF}, Lemma \ref{L^2estimation} and Theorem \ref{MYextension}, there exists a unique $L^\beta(\beta>2)$-solution $(x,y,z^1,z^2,\tilde{z}^1,\tilde{z}^2)$ to state equation \eqref{stateeq2} on $[0,T]$.

\section{Global Maximum Principle for Partially Observed Optimal Control Problem of FBSDEPs}

In this section, we study the partially observed global maximum principe for the optimal control problem of FBSDEPs. The main result is an extension of the work in \cite{Hu17} and \cite{STW20}.

\subsection{Spike variation}

We consider the state equation \eqref{stateeq1}, observation equation
\begin{equation}\label{observation000}
\left\{
\begin{aligned}
dY_t^u&=b_2(t,x_t^u,u_t)dt+\sigma_3(t)d\tilde{W}^2_t+\int_{\mathcal{E}_2}f_3(t,e)\tilde{N}^{\prime}_2(de,dt),\quad t\in[0,T],\\
 Y_0^u&=0,
\end{aligned}
\right.
\end{equation}
and the cost functional \eqref{cf1}. Noting that $b_2$ is independent of $(y,z^1_t,z^2_t,\tilde{z}^1_{(t,e)},\tilde{z}^2_{(t,e)})$ here, comparing with \eqref{observation}. The general $b_2$ case will not be considered in this section, which is our future research topic.

According to Section 2, we can transfer the partially observed optimal control problem into the complete information case with the controlled FBSDEP
\begin{equation*}
\left\{
\begin{aligned}
 dx_t^u&=\tilde{b}_1\big(t,x_t^u,u_t\big)dt+\sum^2_{i=1}\sigma_i(t,x_t^u,u_t)dW^i_t+\sum^2_{i=1}\int_{\mathcal{E}_i}f_i(t,x_{t-}^u,u_t,e)\tilde{N}_i(de,dt),\\
-dy_t^u&=g\big(t,x_t^u,y_t^u,z^{1,u}_t,z^{2,u}_t,\tilde{z}^{1,u}_{(t,e)},\tilde{z}^{2,u}_{(t,e)},u_t\big)dt-\sum^2_{i=1}z^{i,u}_tdW^i_t-\sum^2_{i=1}\int_{\mathcal{E}_i}\tilde{z}^{i,u}_{(t,e)}\tilde{N}_i(de,dt),\quad t\in[0,T],\\
  x_0^u&=x_0,\quad y_T^u=\phi(x_T^u),
\end{aligned}
\right.
\end{equation*}
and \eqref{RN2}, where $\tilde{b}_1$ does not contain $(y,z^1_t,z^2_t,\tilde{z}^1_{(t,e)},\tilde{z}^2_{(t,e)})$, which admits a unique solution $(x^u,y^u,z^{1,u},z^{2,u},\tilde{z}^{1.u},\tilde{z}^{2,u},\tilde{\Gamma}^u)$ under $P$ by {\bf (A1)}.
Moreover, the cost functional \eqref{cf1} becomes
\begin{equation}\label{CCFF}
J(u)=E\bigg[\int_0^T\tilde{\Gamma}_tl\big(t,x_t^u,y_t^u,z^{1,u}_t,z^{2,u}_t,\tilde{z}^{1,u}_{(t,e)},\tilde{z}^{2,u}_{(t,e)},u_t\big)dt+\tilde{\Gamma}_T\Phi(x_T^u)+\Gamma(y_0^u)\bigg].
\end{equation}

Since $U$ is not necessarily convex, we apply the new spike variation technique \cite{STW20}. Suppose that $\bar{u}\in\mathcal{U}_{ad}$ is an optimal control, for any $\bar{t}\in[0,T]$, define $u^{\epsilon}$ as follows. For $s\in[0,T]$,
\begin{equation*}\label{Songyuanzhuo}
\begin{aligned}
u^\epsilon_s=
\begin{cases}
u,&\text{if $(s,\omega)\in\mathcal{O}:=\rrbracket\bar{t},\bar{t}+\epsilon\rrbracket\backslash\bigcup_{n=1}^\infty\llbracket T_n\rrbracket$},\\
\bar{u}_s,&\text{otherwise},
\end{cases}
\end{aligned}
\end{equation*}
where $\llbracket T_n\rrbracket:=\big\{(\omega,t)\in\Omega\times[0,T]|T_n(\omega)=t\big\}$, which is the graph of stopping time $T_n$ when the jump of the observable process appears, is a progressive set, and $u$ is a bounded $\mathcal{F}^Y_{\bar{t}}$-measurable function taking values in $U$. The big difference is that the value of $u^\epsilon_s$ at $T_n(\omega)$ is equal to $\bar{u}_s$ rather than $u$ when some jump appears in $(\bar{t},\bar{t}+\epsilon]$, that is, some $T_n(\omega)$ is in $(\bar{t},\bar{t}+\epsilon]$. Moreover, different from completely observed case and due to technical difficulties, we also assume that there is no jumps of the unobservable process appearing in $(\bar{t},\bar{t}+\epsilon]$. It is easy to show that $u^\epsilon\in\mathcal{U}_{ad}[0,T]$ defined in \eqref{ad_control_set}.

Denoting by $(\bar{x},\bar{y},\bar{z}^1,\bar{z}^2,\bar{\tilde{z}}^1,\bar{\tilde{z}}^2)$ the trajectory of $\bar{u}$ and by $(x^\epsilon,y^\epsilon,z^{1,\epsilon},z^{2,\epsilon},\tilde{z}^{1,\epsilon},\tilde{z}^{2,\epsilon})$ the trajectory of $u^\epsilon$. For simplicity, for $i=1,2$, for $f_i,\varphi=\tilde{b}_1,\sigma_i,g,\phi$ and $\kappa=x,y,z^i,\tilde{z}^i$, denote
\begin{equation*}
\begin{aligned}
f_i(t,e)&=f_i(t,\bar{x}_t,\bar{u}_t,e),\quad f_{ix}(t,e)=f_{ix}(t,\bar{x}_t,\bar{u}_t,e),\quad \delta f_i(t,e)=f_i(t,\bar{x}_t,u_t,e)-f_i(t,e),\\
        \delta f_{ix}(t,e)&=f_{ix}(t,\bar{x}_t,u_t,e)-f_{ix}(t,e),\quad \varphi(t)=\varphi(t,\bar{\Theta}(t),\bar{u}_t),\quad
                        \varphi_\kappa(t)=\varphi_\kappa(t,\bar{\Theta}(t),\bar{u}_t),\\
       \delta\varphi(t)&=\varphi(t,\bar{\Theta}(t),u_t)-\varphi(t),\quad \delta\varphi_\kappa(t)=\varphi_\kappa(t,\bar{\Theta}(t),u_t)-\varphi_\kappa(t),\\
\delta\varphi(t,\Delta^1,\Delta^2)&=\varphi(t,\bar{x}_t,\bar{y}_t,\bar{z}^1_t+\Delta^1(t),\bar{z}^2_t+\Delta^2(t),\bar{\tilde{z}}^1,\bar{\tilde{z}}^2,u_t)-\varphi(t),\\
                        \delta\varphi_\kappa(t,\Delta^1,\Delta^2)&=\varphi_\kappa(t,\bar{x}_t,\bar{y}_t,\bar{z}^1_t+\Delta^1(t),\bar{z}^2_t+\Delta^2(t),\bar{\tilde{z}}^1,\bar{\tilde{z}}^2,u_t)-\varphi_\kappa(t),
\end{aligned}
\end{equation*}
where $\Delta^1(\cdot),\Delta^2(\cdot)$ are $\mathcal{F}_t$-adapted processes, to be determined later.

Denote $\bar{\Theta}(t):=(\bar{x}_t,\bar{y}_t,\bar{z}^1_t,\bar{z}^2_t,\bar{\tilde{z}}^1_{(t,e)},\bar{\tilde{z}}^2_{(t,e)})$ and $\Theta^\epsilon(t):=(x^\epsilon_t,y^\epsilon_t,z^{1,\epsilon}_t,z^{2,\epsilon}_t,\tilde{z}^{1,\epsilon}_{(t,e)},\tilde{z}^{2,\epsilon}_{(t,e)})$, and set
\begin{equation}\label{notation0}
\xi^{1,\epsilon}_t:=x^\epsilon_t-\bar{x}_t,\ \ \eta^{1,\epsilon}_t:=y^\epsilon_t-\bar{y}_t,\ \ \zeta^{i,1,\epsilon}_t:=z^{i,\epsilon}_t-\bar{z}^i_t,\ \ \lambda^{i,1,\epsilon}_{(t,e)}:=\tilde{z}^{i,\epsilon}_{(t,e)}-\bar{\tilde{z}}^i_{(t,e)},\quad i=1,2,
\end{equation}
we have
\begin{equation}\label{eta1}
\left\{
\begin{aligned}
  d\xi^{1,\epsilon}_t&=\big[\tilde{b}^{\epsilon}_{1x}(t)\xi^{1,\epsilon}_t+\delta \tilde{b}_1(t)\mathbbm{1}_{[\bar{t},\bar{t}+\epsilon]}\big]dt
                     +\sum^2_{i=1}\big[\tilde{\sigma}^{\epsilon}_{ix}(t)\xi^{1,\epsilon}_t+\delta\sigma_i(t)\mathbbm{1}_{[\bar{t},\bar{t}+\epsilon]}\big]dW^i_t\\
                     &\quad+\sum^2_{i=1}\int_{\mathcal{E}_i}\big[\tilde{f}^{\epsilon}_{ix}(t,e)\xi^{1,\epsilon}_{t-}+\delta f_i(t,e)\mathbbm{1}_{\mathcal{O}}\big]\tilde{N}_i(de,dt),\\
d\eta^{1,\epsilon}_t&=-\bigg[\tilde{g}^{\epsilon}_x(t)\xi^{1,\epsilon}_t+\tilde{g}^{\epsilon}_y(t)\eta^{1,\epsilon}_t
                     +\sum^2_{i=1}\bigg(\tilde{g}^{\epsilon}_{z^i}(t)\zeta^{i,1,\epsilon}_t+\tilde{g}^{\epsilon}_{\tilde{z}^i}(t)\int_{\mathcal{E}_i}\lambda^{i,1,\epsilon}_{(t,e)}\nu_i(de)\bigg)\\
                     &\quad+\delta g(t)\mathbbm{1}_{[\bar{t},\bar{t}+\epsilon]}\bigg]dt+\sum^2_{i=1}\zeta^{i,1,\epsilon}_tdW^i_t+\sum^2_{i=1}\int_{\mathcal{E}_i}\lambda^{i,1,\epsilon}_{(t,e)}\tilde{N}_i(de,dt),\quad t\in[0,T],\\
  \xi^{1,\epsilon}_0&=0,\quad \eta^{1,\epsilon}_T=\tilde{\phi}_x^\epsilon(T)\xi^{1,\epsilon}_T,
\end{aligned}
\right.
\end{equation}
where
\begin{equation}\label{notation1}
\begin{aligned}
\tilde{b}^{\epsilon}_{1x}(t)&:=\int_0^1\tilde{b}_{1x}(t,\bar{x}_t+\theta(x^\epsilon_t-\bar{x}_t),u^\epsilon_t)d\theta,\\
\tilde{g}^{\epsilon}_x(t)&:=\int_0^1g_x(t,\bar{\Theta}(t)+\theta(\Theta^\epsilon(t)-\bar{\Theta}(t)),u^\epsilon_t)d\theta,
\end{aligned}
\end{equation}
and $\tilde{\sigma}^{\epsilon}_{ix}(t),\tilde{f}^{\epsilon}_{ix}(t,e),\tilde{g}^{\epsilon}_y(t),\tilde{g}^{\epsilon}_{z^i}(t),\tilde{g}^{\epsilon}_{\tilde{z}^i}(t)$ and $\tilde{\phi}_x^\epsilon(T), i=1,2$ have similar definitions.

Then, by $L^{\beta}(\beta\geq2)$-estimate of FBSDEPs in Section 2, we get
\begin{equation*}
\begin{aligned}
&E\bigg[\sup_{0\leq t\leq T}(|\xi^{1,\epsilon}_t|^\beta+|\eta^{1,\epsilon}_t|^\beta)+\sum^2_{i=1}\bigg[\bigg(\int_0^T|\zeta^{i,1,\epsilon}_t|^2dt\bigg)^{\frac{\beta}{2}}
 +\bigg(\int_0^T\int_{\mathcal{E}_i}|\lambda^{i,1,\epsilon}_{(t,e)}|^2\nu_i(de)dt\bigg)^{\frac{\beta}{2}}\bigg]\bigg]\\
\end{aligned}
\end{equation*}
\begin{equation*}
\begin{aligned}
&\leq C E\bigg[\bigg(\int_0^T\delta g(t)\mathbbm{1}_{[\bar{t},\bar{t}+\epsilon]}dt\bigg)^\beta+\bigg(\int_0^T\delta \tilde{b}_1(t)\mathbbm{1}_{[\bar{t},\bar{t}+\epsilon]}dt\bigg)^\beta\\
&\quad +\sum_{i=1}^2\bigg(\int_0^T|\delta \sigma_i(t)\mathbbm{1}_{[\bar{t},\bar{t}+\epsilon]}|^2dt\bigg)^{\frac{\beta}{2}}+\sum_{i=1}^2\bigg(\int_0^T\int_{\mathcal{E}_i}|\delta f_i(t,e)\mathbbm{1}_{\mathcal{O}}|^2N_i(de,dt)\bigg)^{\frac{\beta}{2}}\bigg]\\
&\leq C E\bigg[\bigg(\int_0^T(1+|\bar{x}|+|\bar{y}|+|\bar{z}^1|+|\bar{z}^2|+||\bar{\tilde{z}}^1||+||\bar{\tilde{z}}^2||+|\bar{u}|+|u|)\mathbbm{1}_{[\bar{t},\bar{t}+\epsilon]}dt\bigg)^\beta\\
&\quad +\bigg(\int_0^T(1+|\bar{x}|+|\bar{u}|+|u|)\mathbbm{1}_{[\bar{t},\bar{t}+\epsilon]}dt\bigg)^\beta
 +\bigg(\int_0^T(1+|\bar{x}|+|\bar{u}|+|u|)^2\mathbbm{1}_{[\bar{t},\bar{t}+\epsilon]}dt\bigg)^{\frac{\beta}{2}}\\
&\quad +\sum_{i=1}^2\bigg(\int_0^T\int_{\mathcal{E}_i}|\delta f_i(t,e)\mathbbm{1}_{\mathcal{O}}|^2N_i(de,dt)\bigg)^{\frac{\beta}{2}}\bigg]\\
&\leq C E\bigg[\bigg(\int_{\bar{t}}^{\bar{t}+\epsilon}(1+|\bar{x}|+|\bar{y}|+|\bar{u}|+|u|)dt\bigg)^{\beta}+\sum_{i=1}^2\bigg(\int_{\bar{t}}^{\bar{t}+\epsilon}|\bar{z}^i|dt\bigg)^\beta\\
&\quad +\sum_{i=1}^2\bigg(\int_{\bar{t}}^{\bar{t}+\epsilon}||\bar{\tilde{z}}^i||dt\bigg)^\beta+\bigg(\int_{\bar{t}}^{\bar{t}+\epsilon}(1+|\bar{x}|+|\bar{u}|+|u|)dt\bigg)^\beta\\
&\quad +\bigg(\int_{\bar{t}}^{\bar{t}+\epsilon}(1+|\bar{x}|+|\bar{u}|+|u|)^2dt\bigg)^{\frac{\beta}{2}}+\sum_{i=1}^2\bigg(\int_0^T\int_{\mathcal{E}_i}|\delta f_i(t,e)\mathbbm{1}_{\mathcal{O}}|^2N_i(de,dt)\bigg)^{\frac{\beta}{2}}\bigg]\\
&\leq C\bigg\{\epsilon^{\beta-1}E\bigg[\bigg(\int_{\bar{t}}^{\bar{t}+\epsilon}(1+|\bar{x}|^\beta+|\bar{y}|^\beta+|\bar{u}|^\beta+|u|^\beta)dt\bigg)\bigg]
 +\epsilon^{\frac{\beta}{2}}E\bigg[\sum_{i=1}^2\bigg(\int_{\bar{t}}^{\bar{t}+\epsilon}|\bar{z}^i|^2dt\bigg)^{\frac{\beta}{2}}\bigg]\\
&\quad +\epsilon^{\frac{\beta}{2}}E\bigg[\sum_{i=1}^2\bigg(\int_{\bar{t}}^{\bar{t}+\epsilon}\int_{\mathcal{E}_i}|\bar{\tilde{z}}^i|^2\nu_i(de)dt\bigg)^{\frac{\beta}{2}}\bigg]
 +\epsilon^{\frac{\beta}{2}-1}E\bigg[\bigg(\int_{\bar{t}}^{\bar{t}+\epsilon}(1+|\bar{x}|^\beta+|\bar{u}|^\beta+|u|^\beta)dt\bigg)\bigg]\\
&\quad +\sum_{i=1}^2E\bigg[\bigg(\int_0^T\int_{\mathcal{E}_i}|\delta f_i(t,e)\mathbbm{1}_{\mathcal{O}}|^2N_i(de,dt)\bigg)^{\frac{\beta}{2}}\bigg]\bigg\}\\
&\leq C\bigg\{\epsilon^{\frac{\beta}{2}}+\epsilon^\beta+(\epsilon^\beta+\epsilon^{\frac{\beta}{2}})E\bigg[\sup_{0\leq t\leq T}|\bar{x}|^\beta\bigg]
 +\epsilon^\beta E\bigg[\sup_{0\leq t\leq T}|\bar{y}|^\beta\bigg]+(\epsilon^\beta+\epsilon^{\frac{\beta}{2}})\sup_{0\leq t\leq T}E|\bar{u}|^\beta\\
 &\quad +(\epsilon^\beta+\epsilon^{\frac{\beta}{2}})\sup_{0\leq t\leq T}E|u|^\beta+\epsilon^{\frac{\beta}{2}}E\bigg[\sum_{i=1}^2\bigg(\int_0^T|\bar{z}^i|^2dt\bigg)^{\frac{\beta}{2}}\bigg]
 +\epsilon^{\frac{\beta}{2}}E\bigg[\sum_{i=1}^2\bigg(\int_0^T\int_{\mathcal{E}_i}|\bar{\tilde{z}}^i|^2\nu_i(de)dt\bigg)^{\frac{\beta}{2}}\bigg]\\
&\quad +\sum_{i=1}^2E\bigg[\bigg(\int_0^T\int_{\mathcal{E}_i}|\delta f_i(t,e)\mathbbm{1}_{\mathcal{O}}|^2N_i(de,dt)\bigg)^{\frac{\beta}{2}}\bigg]\bigg\}\leq O(\epsilon^{\frac{\beta}{2}})+O(\epsilon^{\beta}),
\end{aligned}
\end{equation*}
where we have used that, under {\bf (A1)}, by Theorem \ref{the22},
\begin{equation*}
E\bigg[\sup_{0\leq t\leq T}|\bar{x}|^\beta+\sup_{0\leq t\leq T}|\bar{y}|^\beta
+\sum_{i=1}^2\bigg(\int_0^T|\bar{z}^i|^2dt\bigg)^{\frac{\beta}{2}}
+\sum_{i=1}^2\bigg(\int_0^T\int_{\mathcal{E}_i}|\bar{\tilde{z}}^i|^2\nu_i(de)dt\bigg)^{\frac{\beta}{2}}\bigg]<\infty.
\end{equation*}
Noticing that the estimate is derived by subtracting the jump part on $\mathcal{O}$, which means the jump term does not influence the order of variation. If not, the part with jump
\begin{equation*}
E\bigg[\bigg(\int_0^T\int_{\mathcal{E}_i}|\delta f_i(t,e)|^2\mathbbm{1}_{[\bar{t},\bar{t}+\epsilon]}N_i(de,dt)\bigg)^{\frac{\beta}{2}}\bigg],\ i=1,2,
\end{equation*}
is always of order $O(\epsilon)$, no matter how large $\beta$ is. Indeed, we have
\begin{equation*}
\begin{aligned}
E\bigg[\bigg(\int_0^T\int_{\mathcal{E}_i}|\delta f_i(t,e)|^2\mathbbm{1}_{[\bar{t},\bar{t}+\epsilon]}N_i(de,dt)\bigg)^{\frac{\beta}{2}}\bigg]
\leq C E\bigg[\bigg(\int_0^T(1+|\bar{x}|+|\bar{u}|+|u|)^2\mathbbm{1}_{[\bar{t},\bar{t}+\epsilon]}N_i({\mathcal{E}_i},dt)\bigg)^{\frac{\beta}{2}}\bigg].
\end{aligned}
\end{equation*}
Set $I_s:=(1+|\bar{x}_{s-}|+|\bar{u}|+|u|)^2\mathbbm{1}_{[\bar{t},\bar{t}+\epsilon]}$, and $J_{it}:=\int_0^tI_sN_i({\mathcal{E}_i},ds)$, then $J_{it}$ is a pure jump process and so is $J_{it}^{\frac{\beta}{2}}$. Moreover, the jump time of $J_i^{\frac{\beta}{2}}$ is also a jump time of $N_i$ and the jump size of $N_i$ is always equal to 1. Therefore we have
\begin{equation*}
\begin{aligned}
J_{iT}^{\frac{\beta}{2}}&=\sum_{s\leq T}\big(J_{is}^{\frac{\beta}{2}}-J_{is-}^{\frac{\beta}{2}}\big)=\sum_{s\leq T}\big(J_{is}^{\frac{\beta}{2}}-J_{is-}^{\frac{\beta}{2}}\big)\mathbbm{1}_{N_i(\mathcal{E}_i,\{s\})\neq0}
=\sum_{s\leq T}\big(|J_{is-}+I_s|^{\frac{\beta}{2}}-J_{is-}^{\frac{\beta}{2}}\big)N_i(\mathcal{E}_i,\{s\})\\
&=\int_0^T\big(|J_{is-}+I_s|^{\frac{\beta}{2}}-J_{is-}^{\frac{\beta}{2}}\big)N_i(\mathcal{E}_i,ds)\leq C\int_0^T\big(J_{is-}^{\frac{\beta}{2}}+I_s^{\frac{\beta}{2}}\big)N_i(\mathcal{E}_i,ds).
\end{aligned}
\end{equation*}
By the predictability of $J_{i\cdot-}$ and $I$, we get
\begin{equation*}
EJ_{iT}^{\frac{\beta}{2}}\leq C E\int_0^T\big(J_{is}^{\frac{\beta}{2}}+I_s^{\frac{\beta}{2}}\big)ds
\leq C E\int_0^TJ_{is}^{\frac{\beta}{2}}ds+C E\int_0^T(1+|\bar{x}_s|+|\bar{u}_s|+|u_s|)^\beta\mathbbm{1}_{[\bar{t},\bar{t}+\epsilon]}ds.
\end{equation*}
From Gronwall's inequality, we have
\begin{equation*}
EJ_{iT}^{\frac{\beta}{2}}\leq C E\int_{\bar{t}}^{\bar{t}+\epsilon}(1+|\bar{x}_s|+|\bar{u}_s|+|u_s|)^\beta ds\leq O(\epsilon).
\end{equation*}
Therefore, we derive
\begin{equation}\label{estimate2}
\begin{aligned}
&E\bigg[\sup_{0\leq t\leq T}(|\xi^{1,\epsilon}_t|^\beta+|\eta^{1,\epsilon}_t|^\beta)
 +\sum_{i=1}^2\bigg(\int_0^T|\zeta^{i,1,\epsilon}_t|^2dt\bigg)^{\frac{\beta}{2}}+\sum_{i=1}^2\bigg(\int_0^T\int_{\mathcal{E}_i}|\lambda^{i,1,\epsilon}_{(t,e)}|^2\nu_i(de)dt\bigg)^{\frac{\beta}{2}}\bigg]\\
&=O(\epsilon^{\frac{\beta}{2}})+O(\epsilon^{\beta}).
\end{aligned}
\end{equation}
According to the estimate \eqref{estimate2} in the above, we could set
\begin{equation}
\begin{aligned}
&x^\epsilon_t-\bar{x}_t=x^1_t+x^2_t+o(\epsilon),\quad y^\epsilon_t-\bar{y}_t=y^1_t+y^2_t+o(\epsilon),\\
&z^{i,\epsilon}_t-\bar{z}^i_t=z^{i,1}_t+z^{i,2}_t+o(\epsilon),\quad \tilde{z}^{i,\epsilon}_{(t,e)}-\bar{\tilde{z}}^i_{(t,e)}=\tilde{z}^{i,1}_{(t,e)}+\tilde{z}^{i,2}_{(t,e)}+o(\epsilon),\ i=1,2,
\end{aligned}
\end{equation}
where $x^1,y^1,z^{1,1},z^{2,1},\tilde{z}^{1,1},\tilde{z}^{2,1}\sim O(\sqrt{\epsilon})$ and $x^2,y^2,z^{1,2},z^{2,2},\tilde{z}^{1,2},\tilde{z}^{2,2}\sim O(\epsilon)$.

\subsection{Variational equations}

Then, we introduce the first and second order variational equations for the SDEP as follows:
\begin{equation}\label{x1x2}
\begin{aligned}
x^1_t&=\int_0^t\tilde{b}_{1x}(s)x_s^1ds+\sum^2_{i=1}\int_0^t\big[\sigma_{ix}(s)x_s^1+\delta\sigma_i(s)\mathbbm{1}_{[\bar{t},\bar{t}+\epsilon]}\big]dW^i_s\\
     &\quad +\sum^2_{i=1}\int_0^t\int_{\mathcal{\mathcal{E}}_i}f_{ix}(s,e)x_{s-}^1\tilde{N}_i(de,ds),
\end{aligned}
\end{equation}
\begin{equation}
\begin{aligned}
x^2_t&=\int_0^t\Big[\tilde{b}_{1x}(s)x_s^2+\frac{1}{2}\tilde{b}_{1xx}(s)(x_s^1)^2+\delta \tilde{b}_1(s)\mathbbm{1}_{[\bar{t},\bar{t}+\epsilon]}\Big]ds\\
     &\quad +\sum^2_{i=1}\int_0^t\Big[\sigma_{ix}(s)x_s^2+\frac{1}{2}\sigma_{ixx}(s)(x_s^1)^2+\delta\sigma_{ix}(s)x_s^1\mathbbm{1}_{[\bar{t},\bar{t}+\epsilon]}\Big]dW^i_s\\
     &\quad +\sum^2_{i=1}\int_0^t\int_{\mathcal{\mathcal{E}}_i}\Big[f_{ix}(s,e)x_{s-}^2+\frac{1}{2}f_{ixx}(s,e)(x_{s-}^1)^2\Big]\tilde{N}_i(de,ds),
\end{aligned}
\end{equation}
with $x^1_0=x^2_0=0$. Inspired by \cite{Hu17} (or \cite{HJX18}), we guess that $z_t^{i,1}, \tilde{z}_{(t,e)}^{i,1},i=1,2$ have the following forms, respectively:
\begin{equation}\label{ztildez}
z_t^{i,1}=\Delta^i(t)\mathbbm{1}_{[\bar{t},\bar{t}+\epsilon]}+z_t^{i,1'},\quad \tilde{z}_{(t,e)}^{i,1}=\tilde{\Delta}^i(t)\mathbbm{1}_{\mathcal{O}}+\tilde{z}_{(t,e)}^{i,1'},\quad i=1,2,
\end{equation}
where $\Delta^1(t),\Delta^2(t),\tilde{\Delta}^1(t),\tilde{\Delta}^2(t)$ are $\mathcal{F}_t$-adapted processes, to be determined later, and importantly, $z_t^{i,1'},\tilde{z}_{(t,e)}^{i,1'},i=1,2$ have good estimates similarly as $x_t^1$. Then
\begin{equation*}
\begin{aligned}
&g\big(t,x^\epsilon_t,y^\epsilon_t,z^{1,\epsilon}_t,z^{2,\epsilon}_t,\tilde{z}^{1,\epsilon}_{(t,e)},\tilde{z}^{2,\epsilon}_{(t,e)},u^\epsilon_t\big)
 -g\big(t,\bar{x}_t,\bar{y}_t,\bar{z}^1_t,\bar{z}^2_t,\bar{\tilde{z}}^1_{(t,e)},\bar{\tilde{z}}^2_{(t,e)},\bar{u}_t\big)\\
&=g\big(t,\bar{x}_t+x_t^1+x_t^2,\bar{y}_t+y_t^1+y_t^2,\bar{z}^1_t+\Delta^1(t)\mathbbm{1}_{[\bar{t},\bar{t}+\epsilon]}+z_t^{1,1'}+z_t^{1,2},\\
&\qquad\bar{z}^2_t+\Delta^2(t)\mathbbm{1}_{[\bar{t},\bar{t}+\epsilon]}+z_t^{2,1'}+z_t^{2,2},\bar{\tilde{z}}^1_{(t,e)}+\tilde{\Delta}^1(t)\mathbbm{1}_{\mathcal{O}}+\tilde{z}_{(t,e)}^{1,1'}+\tilde{z}_{(t,e)}^{1,2},\\
&\qquad\bar{\tilde{z}}^2_{(t,e)}+\tilde{\Delta}^2(t)\mathbbm{1}_{\mathcal{O}}+\tilde{z}_{(t,e)}^{2,1'}+\tilde{z}_{(t,e)}^{2,2},u^\epsilon_t\big)-g(t)+o(\epsilon)\\
&=g_x(t)(x_t^1+x_t^2)+g_y(t)(y_t^1+y_t^2)+\sum_{i=1}^2g_{z^i}(t)(z_t^{i,1'}+z_t^{i,2})\\
&\quad+\sum_{i=1}^2g_{\tilde{z}^i}(t)\int_{\mathcal{E}_i}(\tilde{z}_{(t,e)}^{i,1'}+\tilde{z}_{(t,e)}^{i,2})\nu_i(de)+\frac{1}{2}\Xi(t)D^2g(t)\Xi(t)^\top\\
&\quad +\delta g(t,\Delta^1,\Delta^2,\tilde{\Delta}^1,\tilde{\Delta}^2)\mathbbm{1}_{[\bar{t},\bar{t}+\epsilon]}+o(\epsilon),
\end{aligned}
\end{equation*}
where $\Xi(t):=\big[x_t^1,y_t^1,z_t^{1,1'},z_t^{2,1'},\int_{\mathcal{E}_1}\tilde{z}_{(t,e)}^{1,1'}\nu_1(de),\int_{\mathcal{E}_2}\tilde{z}_{(t,e)}^{2,1'}\nu_2(de)\big]$ for simplicity.
So we can obtain the following equation for the BSDEP:
\begin{equation}\label{y1y2}
\left\{
\begin{aligned}
-d(y_t^1+y_t^2)&=\bigg\{g_x(t)(x_t^1+x_t^2)+g_y(t)(y_t^1+y_t^2)+\sum^2_{i=1}g_{z^i}(t)(z_t^{i,1'}+z_t^{i,2})\\
               &\qquad+\sum^2_{i=1}g_{\tilde{z}^i}(t)\int_{\mathcal{E}_i}(\tilde{z}_{(t,e)}^{i,1'}+\tilde{z}_{(t,e)}^{i,2})\nu_i(de)
               +\delta g(t,\Delta^1,\Delta^2,\tilde{\Delta}^1,\tilde{\Delta}^2)\mathbbm{1}_{[\bar{t},\bar{t}+\epsilon]}\\
               &\qquad+\frac{1}{2}\Xi(t)D^2g\Xi(t)^\top\bigg\}dt\\
               &\quad-\sum^2_{i=1}(z_t^{i,1}+z_t^{i,2})dW^i_t-\sum^2_{i=1}\int_{\mathcal{E}_i}(\tilde{z}_{(t,e)}^{i,1}+\tilde{z}_{(t,e)}^{i,2})\tilde{N}_i(de,dt),\quad t\in[0,T],\\
    y_T^1+y_T^2&=\phi_x(\bar{x}_T)(x_T^1+x_T^2)+\frac{1}{2}\phi_{xx}(\bar{x}_T)(x_T^1)^2,
\end{aligned}
\right.
\end{equation}

Next, we need derive the corresponding first and second order variational equations for the BSDEP. Noticing that the terminal value in \eqref{y1y2}, $y_T^1=\phi_x(\bar{x}_T)x_T^1$, we guess that $y_t^1=p_tx_t^1$, where $(p_t,q^1_t,q^2_t,\tilde{q}^1_{(t,e)},\tilde{q}^2_{(t,e)})$ satisfies the following adjoint equation:
\begin{equation*}
\left\{
\begin{aligned}
-dp_t&=F(t)dt-\sum_{i=1}^2q^i_tdW^i_t-\sum_{i=1}^2\int_{\mathcal{E}_i}\tilde{q}^i_{(t,e)}\tilde{N}_i(de,dt),\quad t\in[0,T],\\
  p_T&=\phi_x(\bar{x}_T),
\end{aligned}
\right.
\end{equation*}
where the $\mathcal{F}_t$-adapted process $F(\cdot)$ is to be determined.

Applying It\^{o}'s formula to $p_tx_t^1$, we can get
\begin{equation}\label{y1}
\left\{
\begin{aligned}
-dy_t^1&=\bigg\{g_x(t)x_t^1+g_y(t)y_t^1+\sum^2_{i=1}\bigg[g_{z^i}(t)(z_t^{i,1}-\Delta^i(t)\mathbbm{1}_{[\bar{t},\bar{t}+\epsilon]})\\
       &\qquad +g_{\tilde{z}^i}(t)\int_{\mathcal{E}_i}\big(\tilde{z}_{(t,e)}^{i,1}-\tilde{\Delta}^i(t)\mathbbm{1}_{\mathcal{O}}\big)\nu_i(de)\bigg]
        -\sum^2_{i=1}q^i_t\delta\sigma_i(t)\mathbbm{1}_{[\bar{t},\bar{t}+\epsilon]}\bigg\}dt\\
       &\quad -\sum^2_{i=1}z_t^{i,1}dW^i_t-\sum^2_{i=1}\int_{\mathcal{E}_i}\tilde{z}_{(t,e)}^{i,1}\tilde{N}_i(de,dt),\quad t\in[0,T],\\
  y_T^1&=\phi_x(\bar{x}_T)x_T^1.
\end{aligned}
\right.
\end{equation}
and $(p_t,q^1_t,q^2_t,\tilde{q}^1_{(t,e)},\tilde{q}^2_{(t,e)})\in \mathcal{N}^\beta[0,T]$ is the unique solution to the following first order adjoint equation:
\begin{equation}\label{p}
\left\{
\begin{aligned}
-dp_t&=\bigg\{\tilde{b}_{1x}(t)p_t+\sum^2_{i=1}\sigma_{ix}(t)q^i_t+g_x(t)+g_y(t)p_t+\sum^2_{i=1}g_{z^i}(t)k^i_1(t)\\
     &\qquad +\sum^2_{i=1}\bigg[g_{\tilde{z}^i}(t)\int_{\mathcal{E}_i}k^i_2(t,e)\nu_i(de)+\int_{\mathcal{E}_i}\mathbb{E}_i\big[f_{ix}(t,e)\big|\mathcal{P}\otimes\mathcal{B}(\mathcal{E}_i)\big]\tilde{q}^i_{(t,e)}\nu_i(de)\bigg]\bigg\}dt\\
     &\quad -\sum^2_{i=1}q^i_tdW^i_t-\sum^2_{i=1}\int_{\mathcal{E}_i}\tilde{q}^i_{(t,e)}\tilde{N}_i(de,dt),\quad t\in[0,T],\\
  p_T&=\phi_x(\bar{x}_T).
\end{aligned}
\right.
\end{equation}
From the deduction above, it is not hard to get the following relationships:
\begin{equation}
y_t^1=p_tx_t^1,\quad z_t^{i,1}=k^i_1(t)x_t^1+p_t\delta\sigma_i(t)\mathbbm{1}_{[\bar{t},\bar{t}+\epsilon]},\quad\tilde{z}_{(t,e)}^{i,1}=k^i_2(t,e)x_{t-}^1,\ i=1,2,
\end{equation}
where
\begin{equation}\label{k1k2}
\begin{aligned}
k^i_1(t)&:=\sigma_{ix}(t)p_t+q^i_t,\quad k^i_2(t,e):=f_{ix}(t,e)p_{t-}+\tilde{q}^i_{(t,e)}+f_{ix}(t,e)\tilde{q}^i_{(t,e)},\ i=1,2.
\end{aligned}
\end{equation}
Observing \eqref{ztildez}, we obtain
\begin{equation}
\begin{aligned}
z_t^{i,1'}=k^i_1(t)x_t^1,\quad \tilde{z}_{(t,e)}^{i,1'}=k^i_2(t,e)x_{t-}^1,\quad\Delta^i(t)=p_t\delta\sigma_i(t),\quad\tilde{\Delta}^i(t)=0,\ i=1,2.
\end{aligned}
\end{equation}

\begin{remark}\label{rem31}
We notice that, different from the continuous martingale ``$dW^i$'' terms $z_t^{i,1}$, $\tilde{\Delta}^i(t)=0,i=1,2$, that is $\tilde{z}_{(t,e)}^{i,1}=\tilde{z}_{(t,e)}^{i,1'},i=1,2,$ which means the ``jump term" in the first order variation has no change. Indeed, inspired by \cite{Hu17} (or \cite{HJX18}), the diffusion term of the first order variational equation of $x$ should include the term $\delta\sigma_i(t)\mathbbm{1}_{[\bar{t},\bar{t}+\epsilon]}$, so it is natural to guess another form of $z_t^{i,1}$ in \eqref{ztildez}. However, considering the first order variational equation of $x$ in our paper, motivated by the new variation technique in \cite{STW20}, the ``jump term" does not include the similar term $``\delta f_i(t,e)\mathbbm{1}_{\mathcal{O}}"$. Therefore, the form of $\tilde{z}_{(t,e)}^{i,1}$ has no need to change like $z_t^{i,1}$ in heuristic derivation, which has been checked in the above discussion.
\end{remark}

From \eqref{y1y2} and \eqref{y1}, we introduce the first and second variational equations of the BSDEP as follows:
\begin{equation}\label{y1andy2}
\left\{
\begin{aligned}
-dy_t^1&=\Big\{g_x(t)x_t^1+g_y(t)y_t^1+\sum^2_{i=1}g_{z^i}(t)\big(z_t^{i,1}-p_t\delta\sigma_i(t)\mathbbm{1}_{[\bar{t},\bar{t}+\epsilon]}\big)+\sum^2_{i=1}g_{\tilde{z}^i}(t)\int_{\mathcal{E}_i}\tilde{z}_{(t,e)}^{i,1}\nu_i(de)\\
       &\quad -\sum^2_{i=1}q^i_t\delta\sigma_i(t)\mathbbm{1}_{[\bar{t},\bar{t}+\epsilon]}\Big\}dt-\sum^2_{i=1}z_t^{i,1}dW^i_t-\sum^2_{i=1}\int_{\mathcal{E}_i}\tilde{z}_{(t,e)}^{i,1}\tilde{N}_i(de,dt),\\
-dy_t^2&=\bigg\{g_x(t)x_t^2+g_y(t)y_t^2+\sum^2_{i=1}g_{z^i}(t)z_t^{i,2}+\sum^2_{i=1}g_{\tilde{z}^i}(t)\int_{\mathcal{E}_i}\tilde{z}_{(t,e)}^{i,2}\nu_i(de)+\frac{1}{2}\tilde{\Xi}(t)D^2g\tilde{\Xi}(t)^\top\\
       &\quad+\delta g(t,p_t\delta\sigma_1(t),p_t\delta\sigma_2(t))\mathbbm{1}_{[\bar{t},\bar{t}+\epsilon]}+\sum^2_{i=1}q^i_t\delta\sigma_i(t)\mathbbm{1}_{[\bar{t},\bar{t}+\epsilon]}\bigg\}dt\\
       &\quad-\sum^2_{i=1}z_t^{i,2}dW^i_t-\sum^2_{i=1}\int_{\mathcal{E}_i}\tilde{z}_{(t,e)}^{i,2}\tilde{N}_i(de,dt),\quad t\in[0,T],\\
  y_T^1&=\phi_x(\bar{x}_T)x_T^1,\ y_T^2=\phi_x(\bar{x}_T)x_T^2+\frac{1}{2}\phi_{xx}(\bar{x}_T)(x_T^1)^2,
\end{aligned}
\right.
\end{equation}
where {\footnotesize$\tilde{\Xi}(t)=\big[x_t^1,y_t^1,z_t^{1,1}-p_t\delta\sigma_1(t)\mathbbm{1}_{[\bar{t},\bar{t}+\epsilon]},z_t^{2,1}-p_t\delta\sigma_2(t)\mathbbm{1}_{[\bar{t},\bar{t}+\epsilon]},
\int_{\mathcal{E}_1}\tilde{z}_{(t,e)}^{1,1}\nu_1(de),\int_{\mathcal{E}_2}\tilde{z}_{(t,e)}^{2,1}\nu_2(de)\big]$}. It is easy to check that the first and second order variational equations \eqref{x1x2}, \eqref{y1andy2} admit unique solutions $x^1,x^2$, $(y^1,z^{1,1},z^{2,1},\tilde{z}^{1,1},\tilde{z}^{2,1})$ and $(y^2,z^{1,2},z^{2,2},\tilde{z}^{1,2},\tilde{z}^{2,2})$, respectively.

\begin{lemma}\label{lemma31}
Let {\bf (A1)} hold, then we have
\begin{equation}\label{relationshipy1}
\begin{aligned}
y_t^1&=p_tx_t^1,\quad z_t^{i,1}=k^i_1(t)x_t^1+p_t\delta\sigma_i(t)\mathbbm{1}_{[\bar{t},\bar{t}+\epsilon]},\quad\tilde{z}_{(t,e)}^{i,1}=k^i_2(t,e)x_{t-}^1,\ i=1,2,
\end{aligned}
\end{equation}
where $p_t$ satisfies \eqref{p}, and $k^i_1(t),k^i_2(t,e),i=1,2,$ are given by \eqref{k1k2}.
\end{lemma}

\begin{proof}
We introduce the following SDEP:
\begin{equation*}
\left\{
\begin{aligned}
d\hat{x}_t^1&=\tilde{b}_{1x}(t)\hat{x}_t^1dt+\sum_{i=1}^2\big[\sigma_{ix}(t)\hat{x}_t^1+\delta\sigma_i(t)\mathbbm{1}_{[\bar{t},\bar{t}+\epsilon]}\big]dW^i_t
             +\sum_{i=1}^2\int_{\mathcal{E}_i}f_{ix}(t,e)\hat{x}_{t-}^1\tilde{N}_i(de,dt),\quad t\in[0,T],\\
 \hat{x}_0^1&=0,
\end{aligned}
\right.
\end{equation*}
which admits a unique solution $\hat{x}_t^1$, and set
\begin{equation*}
\hat{y}_t^1:=p_t\hat{x}_t^1,\quad \hat{z}_t^{i,1}:=k^i_1(t)\hat{x}_t^1+p_t\delta\sigma_i(t)\mathbbm{1}_{[\bar{t},\bar{t}+\epsilon]},\quad \hat{\tilde{z}}_{(t,e)}^{i,1}:=k^i_2(t,e)\hat{x}_{t-}^1,\ i=1,2.
\end{equation*}
Applying It\^{o}'s formula to $p_t\hat{x}_t^1$, we can get
\begin{equation*}
\begin{aligned}
\hat{y}_t^1&=\phi_{\hat{x}}(\hat{\bar{x}}_T)\hat{x}_T^1+\int_t^T\Big\{g_{\hat{x}}(s)\hat{x}_s^1+g_{\hat{y}}(s)\hat{y}_s^1+\sum_{i=1}^2g_{\hat{z}^i}(s)k^i_1(s)\hat{x}_s^1
            +\sum_{i=1}^2g_{\hat{\tilde{z}}^i}(s)\int_{\mathcal{E}_i}k^i_2(s,e)\nu_i(de)\hat{x}_s^1\\
           &\quad -\sum_{i=1}^2q^i_s\delta\sigma_i(s)\mathbbm{1}_{[\bar{t},\bar{t}+\epsilon]}\Big\}ds-\sum_{i=1}^2\int_t^T\hat{z}_s^{i,1}dW^i_s-\sum_{i=1}^2\int_t^T\int_{\mathcal{E}_i}\hat{\tilde{z}}_{(s,e)}^{i,1}\tilde{N}_i(de,ds).
\end{aligned}
\end{equation*}
It is easy to check that $(\hat{x}^1,\hat{y}^1,\hat{z}^{1,1},\hat{z}^{2,1},\hat{\tilde{z}}^{1,1},\hat{\tilde{z}}^{2,1})$ solves \eqref{x1x2} and \eqref{y1andy2}, and by Theorem \ref{the24}, $(\hat{x}^1,\hat{y}^1,\hat{z}^{1,1},\hat{z}^{2,1},\hat{\tilde{z}}^{1,1},\hat{\tilde{z}}^{2,1})=(x^1,y^1,z^{1,1},z^{2,1},\tilde{z}^{1,1},\tilde{z}^{2,1})$. The proof is complete.
\end{proof}

\subsection{First-order and second-order expansions}

The following standard estimates of FBSDEPs are needed.
\begin{lemma}\label{lemma32}
Suppose that {\bf (A1)} holds, for $\beta\geq2$, we have the following estimates:
\begin{equation}\label{estimate1}
\begin{aligned}
&E\bigg[\sup_{0\leq t\leq T}(|x_t^1|^\beta+|y_t^1|^\beta)+\sum^2_{i=1}\bigg[\bigg(\int_0^T|z_t^{i,1}|^2dt\bigg)^{\frac{\beta}{2}}
 +\bigg(\int_0^T\int_{\mathcal{E}_i}|\tilde{z}_{(t,e)}^{i,1}|^2\nu_i(de)dt\bigg)^{\frac{\beta}{2}}\bigg]\bigg]=O(\epsilon^{\frac{\beta}{2}}),\\
&E\bigg[\sup_{0\leq t\leq T}|x_t^\epsilon-\bar{x}_t-x_t^1|^2+\sup_{0\leq t\leq T}|y_t^\epsilon-\bar{y}_t-y_t^1|^2+\sum^2_{i=1}\int_0^T|z_t^{i,\epsilon}-\bar{z}^i_t-z_t^{i,1}|^2dt\bigg]\\
&\quad +E\bigg[\sum^2_{i=1}\int_0^T\int_{\mathcal{E}_i}|\tilde{z}_{(t,e)}^{i,\epsilon}-\bar{\tilde{z}}^i_{(t,e)}-\tilde{z}_{(t,e)}^{i,1}|^2\nu_i(de)dt\bigg]=o(\epsilon),\\
&E\bigg[\sup_{0\leq t\leq T}|x_t^\epsilon-\bar{x}_t-x_t^1|^4+\sup_{0\leq t\leq T}|y_t^\epsilon-\bar{y}_t-y_t^1|^4+\sum^2_{i=1}\bigg(\int_0^T|z_t^{i,\epsilon}-\bar{z}^i_t-z_t^{i,1}|^2dt\bigg)^2\bigg]\\
&\quad +E\bigg[\sum^2_{i=1}\bigg(\int_0^T\int_{\mathcal{E}_i}|\tilde{z}_{(t,e)}^{i,\epsilon}-\bar{\tilde{z}}^i_{(t,e)}-\tilde{z}_{(t,e)}^{i,1}|^2\nu_i(de)dt\bigg)^2\bigg]=o(\epsilon^2).
\end{aligned}
\end{equation}
\end{lemma}

\begin{proof}
By {\bf (A1)}, for any $\beta\geq2$, we have
\begin{equation*}
E\bigg[\sup_{0\leq t\leq T}|p_t|^\beta+\sum_{i=1}^2\bigg(\int_0^T|q^i_t|^2dt\bigg)^{\frac{\beta}{2}}
 +\sum_{i=1}^2\bigg(\int_0^T\int_{\mathcal{E}_i}|\tilde{q}^i_{(t,e)}|^2\nu_i(de)dt\bigg)^{\frac{\beta}{2}}\bigg]<\infty.
\end{equation*}
and thus
\begin{equation*}
\begin{aligned}
&E\bigg[\sup_{0\leq t\leq T}|x_t^1|^\beta+\sup_{0\leq t\leq T}|y_t^1|^\beta+\sum_{i=1}^2\bigg[\bigg(\int_0^T|z_t^{i,1}|^2dt\bigg)^{\frac{\beta}{2}}
 +\bigg(\int_0^T\int_{\mathcal{E}_i}|\tilde{z}_{(t,e)}^{i,1}|^2\nu_i(de)dt\bigg)^{\frac{\beta}{2}}\bigg]\bigg]\\
&\leq C_\beta E\bigg\{\sum_{i=1}^2\bigg[\bigg(\int_0^T|g_{z^i}(t)p_t\delta\sigma_i(t)|\mathbbm{1}_{[\bar{t},\bar{t}+\epsilon]}+|q^i_t\delta\sigma_i(t)|\mathbbm{1}_{[\bar{t},\bar{t}+\epsilon]}dt\bigg)^{\beta}\\
&\qquad +\bigg(\int_0^T|\delta\sigma_i(t)\mathbbm{1}_{[\bar{t},\bar{t}+\epsilon]}|^2dt\bigg)^{\frac{\beta}{2}}\bigg]\bigg\}\\
&\leq C_\beta\bigg\{E\bigg[\bigg(\int_{\bar{t}}^{\bar{t}+\epsilon}(1+|\bar{x}|+|\bar{u}|+|u|)^2dt\bigg)^{\frac{\beta}{2}}\bigg]
 +E\bigg[\sum_{i=1}^2\bigg(\int_0^T|g_{z^i}(t)p_t\delta\sigma_i(t)|\mathbbm{1}_{[\bar{t},\bar{t}+\epsilon]}dt\bigg)^\beta\bigg]\\
&\qquad +E\bigg[\sum_{i=1}^2\bigg(\int_0^T|q^i_t\delta\sigma_i(t)|\mathbbm{1}_{[\bar{t},\bar{t}+\epsilon]}dt\bigg)^\beta\bigg]\bigg\}\\
&\leq C_\beta\bigg\{E\bigg[\sum_{i=1}^2\bigg(\int_0^T|p_t\delta\sigma_i(t)|\mathbbm{1}_{[\bar{t},\bar{t}+\epsilon]}dt\bigg)^\beta\bigg]
 +E\bigg[\sum_{i=1}^2\bigg(\int_0^T|q^i_t\delta\sigma_i(t)|\mathbbm{1}_{[\bar{t},\bar{t}+\epsilon]}dt\bigg)^\beta\bigg]\bigg\}+O(\epsilon^{\frac{\beta}{2}})\\
&\leq C_\beta\bigg\{E\bigg[\sup_{0\leq t\leq T}|p_t|^\beta\sum_{i=1}^2\bigg(\int_{\bar{t}}^{\bar{t}+\epsilon}|\delta\sigma_i(t)|dt\bigg)^\beta\bigg]\\
 \end{aligned}
 \end{equation*}
\begin{equation*}
\begin{aligned}
&\qquad+\sum_{i=1}^2E\bigg[\bigg(\int_{\bar{t}}^{\bar{t}+\epsilon}|q^i_t|^2dt\bigg)^{\frac{\beta}{2}}
 \bigg(\int_{\bar{t}}^{\bar{t}+\epsilon}|\delta\sigma_i(t)|^2dt\bigg)^{\frac{\beta}{2}}\bigg]\bigg\}+O(\epsilon^{\frac{\beta}{2}})\\
&\leq C_\beta\bigg\{\bigg\{E\bigg[\sup_{0\leq t\leq T}|p_t|^{2\beta}\bigg]\bigg\}^{\frac{1}{2}}
 \bigg\{E\bigg[\bigg(\int_{\bar{t}}^{\bar{t}+\epsilon}(1+|\bar{x}|+|\bar{u}|+|u|)dt\bigg)^{2\beta}\bigg]\bigg\}^{\frac{1}{2}}\\
&\qquad +\sum_{i=1}^2\bigg\{E\bigg[\bigg(\int_{\bar{t}}^{\bar{t}+\epsilon}|q^i_t|^2dt\bigg)^{\beta}\bigg]\bigg\}^{\frac{1}{2}}
 \bigg\{E\bigg[\bigg(\int_{\bar{t}}^{\bar{t}+\epsilon}(1+|\bar{x}|+|\bar{u}|+|u|)^2dt\bigg)^{\beta}\bigg]\bigg\}^{\frac{1}{2}}\bigg\}+O(\epsilon^{\frac{\beta}{2}})\\
&\leq C_\beta\bigg\{\epsilon^{\beta-\frac{1}{2}}\bigg\{E\bigg[\bigg(\int_{\bar{t}}^{\bar{t}+\epsilon}(1+|\bar{x}|^{2\beta}
 +|\bar{u}|^{2\beta}+|u|^{2\beta})dt\bigg)\bigg]\bigg\}^{\frac{1}{2}}
 +\epsilon^{\frac{\beta}{2}-\frac{1}{2}}\sum_{i=1}^2\bigg\{E\bigg[\bigg(\int_{\bar{t}}^{\bar{t}+\epsilon}|q^i_t|^2dt\bigg)^{\beta}\bigg]\bigg\}^{\frac{1}{2}}\\
&\qquad\quad \times\bigg\{E\bigg[\bigg(\int_{\bar{t}}^{\bar{t}+\epsilon}(1+|\bar{x}|^{2\beta}+|\bar{u}|^{2\beta}+|u|^{2\beta})dt\bigg)\bigg]\bigg\}^{\frac{1}{2}}\bigg\}
 +O(\epsilon^{\frac{\beta}{2}})\\
&\leq C_\beta\bigg\{\epsilon^\beta+\epsilon^{\frac{\beta}{2}}\sum_{i=1}^2\bigg\{E\bigg[\bigg(\int_{\bar{t}}^{\bar{t}+\epsilon}|q^i_t|^2dt\bigg)^\beta\bigg]\bigg\}^{\frac{1}{2}}\bigg\}
 +O(\epsilon^{\frac{\beta}{2}})
 \leq O(\epsilon^{\frac{\beta}{2}}).
\end{aligned}
\end{equation*}
The first estimate in \eqref{estimate1} holds. We continue to use the notations $\xi_t^{1,\epsilon},\eta_t^{1,\epsilon},\zeta_t^{1,1,\epsilon},\zeta_t^{2,1,\epsilon},\lambda_{(t,e)}^{1,1,\epsilon}$ and $\lambda_{(t,e)}^{2,1,\epsilon}$ in \eqref{notation0} and let
\begin{equation}\label{notation2}
\begin{aligned}
&\xi_t^{2,\epsilon}:=x^\epsilon_t-\bar{x}_t-x_t^1,\quad \eta_t^{2,\epsilon}:=y^\epsilon_t-\bar{y}_t-y_t^1,\quad \zeta_t^{i,2,\epsilon}:=z^{i,\epsilon}_t-\bar{z}^i_t-z_t^{i,1},\\
&\lambda_{(t,e)}^{i,2,\epsilon}:=\tilde{z}^{i,\epsilon}_{(t,e)}-\bar{\tilde{z}}^i_{(t,e)}-\tilde{z}_{(t,e)}^{i,1},\ i=1,2,\quad \Theta(t,\Delta^1(t)\mathbbm{1}_{[\bar{t},\bar{t}+\epsilon]},\Delta^2(t)\mathbbm{1}_{[\bar{t},\bar{t}+\epsilon]})\\
&:=(\bar{x}_t,\bar{y}_t,\bar{z}^1_t+\Delta^1(t)\mathbbm{1}_{[\bar{t},\bar{t}+\epsilon]},\bar{z}^2_t+\Delta^2(t)\mathbbm{1}_{[\bar{t},\bar{t}+\epsilon]},\bar{\tilde{z}}^1_{(t,e)},\bar{\tilde{z}}^2_{(t,e)}),
\end{aligned}
\end{equation}
thus
\begin{equation*}
\left\{
\begin{aligned}
 d\xi_t^{2,\epsilon}&=\Big[\tilde{b}_{1x}^\epsilon(t)\xi_t^{2,\epsilon}+(\tilde{b}_{1x}^\epsilon(t)-\tilde{b}_{1x}(t))x^1_t+\delta \tilde{b}_1(t)\mathbbm{1}_{[\bar{t},\bar{t}+\epsilon]}\Big]dt\\
                    &\quad +\sum^2_{i=1}\Big[\tilde{\sigma}_{ix}^\epsilon(t)\xi_t^{2,\epsilon}+(\tilde{\sigma}_{ix}^\epsilon(t)-\sigma_{ix}(t))x^1_t\Big]dW^i_t\\
                    &\quad +\sum^2_{i=1}\int_{\mathcal{E}_i}\Big[\tilde{f}_{ix}^\epsilon(t,e)\xi_{t-}^{2,\epsilon}+(\tilde{f}_{ix}^\epsilon(t,e)
                     -f_{ix}(t,e))x^1_{t-}+\delta f_i(t,e)\mathbbm{1}_{\mathcal{O}}\Big]\tilde{N}_i(de,dt),\\
d\eta_t^{2,\epsilon}&=-\Big[\tilde{g}_x^\epsilon(t)\xi_t^{2,\epsilon}+\tilde{g}_y^\epsilon(t)\eta_t^{2,\epsilon}+\sum^2_{i=1}\tilde{g}_{z^i}^\epsilon(t)\zeta_t^{i,2,\epsilon}
                     +\sum^2_{i=1}\tilde{g}_{\tilde{z}^i}^\epsilon(t)\int_{\mathcal{E}_i}\lambda_{(t,e)}^{i,2,\epsilon}\nu_i(de)+A_1^\epsilon(t)\Big]dt\\
                    &\quad +\sum^2_{i=1}\zeta_t^{i,2,\epsilon}dW^i_t+\sum^2_{i=1}\int_{\mathcal{E}_i}\lambda_{(t,e)}^{i,2,\epsilon}\tilde{N}_i(de,dt),\quad t\in[0,T],\\
  \xi_0^{2,\epsilon}&=0,\quad \eta_T^{2,\epsilon}=\tilde{\phi}_x^\epsilon(T)\xi_T^{2,\epsilon}+(\tilde{\phi}_x^\epsilon(T)-\phi_x(\bar{x}_T))x_T^1,
\end{aligned}
\right.
\end{equation*}
where $\tilde{b}_{1x}^\epsilon(t),\tilde{\sigma}_{ix}^\epsilon(t),\tilde{f}_{ix}^\epsilon(t,e),\tilde{g}_x^\epsilon(t),\tilde{g}_y^\epsilon(t),\tilde{g}_{z^i}^\epsilon(t),
\tilde{g}_{\tilde{z}^i}^\epsilon(t)$ and $\tilde{\phi}_x^\epsilon(T)$ are defined in \eqref{notation1}, and
\begin{equation*}\label{A1}
\begin{aligned}
A_1^\epsilon(t)&:=\delta g(t)\mathbbm{1}_{[\bar{t},\bar{t}+\epsilon]}
                +\big[\tilde{g}_x^\epsilon(t)-g_x(t)\big]x_t^1+\big[\tilde{g}_y^\epsilon(t)-g_y(t)\big]y_t^1+\sum^2_{i=1}\Big[\big[\tilde{g}_{z^i}^\epsilon(t)-g_{z^i}(t)\big]z_t^{i,1}\\
               &\quad +\big[\tilde{g}_{\tilde{z}^i}^\epsilon(t)-g_{\tilde{z}^i}(t)\big]\int_{\mathcal{E}_i}\tilde{z}_{(t,e)}^{i,1}\nu_i(de)
                +g_{z^i}(t)\Delta^i(t)\mathbbm{1}_{[\bar{t},\bar{t}+\epsilon]}+q^i_t\delta\sigma_i(t)\mathbbm{1}_{[\bar{t},\bar{t}+\epsilon]}\Big].
\end{aligned}
\end{equation*}
By the $L^\beta$-estimate in Section 2, we have
\begin{equation*}
\begin{aligned}
&E\bigg[\sup_{0\leq t\leq T}(|\xi^{2,\epsilon}_t|^2+|\eta^{2,\epsilon}_t|^2)+\sum^2_{i=1}\bigg[\int_0^T|\zeta^{i,2,\epsilon}_t|^2dt+\int_0^T\int_{\mathcal{E}_i}|\lambda^{i,2,\epsilon}_{(t,e)}|^2\nu_i(de)dt\bigg]\bigg]\\
&\leq C_\beta E\bigg[\big|(\tilde{\phi}_x^\epsilon(T)-\phi_x(\bar{x}_T))x_T^1\big|^2+\bigg(\int_0^TA_1^\epsilon(t)dt\bigg)^2\\
&\quad +\bigg(\int_0^T\big|(\tilde{b}_{1x}^\epsilon(t)-\tilde{b}_{1x}(t))x^1_t+\delta \tilde{b}_1(t)\mathbbm{1}_{[\bar{t},\bar{t}+\epsilon]}\big|dt\bigg)^2+\sum_{i=1}^2\bigg(\int_0^T\big|(\tilde{\sigma}_{ix}^\epsilon(t)-\sigma_{ix}(t))x^1_t\big|^2dt\bigg)\\
&\quad +\sum_{i=1}^2\bigg(\int_0^T\int_{\mathcal{E}_i}\big|(\tilde{f}_{ix}^\epsilon(t,e)-f_{ix}(t,e))x^1_{t-}+\delta f_i(t,e)\mathbbm{1}_{\mathcal{O}}\big|^2N_i(de,dt)\bigg)\bigg]\\
&\leq C_{\beta}\bigg\{E\bigg[\bigg|\int_0^1\phi_x(\bar{x}_T+\theta(x^\epsilon_T-\bar{x}_T))-\phi_x(\bar{x}_T)d\theta\bigg|^2|x_T^1|^2\bigg]
 +E\bigg[\bigg(\int_0^TA_1^\epsilon(t)dt\bigg)^2\bigg]\\
&\quad +E\bigg[\bigg(\int_0^T|(\tilde{b}_{1x}^\epsilon(t)-\tilde{b}_{1x}(t))x^1_t+\delta \tilde{b}_1(t)\mathbbm{1}_{[\bar{t},\bar{t}+\epsilon]}|dt\bigg)^2\bigg]\\
&\quad +\sum_{i=1}^2E\bigg[\bigg(\int_0^T(\tilde{\sigma}_{ix}^\epsilon(t)-\sigma_{ix}(t))^2|x^1_t|^2dt\bigg)\bigg]\\
&\quad +\sum_{i=1}^2E\bigg[\bigg(\int_0^T\int_{\mathcal{E}_i}(\tilde{f}_{ix}^\epsilon(t,e)-f_{ix}(t,e))^2|x^1_{t-}|^2+(\delta f_i(t,e))^2\mathbbm{1}_{\mathcal{O}}N_i(de,dt)\bigg)\bigg]\bigg\}\\
&\leq C_{\beta,T}\bigg\{\bigg[E\sup_{0\leq t\leq T}|\xi^{1,\epsilon}_t|^4\bigg]^{\frac{1}{2}}\bigg[E\sup_{0\leq t\leq T}|x_t^1|^4\bigg]^{\frac{1}{2}}
 +E\bigg[\bigg(\int_0^TA_1^\epsilon(t)dt\bigg)^2\bigg]\\
&\quad +E\bigg[\bigg(\int_0^T(|\xi^{1,\epsilon}_t|^2+|\delta \tilde{b}_{1x}(t)|^2\mathbbm{1}_{[\bar{t},\bar{t}+\epsilon]})|x^1_t|^2dt+\epsilon\int_{\bar{t}}^{\bar{t}+\epsilon}|\delta \tilde{b}_1(t)|^2dt\bigg)\bigg]\\
&\quad +\sum_{i=1}^2E\bigg[\bigg(\int_0^T(|\xi^{1,\epsilon}_t|^2+|\delta\sigma_{ix}(t)|^2\mathbbm{1}_{[\bar{t},\bar{t}+\epsilon]})|x^1_t|^2dt\bigg)\bigg]\\
&\quad +\sum_{i=1}^2E\bigg[\bigg(\int_0^T\int_{\mathcal{E}_i}(|\xi^{1,\epsilon}_{t-}|^2+|\delta f_{ix}(t,e)|^2\mathbbm{1}_{\mathcal{O}})|x^1_{t-}|^2+(\delta f_i(t,e))^2\mathbbm{1}_{\mathcal{O}}N_i(de,dt)\bigg)\bigg]\bigg\}\\
&\leq C_{\beta,T}\bigg\{\epsilon^2+E\bigg[\bigg(\int_0^TA_1^\epsilon(t)dt\bigg)^2\bigg]+\sum_{i=1}^2E\bigg[\sup_{0\leq t\leq T}|x^1_t|^2\int_{\bar{t}}^{\bar{t}+\epsilon}|\delta\sigma_{ix}(t)|^2dt\bigg]\\
&\quad +E\bigg[\sup_{0\leq t\leq T}|\xi^{1,\epsilon}_t|^2\sup_{0\leq t\leq T}|x^1_t|^2+\sup_{0\leq t\leq T}|x^1_t|^2\int_{\bar{t}}^{\bar{t}+\epsilon}|\delta \tilde{b}_{1x}(t)|^2dt
 +\epsilon\int_{\bar{t}}^{\bar{t}+\epsilon}|\delta \tilde{b}_1(t)|^2dt\bigg]\\
&\quad +\sum_{i=1}^2E\bigg[\bigg(\int_0^T\int_{\mathcal{E}_i}|x^1_{t-}|^2|\delta f_{ix}(t,e)|^2\mathbbm{1}_{\mathcal{O}}+(\delta f_i(t,e))^2\mathbbm{1}_{\mathcal{O}}N_i(de,dt)\bigg)\bigg]\bigg\}\\
&\leq C_{\beta,T}\bigg\{\epsilon^2+E\bigg[\bigg(\int_0^TA_1^\epsilon(t)dt\bigg)^2\bigg]+\epsilon E\bigg[\sup_{0\leq t\leq T}|x^1_t|^2\bigg]\\
&\quad +\bigg[E\bigg[\sup_{0\leq t\leq T}|\xi^{1,\epsilon}_t|^4\bigg]\bigg]^{\frac{1}{2}}\bigg[E\bigg[\sup_{0\leq t\leq T}|x^1_t|^4\bigg]\bigg]^{\frac{1}{2}}
 +\epsilon E\bigg[\int_{\bar{t}}^{\bar{t}+\epsilon}(1+|\bar{x}|^2+|\bar{u}|^2+|u|^2)dt\bigg]\bigg\}\\
\end{aligned}
\end{equation*}
\begin{equation*}
\begin{aligned}
&\leq C_{\beta,T}\bigg\{\epsilon^2+E\bigg[\bigg(\int_0^TA_1^\epsilon(t)dt\bigg)^2\bigg]\bigg\}.
\end{aligned}
\end{equation*}
Next, we give the estimate of $E\big[\big(\int_0^TA_1^\epsilon(t)dt\big)^2\big]$. In fact, by {\bf (A1)}, we first have
\begin{equation*}
\begin{aligned}
&E\bigg[\bigg(\int_0^Tx_t^1(\tilde{g}_x^\epsilon(t)-g_x(t))dt\bigg)^2\bigg]
\leq \bigg\{E\bigg[\sup_{0\leq t\leq T}|x_t^1|^4\bigg]\bigg\}^{\frac{1}{2}}\bigg\{E
\bigg[\bigg(\int_0^T(\tilde{g}_x^\epsilon(t)-g_x(t))dt\bigg)^4\bigg]\bigg\}^{\frac{1}{2}}\\
&\leq \epsilon\bigg\{E\bigg[\bigg(\int_0^T\tilde{g}_{xx}^\epsilon(t)\xi_t^{1,\epsilon}+\tilde{g}_{xy}^\epsilon(t)\eta_t^{1,\epsilon}
 +\sum_{i=1}^2\bigg[\tilde{g}_{xz^i}^\epsilon(t)\zeta_t^{i,1,\epsilon}+\tilde{g}_{x\tilde{z}^i}^\epsilon(t)\int_{\mathcal{E}_i}\lambda_{(t,e)}^{i,1,\epsilon}\nu_i(de)\bigg]\\
&\qquad +\delta g_x(t)\mathbbm{1}_{[\bar{t},\bar{t}+\epsilon]}dt\bigg)^4\bigg]\bigg\}^{\frac{1}{2}}\\
&\leq \epsilon C\bigg\{E\bigg[\bigg(\int_0^T\xi_t^{1,\epsilon}+\eta_t^{1,\epsilon}+\sum_{i=1}^2\bigg[\zeta_t^{i,1,\epsilon}+\int_{\mathcal{E}_i}\lambda_{(t,e)}^{i,1,\epsilon}\nu_i(de)\bigg]
 +\delta g_x(t)\mathbbm{1}_{[\bar{t},\bar{t}+\epsilon]}dt\bigg)^4\bigg]\bigg\}^{\frac{1}{2}}\\
 &\leq \epsilon C\bigg\{E\bigg[\sup_{0\leq t\leq T}\big(|\xi_t^{1,\epsilon}|^4+|\eta_t^{1,\epsilon}|^4\big)+\sum_{i=1}^2\bigg[\bigg(\int_0^T|\zeta_t^{i,1,\epsilon}|^2dt\bigg)^2
 +\bigg(\int_0^T||\lambda_{(t,e)}^{i,1,\epsilon}||^2dt\bigg)^2\bigg]\\
&\qquad +\epsilon^2\bigg(\int_{\bar{t}}^{\bar{t}+\epsilon}|\delta g_x(t)|^2dt\bigg)^2\bigg]\bigg\}^{\frac{1}{2}}\\
&\leq \epsilon C\bigg\{\epsilon^2+\epsilon^4+\sum_{i=1}^2E\bigg[\bigg(\int_0^T|\zeta_t^{i,1,\epsilon}|^2dt\bigg)^2
 +\bigg(\int_0^T\int_{\mathcal{E}_i}|\lambda_{(t,e)}^{i,1,\epsilon}|^2\nu_i(de)dt\bigg)^2\bigg]\bigg\}^{\frac{1}{2}}
\leq O(\epsilon^2),
\end{aligned}
\end{equation*}
where
\begin{equation*}\label{notation3}
\begin{aligned}
&\tilde{g}_{xa}^\epsilon(t)=\int_0^1\int_0^1\theta g_{xa}(t,\bar{\Theta}+\lambda\theta(\Theta^\epsilon-\bar{\Theta}),u^\epsilon)d\lambda d\theta,\ a=x,y,z^1,z^2,\tilde{z}^1,\tilde{z}^2,
\end{aligned}
\end{equation*}
and $\tilde{g}_{ya}^\epsilon(t),\tilde{g}_{z^1a}^\epsilon(t),\tilde{g}_{z^2a}^\epsilon(t),\tilde{g}_{\tilde{z}^1a}^\epsilon(t),\tilde{g}_{\tilde{z}^2a}^\epsilon(t)$, for $a=x,y,z^1,z^2,\tilde{z}^1,\tilde{z}^2$ are defined similarly. Moreover, we have, for i=1,2,
\begin{equation*}
\begin{aligned}
&E\bigg[\bigg(\int_0^Tz_t^{i,1}(\tilde{g}_{z^i}^\epsilon(t)-g_{z^i}(t))dt\bigg)^2\bigg]\leq E\bigg[\bigg(\int_0^T|z_t^{i,1}|^2dt\bigg)\bigg(\int_0^T|(\tilde{g}_{z^i}^\epsilon(t)-g_{z^i}(t))|^2dt\bigg)\bigg]\\
&\leq \bigg\{E\bigg[\bigg(\int_0^T|z_t^{i,1}|^2dt\bigg)^2\bigg]\bigg\}^{\frac{1}{2}}
 \bigg\{E\bigg[\bigg(\int_0^T|(\tilde{g}_{z^i}^\epsilon(t)-g_{z^i}(t))|^2dt\bigg)^2\bigg]\bigg\}^{\frac{1}{2}}\\
&\leq \epsilon\bigg\{E\bigg[\bigg(\int_0^T|(\tilde{g}_{z^i}^\epsilon(t)-g_{z^i}(t))|^2dt\bigg)^2\bigg]\bigg\}^{\frac{1}{2}}\leq O(\epsilon^2),
\end{aligned}
\end{equation*}
and similarly,
\begin{equation*}
\begin{aligned}
&E\bigg[\bigg(\int_0^Ty_t^1(\tilde{g}_y^\epsilon(t)-g_y(t))dt\bigg)^2\bigg]\leq O(\epsilon^2),\\
&E\bigg[\bigg(\int_0^T\int_{\mathcal{E}_i}\tilde{z}_{(t,e)}^{i,1}\nu_i(de)(\tilde{g}_{\tilde{z}^i}^\epsilon(t)-g_{\tilde{z}^i}(t))dt\bigg)^2\bigg]\leq O(\epsilon^2),\ i=1,2.
\end{aligned}
\end{equation*}
And
\begin{equation*}
\begin{aligned}
&E\bigg[\bigg(\int_0^Tg_{z^i}(t)p_t\delta\sigma_i(t)\mathbbm{1}_{[\bar{t},\bar{t}+\epsilon]}dt\bigg)^2\bigg]
\leq E\bigg[\sup_{0\leq t\leq T}|p_t|^2\bigg(\int_{\bar{t}}^{\bar{t}+\epsilon}\delta\sigma_i(t)dt\bigg)^2\bigg]\\
&\leq \bigg\{E\bigg[\sup_{0\leq t\leq T}|p_t|^4\bigg]\bigg\}^{\frac{1}{2}}\bigg\{E\bigg[\bigg(\int_{\bar{t}}^{\bar{t}+\epsilon}(1+|\bar{x}|+|\bar{u}|+|u|)dt\bigg)^4\bigg]\bigg\}^{\frac{1}{2}}\\
&\leq \epsilon^{\frac{3}{2}}C\bigg\{E\bigg[\bigg(\int_{\bar{t}}^{\bar{t}+\epsilon}(1+|\bar{x}|^4+|\bar{u}|^4+|u|^4)dt\bigg)\bigg]\bigg\}^{\frac{1}{2}}\leq O(\epsilon^2),\ i=1,2,
\end{aligned}
\end{equation*}
and similarly,
\begin{equation*}
E\bigg[\bigg(\int_0^T\delta g(t)\mathbbm{1}_{[\bar{t},\bar{t}+\epsilon]}dt\bigg)^2\bigg]\leq O(\epsilon^2).
\end{equation*}
However, we only have, for i=1,2,
\begin{equation*}
\begin{aligned}
&E\bigg[\bigg(\int_0^Tq^i_t\delta\sigma_i(t)\mathbbm{1}_{[\bar{t},\bar{t}+\epsilon]}dt\bigg)^2\bigg]\\
&\leq \epsilon^{\frac{1}{2}}\bigg\{E\bigg[\bigg(\int_{\bar{t}}^{\bar{t}+\epsilon}|q^i_t|^2dt\bigg)^2\bigg]\bigg\}^{\frac{1}{2}}
 \bigg\{E\bigg[\bigg(\int_{\bar{t}}^{\bar{t}+\epsilon}(1+|\bar{x}|^4+|\bar{u}|^4+|u|^4)dt\bigg)\bigg]\bigg\}^{\frac{1}{2}}\leq o(\epsilon).
\end{aligned}
\end{equation*}
Thus, we get the second estimate in \eqref{estimate1}. The third estimate in \eqref{estimate1} can be proved similarly. The proof is complete.
\end{proof}

\begin{lemma}\label{lemma33}
Suppose that {\bf (A1)} holds, then for $\beta\geq2$, we have the following estimates:
\begin{equation}\label{estimate4}
\begin{aligned}
&E\bigg[\sup_{0\leq t\leq T}(|x_t^2|^2+|y_t^2|^2)+\sum^2_{i=1}\bigg[\int_0^T|z_t^{i,2}|^2dt
 +\int_0^T\int_{\mathcal{E}_i}|\tilde{z}_{(t,e)}^{i,2}|^2\nu_i(de)dt\bigg]\bigg]=o(\epsilon),\\
&E\bigg[\sup_{0\leq t\leq T}(|x_t^2|^\beta+|y_t^2|^\beta)+\sum^2_{i=1}\bigg[\bigg(\int_0^T|z_t^{i,2}|^2dt\bigg)^{\frac{\beta}{2}}
+\bigg(\int_0^T\int_{\mathcal{E}_i}|\tilde{z}_{(t,e)}^{i,2}|^2\nu_i(de)dt\bigg)^{\frac{\beta}{2}}\bigg]\bigg]=o(\epsilon^{\frac{\beta}{2}}),\\
&E\bigg[\sup_{0\leq t\leq T}|x_t^\epsilon-\bar{x}_t-x_t^1-x_t^2|^2+\sup_{0\leq t\leq T}|y_t^\epsilon-\bar{y}_t-y_t^1-y_t^2|^2+\sum^2_{i=1}\bigg[\int_0^T|z_t^{i,\epsilon}-\bar{z}^i_t-z_t^{i,1}-z_t^{i,2}|^2dt\\
&\quad +\int_0^T\int_{\mathcal{E}_i}|\tilde{z}_{(t,e)}^{i,\epsilon}-\bar{\tilde{z}}^i_{(t,e)}-\tilde{z}_{(t,e)}^{i,1}-\tilde{z}_{(t,e)}^{i,2}|^2\nu_i(de)dt\bigg]\bigg]=o(\epsilon^2).
\end{aligned}
\end{equation}
\end{lemma}
\begin{proof}
We can use the same technique in Lemma \ref{lemma32}, to obtain the first estimate. However, we should note that we fail to make the order reach $O(\epsilon^2)$ due to the appearance of $E\big[\big(\int_0^Tq^i_t\delta\sigma_i(t)\mathbbm{1}_{[\bar{t},\bar{t}+\epsilon]}dt\big)^2\big]$ term, and so does the second estimate. Then, we give the proof of the third estimate directly.
Let
\begin{equation}\label{notation4}
\begin{aligned}
&\xi_t^{3,\epsilon}:=x^\epsilon_t-\bar{x}_t-x_t^1-x_t^2,\quad \eta_t^{3,\epsilon}:=y^\epsilon_t-\bar{y}_t-y_t^1-y_t^2,\\
&\zeta_t^{i,3,\epsilon}:=z^{i,\epsilon}_t-\bar{z}^i_t-z_t^{i,1}-z_t^{i,2},\quad \lambda_{(t,e)}^{i,3,\epsilon}:=\tilde{z}^{i,\epsilon}_{(t,e)}-\bar{\tilde{z}}^i_{(t,e)}-\tilde{z}_{(t,e)}^{i,1}-\tilde{z}_{(t,e)}^{i,2},\ i=1,2,
\end{aligned}
\end{equation}
and recall that $\xi_t^{1,\epsilon},\eta_t^{1,\epsilon},\zeta_t^{i,1,\epsilon},\lambda_{(t,e)}^{i,1,\epsilon},\xi_t^{2,\epsilon},\eta_t^{2,\epsilon},\zeta_t^{i,2,\epsilon}$ and $\lambda_{(t,e)}^{i,2,\epsilon},i=1,2$ are defined in \eqref{notation2}. Further, we define
\begin{equation*}
\begin{aligned}
&\tilde{\Upsilon}_{xx}^{\epsilon}(t):=2\int_0^1\int_0^1\theta \Upsilon_{xx}(t,\bar{x}+\lambda\theta(x^{\epsilon}-\bar{x}),u^\epsilon)d\lambda d\theta,\mbox{ for }\Upsilon=\tilde{b}_1,\sigma_1,\sigma_2,f_1,f_2,\\
&\widetilde{D^2g}^\epsilon(t):=2\int_0^1\int_0^1\theta D^2g(t,\Theta(t,\Delta^1\mathbbm{1}_{[\bar{t},\bar{t}+\epsilon]},\Delta^2\mathbbm{1}_{[\bar{t},\bar{t}+\epsilon]})
 +\lambda\theta(\Theta^\epsilon-\Theta(t,\Delta^1\mathbbm{1}_{[\bar{t},\bar{t}+\epsilon]},\\
 &\Delta^2\mathbbm{1}_{[\bar{t},\bar{t}+\epsilon]})),u^\epsilon)d\lambda d\theta,
\end{aligned}
\end{equation*}
and
$\widetilde{\phi}_{xx}^\epsilon(T):=2\int_0^1\int_0^1\theta \phi_{xx}(\bar{x}_T+\lambda\theta(x_T^\epsilon-\bar{x}_T))d\lambda d\theta$,
then we have the following equation:
\begin{equation}
\left\{
\begin{aligned}
 d\xi_t^{3,\epsilon}&=\big[\tilde{b}_{1x}(t)\xi_t^{3,\epsilon}+B_1^\epsilon(t)\big]dt+\sum^2_{i=1}\Big[\sigma_{ix}(t)\xi_t^{3,\epsilon}+C_{i1}^\epsilon(t)\Big]dW^i_t\\
                    &\quad +\sum^2_{i=1}\int_{\mathcal{E}_i}\Big[f_{ix}(t,e)\xi_{t-}^{3,\epsilon}+D_{i1}^\epsilon(t,e)\Big]\tilde{N}_i(de,dt),\\
d\eta_t^{3,\epsilon}&=-\Big\{g_x(t)\xi_t^{3,\epsilon}+g_y(t)\eta_t^{3,\epsilon}+\sum^2_{i=1}\Big[g_{z^i}(t)\zeta_t^{i,3,\epsilon}
                     +g_{\tilde{z}^i}(t)\int_{\mathcal{E}_i}\lambda_{(t,e)}^{i,3,\epsilon}\nu_i(de)\Big]+A_2^\epsilon(t)\Big\}dt\\
                    &\quad +\sum^2_{i=1}\zeta_t^{i,3,\epsilon}dW^i_t+\sum^2_{i=1}\int_{\mathcal{E}_i}\lambda_{(t,e)}^{i,3,\epsilon}\tilde{N}_i(de,dt),\quad t\in[0,T],\\
  \xi_0^{3,\epsilon}&=0,\quad \eta_T^{3,\epsilon}=\phi_x(\bar{x}_T)\xi_T^{3,\epsilon}+\frac{1}{2}\widetilde{\phi}_{xx}^\epsilon(T)(\xi_T^{1,\epsilon})^2-\frac{1}{2}\phi_{xx}(\bar{x}_T)(x_T^1)^2,
\end{aligned}
\right.
\end{equation}
where
\begin{equation*}
\begin{aligned}
B_1^\epsilon(t)&:=\delta \tilde{b}_{1x}(t)\xi_t^{1,\epsilon}\mathbbm{1}_{[\bar{t},\bar{t}+\epsilon]}+\frac{1}{2}\tilde{b}^{\epsilon}_{1xx}(t)|\xi_t^{1,\epsilon}|^2-\frac{1}{2}\tilde{b}_{1xx}(t)(x^1_t)^2,\\
C_{i1}^\epsilon(t)&:=\delta \sigma_{ix}(t)\xi_t^{2,\epsilon}\mathbbm{1}_{[\bar{t},\bar{t}+\epsilon]}+\frac{1}{2}\tilde{\sigma}^{\epsilon}_{ixx}(t)|\xi_t^{1,\epsilon}|^2-\frac{1}{2}\sigma_{ixx}(t)(x^1_t)^2,\\
D_{i1}^\epsilon(t,e)&:=\delta f_{ix}(t,e)\xi_t^{1,\epsilon}\mathbbm{1}_{\mathcal{O}}+\frac{1}{2}\tilde{f}^{\epsilon}_{ixx}(t,e)|\xi_t^{1,\epsilon}|^2-\frac{1}{2}f_{ixx}(t,e)(x^1_t)^2
 +\delta f_i(t,e)\mathbbm{1}_{\mathcal{O}},\\
A_2^\epsilon(t)&:=\Big\{\delta g_x(t,\Delta^1,\Delta^2)\xi_t^{1,\epsilon}+\delta g_y(t,\Delta^1,\Delta^2)\eta_t^{1,\epsilon}
 +\sum^2_{i=1}\Big[\delta g_{z^i}(t,\Delta^1,\Delta^2)(\zeta_t^{i,1,\epsilon}-\Delta^i(t)\mathbbm{1}_{[\bar{t},\bar{t}+\epsilon]})\\
&\qquad +\delta g_{\tilde{z}^i}(t,\Delta^1,\Delta^2)\int_{\mathcal{E}_i}\lambda_{(t,e)}^{i,1,\epsilon}\nu_i(de)\Big]\Big\}\mathbbm{1}_{[\bar{t},\bar{t}+\epsilon]}
 +\frac{1}{2}\check{\Xi}(t)\widetilde{D^2g}^\epsilon(t)\check{\Xi}(t)^\top-\frac{1}{2}\tilde{\Xi}(t)D^2g(t)\tilde{\Xi}(t)^\top,\\
\check{\Xi}(t)&:=\big[\xi_t^{1,\epsilon},\eta_t^{1,\epsilon},\zeta_t^{1,1,\epsilon}
                -\Delta^1(t)\mathbbm{1}_{[\bar{t},\bar{t}+\epsilon]},\zeta_t^{2,1,\epsilon}
                -\Delta^2(t)\mathbbm{1}_{[\bar{t},\bar{t}+\epsilon]},\int_{\mathcal{E}_1}\lambda_{(t,e)}^{1,1,\epsilon}\nu_1(de),\int_{\mathcal{E}_2}\lambda_{(t,e)}^{2,1,\epsilon}\nu_2(de)\big].
\end{aligned}
\end{equation*}
By the $L^\beta$-estimate of FBSDEP in Section 2, we have
\begin{equation*}
\begin{aligned}
&E\bigg[\sup_{0\leq t\leq T}(|\xi^{3,\epsilon}_t|^2+|\eta^{3,\epsilon}_t|^2)+\sum^2_{i=1}\bigg[\int_0^T|\zeta^{i,3,\epsilon}_t|^2dt
 +\int_0^T\int_{\mathcal{E}_i}|\lambda^{i,3,\epsilon}_{(t,e)}|^2\nu_i(de)dt\bigg]\bigg]\\
 &\leq C E\bigg[\bigg|\frac{1}{2}\widetilde{\phi}_{xx}^\epsilon(T)(\xi_T^{1,\epsilon})^2-\frac{1}{2}\phi_{xx}(\bar{x}_T)(x_T^1)^2\bigg|^2
 +\bigg(\int_0^TA_2^\epsilon(t)dt\bigg)^2+\bigg(\int_0^TB^\epsilon_1(t)dt\bigg)^2\\
\end{aligned}
\end{equation*}
\begin{equation}\label{estimate3}
\begin{aligned}
&\qquad +\sum_{i=1}^2\bigg(\int_0^T|C^\epsilon_{i1}(t)|^2dt\bigg)+\sum_{i=1}^2\bigg(\int_0^T\int_{\mathcal{E}_i}|D^\epsilon_{i1}(t,e)|^2N_i(de,dt)\bigg)\bigg]\\
&\leq C E\bigg[\bigg|\frac{1}{2}\widetilde{\phi}_{xx}^\epsilon(T)\big[(\xi_T^{1,\epsilon})^2-(x_T^1)^2\big]
 +\frac{1}{2}\big[\widetilde{\phi}_{xx}^\epsilon(T)-\phi_{xx}(\bar{x}_T)\big](x_T^1)^2\bigg|^2
 +\bigg(\int_0^TA_2^\epsilon(t)dt\bigg)^2\\
&\qquad +\bigg(\int_0^TB^\epsilon_1(t)dt\bigg)^2+\sum_{i=1}^2\bigg[\bigg(\int_0^T|C^\epsilon_{i1}(t)|^2dt\bigg)+\bigg(\int_0^T\int_{\mathcal{E}_i}|D^\epsilon_{i1}(t,e)|^2N_i(de,dt)\bigg)\bigg]\bigg]\\
&\leq C\bigg\{E\bigg[\bigg|\frac{1}{2}\widetilde{\phi}_{xx}^\epsilon(T)\xi_T^{2,\epsilon}(\xi_T^{1,\epsilon}+x_T^1)
 +\frac{1}{2}\big[\widetilde{\phi}_{xx}^\epsilon(T)-\phi_{xx}(\bar{x}_T)\big](x_T^1)^2\bigg|^2
 +\bigg(\int_0^TA_2^\epsilon(t)dt\bigg)^2\\
 &\qquad +\bigg(\int_0^TB^\epsilon_1(t)dt\bigg)^2+\sum_{i=1}^2\bigg[\bigg(\int_0^T|C^\epsilon_{i1}(t)|^2dt\bigg)+\bigg(\int_0^T\int_{\mathcal{E}_i}|D^\epsilon_{i1}(t,e)|^2N_i(de,dt)\bigg)\bigg]\bigg]\bigg\}\\
&\leq C\bigg\{E\bigg[|\xi_T^{2,\epsilon}|^2|\xi_T^{1,\epsilon}+x_T^1|^2\bigg]+E\bigg[[\widetilde{\phi}_{xx}^\epsilon(T)-\phi_{xx}(\bar{x}_T)]^2(x_T^1)^4\bigg]
 +E\bigg[\bigg(\int_0^TA_2^\epsilon(t)dt\bigg)^2\\
&\qquad +\bigg(\int_0^TB^\epsilon_1(t)dt\bigg)^2+\sum_{i=1}^2\bigg[\bigg(\int_0^T|C^\epsilon_{i1}(t)|^2dt\bigg)+\bigg(\int_0^T\int_{\mathcal{E}_i}|D^\epsilon_{i1}(t,e)|^2N_i(de,dt)\bigg)\bigg]\bigg]\bigg\}\\
&\leq C\bigg\{\bigg\{E\bigg[\sup_{0\leq t\leq T}|\xi_t^{2,\epsilon}|^4\bigg]\bigg\}^\frac{1}{2}\bigg\{E\bigg[\sup_{0\leq t\leq T}|\xi_t^{1,\epsilon}+x_t^1|^4\bigg]\bigg\}^\frac{1}{2}\\
&\qquad +\bigg\{E\bigg[[\widetilde{\phi}_{xx}^\epsilon(T)-\phi_{xx}(\bar{x}_T)]^4\bigg]\bigg\}^\frac{1}{2}
 \bigg\{E\bigg[\sup_{0\leq t\leq T}|x_t^1|^8\bigg]\bigg\}^\frac{1}{2}+E\bigg[\bigg(\int_0^TA_2^\epsilon(t)dt\bigg)^2\\
&\qquad +\bigg(\int_0^TB^\epsilon_1(t)dt\bigg)^2+\sum_{i=1}^2\bigg[\bigg(\int_0^T|C^\epsilon_{i1}(t)|^2dt\bigg)+\bigg(\int_0^T\int_{\mathcal{E}_i}|D^\epsilon_{i1}(t,e)|^2N_i(de,dt)\bigg)\bigg]\bigg]\bigg\}\\
&\leq C\bigg\{o(\epsilon^2)+E\bigg[\bigg(\int_0^TA_2^\epsilon(t)dt\bigg)^2+\bigg(\int_0^TB^\epsilon_1(t)dt\bigg)^2+\sum_{i=1}^2\bigg[\bigg(\int_0^T|C^\epsilon_{i1}(t)|^2dt\bigg)\\
&\qquad +\bigg(\int_0^T\int_{\mathcal{E}_i}|D^\epsilon_{i1}(t,e)|^2N_i(de,dt)\bigg)\bigg]\bigg\}.
\end{aligned}
\end{equation}
Next, we need to give the estimate of $E\big[\big(\int_0^TA_2^\epsilon(t)dt\big)^2\big]$. In fact,
\begin{equation}\label{A2}
\begin{aligned}
&E\bigg[\bigg(\int_0^TA_2^\epsilon(t)dt\bigg)^2\bigg]=E\bigg[\bigg(\int_0^T\bigg[\delta g_x(t,\Delta^1,\Delta^2)\xi_t^{1,\epsilon}+\delta g_y(t,\Delta^1,\Delta^2)\eta_t^{1,\epsilon}\\
&\quad+\sum^2_{i=1}\bigg[\delta g_{z^i}(t,\Delta^1,\Delta^2)(\zeta_t^{i,1,\epsilon}-\Delta^i(t)\mathbbm{1}_{[\bar{t},\bar{t}+\epsilon]})
 +\delta g_{\tilde{z}^i}(t,\Delta^1,\Delta^2)\int_{\mathcal{E}_i}\lambda_{(t,e)}^{i,1,\epsilon}\nu_i(de)\bigg]\bigg]\mathbbm{1}_{[\bar{t},\bar{t}+\epsilon]}\\
&\quad +\frac{1}{2}\check{\Xi}(t)\widetilde{D^2g}^\epsilon(t)\check{\Xi}(t)^\top-\frac{1}{2}\tilde{\Xi}(t)D^2g(t)\tilde{\Xi}(t)^\top dt\bigg)^2\bigg].
\end{aligned}
\end{equation}
We divide the right hand of \eqref{A2} into several parts to get the following estimates:
\begin{equation*}
\begin{aligned}
&E\bigg[\bigg(\int_0^T\delta g_x(t,\Delta^1,\Delta^2)\xi_t^{1,\epsilon}\mathbbm{1}_{[\bar{t},\bar{t}+\epsilon]}dt\bigg)^2\bigg]
\leq C E\bigg[\bigg(\int_{\bar{t}}^{\bar{t}+\epsilon}\xi_t^{1,\epsilon}dt\bigg)^2\bigg]
\leq C\epsilon^2E\bigg[\sup_{0\leq t\leq T}|\xi_t^{1,\epsilon}|^2\bigg]\leq o(\epsilon^2),\\
&E\bigg[\bigg(\int_0^T\delta g_{z^i}(t,\Delta^1,\Delta^2)(\zeta_t^{i,1,\epsilon}-\Delta^i(t)\mathbbm{1}_{[\bar{t},\bar{t}+\epsilon]})\mathbbm{1}_{[\bar{t},\bar{t}+\epsilon]}dt\bigg)^2\bigg]\\
&\leq C E\bigg[\bigg(\int_{\bar{t}}^{\bar{t}+\epsilon}|\delta g_{z^i}(t,\Delta^1,\Delta^2)||(\zeta_t^{i,2,\epsilon}+k^i_1(t)x_t^1)|dt\bigg)^2\bigg]
\leq C E\bigg[\bigg(\int_{\bar{t}}^{\bar{t}+\epsilon}|\zeta_t^{i,2,\epsilon}+k^i_1(t)x_t^1|dt\bigg)^2\bigg]\\
&\leq C\epsilon E\bigg[\bigg(\int_{\bar{t}}^{\bar{t}+\epsilon}|\zeta_t^{i,2,\epsilon}|^2dt\bigg)\bigg]
 +C\epsilon E\bigg[\sup_{0\leq t\leq T}|x_t^1|^2\bigg(\int_{\bar{t}}^{\bar{t}+\epsilon}|k^i_1(t)|^2dt\bigg)\bigg]\\
&\leq C\epsilon E\bigg[\bigg(\int_0^T|\zeta_t^{i,2,\epsilon}|^2dt\bigg)\bigg]
 +C\epsilon\bigg\{E\bigg[\sup_{0\leq t\leq T}|x_t^1|^4\bigg]\bigg\}^{\frac{1}{2}}\bigg\{E\bigg[\bigg(\int_{\bar{t}}^{\bar{t}+\epsilon}|k^i_1(t)|^2dt\bigg)^2\bigg]\bigg\}^{\frac{1}{2}}\leq o(\epsilon^2),
\end{aligned}
\end{equation*}
and similarly,
\begin{equation*}
\begin{aligned}
&E\bigg[\bigg(\int_0^T\delta g_y(t,\Delta^1,\Delta^2)\eta_t^{1,\epsilon}\mathbbm{1}_{[\bar{t},\bar{t}+\epsilon]}dt\bigg)^2\bigg]\leq o(\epsilon^2)\\
&E\bigg[\bigg(\int_0^T\delta g_{\tilde{z}^i}(t,\Delta^1,\Delta^2)\int_{\mathcal{E}_i}\lambda_{(t,e)}^{i,1,\epsilon}\nu_i(de)\mathbbm{1}_{[\bar{t},\bar{t}+\epsilon]}dt\bigg)^2\bigg]\leq o(\epsilon^2),\ i=1,2.
\end{aligned}
\end{equation*}
Then we consider the last part of \eqref{A2} as follows. In which, we have
\begin{equation*}
\begin{aligned}
&E\bigg[\bigg(\int_0^T\widetilde{g}_{z^1z^1}^\epsilon(t)[\zeta_t^{1,1,\epsilon}-\Delta^1(t)\mathbbm{1}_{[\bar{t},\bar{t}+\epsilon]}]^2
 -g_{z^1z^1}(t)[z_t^{1,1}-\Delta^1(t)\mathbbm{1}_{[\bar{t},\bar{t}+\epsilon]}]^2dt\bigg)^2\bigg]\\
&\leq E\bigg[\bigg(\int_0^T\widetilde{g}_{z^1z^1}^\epsilon(t)\big[(\zeta_t^{1,1,\epsilon}-\Delta^1(t)\mathbbm{1}_{[\bar{t},\bar{t}+\epsilon]})^2
 -(z_t^{1,1}-\Delta^1(t)\mathbbm{1}_{[\bar{t},\bar{t}+\epsilon]})^2\big]\\
&\qquad +(\widetilde{g}_{z^1z^1}^\epsilon(t)-g_{z^1z^1}(t))\big[z_t^{1,1}-\Delta^1(t)\mathbbm{1}_{[\bar{t},\bar{t}+\epsilon]}\big]^2dt\bigg)^2\bigg]\\
&\leq E\bigg[\bigg(\int_0^T\widetilde{g}_{z^1z^1}^\epsilon(t)\zeta_t^{1,2,\epsilon}\big[\zeta_t^{1,2,\epsilon}+2k^1_1(t)x_t^1\big]
 +(\widetilde{g}_{z^1z^1}^\epsilon(t)-g_{z^1z^1}(t))\big[k^1_1(t)x_t^1\big]^2dt\bigg)^2\bigg]\\
&\leq C E\bigg[\bigg(\int_0^T\widetilde{g}_{z^1z^1}^\epsilon(t)\zeta_t^{1,2,\epsilon}\big[\zeta_t^{1,2,\epsilon}+2k^1_1(t)x_t^1\big]dt\bigg)^2
 +\bigg(\int_0^T(\widetilde{g}_{z^1z^1}^\epsilon(t)-g_{z^1z^1}(t))\big[k^1_1(t)x_t^1\big]^2dt\bigg)^2\bigg]\\
&\leq C E\bigg[\bigg(\int_0^T\zeta_t^{1,2,\epsilon}\big[\zeta_t^{1,2,\epsilon}+2k^1_1(t)x_t^1\big]dt\bigg)^2\bigg]\\
&\qquad +E\bigg[\sup_{0\leq t\leq T}|x_t^1|^4\bigg(\int_0^T(\widetilde{g}_{z^1z^1}^\epsilon(t)-g_{z^1z^1}(t))|k^1_1(t)|^2dt\bigg)^2\bigg]\\
&\leq C E\bigg[\bigg(\int_0^T|\zeta_t^{1,2,\epsilon}|^2dt\bigg)^2\bigg]+C E\bigg[\bigg(\int_0^T\zeta_t^{1,2,\epsilon}k^1_1(t)x_t^1dt\bigg)^2\bigg]\\
&\qquad +E\bigg[\sup_{0\leq t\leq T}|x_t^1|^4\bigg(\int_0^T(\widetilde{g}_{z^1z^1}^\epsilon(t)-g_{z^1z^1}(t))|k^1_1(t)|^2dt\bigg)^2\bigg]\\
\end{aligned}
\end{equation*}
\begin{equation*}
\begin{aligned}
&\leq o(\epsilon^2)+C E\bigg[\sup_{0\leq t\leq T}|x_t^1|^2\bigg(\int_0^T\zeta_t^{1,2,\epsilon}k^1_1(t)dt\bigg)^2\bigg]\\
&\quad +E\bigg[\sup_{0\leq t\leq T}|x_t^1|^4\bigg(\int_0^T(\widetilde{g}_{z^1z^1}^\epsilon(t)-g_{z^1z^1}(t))|k^1_1(t)|^2dt\bigg)^2\bigg]\\
&\leq o(\epsilon^2)+C E\bigg[\sup_{0\leq t\leq T}|x_t^1|^2\bigg(\int_0^T|\zeta_t^{1,2,\epsilon}|^2dt\bigg)\bigg(\int_0^T|k^1_1(t)|^2dt\bigg)\bigg]\\
&\quad +\bigg\{E\bigg[\sup_{0\leq t\leq T}|x_t^1|^8\bigg]\bigg\}^{\frac{1}{2}}\bigg\{E\bigg[\bigg(\int_0^T(\widetilde{g}_{z^1z^1}^\epsilon(t)-g_{z^1z^1}(t))|k^1_1(t)|^2dt\bigg)^4\bigg]\bigg\}^{\frac{1}{2}}\\
&\leq o(\epsilon^2)+\bigg\{E\bigg[\sup_{0\leq t\leq T}|x_t^1|^6\bigg]\bigg\}^{\frac{1}{3}}\bigg\{E\bigg[\bigg(\int_0^T|\zeta_t^{1,2,\epsilon}|^2dt\bigg)^2\bigg]\bigg\}^{\frac{1}{2}}
 \bigg\{E\bigg[\bigg(\int_0^T|k^1_1(t)|^2dt\bigg)^6\bigg]\bigg\}^{\frac{1}{6}}\\
&\quad +\epsilon^2\bigg\{E\bigg[\bigg(\int_0^T(\widetilde{g}_{z^1z^1}^\epsilon(t)-g_{z^1z^1}(t))|k^1_1(t)|^2dt\bigg)^4\bigg]\bigg\}^{\frac{1}{2}}\leq o(\epsilon^2),
\end{aligned}
\end{equation*}
and similarly, we have
\begin{equation*}
\begin{aligned}
&E\bigg[\bigg(\int_0^T\widetilde{g}_{xx}^\epsilon(t)|\xi_t^{1,\epsilon}|^2-g_{xx}(t)|x_t^1|^2dt\bigg)^2\bigg]\leq o(\epsilon^2)\\
&E\bigg[\bigg(\int_0^T\widetilde{g}_{yy}^\epsilon(t)|\eta_t^{1,\epsilon}|^2-g_{yy}(t)|y_t^1|^2dt\bigg)^2\bigg]\leq o(\epsilon^2)\\
&E\bigg[\bigg(\int_0^T\widetilde{g}_{\tilde{z}^i\tilde{z}^i}^\epsilon(t)\bigg|\int_{\mathcal{E}_i}\lambda_{(t,e)}^{i,1,\epsilon}\nu_i(de)\bigg|^2
-g_{\tilde{z}^i\tilde{z}^i}(t)\bigg|\int_{\mathcal{E}_i}\tilde{z}_{(t,e)}^{i,1}\nu_i(de)\bigg|^2dt\bigg)^2\bigg]\leq o(\epsilon^2),\ i=1,2,
\end{aligned}
\end{equation*}
and the other cross terms have the same estimate.

Then, we give the estimates of $E\big[\big(\int_0^TB_1^\epsilon(t)dt\big)^2\big],E\big[\big(\int_0^T|C_{i1}^\epsilon(t)|^2dt\big)\big]$ and \\
$E\big(\int_0^T\int_{\mathcal{E}_i}|D_{i1}^\epsilon(t,e)|^2N_i(de,dt)\big)$, respectively. In fact, we have
\begin{equation*}
\begin{aligned}
&E\bigg[\bigg(\int_0^TB_1^\epsilon(t)dt\bigg)^2\bigg]
=E\bigg[\bigg(\int_0^T\delta \tilde{b}_{1x}(t)\xi^{1,\epsilon}_t\mathbbm{1}_{[\bar{t},\bar{t}+\epsilon]}
+\frac{1}{2}\tilde{b}^{\epsilon}_{1xx}(t)|\xi^{1,\epsilon}_t|^2-\frac{1}{2}\tilde{b}_{1xx}(t)|x^1_t|^2dt\bigg)^2\bigg]\\
&\leq C E\bigg[\bigg(\int_0^T\delta \tilde{b}_{1x}(t)\xi^{1,\epsilon}_t\mathbbm{1}_{[\bar{t},\bar{t}+\epsilon]}dt\bigg)^2\bigg]\\
&\quad +CE\bigg[\bigg(\int_0^T\frac{1}{2}\tilde{b}^{\epsilon}_{1xx}(t)(|\xi^{1,\epsilon}_t|^2-|x^1_t|^2)+\frac{1}{2}(\tilde{b}^{\epsilon}_{1xx}(t)-\tilde{b}_{1xx}(t))|x^1_t|^2dt\bigg)^2\bigg]\\
&\leq C E\bigg[\sup_{0\leq t\leq T}|\xi^{1,\epsilon}_t|^2\bigg(\int_{\bar{t}}^{\bar{t}+\epsilon}\delta \tilde{b}_{1x}(t)dt\bigg)^2\bigg]
 +C E\bigg[\bigg(\int_0^T\frac{1}{2}\tilde{b}^{\epsilon}_{1xx}(t)\xi^{2,\epsilon}_t(\xi^{1,\epsilon}_t+x^1_t)dt\bigg)^2\bigg]\\
&\quad +C E\bigg[\bigg(\int_0^T\frac{1}{2}(\tilde{b}^{\epsilon}_{1xx}(t)-\tilde{b}_{1xx}(t))|x^1_t|^2dt\bigg)^2\bigg]\\
&\leq C\epsilon^2E\bigg[\sup_{0\leq t\leq T}|\xi^{1,\epsilon}_t|^2\bigg]+C E\bigg[\sup_{0\leq t\leq T}|\xi^{2,\epsilon}_t|^2\sup_{0\leq t\leq T}(\xi^{1,\epsilon}_t+x^1_t)^2\bigg]\\
&\quad +C E\bigg[\sup_{0\leq t\leq T}|x^1_t|^4\bigg(\int_0^T\frac{1}{2}(\tilde{b}^{\epsilon}_{1xx}(t)-\tilde{b}_{1xx}(t))dt\bigg)^2\bigg]\\
\end{aligned}
\end{equation*}
\begin{equation*}
\begin{aligned}
&\leq o(\epsilon^2)+C\bigg\{E\bigg[\sup_{0\leq t\leq T}|\xi^{2,\epsilon}_t|^4\bigg]\bigg\}^{\frac{1}{2}}\bigg\{E\bigg[\sup_{0\leq t\leq T}(\xi^{1,\epsilon}_t+x^1_t)^4\bigg]\bigg\}^{\frac{1}{2}}\\
&\quad +\bigg\{E\bigg[\sup_{0\leq t\leq T}|x^1_t|^8\bigg]\bigg\}^{\frac{1}{2}}\bigg\{E\bigg[\bigg(\int_0^T\frac{1}{2}(\tilde{b}^{\epsilon}_{1xx}(t)-\tilde{b}_{1xx}(t))dt\bigg)^4\bigg]\bigg\}^{\frac{1}{2}}
\leq o(\epsilon^2),\\
&E\bigg[\bigg(\int_0^T|C_{i1}^\epsilon(t)|^2dt\bigg)\bigg]
=E\bigg[\bigg(\int_0^T|\delta \sigma_{ix}(t)\xi_t^{2,\epsilon}\mathbbm{1}_{[\bar{t},\bar{t}+\epsilon]}
 +\frac{1}{2}\tilde{\sigma}^{\epsilon}_{ixx}(t)|\xi_t^{1,\epsilon}|^2-\frac{1}{2}\sigma_{ixx}(t)(x^1_t)^2|^2dt\bigg)\bigg]\\
&\leq C E\bigg[\bigg(\int_0^T|\delta \sigma_{ix}(t)|^2|\xi_t^{2,\epsilon}|^2\mathbbm{1}_{[\bar{t},\bar{t}+\epsilon]}+|\frac{1}{2}\tilde{\sigma}^{\epsilon}_{ixx}(t)(|\xi_t^{1,\epsilon}|^2-|x^1_t|^2)
 +\frac{1}{2}(\tilde{\sigma}^{\epsilon}_{ixx}(t)-\sigma_{ixx}(t))(x^1_t)^2|^2dt\bigg)\bigg]\\
&\leq C E\bigg[\int_{\bar{t}}^{\bar{t}+\epsilon}|\delta \sigma_{ix}(t)|^2|\xi_t^{2,\epsilon}|^2dt\bigg]\\
&\quad +C E\bigg[\bigg(\int_0^T|\frac{1}{2}\tilde{\sigma}^{\epsilon}_{ixx}(t)(|\xi_t^{1,\epsilon}|^2-|x^1_t|^2)|^2
 +|\frac{1}{2}(\tilde{\sigma}^{\epsilon}_{ixx}(t)-\sigma_{ixx}(t))(x^1_t)^2|^2dt\bigg)\bigg]\\
&\leq C\epsilon E\bigg[\sup_{0\leq t\leq T}|\xi_t^{2,\epsilon}|^2\bigg]+CE\bigg[\bigg(\int_0^T|\xi_t^{2,\epsilon}|^2(\xi_t^{1,\epsilon}+x^1_t)^2dt\bigg)\bigg]\\
&\quad +C E\bigg[\sup_{0\leq t\leq T}|x^1_t|^4\int_0^T(\tilde{\sigma}^{\epsilon}_{ixx}(t)-\sigma_{ixx}(t))^2dt\bigg]\leq o(\epsilon^2),
\end{aligned}
\end{equation*}
and
\begin{equation*}
\begin{aligned}
&E\bigg[\bigg(\int_0^T\int_{\mathcal{E}_i}|D_{i1}^\epsilon(t,e)|^2N_i(de,dt)\bigg)\bigg]\\
&=E\bigg[\bigg(\int_0^T\int_{\mathcal{E}_i}|\delta f_{ix}(t,e)\xi_{t-}^{1,\epsilon}\mathbbm{1}_{\mathcal{O}}+\frac{1}{2}\tilde{f}^{\epsilon}_{ixx}(t,e)|\xi_{t-}^{1,\epsilon}|^2
 -\frac{1}{2}f_{ixx}(t,e)(x^1_{t-})^2+\delta f_i(t,e)\mathbbm{1}_{\mathcal{O}}|^2N_i(de,dt)\bigg)\bigg]\\
&\leq E\bigg[\bigg(\int_0^T\int_{\mathcal{E}_i}|\delta f_{ix}(t,e)\xi_{t-}^{1,\epsilon}\mathbbm{1}_{\mathcal{O}}+\frac{1}{2}\tilde{f}^{\epsilon}_{ixx}(t,e)(|\xi_{t-}^{1,\epsilon}|^2-|x^1_{t-}|^2)\\
&\qquad\qquad +\frac{1}{2}(\tilde{f}^\epsilon_{ixx}(t,e)-f_{ixx}(t,e))(x^1_{t-})^2+\delta f_i(t,e)\mathbbm{1}_{\mathcal{O}}|^2N_i(de,dt)\bigg)\bigg]\\
&\leq E\bigg[\bigg(\int_0^T\int_{\mathcal{E}_i}|\delta f_{ix}(t,e)|^2|\xi_{t-}^{1,\epsilon}|^2\mathbbm{1}_{\mathcal{O}}+\Big|\frac{1}{2}\tilde{f}^{\epsilon}_{ixx}(t,e)(|\xi_{t-}^{1,\epsilon}|^2-|x^1_{t-}|^2)\Big|^2\\
&\qquad\qquad +\Big|\frac{1}{2}(\tilde{f}^{\epsilon}_{ixx}(t,e)-f_{ixx}(t,e))(x^1_{t-})^2\Big|^2+|\delta f_i(t,e)|^2\mathbbm{1}_{\mathcal{O}}N_i(de,dt)\bigg)\bigg]\\
&\leq E\bigg[\int_0^T\int_{\mathcal{E}_i}\mathbb{E}\bigg[|\xi_{t-}^{2,\epsilon}|^2(\xi_{t-}^{1,\epsilon}+x^1_{t-})^2
 +(\tilde{f}^{\epsilon}_{ixx}(t,e)-f_{ixx}(t,e))^2(x^1_{t-})^4\bigg|\mathcal{P}\otimes\mathcal{B}(\mathcal{E}_i)\bigg]\nu_i(de)dt\bigg]\\
&\leq E\bigg[\sup_{0\leq t\leq T}|\xi_t^{2,\epsilon}|^2\sup_{0\leq t\leq T}|\xi_t^{1,\epsilon}+x^1_t|^2\bigg]\\
&\quad +E\bigg[\sup_{0\leq t\leq T}|x^1_t|^4\int_0^T\int_{\mathcal{E}_i}\mathbb{E}\bigg[(\tilde{f}^{\epsilon}_{ixx}-f_{ixx})^2\bigg|\mathcal{P}\otimes\mathcal{B}(\mathcal{E}_i)\bigg]\nu_i(de)dt\bigg]\leq o(\epsilon^2).
\end{aligned}
\end{equation*}

Thus, from \eqref{estimate3}, the third estimate in \eqref{estimate4} holds. The proof is complete.
\end{proof}

Then we give the variation equation for $\tilde{\Gamma}$ and its estimates. For simplicity, we rewrite reduced \eqref{RN2} as follows:
\begin{equation}\label{Gamma000}
\left\{
\begin{aligned}
d\tilde{\Gamma}_t^u&=\tilde{b}_2(t,x_t^u,\tilde{\Gamma}_t^u,u_t)dW^2_t+\int_{\mathcal{E}_2}f_4(t,x_{t-}^u,\tilde{\Gamma}_{t-}^u,e)\tilde{N}_2(de,dt),\quad t\in[0,T],\\
\tilde{\Gamma}_0^u&=1,
\end{aligned}
\right.
\end{equation}
where
\begin{equation*}
\begin{aligned}
\tilde{b}_2(t,x_t^u,\tilde{\Gamma}_t^u,u_t)&:=\tilde{\Gamma}_t^u\sigma_3^{-1}(t)b_2(t,x_t^u,u_t),\\
f_4(t,x_{t-}^u,\tilde{\Gamma}_{t-}^u,e)&:=\tilde{\Gamma}_{t-}^u(\lambda(t,x_{t-}^u,e)-1).
\end{aligned}
\end{equation*}

Similarly, we first have the following results.
\begin{lemma}\label{Lemma41}
For the solution $\tilde{\Gamma}$ to \eqref{Gamma000}, we have
\begin{equation*}
E\bigg[\sup_{0\leq t\leq T}|\tilde{\Gamma}_t^u|^\beta\bigg]<\infty,\ \ \forall \beta\in \mathbb{R},\ \ u\in\mathcal{U}_{ad}[0,T],
\end{equation*}
and for any $\beta\geq2$,
\begin{equation*}
\begin{aligned}
&E\bigg[\sup_{0\leq t\leq T}|x_t^u|^\beta\bigg]\leq C\Big(1+\sup_{0\leq t\leq T}E\big[|u_t|^\beta\big]\Big),\quad
 E\bigg[\sup_{0\leq t\leq T}|y_t^u|^\beta\bigg]\leq C\Big(1+\sup_{0\leq t\leq T}E\big[|u_t|^\beta\big]\Big),\\
&E\bigg[\sum_{i=1}^2\bigg(\int_0^T|z^{i,u}_t|^2dt\bigg)^{\frac{\beta}{2}}+\sum_{i=1}^2\bigg(\int_0^T\int_{\mathcal{E}_i}|\tilde{z}^{i,u}_t|^2\nu_i(de)dt\bigg)^{\frac{\beta}{2}}\bigg]\leq C\Big(1+\sup_{0\leq t\leq T}E\big[|u_t|^\beta\big]\Big).
\end{aligned}
\end{equation*}
\end{lemma}

From this, the expansions for $\tilde{b}_2, f_4$ are given as follows:
\begin{equation*}
\begin{aligned}
\tilde{b}_2(t,x^\epsilon_t,\tilde{\Gamma}_t^\epsilon,u^\epsilon_t)-\tilde{b}_2(t)
&=\tilde{b}_{2x}(t)(x_t^1+x_t^2)+\tilde{b}_{2\tilde{\Gamma}}(t)(\tilde{\Gamma}_t^1+\tilde{\Gamma}_t^2)+\frac{1}{2}\big[x_t^1,\tilde{\Gamma}_t^1\big]D^2\tilde{b}_2(t)\big[x_t^1,\tilde{\Gamma}_t^1\big]^\top\\
&\quad +\delta\tilde{b}_2(t)\mathbbm{1}_{[\bar{t},\bar{t}+\epsilon]}+\delta\tilde{b}_{2x}(t)x^1_t\mathbbm{1}_{[\bar{t},\bar{t}+\epsilon]}
 +\delta\tilde{b}_{2\tilde{\Gamma}}(t)\tilde{\Gamma}^1_t\mathbbm{1}_{[\bar{t},\bar{t}+\epsilon]}+o(\epsilon),\\
f_4(t,x^\epsilon_t,\tilde{\Gamma}_t^\epsilon,e)-f_4(t,e)
&=f_{4x}(t,e)(x_t^1+x_t^2)+f_{4\tilde{\Gamma}}(t,e)(\tilde{\Gamma}_t^1+\tilde{\Gamma}_t^2)\\
&\quad+\frac{1}{2}\big[x_t^1,\tilde{\Gamma}_t^1\big]D^2f_4(t,e)\big[x_t^1,\tilde{\Gamma}_t^1\big]^\top+o(\epsilon).
\end{aligned}
\end{equation*}
Then we have
\begin{equation*}
\left\{
\begin{aligned}
d\tilde{\Gamma}^1_t&=\Big[\tilde{b}_{2x}(t)x_t^1+\tilde{b}_{2\tilde{\Gamma}}(t)\tilde{\Gamma}_t^1+\delta \tilde{b}_2(t)\mathbbm{1}_{[\bar{t},\bar{t}+\epsilon]}\Big]dW^2_t+\int_{\mathcal{E}_2}\Big[f_{4x}(t,e)x_{t-}^1+f_{4\tilde{\Gamma}}(t,e)\tilde{\Gamma}_{t-}^1\Big]\tilde{N}_2(de,dt),\\
d\tilde{\Gamma}^2_t&=\Big[\tilde{b}_{2x}(t)x_t^2+\tilde{b}_{2\tilde{\Gamma}}(t)\tilde{\Gamma}_t^2
                    +\frac{1}{2}\big[x_t^1,\tilde{\Gamma}_t^1\big]D^2\tilde{b}_2(t)\big[x_t^1,\tilde{\Gamma}_t^1\big]^\top
                    +\delta\tilde{b}_{2x}(t)x^1_t\mathbbm{1}_{[\bar{t},\bar{t}+\epsilon]}\\
                   &\qquad +\delta\tilde{b}_{2\tilde{\Gamma}}(t)\tilde{\Gamma}^1_t\mathbbm{1}_{[\bar{t},\bar{t}+\epsilon]}\Big]dW^2_t
                    +\int_{\mathcal{E}_2}\Big[f_{4x}(t,e)x_{t-}^2+f_{4\tilde{\Gamma}}(t,e)\tilde{\Gamma}_{t-}^2\\
                   &\qquad +\frac{1}{2}[x_{t-}^1,\tilde{\Gamma}_{t-}^1]D^2f_4(t,e)[x_{t-}^1,\tilde{\Gamma}_{t-}^1]^\top\Big]\tilde{N}_2(de,dt),\quad t\in[0,T],\\
 \tilde{\Gamma}_0^1&=0,\ \tilde{\Gamma}_0^2=0,
\end{aligned}
\right.
\end{equation*}
and the following result holds.
\begin{lemma}\label{lemma42}
For $\beta\geq2$, by {\bf (A1)}, we have
\begin{equation*}
\begin{aligned}
&E\bigg[\sup_{0\leq t\leq T}|\tilde{\Gamma}^1_t|^\beta\bigg]=O(\epsilon^{\frac{\beta}{2}}),\quad
 E\bigg[\sup_{0\leq t\leq T}|\tilde{\Gamma}^2_t|^\beta\bigg]=O(\epsilon^{\beta}),\\
&E\bigg[\sup_{0\leq t\leq T}|\tilde{\Gamma}^{\epsilon}_t-\bar{\tilde{\Gamma}}_t-\tilde{\Gamma}^1_t|^2\bigg]=O(\epsilon^2),\quad
 E\bigg[\sup_{0\leq t\leq T}|\tilde{\Gamma}^{\epsilon}_t-\bar{\tilde{\Gamma}}_t-\tilde{\Gamma}^1_t-\tilde{\Gamma}^2_t|^2\bigg]=o(\epsilon^2).
\end{aligned}
\end{equation*}
\end{lemma}

\subsection{Adjoint equations and the maximum principle}

Let $\tilde{l}(t,\Theta^u(t),\tilde{\Gamma}^u_t,u_t):=\tilde{\Gamma}^u_tl(t,\Theta^u(t),u_t),\
\tilde{\Phi}(x_T^u,\tilde{\Gamma}^u_T):=\tilde{\Gamma}^u_T\Phi(x_T^u)$,
and the cost functional \eqref{CCFF} can be rewritten as
\begin{equation}\label{costf1}
J(u)=E\bigg\{\int_0^T\tilde{l}\big(t,\Theta^{u}(t),\tilde{\Gamma}^u_t,u_t\big)dt
+\tilde{\Phi}(x_T^u,\tilde{\Gamma}^u_T)+\Gamma(y_0^u)\bigg\}.
\end{equation}
Therefore, the expansion of $\tilde{l},\tilde{\Phi}$ and $\Gamma$ can be given as:
\begin{equation*}
\begin{aligned}
&\tilde{l}\big(t,x^\epsilon_t,y^\epsilon_t,z^{1,\epsilon}_t,z^{2,\epsilon}_t,\tilde{z}^{1,\epsilon}_{(t,e)},\tilde{z}^{2,\epsilon}_{(t,e)},\tilde{\Gamma}_t^\epsilon,u^\epsilon_t\big)
-\tilde{l}(t)\\
&=\tilde{l}_x(t)(x_t^1+x_t^2)+\tilde{l}_y(t)(y_t^1+y_t^2)+\sum_{i=1}^2\bigg[\tilde{l}_{z^i}(t)(z_t^{i,1'}+z_t^{i,2})+\tilde{l}_{\tilde{z}^i}(t)\int_{\mathcal{E}_i}(\tilde{z}_{(t,e)}^{i,1}+\tilde{z}_{(t,e)}^{i,2})\nu_i(de)\bigg]\\
&\quad +\frac{1}{2}\chi(t)D^2\tilde{l}(t)\chi(t)^\top+\tilde{l}_{\tilde{\Gamma}}(t)(\tilde{\Gamma}_t^1+\tilde{\Gamma}_t^2)+\delta \tilde{l}(t,\Delta^1,\Delta^2)\mathbbm{1}_{[\bar{t},\bar{t}+\epsilon]}+o(\epsilon),\\
&\tilde{\Phi}(x_T^\epsilon,\tilde{\Gamma}^\epsilon_T)-\tilde{\Phi}(\bar{x}_T,\bar{\tilde{\Gamma}}_T)\\
&=\tilde{\Phi}_x(\bar{x}_T,\bar{\tilde{\Gamma}}_T)(x^1_T+x^2_T)+\tilde{\Phi}_{\tilde{\Gamma}}(\bar{x}_T,\bar{\tilde{\Gamma}}_T)(\tilde{\Gamma}_T^1+\tilde{\Gamma}_T^2)
 +\frac{1}{2}\big[x^1_T,\tilde{\Gamma}_T^1\big]D^2\tilde{\Phi}(T)\big[x^1_T,\tilde{\Gamma}_T^1\big]^\top+o(\epsilon),\\
&\Gamma(y^\epsilon_0)-\Gamma(\bar{y}_0)=\Gamma_y(\bar{y}_0)(y^1_0+y^2_0)+\frac{1}{2}\Gamma_{yy}(\bar{y}_0)(y^1_0)^2+o(\epsilon),
\end{aligned}
\end{equation*}
where $\chi(t)=\big[x_t^1,y_t^1,z_t^{1,1'},z_t^{2,1'},\int_{\mathcal{E}_1}\tilde{z}_{(t,e)}^{1,1}\nu_1(de),\int_{\mathcal{E}_2}\tilde{z}_{(t,e)}^{2,1}\nu_2(de),\tilde{\Gamma}_t^1\big]$. Then the expansion of cost functional \eqref{costf1} is given by
\begin{equation}\label{costf2}
\begin{aligned}
&J(u^\epsilon)-J(\bar{u})=E\bigg\{\int_0^T\bigg[\tilde{l}_x(t)(x_t^1+x_t^2)+\tilde{l}_{\tilde{\Gamma}}(t)(\tilde{\Gamma}_t^1+\tilde{\Gamma}_t^2)+\tilde{l}_y(t)(y_t^1+y_t^2)\\
&\qquad+\sum^2_{i=1}\tilde{l}_{z^i}(t)(z_t^{i,1'}+z_t^{i,2})+\sum^2_{i=1}\tilde{l}_{\tilde{z}^i}(t)\int_{\mathcal{E}_i}(\tilde{z}_{(t,e)}^{i,1}+\tilde{z}_{(t,e)}^{i,2})\nu_i(de)
 +\frac{1}{2}\chi(t)D^2\tilde{l}(t)\chi(t)^\top\\
 &\qquad+\delta \tilde{l}(t,\Delta^1,\Delta^2)\mathbbm{1}_{[\bar{t},\bar{t}+\epsilon]}\bigg]dt+\tilde{\Phi}_x(\bar{x}_T,\bar{\tilde{\Gamma}}_T)(x^1_T+x^2_T)
 +\tilde{\Phi}_{\tilde{\Gamma}}(\bar{x}_T,\bar{\tilde{\Gamma}}_T)(\tilde{\Gamma}_T^1+\tilde{\Gamma}_T^2)\\
 &\qquad+\frac{1}{2}\big[x^1_T,\tilde{\Gamma}_T^1\big]D^2\tilde{\Phi}(T)\big[x^1_T,\tilde{\Gamma}_T^1\big]^\top
 +\Gamma_y(\bar{y}_0)(y^1_0+y^2_0)+\frac{1}{2}\Gamma_{yy}(\bar{y}_0)(y^1_0)^2\bigg\}+o(\epsilon).
\end{aligned}
\end{equation}

Then we introduce the new adjoint equations for $x^1+x^2,y^1+y^2,z^{1,1}+z^{1,2},z^{2,1}+z^{2,2},\tilde{z}^{1,1}+\tilde{z}^{1,2},\tilde{z}^{2,1}+\tilde{z}^{2,2}$ and $\tilde{\Gamma}^1+\tilde{\Gamma}^2$:
\begin{equation}\label{hr}
\left\{
\begin{aligned}
 dh_t&=\big[\tilde{l}_y(t)+h_tg_y(t)\big]dt+\sum^2_{i=1}\big[\tilde{l}_{z^i}(t)+h_tg_{z^i}(t)\big]dW^i_t\\
     &\quad+\sum^2_{i=1}\int_{\mathcal{E}_i}\big[\tilde{l}_{\tilde{z}^i}(t)+h_{t-}g_{\tilde{z}^i}(t)\big]\tilde{N}_i(de,dt),\\
-dr_t&=\bigg[\tilde{l}_{\tilde{\Gamma}}(t)+s^2_t\tilde{b}_{2\tilde{\Gamma}}+\int_{\mathcal{E}_2}\mathbb{E}_2\big[f_{4\tilde{\Gamma}}(t,e)|\mathcal{P}\otimes\mathcal{B}({\mathcal{E}_2})\big]\tilde{s}^2_{(t,e)}\nu_2(de)\bigg]dt\\
     &\quad -\sum^2_{i=1}s^i_tdW^i_t-\sum^2_{i=1}\int_{\mathcal{E}_i}\tilde{s}^i_{(t,e)}\tilde{N}_i(de,dt),\\
  h_0&=\Gamma_y(\bar{y}_0),\quad r_T=\tilde{\Phi}_{\tilde{\Gamma}}(\bar{x}_T,\bar{\tilde{\Gamma}}_T),
\end{aligned}
\right.
\end{equation}
\begin{equation}\label{m}
\left\{
\begin{aligned}
-dm_t&=\bigg[\tilde{l}_x(t)+h_tg_x(t)+m_t\tilde{b}_{1x}(t)+\sum^2_{i=1}n^i_t\sigma_{ix}(t)+s^2_t\tilde{b}_{2x}(t)\\
     &\qquad +\sum^2_{i=1}\int_{\mathcal{E}_i}\mathbb{E}_i\big[f_{ix}(t,e)|\mathcal{P}\otimes\mathcal{B}({\mathcal{E}_i})\big]\tilde{n}^i_{(t,e)}\nu_i(de)\\
     &\qquad +\int_{\mathcal{E}_2}\mathbb{E}_2\big[f_{4x}(t,e)|\mathcal{P}\otimes\mathcal{B}({\mathcal{E}_2})\big]\tilde{s}^2_{(t,e)}\nu_2(de)\bigg]dt\\
     &\quad -\sum^2_{i=1}n^i_tdW^i_t-\sum^2_{i=1}\int_{\mathcal{E}_i}\tilde{n}^i_{(t,e)}\tilde{N}_i(de,dt),\quad t\in[0,T]\\
 m_T&=\tilde{\Phi}_x(\bar{x}_T,\bar{\tilde{\Gamma}}_T)+h_T\phi_x(\bar{x}_T),
\end{aligned}
\right.
\end{equation}
with unique solution $(h,m,n^1,n^2,\tilde{n}^1,\tilde{n}^2,r,s^1,s^2,\tilde{s}^1,\tilde{s}^2)\in \mathcal{M}^\beta[0,T]\times \mathcal{N}^\beta[0,T]$. Applying It\^{o}'s formula to $m_t(x^1_t+x^2_t)-h_t(y_t^1+y_t^2)+r_t(\tilde{\Gamma}_t^1+\tilde{\Gamma}_t^2)$, and by \eqref{costf2}, we have
\begin{equation}\label{costf3}
\begin{aligned}
&J(u^\epsilon)-J(\bar{u})=E\bigg\{\int_0^T\bigg[\Big(-\sum_{i=1}^2\bigg[\tilde{l}_{z^i}p_t\delta\sigma_i(t)
 +h_tg_{z^i}p_t\delta\sigma_i(t)-n^i_t\delta \sigma_i(t)\bigg]\\
&+h_t\delta g(t,\Delta^1,\Delta^2)+\delta \tilde{l}(t,\Delta^1,\Delta^2)+m_t\delta \tilde{b}_1+s^2_t\delta\tilde{b}_2\Big)\mathbbm{1}_{[\bar{t},\bar{t}+\epsilon]}+\frac{1}{2}m_t\tilde{b}_{1xx}(x^1_t)^2\\
&+\sum_{i=1}^2\bigg[\frac{1}{2}n^i_t\sigma_{ixx}(x^1_t)^2+n^i_t\delta\sigma_{ix}x^1_t\mathbbm{1}_{[\bar{t},\bar{t}+\epsilon]}\bigg]+\frac{1}{2}s^2_t[x^1_t,\tilde{\Gamma}^1_t]D^2\tilde{b}_2(t)[x^1_t,\tilde{\Gamma}^1_t]^\top\\
&+s^2_t\delta\tilde{b}_{2x}x^1_t\mathbbm{1}_{[\bar{t},\bar{t}+\epsilon]}
+s^2_t\delta\tilde{b}_{2\tilde{\Gamma}}\tilde{\Gamma}^1_t\mathbbm{1}_{[\bar{t},\bar{t}+\epsilon]}+\frac{1}{2}\chi(t) D^2\tilde{l}(t)\chi(t)^\top+\frac{1}{2}h_t\Xi(t) D^2g(t)\Xi(t)^\top\\
&+\frac{1}{2}\sum_{i=1}^2\int_{\mathcal{E}_i}\mathbb{E}_i\big[f_{ixx}(t,e)\big|\mathcal{P}\otimes\mathcal{B}(\mathcal{E}_i)\big]\tilde{n}^i_{(t,e)}\nu_i(de)(x^1_t)^2\\
&+\frac{1}{2}\int_{\mathcal{E}_2}[x^1_t,\tilde{\Gamma}^1_t]\mathbb{E}_2\big[D^2f_4(t,e)\big|\mathcal{P}\otimes\mathcal{B}(\mathcal{E}_2)\big]
[x^1_t,\tilde{\Gamma}^1_t]^\top\tilde{s}^2_{(t,e)}\nu_2(de)\bigg]dt\\
&+\frac{1}{2}h_T\phi_{xx}(\bar{x}_T)(x_T^1)^2+\frac{1}{2}\tilde{\Phi}_{xx}(\bar{x}_T,\bar{\tilde{\Gamma}}_T)(x^1_T)^2
+\tilde{\Phi}_{x\tilde{\Gamma}}(\bar{x}_T,\bar{\tilde{\Gamma}}_T)x^1_T\tilde{\Gamma}_T^1\bigg\}+o(\epsilon).
\end{aligned}
\end{equation}
Then we apply It\^{o}'s formula to $x^1_t\tilde{\Gamma}^1_t$ to obtain
\begin{equation*}
\begin{aligned}
d(x^1_t\tilde{\Gamma}^1_t)&=\bigg[\tilde{b}_{1x}(t)x^1_t\tilde{\Gamma}^1_t+\sigma_{2x}(t)\tilde{b}_{2x}(t)(x^1_t)^2+\sigma_{2x}(t)\tilde{b}_{2\tilde{\Gamma}}x^1_t\tilde{\Gamma}^1_t
 +\sigma_{2x}(t)\delta \tilde{b}_2(t)x^1_t\mathbbm{1}_{[\bar{t},\bar{t}+\epsilon]}\\
&\qquad +\tilde{b}_{2x}(t)\delta\sigma_2(t)x^1_t\mathbbm{1}_{[\bar{t},\bar{t}+\epsilon]}+\tilde{b}_{2\tilde{\Gamma}}(t)\delta\sigma_2(t)\tilde{\Gamma}^1_t\mathbbm{1}_{[\bar{t},\bar{t}+\epsilon]}
 +\delta\sigma_2(t)\delta\tilde{b}_2(t)\mathbbm{1}_{[\bar{t},\bar{t}+\epsilon]}\\
&\qquad +\int_{\mathcal{E}_2}\mathbb{E}_2\big[f_{2x}(t,e)f_{4x}(t,e)|\mathcal{P}\otimes\mathcal{B}(\mathcal{E}_2)\big]\nu_2(de)(x^1_t)^2\\
&\qquad +\int_{\mathcal{E}_2}\mathbb{E}_2\big[f_{2x}(t,e)f_{4\tilde{\Gamma}}(t,e)|\mathcal{P}\otimes\mathcal{B}(\mathcal{E}_2)\big]\nu_2(de)x^1_t\tilde{\Gamma}^1_t\bigg]dt\\
&\quad +\big[\sigma_{1x}x^1_t\tilde{\Gamma}^1_t+\delta\sigma_1(t)\tilde{\Gamma}^1_t\mathbbm{1}_{[\bar{t},\bar{t}+\epsilon]}\big]dW^1_t
 +\big[\tilde{b}_{2x}(x^1_t)^2+\tilde{b}_{2\tilde{\Gamma}}x^1_t\tilde{\Gamma}^1_t+\delta\tilde{b}_2(t)x^1_t\mathbbm{1}_{[\bar{t},\bar{t}+\epsilon]}\\
 &\qquad +\sigma_{2x}x^1_t\tilde{\Gamma}^1_t+\delta\sigma_2(t)\tilde{\Gamma}^1_t\mathbbm{1}_{[\bar{t},\bar{t}+\epsilon]}\big]dW^2_t+\int_{\mathcal{E}_1}f_{1x}(t,e)x^1_{t-}\tilde{\Gamma}^1_{t-}\tilde{N}_1(de,dt)\\
&\quad +\int_{\mathcal{E}_2}\Big[f_{4x}(t,e)(x^1_{t-})^2+f_{4\tilde{\Gamma}}(t,e)x^1_{t-}\tilde{\Gamma}^1_{t-}+f_{2x}(t,e)x^1_{t-}\tilde{\Gamma}^1_{t-}\\
&\qquad +f_{2x}(t,e)f_{4x}(t,e)(x^1_{t-})^2+f_{2x}(t,e)f_{4\tilde{\Gamma}}(t,e)x^1_{t-}\tilde{\Gamma}^1_{t-}\Big]\tilde{N}_2(de,dt).
\end{aligned}
\end{equation*}
We continue to give its adjoint equation as follows:
\begin{equation}\label{alpha}
\left\{
\begin{aligned}
-d\alpha_t&=\bigg[l_x(t)+l_y(t)p_t+\sum^2_{i=1}\Big[l_{z^i}k_1^i(t)+l_{\tilde{z}^i}\int_{\mathcal{E}_i}k_2^i(t,e)\nu_i(de)+\sigma_{ix}\beta^i_t\Big]\\
          &\qquad +s^2_t\tilde{b}_{2x\tilde{\Gamma}}(t)+\tilde{b}_{1x}\alpha_t+\sigma_{2x}(t)\tilde{b}_{2\tilde{\Gamma}}(t)\alpha_t+\tilde{b}_{2\tilde{\Gamma}}(t)\beta^2_t\\
          &\qquad +\int_{\mathcal{E}_1}\mathbb{E}_1\big[f_{1x}(t,e)|\mathcal{P}\otimes\mathcal{B}(\mathcal{E}_1)\big]\nu_1(de)
           +\int_{\mathcal{E}_2}\Big[\mathbb{E}_2\big[f_{4x\tilde{\Gamma}}(t,e)\big|\mathcal{P}\otimes\mathcal{B}(\mathcal{E}_2)\big]\tilde{s}^2_{(t,e)}\\
          &\qquad +\mathbb{E}_2\big[f_{4\tilde{\Gamma}}(t,e)+f_{2x}(t,e)f_{4\tilde{\Gamma}}(t,e)+f_{2x}(t,e)\big|\mathcal{P}\otimes\mathcal{B}(\mathcal{E}_2)\big]\tilde{\beta}^2_{(t,e)}\Big]\nu_2(de)\\
          &\qquad +\int_{\mathcal{E}_2}\mathbb{E}_2\big[f_{2x}(t,e)f_{4\tilde{\Gamma}}(t,e)\big|\mathcal{P}\otimes\mathcal{B}(\mathcal{E}_2)\big]\nu_2(de)\alpha_t\bigg]dt\\
          &\quad -\sum^2_{i=1}\beta^i_tdW^i_t-\sum^2_{i=1}\int_{\mathcal{E}_i}\tilde{\beta}^i_{(t,e)}\tilde{N}_i(de,dt),\quad t\in[0,T],\\
  \alpha_T&=\tilde{\Phi}_{x\tilde{\Gamma}}(\bar{x}_T,\bar{\tilde{\Gamma}}_T),
\end{aligned}
\right.
\end{equation}
which admits a unique solution $(\alpha,\beta^1,\beta^2,\tilde{\beta}^1,\tilde{\beta}^2)\in \mathcal{N}^\beta[0,T]$.
Then applying It\^{o}'s formula to $\alpha_tx^1_t\tilde{\Gamma}^1_t$, and combining with \eqref{costf3}, we get
\begin{equation*}
\begin{aligned}
&J(u^\epsilon)-J(\bar{u})=E\bigg\{\int_0^T\bigg[\Big(-\sum_{i=1}^2\bigg[\tilde{l}_{z^i}p_t\delta\sigma_i(t)+h_tg_{z^i}p_t\delta\sigma_i(t)-n^i_t\delta \sigma_i(t)\bigg]\\
&\quad +h_t\delta g(t,\Delta^1,\Delta^2)+\delta \tilde{l}(t,\Delta^1,\Delta^2)+m_t\delta \tilde{b}_1+s^2_t\delta\tilde{b}_2\Big)\mathbbm{1}_{[\bar{t},\bar{t}+\epsilon]}+\frac{1}{2}m_t\tilde{b}_{1xx}(x^1_t)^2\\
&\quad +\sum_{i=1}^2\bigg[\frac{1}{2}n^i_t\sigma_{ixx}(x^1_t)^2+n^i_t\delta\sigma_{ix}x^1_t\mathbbm{1}_{[\bar{t},\bar{t}+\epsilon]}\bigg]
+\frac{1}{2}s^2_t\tilde{b}_{2xx}(t)(x^1_t)^2+s^2_t\delta\tilde{b}_{2x}x^1_t\mathbbm{1}_{[\bar{t},\bar{t}+\epsilon]}\\
\end{aligned}
\end{equation*}
\begin{equation}\label{costf4}
\begin{aligned}
&+\frac{1}{2}\sum_{i=1}^2\int_{\mathcal{E}_i}\mathbb{E}_i\big[f_{ixx}(t,e)\big|\mathcal{P}\otimes\mathcal{B}(\mathcal{E}_i)\big]\tilde{n}^i_{(t,e)}\nu_1(de)(x^1_t)^2
 +s^2_t\delta\tilde{b}_{2\tilde{\Gamma}}\tilde{\Gamma}^1_t\mathbbm{1}_{[\bar{t},\bar{t}+\epsilon]}\\
&+\frac{1}{2}\int_{\mathcal{E}_2}\mathbb{E}_2\Big[f_{4xx}(t,e)\Big|\mathcal{P}\otimes\mathcal{B}(\mathcal{E}_2)\Big]\tilde{s}^2_{(t,e)}\nu_2(de)(x^1_t)^2+\frac{1}{2}(x^1_t)^2\Psi^0(t) D^2\tilde{l}(t)\Psi^0(t)^\top\\
&+\frac{1}{2}(x^1_t)^2h_t\Psi(t) D^2g(t)\Psi(t)^\top+\sigma_{2x}(t)\tilde{b}_{2x}(t)\alpha_t(x^1_t)^2+\sigma_{2x}(t)\alpha_tx^1_t\delta\tilde{b}_2\mathbbm{1}_{[\bar{t},\bar{t}+\epsilon]}\\
&+\tilde{b}_{2x}\alpha_tx^1_t\delta\sigma_2\mathbbm{1}_{[\bar{t},\bar{t}+\epsilon]}
 +\tilde{b}_{2\tilde{\Gamma}}\alpha_t\tilde{\Gamma}^1_t\delta\sigma_2\mathbbm{1}_{[\bar{t},\bar{t}+\epsilon]}+\delta\sigma_2(t)\delta\tilde{b}_2\alpha_t\mathbbm{1}_{[\bar{t},\bar{t}+\epsilon]}\\
&+\int_{\mathcal{E}_2}\mathbb{E}_2\big[f_{2x}(t,e)f_{4x}(t,e)|\mathcal{P}\otimes\mathcal{B}(\mathcal{E}_2)\big]\nu_2(de)(x^1_t)^2+\delta\sigma_1(t)\beta^1_t\tilde{\Gamma}^1_t\mathbbm{1}_{[\bar{t},\bar{t}+\epsilon]}
 +\tilde{b}_{2x}\beta^2_t(x^1_t)^2\\
&+\delta\tilde{b}_2\beta^2_tx^1_t\mathbbm{1}_{[\bar{t},\bar{t}+\epsilon]}+\delta\sigma_2(t)\beta^2_t\tilde{\Gamma}^1_t\mathbbm{1}_{[\bar{t},\bar{t}+\epsilon]}
 +\int_{\mathcal{E}_2}\mathbb{E}_2\big[f_{4x}(t,e)\tilde{\beta}^2_{(t,e)}+f_{2x}(t,e)f_{4x}(t,e)\\
& \times\tilde{\beta}^2_{(t,e)}\big|\mathcal{P}\otimes\mathcal{B}(\mathcal{E}_2)\big]\nu_2(de)(x^1_t)^2\bigg]dt
 +\frac{1}{2}\big[\tilde{\Phi}_{xx}(\bar{x}_T,\bar{\tilde{\Gamma}}_T)+h_T\phi_{xx}(\bar{x}_T)\big](x^1_T)^2\bigg\}+o(\epsilon),
\end{aligned}
\end{equation}
where in \eqref{costf4} $\Psi(t):=\big[1,p,k_1^1,k_1^2,\int_{\mathcal{E}_1}k_2^1(t,e)\nu_1(de),\int_{\mathcal{E}_2}k_2^2(t,e)\nu_2(de)\big]$ and $\Psi^0(t):=\big[1,p,k_1^1,k_1^2,\\\int_{\mathcal{E}_1}k_2^1(t,e)\nu_1(de),\int_{\mathcal{E}_2}k_2^2(t,e)\nu_2(de),0\big]$.

Observing the above equation, we apply It\^{o}'s formula to $(x^1_t)^2$ to get
\begin{equation*}
\begin{aligned}
d(x^1_t)^2&=\bigg[2\tilde{b}_{1x}(t)(x^1_t)^2+\sum_{i=1}^2\big[\sigma^2_{ix}(x^1_t)^2+(\delta\sigma_i(t))^2\mathbbm{1}_{[\bar{t},\bar{t}+\epsilon]}+2\sigma_{ix}x^1_t\delta\sigma_i\mathbbm{1}_{[\bar{t},\bar{t}+\epsilon]}\big]\\
          &\qquad+\sum_{i=1}^2\int_{\mathcal{E}_i}\mathbb{E}_i\big[f_{ix}^2(t,e)\big|\mathcal{P}\otimes\mathcal{B}(\mathcal{E}_i)\big]\nu_i(de)(x^1_t)^2\bigg]dt+\sum_{i=1}^2\Big\{\big[2\sigma_{ix}(t)(x^1_t)^2\\
          &\qquad +2\delta\sigma_ix^1_t\mathbbm{1}_{[\bar{t},\bar{t}+\epsilon]}\big]dW^i_t+\int_{\mathcal{E}_i}\big[2f_{ix}(t,e)+f_{ix}^2(t,e)\big](x^1_{t-})^2\tilde{N}_i(de,dt)\Big\},\\
\end{aligned}
\end{equation*}
then we introduce the adjoint equation for $(x^1_t)^2$ as follows
\begin{equation}\label{P3}
\left\{
\begin{aligned}
&-dP_t=\bigg\{\Psi^0(t)D^2\tilde{l}(t)\Psi^0(t)^\top+s^2_t\tilde{b}_{2xx}(t)+2\sigma_{2x}(t)\tilde{b}_{2x}(t)\alpha_t+h_t\Psi(t)D^2g(t)\Psi(t)^\top\\
&\quad +m_t\tilde{b}_{1xx}(t)+2\tilde{b}_{2x}(t)\beta^2_t+2\tilde{b}_{1x}(t)P_t+\sum^2_{i=1}\big[n^i_t\sigma_{ixx}(t)+2\sigma_{ix}(t)Q^i_t+\sigma^2_{ix}(t)P_t\big]\\
&\quad +\sum^2_{i=1}\int_{\mathcal{E}_i}\Big[\mathbb{E}_i\big[f_{ixx}(t,e)\big|\mathcal{P}\otimes\mathcal{B}(\mathcal{E}_i)\big]\tilde{n}^i_{(t,e)}+\mathbb{E}_i\big[2f_{ix}(t,e)+f^2_{ix}(t,e)\big|\mathcal{P}\otimes\mathcal{B}(\mathcal{E}_i)\big]\tilde{Q}^i_{(t,e)}\\
&\quad +\mathbb{E}_i\big[f^2_{ix}(t,e)\big|\mathcal{P}\otimes\mathcal{B}(\mathcal{E}_i)\big]P_t\Big]\nu_i(de)+\int_{\mathcal{E}_2}\Big[\mathbb{E}_2\big[f_{4xx}(t,e)\big|\mathcal{P}\otimes\mathcal{B}(\mathcal{E}_2)\big]\tilde{s}^2_{(t,e)}\\
&\quad +\mathbb{E}_2\big[2f_{4x}(t,e)+2f_{2x}(t,e)f_{4x}(t,e)\big|\mathcal{P}\otimes\mathcal{B}(\mathcal{E}_2)\big]\tilde{\beta}^2_{(t,e)}\\
&\quad +\mathbb{E}_2\big[2f_{2x}(t,e)f_{4x}(t,e)\big|\mathcal{P}\otimes\mathcal{B}(\mathcal{E}_2)\big]\alpha_t\Big]\nu_2(de)\bigg\}dt\\
&\quad -\sum^2_{i=1}Q^i_tdW^i_t-\sum^2_{i=1}\int_{\mathcal{E}_i}\tilde{Q}^i_{(t,e)}\tilde{N}_i(de,dt),\quad t\in[0,T],\\
& P_T=h_T\phi_{xx}(\bar{x}_T)+\tilde{\Phi}_{xx}(\bar{x}_T,\bar{\tilde{\Gamma}}_T),
\end{aligned}
\right.
\end{equation}
which admits a unique solution $(P,Q^1,Q^2,\tilde{Q}^1,\tilde{Q}^2)\in \mathcal{N}^\beta[0,T]$.
Applying It\^{o}'s formula to $P_t(x^1_t)^2$ and combining \eqref{costf4}, we have
\begin{equation}\label{costf5}
\begin{aligned}
&J(u^\epsilon)-J(\bar{u})=E\bigg\{\int_0^T\bigg[\Big(-\sum_{i=1}^2\big[\tilde{l}_{z^i}p_t\delta\sigma_i(t)+h_tg_{z^i}p_t\delta\sigma_i(t)-n^i_t\delta \sigma_i(t)\big]\\
&\quad +h_t\delta g(t,\Delta^1,\Delta^2)+\delta \tilde{l}(t,\Delta^1,\Delta^2)+m_t\delta \tilde{b}_1+s^2_t\delta\tilde{b}_2+\delta\sigma_2(t)\delta\tilde{b}_2\alpha_t\\
&\quad +\frac{1}{2}(\delta\sigma_1(t))^2P_t+\frac{1}{2}(\delta\sigma_2(t))^2P_t\Big)\mathbbm{1}_{[\bar{t},\bar{t}+\epsilon]}+\Big(n^1_t\delta\sigma_{1x}x^1_t+n^2_t\delta\sigma_{2x}x^1_t+s^2_t\delta\tilde{b}_{2x}x^1_t\\
&\quad +s^2_t\delta\tilde{b}_{2\tilde{\Gamma}}\tilde{\Gamma}^1_t+\sigma_{2x}(t)\alpha_tx^1_t\delta\tilde{b}_2
 +\tilde{b}_{2x}\alpha_tx^1_t\delta\sigma_2+\tilde{b}_{2\tilde{\Gamma}}\alpha_t\tilde{\Gamma}^1_t\delta\sigma_2+\delta\sigma_1(t)\beta^1_t\tilde{\Gamma}^1_t+\delta\tilde{b}_2\beta^2_tx^1_t\\
&\quad +\delta\sigma_2(t)\beta^2_t\tilde{\Gamma}^1_t+\sigma_{1x}(t)x^1_tP_t\delta\sigma_1(t)+\sigma_{2x}(t)x^1_tP_t\delta\sigma_2(t)+\delta\sigma_2(t)Q^2_tx^1_t\\
&\quad +\delta\sigma_1(t)Q^1_tx^1_t\Big)\mathbbm{1}_{[\bar{t},\bar{t}+\epsilon]}\bigg]dt\bigg\}+o(\epsilon).
\end{aligned}
\end{equation}
We can check that
\begin{equation*}
\begin{aligned}
E\int_0^T&\Big(n^1_t\delta\sigma_{1x}x^1_t+n^2_t\delta\sigma_{2x}x^1_t+s^2_t\delta\tilde{b}_{2x}x^1_t+s^2_t\delta\tilde{b}_{2\tilde{\Gamma}}\tilde{\Gamma}^1_t+\sigma_{2x}\alpha_tx^1_t\delta\tilde{b}_2+\tilde{b}_{2x}\alpha_tx^1_t\delta\sigma_2\\
&+\tilde{b}_{2\tilde{\Gamma}}\alpha_t\tilde{\Gamma}^1_t\delta\sigma_2+\delta\sigma_1\beta^1_t\tilde{\Gamma}^1_t+\delta\tilde{b}_2\beta^2_tx^1_t+\delta\sigma_2\beta^2_t\tilde{\Gamma}^1_t+\sigma_{1x}x^1_tP_t\delta\sigma_1\\
&+\sigma_{2x}x^1_tP_t\delta\sigma_2+\delta\sigma_2Q^2_tx^1_t+\delta\sigma_1Q^1_tx^1_t\Big)\mathbbm{1}_{[\bar{t},\bar{t}+\epsilon]}dt=o(\epsilon),
\end{aligned}
\end{equation*}
therefore, we have
\begin{equation}\label{costf6}
\begin{aligned}
&J(u^\epsilon)-J(\bar{u})=E\bigg\{\int_0^T\Big(-\sum^2_{i=1}\Big[\tilde{l}_{z^i}p_t\delta\sigma_i(t)+h_tg_{z^i}p_t\delta\sigma_i(t)-n^i_t\delta \sigma_i(t)-\frac{1}{2}(\delta\sigma_i(t))^2P_t\Big]
\\
&+h_t\delta g(t,\Delta^1,\Delta^2)+m_t\delta \tilde{b}_1+\delta\tilde{l}(t,\Delta^1,\Delta^2)+s^2_t\delta\tilde{b}_2+\delta\sigma_2(t)\delta\tilde{b}_2\alpha_t\Big)\mathbbm{1}_{[\bar{t},\bar{t}+\epsilon]}dt\bigg\}+o(\epsilon)\geq0,
\end{aligned}
\end{equation}
i.e.,
\begin{equation}\label{mp}
\begin{aligned}
&E\Big[\sum^2_{i=1}\big[-(\tilde{l}_{z^i}+h_tg_{z^i})p_t+n^i_t\big]\delta\sigma_i(t)+h_t\delta g(t,\Delta^1,\Delta^2)+\delta\tilde{l}(t,\Delta^1,\Delta^2)+m_t\delta \tilde{b}_1+s^2_t\delta\tilde{b}_2\\
&+\delta\sigma_2(t)\delta\tilde{b}_2\alpha_t
+\sum^2_{i=1}\frac{1}{2}(\delta\sigma_i(t))^2P_t\Big|\mathcal{F}^Y_t\Big]\geq0,\ a.e.\ t\in[0,T],\ P\mbox{-}a.s.,
\end{aligned}
\end{equation}
Define the Hamiltonian function
\begin{equation}\label{HF1}
\begin{aligned}
&\mathcal{H}(t,x,y,z^1,z^2,\tilde{z}^1,\tilde{z}^2,\tilde{\Gamma};p,h,m,n^1,n^2,s^2,\alpha,P,u)\\
&:=\sum^2_{i=1}\big[-(\tilde{l}_{z^i}+h_tg_{z^i})p_t+n^i_t\big]\sigma_i(t,x,u)+\sum^2_{i=1}\Big[\frac{1}{2}\sigma^2_i(t,x,u)P_t-\sigma_i(t,x,u)\sigma_i(t)P_t\Big]\\
&\quad +h_tg\big(t,x,y,z^1+p(\sigma_1(t,x,u)-\sigma_1(t)),z^2+p(\sigma_2(t,x,u)-\sigma_2(t)),\tilde{z}^1,\tilde{z}^2,u\big)\\
&\quad +\tilde{l}\big(t,x,y,z^1+p(\sigma_1(t,x,u)-\sigma_1(t)),z^2+p(\sigma_2(t,x,u)-\sigma_2(t)),\tilde{z}^1,\tilde{z}^2,\tilde{\Gamma},u\big)\\
&\quad +m_t\tilde{b}_1(t,x,u)+s^2_t\tilde{b}_2(t,x,\tilde{\Gamma},u)+(\sigma_2(t,x,u)-\sigma_2(t))(\tilde{b}_2(t,x,\tilde{\Gamma},u)-\tilde{b}_2(t))\alpha_t,\\
\end{aligned}
\end{equation}
the maximum conditions \eqref{mp} is equivalent to
\begin{equation}\label{mp2}
\begin{aligned}
&E\big[\mathcal{H}(t,\bar{x},\bar{y},\bar{z}^1,\bar{z}^2,\bar{\tilde{z}}^1,\bar{\tilde{z}}^2,\bar{\tilde{\Gamma}};p,h,m,n^1,n^2,s^2,\alpha,P,u)-\mathcal{H}(t)\big|\mathcal{F}^Y_t\big]\geq0,\\
&\hspace{7cm} a.e.\ t\in[0,T],\ P\mbox{-}a.s.,
\end{aligned}
\end{equation}
where $\mathcal{H}(t)=\mathcal{H}(t,\bar{x},\bar{y},\bar{z}^1,\bar{z}^2,\bar{\tilde{z}}^1,\bar{\tilde{z}}^2,\bar{\tilde{\Gamma}};p,h,m,n^1,n^2,s^2,\alpha,P,\bar{u})$.

The main result in this paper is the following theorem.

\begin{theorem}\label{the31}
Let {\bf (A1)-(A3)} hold, and $\bar{u}$ be an optimal control, $(\bar{x},\bar{y},\bar{z}^1,\bar{z}^2,\bar{\tilde{z}}^1,\bar{\tilde{z}}^2,\bar{\tilde{\Gamma}})$ be the corresponding solution to the FBSDEP \eqref{stateeq1} and \eqref{Gamma000}. Then the maximum condition \eqref{mp2} holds, where $(p,q^1,q^2,\tilde{q}^1,\tilde{q}^2)$ satisfies (\ref{p}), $(h,m,n^1,n^2,\tilde{n}^1,\tilde{n}^2,r,s^1,s^2,\tilde{s}^1,\tilde{s}^2)$ satisfies (\ref{hr}),(\ref{m}), $(\alpha,\beta^1,\beta^2,\tilde{\beta}^1,\tilde{\beta}^2)$ satisfies (\ref{alpha}) and $(P,Q^1,Q^2,\tilde{Q}^1,\tilde{Q}^2)$ satisfies (\ref{P3}).
\end{theorem}

\begin{remark}\label{rem43}
Theorem \ref{the31} is a new partially observed maximum principle which, in fact, can not be covered by any other results in the existing literature but can cover many existing results in some special cases. In fact, in (\ref{stateeq1}), (\ref{observation}) and (\ref{cf1}) of our problem,

(1)\quad considering the LQ case, if $\sigma_2=f_1=f_2=b_2=\sigma_3=f_3\equiv0$ with $z^2=\tilde{z}^1=\tilde{z}^2\equiv0$, Theorem \ref{the31} reduces to Theorem 4.1 of \cite{HJX18} if $b,\sigma$ are independent of backward components $(Y,Z)$.

Interestingly, we could find some new and important relationships when making comparisons with the existing literature.

(2)\quad If $\sigma_2=f_1=f_2=b_2=\sigma_3=f_3=l=\Phi\equiv0$, $\Gamma(y)=y$, $g$ is independent of $(z^2,\tilde{z}^1,\tilde{z}^2)$ with $z^2=\tilde{z}^1=\tilde{z}^2\equiv0$, and the Hamiltonian function \eqref{HF1} reduces to
\begin{equation}
\begin{aligned}
&\tilde{\mathcal{H}}(t,x,y,z^1;\tilde{p},h,m,n^1,\tilde{P},u)=(-h_tg_{z^1}\tilde{p}_t+n^1_t)\sigma_1(t,x,u)+m_tb_1(t,x,u)\\
&+h_tg(t,x,y,z^1+\tilde{p}(\sigma_1(t,x,u)-\sigma_1(t)),u)+\frac{1}{2}\sigma_1^2(t,x,u)\tilde{P}_t-\sigma_1(t,x,u)\sigma_1(t)\tilde{P}_t.
\end{aligned}
\end{equation}
For avoiding confusion, we reset the solutions to \eqref{p} and \eqref{P3} as $(\tilde{p},\tilde{q}^1)$ and $(\tilde{P},\tilde{Q}^1)$.
Then Theorem \ref{the31} reduces to Theorem 2 of Hu \cite{Hu17}, the global maximum principle of completely observed case for decoupled FBSDEs with the following relations hold:
\begin{equation}
\begin{aligned}
&p_t=\tilde{p}_t,\ q_t=\tilde{q}^1_t,\ p_th_t=m_t,\ h_tq_t+f_zp_th_t=n^1_t,\\
&P_th_t=\tilde{P}_t,\ h_tQ_t+f_zP_th_t=\tilde{Q}^1_t,
\end{aligned}
\end{equation}
and maximum condition have the relation
\begin{equation}
\begin{aligned}
&\frac{1}{h_t}\Big[\tilde{\mathcal{H}}(t,\bar{x},\bar{y},\bar{z}^1;\tilde{p},h,m,n^1,\tilde{P},u)-\tilde{\mathcal{H}}(t,\bar{x},\bar{y},\bar{z}^1;\tilde{p},h,m,n^1,\tilde{P},\bar{u})\Big]\\
&\quad=\mathcal{H}(t,\bar{x},\bar{y},\bar{z},u,p,q,P)-\mathcal{H}(t,\bar{x},\bar{y},\bar{z},\bar{u},p,q,P)\geq0
\end{aligned}
\end{equation}
where the generator $g=f$, $(\tilde{p}_t,\tilde{q}^1_t), h_t, (m_t,n^1_t), (\tilde{P}_t,\tilde{Q}^1_t)$ satisfy \eqref{p},\eqref{hr},\eqref{m} and \eqref{P3}, respectively. And $(p_t,q_t),(P_t,Q_t)$ satisfy (15), (16) and $\mathcal{H}$ is defined in (44) in \cite{Hu17}.

(3)\quad If $\sigma_2=f_1=f_2=b_2=\sigma_3=f_3\equiv0$, $\phi(x)=y_T$(deterministic), $l,g$ are independent of $(z^2,\tilde{z}^1,\tilde{z}^2)$ with $z^2=\tilde{z}^1=\tilde{z}^2\equiv0$, and the corresponding Hamiltonian function (with convex control domain) reduces to:
\begin{equation}
\begin{aligned}
&\mathcal{H}(t,x,y,z^1,p,h,m,n^1,u)=[-(l_{z^1}+h_tg_{z^1})p_t+n^1_t]\sigma_1(t,x,u)+m_tb_1(t,x,u)\\
&\quad+h_tg(t,x,y,z^1+p(\sigma_1(t,x,u)-\sigma_1(t)),u)+l(t,x,y,z^1+p(\sigma_1(t,x,u)-\sigma_1(t)),u).
\end{aligned}
\end{equation}
In our paper, adjoint equation $(p,q)$  is deduced by the relation between $y^1$ and $x^1$ due to its coupling at $T$ with $y_T=\phi(x_T)$. However, when we degenerate to the case $y_T=y_T$ being deterministic, it is natural that $p$ disappears.
Then Theorem \ref{the31} reduces to Theorem 4.4 of Peng \cite{Peng93}, the local maximum principle of completely observed case for decoupled FBSDEs with the following relations:
\begin{equation}
\begin{aligned}
&h_t=q_t,\ m_t=p_t,\ n^1_t=k_t,\\
&\mathcal{H}_u(t,\bar{x},\bar{y},\bar{z}^1,h,m,n^1,\bar{u})=H_u(\bar{x},\bar{y},\bar{z},\bar{u},p,k,q,t),
\end{aligned}
\end{equation}
where $h_t,m_t,n^1_t$ satisfy \eqref{hr},\eqref{m}, and $q_t,p_t,k_t$ satisfy (4.12) and $H$ is defined in the \cite{Peng93}. Noting the backward state equation, we can ensure that the relation can be built equivalently when the generators $g$ between our setting and \cite{Peng93} exist a minus sign.

(4)\quad If $\sigma_2=f_2=b_2=\sigma_3=f_3=g=\phi=\Gamma\equiv0$, $l$ is independent of $(y,z^1,z^2,\tilde{z}^1,\tilde{z}^2)$ with $y=z^1=z^2=\tilde{z}^1=\tilde{z}^2\equiv0$, and the corresponding Hamiltonian function $\mathcal{H}$ reduces to:
\begin{equation}
\begin{aligned}
&\mathcal{H}(t,x,m,n^1,P,u)\\
&\quad=n^1_t\sigma_1(t,x,u)+l(t,x,u)+m_tb_1(t,x,u)+\frac{1}{2}\sigma_1^2(t,x,u)P_t-\sigma_1(t,x,u)\sigma_1(t)P_t.
\end{aligned}
\end{equation}
Then Theorem \ref{the31} reduces to Theorem 5.2 of Song et al. \cite{STW20}, the global maximum principle of completely observed case for SDEPs with the following relation
\begin{equation}
(m,n^1,\tilde{n}^1)=(p,q,k),\quad (\tilde{P},Q^1,\tilde{Q}^1)=(P,Q,K),
\end{equation}
where, for avoiding confusion, we reset $(\tilde{P},Q^1,\tilde{Q}^1)$ as the solution to the reduced second adjoint equation \eqref{P3}, $(m,n^1,\tilde{n}^1)$ satisfies the redecued equation \eqref{m}, and $(p,q,k),(P,Q,K)$ satisfy (5.1),(5.2) in \cite{STW20}, respectively.

(5)\quad If the control domain is convex, $\sigma_2=f_2=b_2=\sigma_3=f_3\equiv0$, $g,l$ are independent of $(z^2,\tilde{z}^2)$ with $z^2=\tilde{z}^2\equiv0$, Theorem \ref{the31} reduces to Theorem 2.1 of Shi and Wu \cite{SW10}, the local maximum principle of completely observed case for FBSDEPs with $h=-p,(m,n^1,\tilde{n}^1)=(q,k,R)$, where $h,(m,n^1,\tilde{n}^1)$ satisfy \eqref{hr},\eqref{m}, and $p_t,(q_t,k_t,R_{(t,e)})$ satisfy (7) in \cite{SW10}.
\end{remark}

\begin{remark}\label{rem44}
As mentioned in Remark \ref{rem22}, for the partially observed case, we can also degenerate and cover the following cases and obtain its equivalence under two different kinds of frameworks. In fact, in (\ref{stateeq1}), (\ref{observation}) and (\ref{cf1}) of our problem,

(1)\quad if the control domain is convex, $f_1=f_2=f_3\equiv0$, $\sigma_3\equiv1$, $b_2$ is independent of $(y,z^1,z^2,\tilde{z}^1,\tilde{z}^2,u)$, $l,g$ are independent of $(\tilde{z}^1,\tilde{z}^2)$ with $\tilde{z}^1=\tilde{z}^2\equiv0$, Theorem \ref{the31} reduces to Theorem 2.1 of Wang et al. \cite{WWX13}, the local maximum principle of partially observed case for FBSDEs with correlated Brownian noises with the following important relations:
\begin{equation}
\begin{aligned}
&\tilde{\Gamma}_t=Z_t,\quad b_2(\cdot,\cdot)=h(\cdot,\cdot),\quad h_t=-p_tZ_t,\quad (r_t,s^1_t,s^2_t)=(P_t,Q_t,\tilde{Q}_t),\\
&m_t=q_tZ_t,\quad n^1_t=k_tZ_t,\quad n^2_t=h(t,\bar{x})q_tZ_t+\tilde{k}_tZ_t,
\end{aligned}
\end{equation}
and the reduced maximum conditions also have the following equivalent relations
\begin{equation}
\begin{aligned}
&E\big[\mathcal{H}_u(t,\bar{x},\bar{y},\bar{z}^1,\bar{z}^2,\bar{\tilde{\Gamma}};h,m,n^1,n^2,s^2,\bar{u})\big|\mathcal{F}^Y_t\big]\\
&=\bar{E}\big[H_u(t,\bar{x},\bar{y},\bar{z},\bar{\tilde{z}},\bar{u};p,q,k,\tilde{k},\tilde{Q})\big|\mathcal{Y}_t\big]=0,\quad\mbox{with}\quad\bar{E}\equiv E^{\bar{u}},
\end{aligned}
\end{equation}
where $\tilde{\Gamma}$ is the Radon-Nikodym derivative in the second framework (for our setting), and $Z$ is the Radon-Nikodym derivative in the first framework (\cite{WWX13}). $h,(r,s^1,s^2),(m,n^1,n^2)$ satisfy \eqref{hr}, \eqref{m}, and $(P,Q,\tilde{Q}),p,(q,k,\tilde{k})$ satisfy (2.3), (2.4) in \cite{WWX13} in which the Hamiltonian function $H$ is defined. Therefore, in the non-Poisson case, the equivalence between the two kinds of frameworks can be proved with the relations above.

(2)\quad If $f_1=f_2=f_3=g=\phi=\Gamma\equiv0$, $\sigma_3\equiv1$, $l,b_2$ are independent of $(y,z^1,z^2,\tilde{z}^1,\tilde{z}^2)$ with $y=z^1=z^2=\tilde{z}^1=\tilde{z}^2\equiv0$, Theorem \ref{the31} reduces to Theorem 2.1 of Tang \cite{Tang98}, the global maximum principle of partially observed case for SDEs with correlated Brownian noises with the following relations:
\begin{equation}
\begin{aligned}
&m_t=\rho_tq_t,\ n^1_t=\rho_tk_t,\ n^2_t=\rho_tq_th(t,\bar{x},\bar{u})+\rho_t\tilde{k}_t,\\
&r_t=r_t,\ s^1_t=R_t,\ s^2_t=\tilde{R}_t,\\
&\alpha_t=q_t,\ \beta^1_t=k_t,\ \beta^2_t=\tilde{k}_t,\\
&P_t=\rho_tQ_t, Q^1_t=\rho_tK_t,\ Q^2_t=\rho_tQ_th(t,\bar{x},\bar{u})+\rho_t\tilde{K}_t,
\end{aligned}
\end{equation}
where $(r,s^1,s^2),(m,n^1,n^2),(\alpha,\beta^1,\beta^2),(P,Q^1,Q^2)$ satisfy \eqref{hr},\eqref{m}, \eqref{alpha} and \eqref{P3}, respectively and $(r,R,\tilde{R}),(q,k,\tilde{k}),(Q,K,\tilde{K})$ satisfy (2.23), (2.24) and (2.25) in \cite{Tang98}, respectively.

Similarly, due to the equivalent transformation between these two kinds of frameworks, we can direct give the following reduced cases:

(3)\quad if $\sigma_1$ is independent of $u$, $\sigma_2=f_1=f_2=f_3\equiv0$, $\sigma_3\equiv1$, $l,b_2$ are independent of $(y,z^1,z^2,\tilde{z}^1,\tilde{z}^2)$, $g$ is independent of $(z^2,\tilde{z}^1,\tilde{z}^2)$ with $z^2=\tilde{z}^1=\tilde{z}^2\equiv0$, Theorem \ref{the31} reduces to Theorem 2.1 of Wang and Wu \cite{WW09}, the global maximum principle of partially observed case for FBSDEs;

(4)\quad if the control domain is convex, $\sigma_2=f_1=f_2=f_3\equiv0$, $\sigma_3\equiv1$, $b_2$ is independent of $(y,z^1,z^2,\tilde{z}^1,\tilde{z}^2)$, $l,g$ are independent of $(z^2,\tilde{z}^1,\tilde{z}^2)$ with $z^2=\tilde{z}^1=\tilde{z}^2\equiv0$, Theorem \ref{the31} reduces to Theorem 2.3 of Wu \cite{Wu10}, the local maximum principle of partially observed case for FBSDEs;

(5)\quad if the control domain is convex, $\sigma_2=f_2=f_3\equiv0$, $\sigma_3\equiv1$, $l,g$ are independent of $(z^2,\tilde{z}^2)$, $b_2$ is independent of $y,z^1,z^2,\tilde{z}^1,\tilde{z}^2$ with $z^2=\tilde{z}^2\equiv0$, Theorem \ref{the31} reduces to Theorem 1 of Xiao \cite{Xiao11} with $h$ independent of $(y,z,r)$, the local maximum principle for FBSDEPs with Brownian-noised observation;

(6)\quad if the control domain is convex, $f_2=f_3=g=\phi=\Gamma\equiv0$, $\sigma_3\equiv1$, $l,b_2$ are independent of $(y,z^1,z^2,\tilde{z}^1,\tilde{z}^2)$ with $y=z^1=z^2=\tilde{z}^1=\tilde{z}^2\equiv0$, Theorem \ref{the31} reduces to Theorem 2 of Xiao \cite{Xiao13}, the local maximum principle for SDEPs with correlated Brownian noise;

(7)\quad if $\sigma_2=f_1=f_2=f_3=g=\phi=\Gamma\equiv0$, $\sigma_3\equiv1$, $l,b_2$ are independent of $(y,z^1,z^2,\tilde{z}^1,\tilde{z}^2)$ with $y=z^1=z^2=\tilde{z}^1=\tilde{z}^2\equiv0$, Theorem \ref{the31} reduces to Theorem 2.1 of Li and Tang \cite{LT95}, the global maximum principle of partially observed case for SDEs;

(8)\quad if $\sigma_1=\sigma_2=f_2=f_3=g=\phi=\Gamma\equiv0$, $\sigma_3\equiv1$, $l,b_2$ are independent of $(y,z^1,z^2,\tilde{z}^1,\tilde{z}^2)$ and $y=z^1=z^2=\tilde{z}^1=\tilde{z}^2\equiv0$, Theorem \ref{the31} reduces to Theorem 2.1 of Tang and Hou \cite{TH02}, the global maximum principle for point processes with Brownian-noised observation.
\end{remark}

\begin{remark}\label{rem45}
However, Theorem \ref{the31} could not cover results in Shi and Wu \cite{SW10} or Hu et al. \cite{HJX18}, since their controlled systems are fully coupled FBSDEs. Moreover, results in Wu \cite{Wu13} and Yong \cite{Yong10} could not be covered either, since unknown parameters have to be involved.
\end{remark}

\section{Partially Observed Linear Quadratic Optimal Control Problem of FBSDEPs}

In this section, as applications, we study a partially observed LQ optimal control problem of FBSDEPs, and obtain the state estimate feedback form of the optimal control with the filtering technique with random jumps.

\subsection{Problem Formulation and Verification Theorem}

Consider the following linear state equation
\begin{equation}\label{stateeq11}
\left\{
\begin{aligned}
  dx^u_t&=\big[b_{11}(t)x^u_t+b_{12}(t)u_t+b_{13}(t)\big]dt+\big[\sigma_{11}(t)x^u_t+\sigma_{12}(t)u_t+\sigma_{13}(t)\big]dW^1_t\\
        &\quad +\big[\sigma_{21}(t)x^u_t+\sigma_{22}(t)u_t+\sigma_{23}(t)\big]d\tilde{W}^2_t+\int_{\mathcal{E}_1}\big[f_{11}(t,e)x^u_{t-}+f_{12}(t,e)\big]\tilde{N}_1(de,dt)\\
        &\quad +\int_{\mathcal{E}_2}\big[f_{21}(t,e)x^u_{t-}+f_{22}(t,e)\big]\tilde{N}^\prime_2(de,dt),\\
 -dy_t^u&=\bigg[g_{11}(t)x^u_t+g_{12}(t)y^u_t+g_{13}(t)z^{1,u}_t+g_{14}(t)z^{2,u}_t+\int_{\mathcal{E}_1}g_{15}(t,e)\tilde{z}^{1,u}_{(t,e)}\nu_1(de)\\
        &\qquad +\int_{\mathcal{E}_2}g_{16}(t,e)\tilde{z}^{2,u}_{(t,e)}\nu_2(de)+g_{17}(t)u_t+g_{18}(t)\bigg]dt-z^{1,u}_tdW^1_t-z^{2,u}_tdW^2_t\\
        &\quad -\int_{\mathcal{E}_1}\tilde{z}^{1,u}_{(t,e)}\tilde{N}_1(de,dt)-\int_{\mathcal{E}_2}\tilde{z}^{2,u}_{(t,e)}\tilde{N}_2(de,dt),\quad t\in[0,T],\\
   x_0^u&=x_0,\ y_T^u=\phi_{11}x^u_T+\phi_{12},
\end{aligned}
\right.
\end{equation}
and observation equation
\begin{equation}\label{observation11}
\left\{
\begin{aligned}
dY_t&=b_{22}(t)dt+\sigma_3(t)d\tilde{W}^2_t+\int_{\mathcal{E}_2}f_3(t,e)\tilde{N}^\prime_2(de,dt),\quad t\in[0,T],\\
 Y_0&=0,
\end{aligned}
\right.
\end{equation}
where we set $b_2(t,x,u)=b_{22}(t)$, which is independent of $x$ and $u$.
The cost functional is given as a quadratic form:
\begin{equation}\label{cf11}
J(u)=\bar{E}\bigg[\int_0^Tl_{11}(t)(u_t)^2dt+(y_0^u)^2\bigg].
\end{equation}
Here, $l_{11}(t)>0$ and we set $U=(-\infty,-1]\cup[1,+\infty)$. We assume that all the coefficients are deterministic functions, $|\sigma_3|^{-1},|f_3(t,e)|^{-1}$ exist and are bounded.
As Section 3, under the new probability measure $P$, the state equation \eqref{stateeq11} becomes
\begin{equation}\label{LQsystem}
\left\{
\begin{aligned}
 dx^u_t&=\bigg[\Big(b_{11}(t)-\sigma_3^{-1}(t)b_{22}(t)\sigma_{21}(t)-\int_{\mathcal{E}_2}(\lambda_{11}(t,e)-1)f_{21}(t,e)\nu_2(de)\Big)x^u_t\\
       &\qquad+\Big(b_{12}(t)-\sigma_3^{-1}(t)b_{22}(t)\sigma_{22}(t)\Big)u_t+b_{13}(t)-\sigma_3^{-1}(t)b_{22}(t)\sigma_{23}(t)\\
       &\qquad-\int_{\mathcal{E}_2}(\lambda_{11}(t,e)-1)f_{22}(t,e)\nu_2(de)\bigg]dt\\
       &\quad +\sum_{i=1}^2\big[\sigma_{i1}(t)x^u_t+\sigma_{i2}(t)u_t+\sigma_{i3}(t)\big]dW^i_t\\
       &\quad +\sum_{i=1}^2\int_{\mathcal{E}_i}\big[f_{i1}(t,e)x^u_{t-}+f_{i2}(t,e)\big]\tilde{N}_i(de,dt)\\
-dy_t^u&=\bigg[g_{11}(t)x^u_t+g_{12}(t)y^u_t+g_{13}(t)z^{1,u}_t+g_{14}(t)z^{2,u}_t+\int_{\mathcal{E}_1}g_{15}(t,e)\tilde{z}^{1,u}_{(t,e)}\nu_1(de)\\
       &\qquad +\int_{\mathcal{E}_2}g_{16}(t,e)\tilde{z}^{2,u}_{(t,e)}\nu_2(de)+g_{17}(t)u_t+g_{18}(t)\bigg]dt-\sum_{i=1}^2z^{i,u}_tdW^i_t\\
       &\quad -\sum_{i=1}^2\int_{\mathcal{E}_i}\tilde{z}^{i,u}_{(t,e)}\tilde{N}_i(de,dt),\quad t\in[0,T],\\
  x_0^u&=x_0,\ y_T^u=\phi_{11}x^u_T+\phi_{12}.
\end{aligned}
\right.
\end{equation}

Applying Theorem \ref{the31} in the previous section, the maximum condition for the optimal control is as follows:
\begin{equation}
\begin{aligned}
&E\Big[[n^1_t\sigma_{12}(t)+n^2_t\sigma_{22}(t)+2l_{11}(t)\bar{\tilde{\Gamma}}_t\bar{u}_t+g_{17}(t)h_t+(b_{12}(t)-\sigma_3^{-1}(t)b_{22}(t)\sigma_{22}(t))m_t](u-\bar{u}_t)\\
&\quad +[l_{11}(t)\bar{\tilde{\Gamma}}_t+\frac{1}{2}(\sigma^2_{12}(t)+\sigma^2_{22}(t))P_t](u-\bar{u}_t)^2\Big|\mathcal{F}^Y_t\Big]\geq0,\ a.e.\ t\in[0,T],\ P\mbox{-}a.s.
\end{aligned}
\end{equation}
where we only consider that $\lambda(t,x,e)=\lambda_{11}(t,e)$, which is independent of $x$. Moreover, we know that the candidate optimal control $\bar{u}$ should satisfy
\begin{equation}
\begin{aligned}
&E\Big[l_{11}(t)\bar{\tilde{\Gamma}}_tu_t^2+\big[\sigma_{12}(t)n^1_t+\sigma_{22}(t)n^2_t+g_{17}(t)h_t+(b_{12}(t)-\sigma_3^{-1}(t)b_{22}(t)\sigma_{22}(t))m_t\big]u_t\\
&\quad +\frac{1}{2}\sigma^2_{12}(t)P_tu_t^2-\sigma^2_{12}(t)P_t\bar{u}_tu_t+\frac{1}{2}\sigma^2_{22}(t)P_tu_t^2-\sigma^2_{22}(t)P_t\bar{u}_tu_t\Big|\mathcal{F}^Y_t\Big]\\
&\geq E\bigg[l_{11}(t)\bar{\tilde{\Gamma}}_t\bar{u}_t^2+\big[\sigma_{12}(t)n^1_t+\sigma_{22}(t)n^2_t+g_{17}(t)h_t+(b_{12}(t)-\sigma_3^{-1}(t)b_{22}(t)\sigma_{22}(t))m_t\big]\bar{u}_t\\
&\qquad -\frac{1}{2}\sigma^2_{12}(t)P_t\bar{u}_t^2-\frac{1}{2}\sigma^2_{22}(t)P_t\bar{u}_t^2\Big|\mathcal{F}^Y_t\Big],\quad \forall u\in U,\ a.e.\ t\in[0,T],\ P\mbox{-}a.s.
\end{aligned}
\end{equation}
Then we have
\begin{equation}\label{baru}
\bar{u}_t=
\begin{cases}
\mu_t,\ &\mu_t\in(-\infty,-1]\cup[1,+\infty),\\
1,&0\leq\mu_t<1,\\
-1,&-1<\mu_t<0,
\end{cases}
\end{equation}
with
\begin{equation}\label{barmu}
\mu_t:=-\frac{1}{2}(l_{11}(t)\bar{\tilde{\Gamma}}_t)^{-1}\big[\sigma_{12}(t)\hat{n}^1_t+\sigma_{22}(t)\hat{n}^2_t+g_{17}(t)\hat{h}_t+(b_{12}(t)-\sigma_3^{-1}(t)b_{22}(t)\sigma_{22}(t))\hat{m}_t\big],
\end{equation}
where $\hat{\kappa}_t:=E[\kappa_t|\mathcal{F}_t^Y]$ are optimal estimates for $\kappa=h,m,n^1,n^2$, and $(h,m,n^1,n^2,\tilde{n}^1,\tilde{n}^2)$ satisfies the following linear FBSDEP:
\begin{equation}\label{hm}
\left\{
\begin{aligned}
 dh_t&=g_{12}(t)h_tdt+g_{13}(t)h_tdW^1_t+g_{14}(t)h_tdW^2_t+\int_{\mathcal{E}_1}g_{15}(t,e)h_{t-}\tilde{N}_1(de,dt)\\
     &\quad +\int_{\mathcal{E}_2}g_{16}(t,e)h_{t-}\tilde{N}_2(de,dt)\\
-dm_t&=\bigg\{g_{11}(t)h_t+\Big(b_{11}(t)-\sigma_3^{-1}(t)b_{22}(t)\sigma_{21}(t)-\int_{\mathcal{E}_2}(\lambda_{11}(t,e)-1)f_{21}(t,e)\nu_2(de)\Big)m_t\\
     &\quad +\sum_{i=1}^2\Big[\sigma_{i1}(t)n^i_t+\int_{\mathcal{E}_i}f_{i1}(t,e)\tilde{n}^i_{(t,e)}\nu_i(de)\Big]\bigg\}dt\\
     &\quad -\sum_{i=1}^2n^i_tdW^i_t-\sum_{i=1}^2\int_{\mathcal{E}_i}\tilde{n}^i_{(t,e)}\tilde{N}_i(de,dt),\quad t\in[0,T],\\
  h_0&=2\bar{y}_0,\quad m_T=\phi_{11}h_T.
\end{aligned}
\right.
\end{equation}

Before we represent the optimal control $\bar{u}$ as a form of state estimate feedback based on the known information filtration $\mathcal{F}_t^Y$, we should verify that the candidate control is indeed optimal.

\begin{proposition}\label{pro51}
The $\bar{u}_t$ defined by \eqref{baru} is an optimal control.
\end{proposition}

\begin{proof}
It is easy to check the optimality of \eqref{baru} by common technique, and we omit the details.
\end{proof}

\subsection{State Filtering Estimate}

In this section, we aim to obtain the state estimate feedback of the optimal control $\bar{u}$ in the filtering form. Moreover, we wish to get a more explicit form by introducing some ODEs.

First, assuming that $m_t=\Pi_t^1h_t$, $t\in[0,T]$, with $\Pi_T^1=\phi_{11}$, and applying It\^{o}'s formula to $\Pi_t^1h_t$, where $\Pi$ is deterministic function, we achieve
\begin{equation}
\begin{aligned}
dm_t&=\big[\Pi_t^1g_{12}(t)h_t+\dot{\Pi}_t^1h_t\big]dt+\Pi_t^1g_{13}(t)h_tdW^1_t+\Pi_t^1g_{14}(t)h_tdW^2_t\\
    &\quad +\int_{\mathcal{E}_1}\Pi_t^1g_{15}(t,e)h_{t-}\tilde{N}_1(de,dt)+\int_{\mathcal{E}_2}\Pi_t^1g_{16}(t,e)h_{t-}\tilde{N}_2(de,dt).
\end{aligned}
\end{equation}
Comparing this with the coefficients in the second equation of \eqref{hm}, we have
\begin{equation}\label{n1234}
n_t^1=\Pi_t^1g_{13}(t)h_t,\ n_t^2=\Pi_t^1g_{14}(t)h_t,\ \tilde{n}^1_{(t,e)}=\Pi_t^1g_{15}(t,e)h_{t-},\ \tilde{n}^2_{(t,e)}=\Pi_t^1g_{16}(t,e)h_{t-},
\end{equation}
and
\begin{equation}\label{dtterm}
\begin{aligned}
&\Pi_t^1g_{12}(t)h_t+\dot{\Pi}_t^1h_t+g_{11}(t)h_t+\Big(b_{11}(t)-\sigma_3^{-1}(t)b_{22}(t)\sigma_{21}(t)\\
&-\int_{\mathcal{E}_2}(\lambda_{11}(t,e)-1)f_{21}(t,e)\nu_2(de)\Big)m_t+\sigma_{11}(t)n^1_t+\sigma_{21}(t)n^2_t\\
&+\int_{\mathcal{E}_1}f_{11}(t,e)\tilde{n}^1_{(t,e)}\nu_1(de)+\int_{\mathcal{E}_2}f_{21}(t,e)\tilde{n}^2_{(t,e)}\nu_2(de)=0,\quad t\in[0,T].
\end{aligned}
\end{equation}
Substituting \eqref{n1234} into \eqref{dtterm}, we have
\begin{equation}
\begin{aligned}
&\bigg[\Pi_t^1g_{12}(t)+\dot{\Pi}_t^1+g_{11}(t)+\Big(b_{11}(t)-\sigma_3^{-1}(t)b_{22}(t)\sigma_{21}(t)\\
&-\int_{\mathcal{E}_2}(\lambda_{11}(t,e)-1)f_{21}(t,e)\nu_2(de)\Big)\Pi_t^1+\sigma_{11}(t)\Pi_t^1g_{13}(t)+\sigma_{21}(t)\Pi_t^1g_{14}(t)\\
&+\int_{\mathcal{E}_1}f_{11}(t,e)\Pi_t^1g_{15}(t,e)\nu_1(de)+\int_{\mathcal{E}_2}f_{21}(t,e)\Pi_t^1g_{16}(t,e)\nu_2(de)\bigg]h_t=0,\quad t\in[0,T].
\end{aligned}
\end{equation}
If we introduce the following ODE:
\begin{equation}\label{pi1}
\left\{
\begin{aligned}
&\dot{\Pi}_t^1+\bigg[g_{12}(t)+b_{11}(t)-\sigma_3^{-1}(t)b_{22}(t)\sigma_{21}(t)+\int_{\mathcal{E}_2}(g_{16}(t,e)-\lambda_{11}(t,e)+1)f_{21}(t,e)\nu_2(de)\\
&+\sigma_{11}(t)g_{13}(t)+\sigma_{21}(t)g_{14}(t)+\int_{\mathcal{E}_1}f_{11}(t,e)g_{15}(t,e)\nu_1(de)\bigg]\Pi_t^1+g_{11}(t)=0,\quad t\in[0,T],\\
&\Pi_T^1=\phi_{11},
\end{aligned}
\right.
\end{equation}
noting \eqref{baru}, \eqref{barmu}, then the optimal control in $U=(-\infty,-1]\cup[1,+\infty)$ becomes
\begin{equation}\label{baru2}
\begin{aligned}
\bar{u}_t&=-\frac{1}{2}L(\Pi_t^1)\hat{h}_t,\quad a.e.\ t\in[0,T],\ P\mbox{-}a.s.,
\end{aligned}
\end{equation}
where we set
\begin{equation*}
\begin{aligned}
L(\Pi_t^1)&:=(l_{11}(t)\tilde{\Gamma}_t)^{-1}\big[\sigma_{12}(t)\Pi_t^1g_{13}(t)+\sigma_{22}(t)\Pi_t^1g_{14}(t)+g_{17}(t)+(b_{12}(t)-\sigma_3^{-1}(t)b_{22}(t)\sigma_{22}(t))\Pi_t^1\big].
\end{aligned}
\end{equation*}
Putting \eqref{baru2} into \eqref{LQsystem}, we get
\begin{equation}\label{LQsystem21}
\left\{
\begin{aligned}
 d\bar{x}_t&=\bigg\{\Big[b_{11}(t)-\sigma_3^{-1}(t)b_{22}(t)\sigma_{21}(t)-\int_{\mathcal{E}_2}(\lambda_{11}(t,e)-1)f_{21}(t,e)\nu_2(de)\Big]\bar{x}_t\\
           &\qquad -\frac{1}{2}\big[b_{12}(t)-\sigma_3^{-1}(t)b_{22}(t)\sigma_{22}(t)\big]L(\Pi_t^1)\hat{h}_t+b_{13}(t)-\sigma_3^{-1}(t)b_{22}(t)\sigma_{23}(t)\\
           &\qquad -\int_{\mathcal{E}_2}(\lambda_{11}(t,e)-1)f_{22}(t,e)\nu_2(de)\bigg\}dt+\sum_{i=1}^2\bigg[\sigma_{i1}(t)\bar{x}_t-\frac{1}{2}\sigma_{i2}(t)L(\Pi_t^1)\hat{h}_t\\
           &\qquad+\sigma_{i3}(t)\bigg]dW^i_t+\sum_{i=1}^2\int_{\mathcal{E}_i}\big[f_{i1}(t,e)\bar{x}_{t-}+f_{i2}(t,e)\big]\tilde{N}_i(de,dt),\\
            \bar{x}_0&=x_0,
\end{aligned}
\right.
\end{equation}
\begin{equation}\label{LQsystem22}
\left\{
\begin{aligned}
-d\bar{y}_t&=\bigg[g_{11}(t)\bar{x}_t+g_{12}(t)\bar{y}_t+g_{13}(t)\bar{z}^1_t+g_{14}(t)\bar{z}^2_t+\int_{\mathcal{E}_1}g_{15}(t,e)\bar{\tilde{z}}^1_{(t,e)}\nu_1(de)\\
           &\qquad +\int_{\mathcal{E}_2}g_{16}(t,e)\bar{\tilde{z}}^2_{(t,e)}\nu_2(de)-\frac{1}{2}g_{17}(t)L(\Pi_t^1)\hat{h}_t+g_{18}(t)\bigg]dt\\
           &\quad -\sum_{i=1}^2\bar{z}^i_tdW^i_t-\sum_{i=1}^2\int_{\mathcal{E}_i}\bar{\tilde{z}}^i_{(t,e)}\tilde{N}_i(de,dt),\quad t\in[0,T],\\
\bar{y}_T&=\phi_{11}\bar{x}_T+\phi_{12},
\end{aligned}
\right.
\end{equation}
and the observation equation becomes
\begin{equation}
\left\{
\begin{aligned}
&dY_t=\sigma_3(t)dW_t^2+\int_{\mathcal{E}_2}f_3(t,e)\tilde{N}_2(de,dt),\quad t\in[0,T],\\
&Y_0=0.
\end{aligned}
\right.
\end{equation}

We conjecture that
\begin{equation}\label{bary}
\bar{y}_t=\Pi_t^2\bar{x}_t+\Pi_t^3h_t+\eta_t,
\end{equation}
with $\Pi_T^2=\phi_{11},\Pi_T^3=0,\eta_T=\phi_{12}$. Applying It\^{o}'s formula, we have
\begin{equation*}
\begin{aligned}
d\bar{y}_t&=\bigg\{\bigg[b_{11}(t)-\sigma_3^{-1}(t)b_{22}(t)\sigma_{21}(t)-\int_{\mathcal{E}_2}(\lambda_{11}(t,e)-1)f_{21}(t,e)\nu_2(de)\bigg]\Pi_t^2\bar{x}_t\\
          &\qquad -\frac{1}{2}\big[b_{12}(t)-\sigma_3^{-1}(t)b_{22}(t)\sigma_{22}(t)\big]\Pi_t^2L(\Pi_t^1)\hat{h}_t+\bigg[b_{13}(t)-\sigma_3^{-1}(t)b_{22}(t)\sigma_{23}(t)\\
          &\qquad -\int_{\mathcal{E}_2}(\lambda_{11}(t,e)-1)f_{22}(t,e)\nu_2(de)\bigg]\Pi_t^2+\dot{\Pi}_t^2\bar{x}_t+\Pi_t^3g_{12}(t)h_t+\dot{\Pi}_t^3h_t+\dot{\eta}_t\bigg\}dt\\
          &\quad +\bigg[\sigma_{11}(t)\Pi_t^2\bar{x}_t-\frac{1}{2}\Pi_t^2\sigma_{12}(t)L(\Pi_t^1)\hat{h}_t+\sigma_{13}(t)\Pi_t^2+\Pi_t^3g_{13}(t)h_t\bigg]dW_t^1\\
          &\quad +\bigg[\sigma_{21}(t)\Pi_t^2\bar{x}_t-\frac{1}{2}\Pi_t^2\sigma_{22}(t)L(\Pi_t^1)\hat{h}_t+\sigma_{23}(t)\Pi_t^2+\Pi_t^3g_{14}(t)h_t\bigg]dW_t^2\\
          &\quad +\int_{\mathcal{E}_1}\big[\Pi_t^2f_{11}(t,e)\bar{x}_{t-}+f_{12}(t,e)\Pi_t^2+g_{15}(t,e)\Pi_t^3h_{t-}\big]\tilde{N}_1(de,dt)\\
\end{aligned}
\end{equation*}
\begin{equation}
\begin{aligned}
          &\quad +\int_{\mathcal{E}_2}\big[\Pi_t^2f_{21}(t,e)\bar{x}_{t-}+f_{22}(t,e)\Pi_t^2+g_{16}(t,e)\Pi_t^3h_{t-}\big]\tilde{N}_2(de,dt).
\end{aligned}
\end{equation}
Comparing this with the second equation in \eqref{LQsystem22}, we get
\begin{equation}\label{z1234}
\begin{aligned}
&\bar{z}^1_t=\sigma_{11}(t)\Pi_t^2\bar{x}_t-\frac{1}{2}\Pi_t^2\sigma_{12}(t)L(\Pi_t^1)\hat{h}_t+\sigma_{13}(t)\Pi_t^2+\Pi_t^3g_{13}(t)h_t,\\
&\bar{z}^2_t=\sigma_{21}(t)\Pi_t^2\bar{x}_t-\frac{1}{2}\Pi_t^2\sigma_{22}(t)L(\Pi_t^1)\hat{h}_t+\sigma_{23}(t)\Pi_t^2+\Pi_t^3g_{14}(t)h_t,\\
&\bar{\tilde{z}}^1_{(t,e)}=\Pi_t^2f_{11}(t,e)\bar{x}_{t-}+f_{12}(t,e)\Pi_t^2+g_{15}(t,e)\Pi_t^3h_{t-},\\
&\bar{\tilde{z}}^2_{(t,e)}=\Pi_t^2f_{21}(t,e)\bar{x}_{t-}+f_{22}(t,e)\Pi_t^2+g_{16}(t,e)\Pi_t^3h_{t-},
\end{aligned}
\end{equation}
and
\begin{equation}\label{dtterm2}
\begin{aligned}
&\bigg[b_{11}(t)-\sigma_3^{-1}(t)b_{22}(t)\sigma_{21}(t)-\int_{\mathcal{E}_2}(\lambda_{11}(t,e)-1)f_{21}(t,e)\nu_2(de)\bigg]\Pi_t^2\bar{x}_t\\
&-\frac{1}{2}\big[b_{12}(t)-\sigma_3^{-1}(t)b_{22}(t)\sigma_{22}(t)\big]\Pi_t^2L(\Pi_t^1)\hat{h}_t+\bigg[b_{13}(t)-\sigma_3^{-1}(t)b_{22}(t)\sigma_{23}(t)\\
&-\int_{\mathcal{E}_2}(\lambda_{11}(t,e)-1)f_{22}(t,e)\nu_2(de)\bigg]\Pi_t^2+\dot{\Pi}_t^2\bar{x}_t+\Pi_t^3g_{12}(t)h_t+\dot{\Pi}_t^3h_t+\dot{\eta}_t+g_{11}(t)\bar{x}_t\\
&+g_{12}(t)\big[\Pi_t^2\bar{x}_t+\Pi_t^3h_t+\eta_t\big]+g_{13}(t)\big[\sigma_{11}(t)\Pi_t^2\bar{x}_t-\frac{1}{2}\Pi_t^2\sigma_{12}(t)L(\Pi_t^1)\hat{h}_t+\sigma_{13}(t)\Pi_t^2\\
&+\Pi_t^3g_{13}(t)h_t\big]+g_{14}(t)\big[\sigma_{21}(t)\Pi_t^2\bar{x}_t-\frac{1}{2}\Pi_t^2\sigma_{22}(t)L(\Pi_t^1)\hat{h}_t+\sigma_{23}(t)\Pi_t^2+\Pi_t^3g_{14}(t)h_t\big]\\
&+\int_{\mathcal{E}_1}g_{15}(t,e)\big[\Pi_t^2f_{11}(t,e)\bar{x}_t+f_{12}(t,e)\Pi_t^2+g_{15}(t,e)\Pi_t^3h_t\big]\nu_1(de)+\int_{\mathcal{E}_2}g_{16}(t,e)\big[\Pi_t^2f_{21}(t,e)\bar{x}_t\\
&+f_{22}(t,e)\Pi_t^2+g_{16}(t,e)\Pi_t^3h_t\big]\nu_2(de)-\frac{1}{2}g_{17}(t)L(\Pi_t^1)\hat{h}_t+g_{18}(t)=0,\quad t\in[0,T].
\end{aligned}
\end{equation}
Taking $E[\cdot|\mathcal{F}_t^Y]$ on both sides of \eqref{bary}, \eqref{z1234} and \eqref{dtterm2}, we have
\begin{equation}
\begin{aligned}
&\hat{\bar{y}}_t=\Pi_t^2\hat{\bar{x}}_t+\Pi_t^3\hat{h}_t+\eta_t,\\
&\hat{\bar{z}}^1_t=\sigma_{11}(t)\Pi_t^2\hat{\bar{x}}_t+(\Pi_t^3g_{13}(t)-\frac{1}{2}\Pi_t^2\sigma_{12}(t)L(\Pi_t^1))\hat{h}_t+\sigma_{13}(t)\Pi_t^2,\\
&\hat{\bar{z}}^2_t=\sigma_{21}(t)\Pi_t^2\hat{\bar{x}}_t+(\Pi_t^3g_{14}(t)-\frac{1}{2}\Pi_t^2\sigma_{22}(t)L(\Pi_t^1))\hat{h}_t+\sigma_{23}(t)\Pi_t^2,\\
&\hat{\bar{\tilde{z}}}^1_{(t,e)}=\Pi_t^2f_{11}(t,e)\hat{\bar{x}}_{t-}+f_{12}(t,e)\Pi_t^2+g_{15}(t,e)\Pi_t^3\hat{h}_{t-},\\
&\hat{\bar{\tilde{z}}}^2_{(t,e)}=\Pi_t^2f_{21}(t,e)\hat{\bar{x}}_{t-}+f_{22}(t,e)\Pi_t^2+g_{16}(t,e)\Pi_t^3\hat{h}_{t-},
\end{aligned}
\end{equation}
and
\begin{equation*}
\begin{aligned}
&\bigg\{\bigg[b_{11}(t)-\sigma_3^{-1}(t)b_{22}(t)\sigma_{21}(t)-\int_{\mathcal{E}_2}(\lambda_{11}(t,e)-1)f_{21}(t,e)\nu_2(de)\bigg]\Pi_t^2+\dot{\Pi}_t^2+g_{11}(t)\\
&\quad +g_{13}(t)\sigma_{11}(t)\Pi_t^2+g_{14}(t)\sigma_{21}(t)\Pi_t^2+\int_{\mathcal{E}_1}g_{15}(t,e)f_{11}(t,e)\nu_1(de)\Pi_t^2\\
&\quad +\int_{\mathcal{E}_2}g_{16}(t,e)f_{21}(t,e)\nu_2(de)\Pi_t^2+g_{12}(t)\Pi_t^2\bigg]\hat{\bar{x}}_t\\
\end{aligned}
\end{equation*}
\begin{equation}
\begin{aligned}
&+\bigg[\dot{\Pi}_t^3+\Pi_t^3g_{12}(t)-\frac{1}{2}(b_{12}(t)-\sigma_3^{-1}(t)b_{22}(t)\sigma_{22}(t))\Pi_t^2L(\Pi_t^1)-\frac{1}{2}g_{13}(t)\Pi_t^2\sigma_{12}(t)L(\Pi_t^1)\\
&\quad -\frac{1}{2}g_{14}(t)\Pi_t^2\sigma_{22}(t)L(\Pi_t^1)+g^2_{13}(t)\Pi_t^3+g^2_{14}(t)\Pi_t^3+\int_{\mathcal{E}_1}g^2_{15}(t,e)\nu_1(de)\Pi_t^3\\
&\quad +\int_{\mathcal{E}_2}g^2_{16}(t,e)\nu_2(de)\Pi_t^3-\frac{1}{2}g_{17}(t)L(\Pi_t^1)+g_{12}(t)\Pi_t^3\bigg]\hat{h}_t\\
&+\bigg[b_{13}(t)-\sigma_3^{-1}(t)b_{22}(t)\sigma_{23}(t)-\int_{\mathcal{E}_2}(\lambda_{11}(t,e)-1)f_{22}(t,e)\nu_2(de)\bigg]\Pi_t^2\\
&+\dot{\eta}_t+g_{12}(t)\eta_t+g_{13}(t)\sigma_{13}(t)\Pi_t^2+g_{14}(t)\sigma_{23}(t)\Pi_t^2+\int_{\mathcal{E}_1}g_{15}(t,e)f_{12}(t,e)\nu_1(de)\Pi_t^2\\
&+\int_{\mathcal{E}_2}g_{16}(t,e)f_{22}(t,e)\nu_2(de)\Pi_t^2+g_{18}(t)=0,\quad t\in[0,T].
\end{aligned}
\end{equation}

Applying Theorem \ref{theA2} in the Appendix to \eqref{LQsystem21},\eqref{LQsystem22} and \eqref{hm}, we get the {\it forward-backward stochastic differential filtering equations} (FBSDFEs for short) of $(\hat{\bar{x}},\hat{\bar{y}},\hat{\bar{z}}^2,\hat{\bar{\tilde{z}}}^2,\hat{h})$:
\begin{equation}
\left\{
\begin{aligned}
 d\hat{\bar{x}}_t&=\bigg\{\bigg[b_{11}(t)-\sigma_3^{-1}(t)b_{22}(t)\sigma_{21}(t)-\int_{\mathcal{E}_2}(\lambda_{11}(t,e)-1)f_{21}(t,e)\nu_2(de)\bigg]\hat{\bar{x}}_t\\
                 &\qquad -\frac{1}{2}\big[b_{12}(t)-\sigma_3^{-1}(t)b_{22}(t)\sigma_{22}(t)\big]L(\Pi_t^1)\hat{h}_t+b_{13}(t)-\sigma_3^{-1}(t)b_{22}(t)\sigma_{23}(t)\\
                 &\qquad -\int_{\mathcal{E}_2}(\lambda_{11}(t,e)-1)f_{22}(t,e)\nu_2(de)\bigg\}dt\\
                 &\quad +\Big[\sigma_{21}(t)\hat{\bar{x}}_t-\frac{1}{2}\sigma_{22}(t)L(\Pi_t^1)\hat{h}_t+\sigma_{23}(t)\Big]dW^2_t\\
                 &\quad +\int_{\mathcal{E}_2}\big[f_{21}(t,e)\hat{\bar{x}}_{t-}+f_{22}(t,e)\big]\tilde{N}_2(de,dt),\\
       d\hat{h}_t&=g_{12}(t)\hat{h}_tdt+g_{14}(t)\hat{h}_tdW^2_t+\int_{\mathcal{E}_2}g_{16}(t,e)\hat{h}_{t-}\tilde{N}_2(de,dt),\\
-d\hat{\bar{y}}_t&=\bigg\{\Big[g_{11}(t)+g_{13}(t)\sigma_{11}(t)\Pi_t^2+\int_{\mathcal{E}_1}g_{15}(t,e)f_{11}(t,e)\nu_1(de)\Pi_t^2\Big]\hat{\bar{x}}_t+g_{12}(t)\hat{\bar{y}}_t\\
                 &\qquad +\Big[g_{13}(t)\sigma_{13}(t)+\int_{\mathcal{E}_1}g_{15}(t,e)f_{12}(t,e)\nu_1(de)\Big]\Pi_t^2+g_{14}(t)\hat{\bar{z}}^2_t\\
                 &\qquad +\int_{\mathcal{E}_2}g_{16}(t,e)\hat{\bar{\tilde{z}}}^2_{(t,e)}\nu_2(de)+\Big[\Big(g^2_{13}(t)+\int_{\mathcal{E}_1}g^2_{15}(t,e)\nu_1(de)\Big)\Pi_t^3\\
                 &\qquad -\frac{1}{2}g_{13}(t)\sigma_{12}(t)L(\Pi_t^1)\Pi_t^2-\frac{1}{2}g_{17}(t)L(\Pi_t^1)\Big]\hat{h}_t+g_{18}(t)\bigg\}dt\\
                 &\quad -\hat{\bar{z}}^2_tdW^2_t-\int_{\mathcal{E}_2}\hat{\bar{\tilde{z}}}^2_{(t,e)}\tilde{N}_2(de,dt),\quad t\in[0,T],\\
  \hat{\bar{x}}_0&=x_0,\ \hat{h}_0=2\hat{\bar{y}}_0,\ \hat{\bar{y}}_T=\phi_{11}\hat{\bar{x}}_T+\phi_{12}.
\end{aligned}
\right.
\end{equation}
Note that the innovation process $\bar{W}^2\equiv W^2$.
Moreover, $\Pi^2,\Pi^3$ and $\eta$ are the solutions to the following ODEs, respectively:
\begin{equation}\label{pi2}
\left\{
\begin{aligned}
&\dot{\Pi}_t^2+\bigg[b_{11}(t)-\sigma_3^{-1}(t)b_{22}(t)\sigma_{21}(t)-\int_{\mathcal{E}_2}(\lambda_{11}(t,e)-1)f_{21}(t,e)\nu_2(de)+g_{13}(t)\sigma_{11}(t)\\
&+g_{14}(t)\sigma_{21}(t)+\int_{\mathcal{E}_1}g_{15}(t,e)f_{11}(t,e)\nu_1(de)+\int_{\mathcal{E}_2}g_{16}(t,e)f_{21}(t,e)\nu_2(de)+g_{12}(t)\bigg]\Pi_t^2\\
&+g_{11}(t)=0,\quad t\in[0,T],\\
&\Pi_T^2=\phi_{11},
\end{aligned}
\right.
\end{equation}
\begin{equation}\label{pi3}
\left\{
\begin{aligned}
&\dot{\Pi}_t^3+\bigg[2g_{12}(t)+g^2_{13}(t)+g^2_{14}(t)+\int_{\mathcal{E}_1}g^2_{15}(t,e)\nu_1(de)+\int_{\mathcal{E}_2}g^2_{16}(t,e)\nu_2(de)\bigg]\Pi_t^3\\
&-\frac{1}{2}\bigg[b_{12}(t)-\sigma_3^{-1}(t)b_{22}(t)\sigma_{22}(t)+g_{13}(t)\sigma_{12}(t)+g_{14}(t)\sigma_{22}(t)\bigg]L(\Pi_t^1)\Pi_t^2\\
&-\frac{1}{2}g_{17}(t)L(\Pi_t^1)=0,\quad t\in[0,T],\\
&\Pi_T^3=0,
\end{aligned}
\right.
\end{equation}
\begin{equation}\label{eta}
\left\{
\begin{aligned}
&\dot{\eta}_t+g_{12}(t)\eta_t+\bigg[b_{13}(t)-\sigma_3^{-1}(t)b_{22}(t)\sigma_{23}(t)-\int_{\mathcal{E}_2}(\lambda_{11}(t,e)-1)f_{22}(t,e)\nu_2(de)\bigg]\Pi_t^2\\
&+g_{13}(t)\sigma_{13}(t)\Pi_t^2+g_{14}(t)\sigma_{23}(t)\Pi_t^2+\int_{\mathcal{E}_1}g_{15}(t,e)f_{12}(t,e)\nu_1(de)\Pi_t^2\\
&+\int_{\mathcal{E}_2}g_{16}(t,e)f_{22}(t,e)\nu_2(de)\Pi_t^2+g_{18}(t)=0,\quad t\in[0,T],\\
&\eta_T=\phi_{12}.
\end{aligned}
\right.
\end{equation}
From \eqref{pi1} and \eqref{pi2}, we find that $\Pi_t^1=\Pi_t^2$, for $t\in[0,T]$. Therefore, we have
\begin{equation}
\hat{\bar{y}}_0=\Pi_0^1x_0+\Pi_0^3\hat{h}_0+\eta_0=\Pi_0^1x_0+2\Pi_0^3\hat{\bar{y}}_0+\eta_0,
\end{equation}
i.e.,
\begin{equation}
\hat{\bar{y}}_0=(1-2\Pi_0^3)^{-1}(\Pi_0^1x_0+\eta_0),\quad \hat{h}_0=2(1-2\Pi_0^3)^{-1}(\Pi_0^1x_0+\eta_0).
\end{equation}
And it is not hard to get the explicit form of $\hat{h}_t$:
\begin{equation}
\begin{aligned}
\hat{h}_t&=\hat{h}_0\exp\bigg\{\int_0^t\bigg[g_{12}(s)-\frac{1}{2}g^2_{14}(s)+\int_{\mathcal{E}_2}(\ln(1+g_{16}(s,e))-g_{16}(s,e))\nu_2(de)\bigg]ds\\
         &\qquad\qquad +\int_0^tg_{14}(s)dW^2_s+\int_0^t\int_{\mathcal{E}_2}\ln(1+g_{16}(s,e))\tilde{N}_2(de,ds)\bigg\}.
\end{aligned}
\end{equation}
Then, the optimal control \eqref{baru2} in $U=(-\infty,-1]\cup[1,+\infty)$ has an explicit form:
\begin{equation}
\begin{aligned}
\bar{u}_t&=-L(\Pi_t^1)(1-2\Pi_0^3)^{-1}(\Pi_0^1x_0+\eta_0)\exp\bigg\{\int_0^t\bigg[g_{12}(s)-\frac{1}{2}g^2_{14}(s)+\int_{\mathcal{E}_2}(\ln(1+g_{16}(s,e))\\
         &\qquad -g_{16}(s,e))\nu_2(de)\bigg]ds+\int_0^tg_{14}(s)dW^2_s+\int_0^t\int_{\mathcal{E}_2}\ln(1+g_{16}(s,e))\tilde{N}_2(de,ds)\bigg\},\\
         &\hspace{7cm}\quad a.e.\ t\in[0,T],\ P\mbox{-}a.s.,
\end{aligned}
\end{equation}
where $\Pi^1,\Pi^3,\eta$ are the solutions to \eqref{pi1}, \eqref{pi3}, \eqref{eta}, respectively, and $\Pi_0^1,\Pi_0^3,\eta_0$ are the corresponding values at initial time $t=0$.

\section{Concluding remarks}

In this paper, we have derived a global maximum principle for progressive optimal control of partially observed forward-backward stochastic systems with random jumps. Compared to \cite{Hu17}, we extend it to the case with partial information and random jumps. Compared to \cite{STW20}, we extend it to the partially observed forward-backward stochastic system. Compared to \cite{LT95}, we extend it to the forward-backward system with random jumps. Compared to \cite{TH02}, we extend it to the forward-backward system with both Brownian motion and Poisson martingale measure. We should notice that, different from \cite{LT95} and \cite{TH02}, we extend the continuous observation equation into discontinuous case, in which Brownian motion and Poisson random measure are both considered. Meanwhile, by using the technique introduced by \cite{MY21arxiv}, and extending it into the case with jumps, then combining with the random field method, the $L^\beta$-solution to fully coupled FBSDEPs and its $L^\beta(\beta\geq2)$-estimate under some assumptions are given to prepare for the maximum principle, and the non-linear filtering equation for partially observed stochastic system with random jumps, Theorem \ref{theA2}, is derived for obtaining the optimal control's state estimate feedback for the LQ case. Also, different from the work in \cite{XW10}, we extend it to the general non-linear forward-backward system and observation equation with Poisson random measure, where the control domain is non-convex, instead of the LQ stochastic system with Poisson process. Moreover, we have given an explicit form of the optimal control by introducing some ODEs.

Global maximum principles for optimal control of fully coupled FBSDEPs and of mean-field's type FBSDEs are very challenging and interesting to research. We will work on these topics in the near future.

\section*{Appendix: Non-linear Filtering Equation of SDEPs}
\setcounter{equation}{0}
\renewcommand\theequation{A.\arabic{equation}}
\setcounter{theorem}{0}
\renewcommand\thetheorem{A.\arabic{theorem}}
\setcounter{lemma}{0}
\renewcommand\thelemma{A.\arabic{lemma}}

Consider the following signal-observation system:
\begin{equation}\label{so}
\left\{
\begin{aligned}
h_t&=h_0+\int_0^tH_sds+y_t,\\
Y_t&=Y_0+\int_0^tA_s(\omega)ds+\int_0^tB_s(Y)dW^2_s+\int_0^t\int_{\mathcal{E}_2}C_{s-}(Y,e)\tilde{N}_2(de,ds),
\end{aligned}
\right.
\end{equation}
where
$$y_t\equiv\sum_{i=1}^2\int_0^tb_{i,s}(\omega)dW^i_s+\sum_{i=1}^2\int_0^t\int_{\mathcal{E}_i}c_{i,s}(\omega,e)\tilde{N}_i(de,ds)$$
is a one-dimensional RCLL square integrable martingale with $y_0=0$ and $b_{1,s}(\omega)(b_{2,s}(\omega)):[0,T]\times\Omega\rightarrow \mathbb{R}$ is $\mathcal{G}/\mathcal{B}(\mathbb{R})$ measurable, $c_{1,s}(\omega,e)(c_{2,s}(\omega,e)):[0,T]\times\Omega\times\mathcal{E}_1(\mathcal{E}_2)\rightarrow \mathbb{R}$ is $\mathcal{G}\otimes\mathcal{B}(\mathcal{E}_1)(\mathcal{B}(\mathcal{E}_2))/\mathcal{B}(\mathbb{R})$ measurable.

On a given filtered probability space $(\Omega,\mathcal{F},(\mathcal{F}_t),P)$, $(W^1_t,W^2_t)$ is a two-dimensional standard $\mathcal{F}_t$-Brownian motion and $(\tilde{N}_1(de,dt),\tilde{N}_2(de,dt))$ is a standard Poisson martingale measure with $\tilde{N}_i(de,dt)=N_i(de,dt)-\nu_i(de)dt$, $\nu_i(A_i)<\infty, \forall A_i\in\mathcal{E}_i$, where $N_i(de,dt)$ is Poisson random measure with $EN_i(de,dt)=\nu_i(de)dt$, i=1,2.

We need the following assumption.

{\bf (H1)}\quad $A_t(\omega):[0,T]\times\Omega\rightarrow \mathbb{R}$ is $\mathcal{G}/\mathcal{B}(\mathcal{\mathbb{R}})$ measurable, $B_t(Y):=B(t,Y_r,r\leq t):[0,T]\times D([0,T];\mathbb{R})\rightarrow \mathbb{R}$ is $\mathcal{G}\otimes\mathcal{B}(D([0,T];\mathbb{R}))/\mathcal{B}(R)$-measurable and for each $t\geq0$, $B_t(Y)$ is $\mathcal{B}(D([0,T];\\\mathbb{R}))/\mathcal{B}(\mathbb{R})$-measurable. $C_t(Y,e):=C(t,Y_r,e,r\leq t):[0,T]\times D([0,T];\mathbb{R})\times\mathcal{E}_2\rightarrow \mathbb{R}$ is $\mathcal{G}\otimes\mathcal{B}(D([0,T];\mathbb{R}))\otimes\mathcal{B}(\mathcal{E}_2)/\mathcal{B}(\mathbb{R})$-measurable. We set $C_{t-}(Y,e):=C(t,Y_r,e,r<t)$, where $Y$ is an $\mathcal{F}_t$-adapted RCLL process, then $C_{t-}(Y,e)$ is $\mathcal{F}_t$-$\mathcal{E}_2$-predictable, where we define $D([0,T];\mathbb{R})$ as the totality of RCLL maps from $[0,T]$ to $\mathbb{R}$.

The following result belongs to Situ \cite{Situ05}.

\begin{theorem}\label{theA1}
Assume that $(h,Y)$ is an $\mathcal{F}_t$-adapted solution to \eqref{so} and for all $T>0$,
\begin{equation}
E\int_0^T|A_s(\omega)|ds<\infty,\ E\int_0^T|B_s(Y)|^2ds<\infty,\ E\int_0^T\int_{\mathcal{E}_2}|C_s(Y,e)|^2\nu_2(de)ds<\infty,
\end{equation}
and $B^{-1}$ exists and is bounded. Then

\noindent (1)\ $\bar{M}_t:=\int_0^tB^{-1}_{s-}(Y)\big[dY_s-\pi_s(A)ds\big]$ is a locally square integrable martingale on the probability space $(\Omega,\mathcal{F},(\mathcal{F}_t^Y),P)$ and the observation equation of \eqref{so} can be rewritten as
\begin{equation}
Y_t=Y_0+\int_0^t\pi_s(A)ds+\int_0^tB_{s-}(Y)d\bar{M}_s.
\end{equation}

\noindent (2)\ Furthermore, $\bar{M}_t=\bar{W}^2_t+\int_0^t\int_{\mathcal{E}_2}B_{s-}^{-1}(Y)C_{s-}(Y,e)\tilde{N}_2(de,ds)$, where $\pi_s(A):=E\big[A_s|\mathcal{F}_s^Y\big]$, $\bar{W}^2_t=\int_0^tB_s^{-1}(Y)(A_s-\pi_s(A))ds+W^2_t$ is a Brownian motion on $(\Omega,\mathcal{F},(\mathcal{F}_t^Y),P)$, and $\tilde{N}_2(de,dt)$ is an $\mathcal{F}_t^Y$-adapted standard Poisson martingale measure.

Usually, we call $\bar{W}^2_t$ and $\tilde{N}_2(de,ds)$ the innovation processes for $W^2_t$ and $\tilde{N}_2(de,ds)$. Therefore, the observation equation becomes
\begin{equation}
Y_t=Y_0+\int_0^t\pi_s(A)ds+\int_0^tB_s(Y)d\bar{W}^2_s+\int_0^t\int_{\mathcal{E}_2}C_{s-}(Y,e)\tilde{N}_2(de,ds).
\end{equation}
\end{theorem}

Before giving our main theorem, we need the following lemmas, similar to \cite{Situ05}.
\begin{lemma}\label{lemmaA1}
$\big\{E\big[y_t\big|\mathcal{F}_t^Y]\big\}_{t\geq0}$ is an $\mathcal{F}_t^Y$-adapted square integrable martingale.
\end{lemma}
\begin{lemma}\label{lemmaA2}
$\big\{E\big[\int_0^tH_sds\big|\mathcal{F}_t^Y\big]-\int_0^t\pi_s(H)ds\big\}_{t\geq0}$ is an $\mathcal{F}_t^Y$-adapted square integrable martingale.
\end{lemma}
\begin{lemma}\label{lemmaA3}
There exists an $f(s,\omega)\in M^2_{\mathcal{F}_t^Y}[0,t]$ and a $g(s,\omega,e)\in F^2_{\mathcal{F}_t^Y}[0,t]$ such that
\begin{equation}
\begin{aligned}
&E\big[h_0\big|\mathcal{F}_t^Y\big]+E\bigg[\int_0^tH_sds\bigg|\mathcal{F}_t^Y\bigg]+E\big[y_t\big|\mathcal{F}_t^Y\big]-\int_0^t\pi_s(H)ds\\
&=\pi_0(h)+\int_0^tf(s,\omega)d\bar{W}^2_s+\int_0^t\int_{\mathcal{E}_2}g(s,\omega,e)\tilde{N}_2(de,ds).
\end{aligned}
\end{equation}
Hence
\begin{equation}
\begin{aligned}
\pi_t(h)&=E\big[h_0\big|\mathcal{F}_t^Y\big]+E\bigg[\int_0^tH_sds\bigg|\mathcal{F}_t^Y\bigg]+E\big[y_t\big|\mathcal{F}_t^Y\big]\\
&=\pi_0(h)+\int_0^t\pi_s(H)ds+\int_0^tf(s,\omega)d\bar{W}^2_s+\int_0^t\int_{\mathcal{E}_2}g(s,\omega,e)\tilde{N}_2(de,ds),
\end{aligned}
\end{equation}
where $M^2_{\mathcal{F}_t^Y}[0,t]$ is the space of $\mathcal{F}_t^Y$-predictable stochastic processes satisfying $E\int_0^t|f(s,\omega)|^2ds\\<\infty$, and $F^2_{\mathcal{F}_t^Y}[0,t]$ is that of $\mathcal{F}_t^Y$-$\mathcal{E}_2$-predictable stochastic processes satisfying $E\int_0^t\int_{\mathcal{E}_2}|g(s,\omega,e)|^2\\\nu_2(de)ds<\infty$.
\end{lemma}

Next, we give the non-linear filtering equation of partially observed SDEP system as follows.

\begin{theorem}\label{theA2}
Assume that (1)\ $|B_t(x)|^2+||C_t(x,e)||^2\leq c_0(1+|x|^2),\forall x\in D([0,T];\mathbb{R})$ with some constant $c_0$, and $C^{-1}$ exists. For any $T<\infty$, $E\int_0^T|A_s(\omega)|^2ds<\infty$.\\
(2)\ $B^{-1}$ exists and is bounded. $|B_t(x_1)-B_t(x_2)|^2+||C_t(x_1,e)-C_t(x_2,e)||^2\leq k|x_1-x_2|^2$ for any $x_1,x_2\in D([0,T];\mathbb{R})$ and $t\geq0$.\\
(3)\ $E|h_0|^2<\infty$, $E|Y_0|^2<\infty$, $E\int_0^T|H_s|^2ds<\infty$, $E\big[\sup_{0\leq t\leq T}|h_t|^2\big]<\infty$, for all $T<\infty$.
Then the optimal filter $\pi_t(h):=E\big[h_t\big|\mathcal{F}_t^Y\big]$ will satisfy the following non-linear filtering equation:
\begin{equation}\label{filter}
\begin{aligned}
\pi_t(h)&=\pi_0(h)+\int_0^t\pi_s(H)ds+\int_0^t\big[\pi_s(b_2)+B_s^{-1}(Y)(\pi_s(Ah)-\pi_s(A)\pi_s(h))\big]d\bar{W}^2_s\\
        &\quad +\int_0^t\int_{\mathcal{E}_2}\pi_{s-}(c_2)\tilde{N}_2(de,ds),\quad P\mbox{-}a.s.
\end{aligned}
\end{equation}
\end{theorem}

\begin{proof}
Let $\tilde{y}_t:=\int_0^tf(s,\omega)d\bar{W}^2_s+\int_0^t\int_{\mathcal{E}_2}g(s,\omega,e)\tilde{N}_2(de,ds)$, and let
\begin{equation*}
z_t:=\int_0^t\lambda^1_s(Y)d\bar{W}^2_s\\+\int_0^t\int_{\mathcal{E}_2}\lambda^2_s(Y,e)\tilde{N}_2(de,ds),
\end{equation*}
where $\lambda^1_t(Y)$ is a scalar-valued bounded $\mathcal{F}_t^Y$-adapted stochastic process and $\lambda^2_s(Y,e)$ is any given one-dimensional $\mathcal{F}_t^Y$-$\mathcal{E}_2$-predictable stochastic process. Then $z$ is an $\mathcal{F}_t^Y$-predictable square integrable martingale.

Applying It\^{o}'s formula to $\tilde{y}_tz_t$, we have
\begin{equation}\label{tildeyz0}
E[\tilde{y}_tz_t]=E\int_0^t\lambda_s^1(Y)f(s,\omega)ds+E\int_0^t\int_{\mathcal{E}_2}\lambda_s^2(Y,e)g(s,\omega,e)\nu_2(de)ds.
\end{equation}
On the other hand, since $\tilde{y}_t=\pi_t(h)-\pi_0(h)-\int_0^t\pi_s(H)ds\in\mathcal{F}_t^Y$, then we have
\begin{equation}\label{tildeyz}
\begin{aligned}
E[\tilde{y}_tz_t]&=E[\pi_t(h)z_t]-E[\pi_0(h)z_t]-E\bigg[\int_0^t\pi_s(H)dsz_t\bigg]\\
&=E\bigg[h_tz_t-\int_0^tH_sz_sds\bigg].
\end{aligned}
\end{equation}
Since $\bar{W}^2_t=\int_0^tB_{s-}^{-1}(Y)(A_s-\pi_s(A))ds+W^2_t$, we get
\begin{equation}
\begin{aligned}
z_t=\tilde{z}_t+\int_0^t\lambda^1_s(Y)B_{s-}^{-1}(Y)(A_s-\pi_s(A))ds,
\end{aligned}
\end{equation}
where $\tilde{z}_t=\int_0^t\lambda^1_s(Y)dW^2_s+\int_0^t\int_{\mathcal{E}_2}\lambda^2_s(Y,e)\tilde{N}_2(de,ds)$. Then \eqref{tildeyz} becomes
\begin{equation}\label{tildeyz2}
\begin{aligned}
E[\tilde{y}_tz_t]&=E\bigg[h_t\tilde{z}_t-\int_0^tH_s\tilde{z}_sds\bigg]+E\bigg[h_t\int_0^t\lambda^1_s(Y)B^{-1}_{s-}(Y)(A_s-\pi_s(A))ds\\
&\qquad -\int_0^tH_s\int_0^s\lambda^1_u(Y)B^{-1}_{u-}(Y)(A_u-\pi_u(A))duds\bigg].
\end{aligned}
\end{equation}
Here, we should note that $\tilde{z}$ is an $\mathcal{F}_t$-adapted square integrable martingale w.r.t. $P$.
Note that
\begin{equation}
E[\tilde{z}_th_0]=E\big[E[\tilde{z}_th_0|\mathcal{F}_0]\big]=E[\tilde{z}_0h_0]=0,
\end{equation}
and
\begin{equation}
\begin{aligned}
E\bigg[\int_0^t\tilde{z}_sH_sds\bigg]&=E\bigg[\int_0^tE[\tilde{z}_t|\mathcal{F}_s]H_sds\bigg]=E\bigg[\int_0^tE[\tilde{z}_tH_s|\mathcal{F}_s]ds\bigg]\\
&=E\bigg[\int_0^t\tilde{z}_tH_sds\bigg]=E\bigg[\tilde{z}_t\int_0^tH_sds\bigg].
\end{aligned}
\end{equation}
Therefore, we have
\begin{equation}\label{part1}
\begin{aligned}
&E\bigg[h_t\tilde{z}_t-\int_0^tH_s\tilde{z}_sds\bigg]=E\bigg[\tilde{z}_t(h_t-h_0)-\tilde{z}_t\int_0^tH_sds\bigg]\\
&=E\bigg[\tilde{z}_t(h_t-h_0-\int_0^tH_sds)\bigg]=E[\tilde{z}_ty_t]\\
&=E\bigg[\tilde{z}_t\sum_{i=1}^2\bigg(\int_0^tb_{i,s}(\omega)dW^i_s+\int_0^t\int_{\mathcal{E}_i}c_{i,s}(\omega,e)\tilde{N}_i(de,ds)\bigg)\bigg]\\
&=E\bigg\langle\int_0^t\lambda^1_s(Y)dW^2_s,\int_0^tb_{2,s}(\omega)dW^2_s\bigg\rangle\\
&\qquad +E\bigg\langle\int_0^t\int_{\mathcal{E}_2}\lambda^2_s(Y,e)\tilde{N}_2(de,ds),\int_0^t\int_{\mathcal{E}_2}c_{2,s}(\omega,e)\tilde{N}_2(de,ds)\bigg\rangle\\
&=E\int_0^t\lambda^1_s(Y)b_{2,s}(\omega)ds+E\bigg[\int_0^t\int_{\mathcal{E}_2}\lambda^2_s(Y,e)\mathbb{E}_2\big[c_{2,s}(\omega,e)\big|\mathcal{P}\otimes\mathcal{B}(\mathcal{E}_2)\big]\nu_2(de)ds\bigg],
\end{aligned}
\end{equation}
and
\begin{equation*}
\begin{aligned}
&E\bigg[h_t\int_0^t\lambda^1_s(Y)B^{-1}_s(Y)(A_s-\pi_s(A))ds\bigg]\\
&=E\bigg[\int_0^t\lambda^1_s(Y)B^{-1}_s(Y)(A_s-\pi_s(A))(h_t-h_s)ds\bigg]\\
&\quad +E\bigg[\int_0^t\lambda^1_s(Y)B^{-1}_s(Y)(A_sh_s-\pi_s(A)h_s)ds\bigg]\\
\end{aligned}
\end{equation*}
\begin{equation}\label{hint}
\begin{aligned}
&=E\bigg[\int_0^t\lambda^1_s(Y)B^{-1}_s(Y)(A_s-\pi_s(A))(h_t-h_s)ds\bigg]\\
&\quad +E\bigg[\int_0^t\lambda^1_s(Y)B^{-1}_s(Y)(\pi_s(Ah)-\pi_s(A)\pi_s(h))ds\bigg].
\end{aligned}
\end{equation}
Notice that $h_t-h_s=\int_s^tH_udu+y_t-y_s$, then we have
\begin{equation*}
\begin{aligned}
&E\bigg[\int_0^t\lambda^1_s(Y)B^{-1}_s(Y)(A_s-\pi_s(A))(h_t-h_s)ds\bigg]\\
&=E\bigg[\int_0^t\lambda^1_s(Y)B^{-1}_s(Y)(A_s-\pi_s(A))(y_t-y_s)ds\bigg]\\
\end{aligned}
\end{equation*}
\begin{equation}\label{hint2}
\begin{aligned}
&\quad +E\bigg[\int_0^t\lambda^1_s(Y)B^{-1}_s(Y)(A_s-\pi_s(A))\int_s^tH_ududs\bigg]\\
&=E\bigg[\int_0^t\lambda^1_s(Y)B^{-1}_s(Y)(A_s-\pi_s(A))E\big[(y_t-y_s)|\mathcal{F}_s\big]ds\bigg]\\
&\quad +E\bigg[\int_0^t\int_0^u\lambda^1_s(Y)B^{-1}_s(Y)(A_s-\pi_s(A))dsH_udu\bigg]\\
&=E\bigg[\int_0^t\int_0^u\lambda^1_s(Y)B^{-1}_s(Y)(A_s-\pi_s(A))dsH_udu\bigg].
\end{aligned}
\end{equation}
Combining \eqref{hint} with \eqref{hint2}, we get
\begin{equation}\label{part2}
\begin{aligned}
&E\bigg[h_t\int_0^t\lambda^1_s(Y)B^{-1}_s(Y)(A_s-\pi_s(A))ds\bigg]-E\bigg[\int_0^t\int_0^u\lambda^1_s(Y)B^{-1}_s(Y)(A_s-\pi_s(A))dsH_udu\bigg]\\
&=E\bigg[\int_0^t\lambda^1_s(Y)B^{-1}_s(Y)(\pi_s(Ah)-\pi_s(A)\pi_s(h))ds\bigg].
\end{aligned}
\end{equation}
Substituting \eqref{part1} and \eqref{part2} into \eqref{tildeyz2}, we have
\begin{equation}
\begin{aligned}
E[\tilde{y}_tz_t]&=E\bigg[\int_0^t\lambda^1_s(Y)\bigg(b_{2,s}(\omega)+B^{-1}_s(Y)(\pi_s(Ah)-\pi_s(A)\pi_s(h))\bigg)ds\bigg]\\
&\quad +E\bigg[\int_0^t\int_{\mathcal{E}_2}\lambda^2_s(Y,e)\mathbb{E}_2[c_{2,s}(\omega,e)|\mathcal{P}\otimes\mathcal{B}(\mathcal{E}_2)]\nu_2(de)ds\bigg]\\
&=E\bigg[\int_0^t\lambda^1_s(Y)\bigg(\pi_s(b_2)+B^{-1}_s(Y)(\pi_s(Ah)-\pi_s(A)\pi_s(h))\bigg)ds\bigg]\\
&\quad +E\bigg[\int_0^t\int_{\mathcal{E}_2}\lambda^2_s(Y,e)\mathbb{E}_2[\pi_s(c_2)|\mathcal{P}\otimes\mathcal{B}(\mathcal{E}_2)]\nu_2(de)ds\bigg].
\end{aligned}
\end{equation}
Compared this to \eqref{tildeyz0}, we get
\begin{equation}
\begin{aligned}
f(s,Y)&=\pi_s(b_2)+B^{-1}_s(Y)\big(\pi_s(Ah)-\pi_s(A)\pi_s(h)\big),\\
g(s,Y,e)&=\mathbb{E}_2\big[\pi_{s-}(c_2)\big|\mathcal{P}\otimes\mathcal{B}(\mathcal{E}_2)\big]=\pi_{s-}(c_2).
\end{aligned}
\end{equation}
By Lemma \ref{lemmaA3}, we can derive \eqref{filter}. The proof is complete.
\end{proof}

\section*{Acknowledgement}

Many thanks for discussion and suggestions with Prof. Jie Xiong at Southern University of Science and Technology, Postdoc. Yuanzhuo Song at Chinese Academy of Sciences and Dr. Weijun Meng at Shandong University.

\end{document}